\documentclass[10pt,reqno]{amsart}

\usepackage[colorlinks=true,linkcolor=blue,citecolor=blue,linktocpage=true,urlcolor=black]{hyperref}

\usepackage{amsmath}
\usepackage{amssymb}
\usepackage{amsthm}
\usepackage{comment}

\usepackage{mathabx,mathrsfs}

\usepackage{url}
\usepackage{setspace}   
\usepackage{enumitem}
\usepackage{bbm}
\usepackage{times}
\usepackage{tabularx}
\numberwithin{equation}{section}

\usepackage{caption}  
\usepackage{makecell}
\usepackage{float}
\restylefloat{table}

\newcommand{\interval}[1]{\widehat{#1}}
\usepackage{color}
\usepackage[dvipsnames]{xcolor}

\def\leb{\mathrm{Leb}}
\def \Aut{\mathrm{Aut}}
\def\calA{\mathcal A}
\def\diag{\mathrm{diag}}
	
\def\what{\widehat}
\def\rank{\mathrm{rank}}

\def\Im{\mathrm{Im}}
\renewcommand{\emph}[1]{{\bf #1}}

\DeclareMathOperator{\cov}{cov}
\DeclareMathOperator{\card}{card}
\DeclareMathOperator{\im}{Im}
\DeclareMathOperator{\Hol}{Hol}

\usepackage{mathtools}
\usepackage{tikz}
\usetikzlibrary{arrows,decorations.markings}
\usetikzlibrary{matrix}
\usetikzlibrary{graphs}
\usetikzlibrary{backgrounds}

    \usepackage[abbrev,nobysame]{amsrefs}
\renewcommand{\MR}[1]{}

\usepackage[toc, title]{appendix}
 \def\Folner{F{\o}lner }

\usepackage{aliascnt}
\usepackage[capitalize] {cleveref}
\theoremstyle{theorem}
\newtheorem{theorem}{Theorem} [section]

\newaliascnt{proposition}{theorem}
\newtheorem{proposition}[proposition]{Proposition}
\aliascntresetthe{proposition}
\newaliascnt{claim}{theorem}
\newtheorem{claim}[claim]{Claim}
\aliascntresetthe{claim}
\newaliascnt{lemma}{theorem}
\newtheorem{lemma}[lemma]{Lemma}
\aliascntresetthe{lemma}
\newaliascnt{corollary}{theorem}
\newtheorem{corollary}[corollary]{Corollary}
\aliascntresetthe{corollary}
 
\theoremstyle{definition}
\newaliascnt{definition}{theorem}
\newtheorem{definition}[definition]{Definition}
\aliascntresetthe{definition}
\theoremstyle{remark}

\theoremstyle{definition}

\newaliascnt{remark}{theorem}
\newtheorem{remark}[remark]{Remark}
\aliascntresetthe{remark}

\newtheorem*{theorem*}{Theorem}
\newtheorem{principle}{Theme}
\theoremstyle{definition}
\newaliascnt{hypot}{theorem}
\newtheorem{hypot}[hypot]{Hypothesis}
\aliascntresetthe{hypot}
\newaliascnt{question}{theorem}
\newtheorem{question}[question]{Question}
\aliascntresetthe{question}
\newaliascnt{example}{theorem}
\newtheorem{example}[example]{Example}
\aliascntresetthe{example}
\crefformat{section}{\S#2#1#3}
\theoremstyle{remark}

\usepackage{marginnote}

\newcommand{\Aff}{\mathrm{Aff}}
\newcommand{\diff}{\mathrm{Diff}}
\newcommand{\Diff}{\diff}
\newcommand{\Homeo}{\mathrm{Homeo}}
\newcommand{\supp}{\mathrm{supp}}
\newcommand{\td}{\tilde}
\newcommand{\wtd}{\widetilde}

\newcommand{\id}{\mathrm{Id}}
\def\Id{\id}
\newcommand{\restrict}[2]{{#1}{\restriction_{{ #2}}}}
\newcommand{\Fol}{\fol} 

\DeclareMathOperator{\diam}{diam}
\DeclareMathOperator{\Ad}{Ad}
\newcommand{\sm}{\smallsetminus}

\newcommand{\R}{\mathbb {R}}
\newcommand{\Q}{\mathbb {Q}}
\newcommand{\Z}{\mathbb {Z}}
\newcommand{\N}{\mathbb {N}}
\newcommand{\T}{\mathbb {T}}

\newcommand{\inv}{^{-1}}

\newcommand{\Sl}{\mathrm{SL}}
\def\SL{\Sl}
\newcommand{\Gl}{\mathrm{GL}}
\def\GL{\Gl}

\newcommand{\SO}{\mathrm{SO}}
\def\So{\SO}

\newcommand{\so}{\mathfrak{so}}
\renewcommand{\sl}{\mathfrak{sl}}
\renewcommand{\sp}{\mathfrak{sp}}
\newcommand{\Sp}{\mathrm{Sp}}

\def\calE{\mathcal E}

\def\fol{\mathscr F}

\def\loc{\mathrm{loc}}

 \newcommand{\oldepsilon}{\mathchar"10F}
\newcommand{\eps}{\oldepsilon}

\def\bfF{\mathbf F}
\newcommand{\bfG}{\mathbf{G}}
\newcommand{\bfS}{\mathbf{S}}

\DeclareMathOperator{\1}{\mathbf 1}

\newcommand{\lieg}{\mathfrak g}
\newcommand{\lieh}{\mathfrak h}

\newcommand{\liea}{\mathfrak a}

\newcommand{\lieq}{\mathfrak q}

\def\calB{\mathcal B}

\renewcommand{\bar}{\overline}

\renewcommand{\top}{\mathrm{top}}

\DeclareFontFamily{U}{wncy}{}
\DeclareFontShape{U}{wncy}{m}{n}{<->wncyr10}{}
\DeclareSymbolFont{mcy}{U}{wncy}{m}{n}
\DeclareMathSymbol{\Sh}{\mathord}{mcy}{"58}

\newcommand{\Lie}{{\rm Lie}}

\renewcommand{\phi}{\varphi}

\newcommand{\Ucal}{ {\mathcal U}}

\newcommand{\Ecal}{ {\mathcal E}}
\newcommand{\Acal}{ {\mathcal A}}
\newcommand{\Ccal}{ {\mathcal C}}

\newcommand{\calQ}{{\mathcal Q}}
\newcommand{\Tcal}{ {\mathcal T}}
\newcommand{\Lcal}{ {\mathcal L}}

\newcommand{\Wcal}{ {\mathcal W}}

\newcommand{\Qbb}{\mathbb{Q}}
\newcommand{\Tbb}{{\mathbb T}}
\newcommand{\Rbb}{{\mathbb R}}
\newcommand{\Zbb}{{\mathbb Z}}

\newcommand{\Tbf}{{\mathbf T}}
\newcommand{\Dbf}{{\mathbf D}}
\newcommand{\Sbf}{{\mathbf S}}
\newcommand{\Hbf}{{\mathbf H}}

\newcommand{\hlie}{ {\mathfrak h}}

\newcommand{\glie}{ {\mathfrak g}}

\newcommand{\qlie}{ {\mathfrak q}}

\newcommand{\alie}{ {\mathfrak a}}

\newcommand{\Bsc}{\mathscr {B} }

\newcommand{\Fsc}{\mathscr {F} }

\def\calE{\mathcal E}

\def\scrF{\mathscr F}

\def\loc{\mathrm{loc}}

\def\bfF{\mathbf F}

\title{Positive entropy actions by higher-rank lattices}

\author[A.~Brown]{Aaron Brown}
\author[H.~Lee]{Homin Lee}
\address{Northwestern University, Evanston, IL 60208, USA}
\email{awb@northwestern.edu}
\address{Northwestern University, Evanston, IL 60208, USA}
\email{homin.lee@northwestern.edu}

\long\def\symbolfootnote[#1]#2{\begingroup\def\thefootnote{\fnsymbol{footnote}}
\footnote[#1]{#2}\endgroup}

\usepackage{bbm}

\def\bfQ{\mathbf Q}

\def\scrA{\mathscr A}
\def\calP{\mathcal P}

\def\scrM{\mathscr M}

\begin{document}




\def\DynkinNodeSize{1.8mm}
\def\DynkinArrowLength{1.7mm}
\tikzset{
  node distance={2cm},
  dnode/.style={
    circle,
    inner sep=0pt,
    minimum size=\DynkinNodeSize,
    fill=white,
    draw},
  middlearrow/.style={
    decoration={markings,
      mark=at position 0.6 with
      {\draw (0:0mm) -- +(+135:\DynkinArrowLength); \draw (0:0mm) -- +(-135:\DynkinArrowLength); \draw (0:0mm) -- +(+315:.2pt); },
    },
    postaction={decorate}
  },
  leftrightarrow/.style={
    decoration={markings,
      mark=at position 0.999 with
      {
      \draw (0:0mm) -- +(+135:\DynkinArrowLength); \draw (0:0mm) -- +(-135:\DynkinArrowLength);
      },
      mark=at position 0.001 with
      {
      \draw (0:0mm) -- +(+45:\DynkinArrowLength); \draw (0:0mm) -- +(-45:\DynkinArrowLength);
      },
    },
    postaction={decorate}
  },
  sedge/.style={
  },
  dedge/.style={
    middlearrow,
    double distance=0.5mm,
  },
  dedgeplain/.style={
    double distance=0.5mm,
  },
  tedge/.style={
    middlearrow,
    double distance=1.0mm+\pgflinewidth,
    postaction={draw}, 
  },
  infedge/.style={
    leftrightarrow,
    double distance=0.5mm,
  },
}




\def\ADYNK{
\begin{tikzpicture}[scale=\scales]
\useasboundingbox (-.25,-.7) rectangle (4,.5);
    \node[dnode,label=below:$\alpha_1$] (1) at (0,0) {};
    \node[dnode,label=below:$\alpha_2$] (2) at (1,0) {};
    \node[dnode,label=below:$\alpha_{\ell-1}$] (3) at (2.5,0) {};
    \node[dnode,label=below:$\alpha_\ell$] (4) at (3.5,0) {};
    \path (1) edge[sedge] (2)
          (2) edge[sedge,dashed, dash phase=1.4pt] (3)
          (3) edge[sedge] (4);
 \end{tikzpicture}
}
\def\BDYNK{
 \begin{tikzpicture}[scale=\scales]
\useasboundingbox (-.25,-.7) rectangle (4,.5);
\path           (2.45,0) edge[dedge] (3.55,0);
    \node[dnode,label=below:$\alpha_1$] (1) at (0,0) {};
    \node[dnode,label=below:$\alpha_2$] (2) at (1,0) {};
    \node[dnode,label=below:$\alpha_{\ell-1}$] (3) at (2.5,0) {};
    \node[dnode,label=below:$\alpha_\ell$] (4) at (3.5,0) {};

    \path (1) edge[sedge] (2)
          (2) edge[sedge,dashed, dash phase=1.4pt] (3)
          ;
\end{tikzpicture}
}

\def\CDYNK{
\begin{tikzpicture}[scale=\scales]
\useasboundingbox (-.25,-.7) rectangle (4,.5);
\path        (3.55,0) edge[dedge]    (2.45,0);

    \node[dnode,label=below:$\alpha_1$] (1) at (0,0) {};
    \node[dnode,label=below:$\alpha_2$] (2) at (1,0) {};
    \node[dnode,label=below:$\alpha_{\ell-1}$] (3) at (2.5,0) {};
    \node[dnode,label=below:$\alpha_\ell$] (4) at (3.5,0) {};

    \path (1) edge[sedge] (2)
          (2) edge[sedge,dashed, dash phase=1.4pt] (3)
          ;
\end{tikzpicture}}

\def\BCDYNK{
\begin{tikzpicture}[scale=\scales]
\useasboundingbox (-.25,-.7) rectangle (4,.5);
\path       (2.45,0) edge[dedge] (3.55,0); 
    \node[dnode,label=below:$\alpha_1$] (1) at (0,0) {};
    \node[dnode,label=below:$\alpha_2$] (2) at (1,0) {};
    \node[dnode,label=below:$\alpha_{\ell-1}$] (3) at (2.5,0) {};
    \node[dnode,label=below:$\alpha_\ell$] (4) at (3.5,0) {};

    \path (1) edge[sedge] (2)
          (2) edge[sedge,dashed, dash phase=1.4pt] (3)
          ;
\end{tikzpicture}
}
\def\DDYNKside{
\begin{tikzpicture}[scale=\scales]

\useasboundingbox (-.25,-.75) rectangle (4,1);

    \node[dnode,label=below:$\alpha_1$] (1) at (0,0) {};
    \node[dnode,label=below:$\alpha_2$] (2) at (1,0) {};
    \node[dnode,label=below:$\alpha_{\ell-2}$] (4) at (2.5,0) {};
    \node[dnode,label=right:$\alpha_{\ell-1}$] (5) at (3.5,0.5) {};
    \node[dnode,label=right:$\alpha_\ell$] (6) at (3.5,-0.5) {};

    \path (1) edge[sedge] (2)
          (2) edge[sedge,dashed, dash phase=1.4pt] (4)
          (4) edge[sedge] (5)
              edge[sedge] (6);
\end{tikzpicture}
}
\def\DDYNKup{
\begin{tikzpicture}[scale=\scales]

\useasboundingbox (-.25,-1.0) rectangle (4,1.25);

    \node[dnode,label=below:$\alpha_1$] (1) at (0,0) {};
    \node[dnode,label=below:$\alpha_2$] (2) at (1,0) {};
    \node[dnode,label=below:$\alpha_{\ell-2}$] (4) at (2.5,0) {};
    \node[dnode,label=above:$\alpha_{\ell-1}$] (5) at (3.5,0.5) {};
    \node[dnode,label=below:$\alpha_\ell$] (6) at (3.5,-0.5) {};

    \path (1) edge[sedge] (2)
          (2) edge[sedge,dashed, dash phase=1.4pt] (4)
          (4) edge[sedge] (5)
              edge[sedge] (6);
\end{tikzpicture}
}

\def\EDYNK{
 \begin{tikzpicture}[scale=\scales] 
\useasboundingbox (-.25,-.7) rectangle (4,1.5);
    \node[dnode,label=below:$\alpha_1$] (1) at (0,0) {};
    \node[dnode,label=below:$\alpha_3$] (3) at (1.5,0) {};
    \node[dnode,label=right:$\alpha_6$] (4) at (1.5,1) {};
    \node[dnode,label=below:$\alpha_4$] (5) at (2.5,0) {};
    \node[dnode,label=below:$\alpha_5$] (6) at (3.5,0) {};

    \path (1) edge[sedge,dashed, dash phase=1.4pt] (3)
          (3) edge[sedge] (4)
          (3) edge[sedge] (5)
          (5) edge[sedge] (6);
\end{tikzpicture}
}
\def\EEDYNK{
 \begin{tikzpicture}[scale=\scales] 
\useasboundingbox (-.25,-.7) rectangle (4,1.5);

    \node[dnode,label=below:$\alpha_1$] (1) at (0,0) {};
    \node[dnode,label=below:$\alpha_4$] (3) at (1.5,0) {};
    \node[dnode,label=right:$\alpha_7$] (4) at (1.5,1) {};
    \node[dnode,label=below:$\alpha_5$] (5) at (2.5,0) {};
    \node[dnode,label=below:$\alpha_6$] (6) at (3.5,0) {};

    \path (1) edge[sedge,dashed, dash phase=1.4pt] (3)
          (3) edge[sedge] (4)
          (3) edge[sedge] (5)
          (5) edge[sedge] (6);
\end{tikzpicture}}

\def\EEEDYNK{
\begin{tikzpicture}[scale=\scales] 
\useasboundingbox (-.25,-.7) rectangle (4,1.5);

    \node[dnode,label=below:$\alpha_1$] (1) at (0,0) {};
    \node[dnode,label=below:$\alpha_5$] (3) at (1.5,0) {};
    \node[dnode,label=right:$\alpha_8$] (4) at (1.5,1) {};
    \node[dnode,label=below:$\alpha_6$] (5) at (2.5,0) {};
    \node[dnode,label=below:$\alpha_7$] (6) at (3.5,0) {};

    \path (1) edge[sedge,dashed, dash phase=1.4pt] (3)
          (3) edge[sedge] (4)
          (3) edge[sedge] (5)
          (5) edge[sedge] (6);
\end{tikzpicture}
}
\def\FDYNK{
\begin{tikzpicture}[scale=\scales]
\useasboundingbox (-.25,-.7) rectangle (4,.5);
\path (.95,0) edge[dedge] (2.05,0);
    \node[dnode,label=below:$\alpha_1$] (1) at (0,0) {};
    \node[dnode,label=below:$\alpha_2$] (2) at (1,0) {};
    \node[dnode,label=below:$\alpha_3$] (3) at (2,0) {};
    \node[dnode,label=below:$\alpha_4$] (4) at (3,0) {};

    \path (1) edge[sedge] (2)
          (3) edge[sedge] (4)
          ;
\end{tikzpicture}
}
\def\GDYNK{
          \begin{tikzpicture}[scale=\scales]
\useasboundingbox (-.25,-.7) rectangle (4,.5);
  \path (-.05,0) edge[tedge] (1.05,0);
    \node[dnode,label=below:$\alpha_1$] (1) at (0,0) {};
    \node[dnode,label=below:$\alpha_2$] (2) at (1,0) {};

\end{tikzpicture}
}

 
\def\DynkTable{
\begin{table}[h]
\footnotesize
\def\scales{.7}
\caption{Roots systems, highest and 2nd highest roots, and resonant codimension of maximal parabolic subalgebras}
\label{table:bull}
\begin{center}
\begin{tabular}{|c|m{83pt} |l|}
\hline
& Dynkin diagram and
& Highest root  $\delta$ and second-highest root  $\delta'$; 
\\
&   simple roots 
&   resonant codimension $\bar r(\lieq_{j})$ where $\lieq_j = \lieq_{\Pi\sm \{ \alpha_j\}}$
\\    \hline

$A_\ell$&
\ADYNK
&
\begin{tabular}{l  }

$ \delta = \alpha_1 + \dots + \alpha _\ell$  \Ts  \Ms
 \\
 $\bar r(\lieq_{j})= \frac1 2 
\big(
(\ell+1)^2 - j^2 - (\ell+1-j)^2 
\big)$\Bs 
 \end{tabular}
 \\		\hline
$B_\ell$&
\BDYNK
&
 \begin{tabular}{l}
$\delta =  \alpha _1 + 2 \alpha _2 + \dots + 2 \alpha _\ell  $\Ts  \Ms
\\
$\bar r(\lieq_{j})= \frac1 2 
\big(
\ell (2\ell+1)  - j^2 - (\ell-j) (2(\ell-j)+1)
\big)$\Bs
\\
\end{tabular}
\\      \hline
$C_\ell$&
\CDYNK
&
\begin{tabular}{l l }
$\delta  = 2\alpha_1 + 2\alpha_2 + \dots + 2 \alpha_{\ell-1} + \alpha _ \ell   $\Ts   \\
$\delta'  = \alpha_1 +  2\alpha_2 + \dots + 2 \alpha_{\ell-1} + \alpha _ \ell  $\Ms
\\
$\bar r(\lieq_{j})= \frac1 2 
\big(
\ell (2\ell+1)  - j^2 - (\ell-j) (2(\ell-j)+1)
\big)$\Bs
\\
\end{tabular}
\\		\hline
$BC_\ell$&
\BCDYNK
&
\begin{tabular}{l l }
$ \delta= 2 \alpha_1 + 2\alpha _{2 } + \dots + 2\alpha _{\ell -1 }  +2 \alpha _\ell $ \Ts\\
$ \delta' = \alpha_1 + 2\alpha _{2 }  + \dots + 2\alpha _{\ell -1 } + 2\alpha _\ell   $\Bs
\end{tabular}
\\ 		\hline
$D_\ell$&
\DDYNKup
&
\begin{tabular}{  l  }
$ \delta = \alpha_1 + 2\alpha _2 +\dots + 2\alpha _{\ell -2 } +   \alpha _{\ell-1}+  \alpha _\ell  $ \Ts \Ms\\
$\bar r(\lieq_{j})= \frac1 2 
\big(
 \ell (2\ell-1)  - j^2 - (\ell-j) (2(\ell-j)-1) 
\big)$
\\
\hfill for $\phantom{\ell}1\le j\le \ell-2$  \\
$\bar r(\lieq_{j})= \frac1 2 
\big(
 \ell (2\ell-1) - \ell^2
\big)$
\hfill   \Bs
 for  $\phantom{1}\ell-1\le j\le \ell$
\end{tabular}
\\		 \hline

$E_6$&
\EDYNK
&
  \begin{tabular}{l}
$ \delta =  \alpha _1 + 2\alpha _2 + 3 \alpha _3 + 2\alpha _4 +  \alpha _5 + 2  \alpha _6$\Ts \Ms \\
$\bar r(\lieq_{1})= 16 \quad
\bar r(\lieq_{2})= 25 \quad 
\bar r(\lieq_{3})= 29 \quad  $\\$
\bar r(\lieq_{4})= 26\quad
\bar r(\lieq_{5})= 16\quad 
\bar r(\lieq_{6})= 21$\Bs
\end{tabular}
  \\ 		\hline
$E_7$&
\EEDYNK
&
 \begin{tabular}{l}
$\delta = \alpha_1+2\alpha_2+3\alpha_3+4\alpha_4+3\alpha_5+2\alpha_6+ 2\alpha_7$\Ms \Ts \\
$\bar r(\lieq_{1})= 27\quad 
\bar r(\lieq_{2})= 42 \quad
\bar r(\lieq_{3})= 50\quad  $\\$
\bar r(\lieq_{4})= 53\quad 
\bar r(\lieq_{5})= 47\quad 
\bar r(\lieq_{6})= 33\quad  $\\$
\bar r(\lieq_{7})= 42$\Bs
 \end{tabular}
\\ 			
\hline
$E_8$&
\EEEDYNK
&
 \begin{tabular}{l}
 $ \delta = 2 \alpha _1 + 3\alpha _2 + 4 \alpha _3 + 5\alpha _4 + 6\alpha _5 +   $\Ts \\
\quad \quad  \quad\quad $ 4\alpha _6 + 2\alpha _7+3\alpha_8$\\
 $ \delta' =  \alpha _1 + 3\alpha _2 + 4 \alpha _3 + 5\alpha _4 + 6\alpha _5 + $ \\
\quad \quad  \quad\quad $   4\alpha _6 + 2\alpha _7+3\alpha_8$\Ms 
\\
$\bar r(\lieq_{1})= 57\phantom{0}\quad 
\bar r(\lieq_{2})= 83\phantom{0} \quad
\bar r(\lieq_{3})= 97\phantom{0}\quad  $\\$
\bar r(\lieq_{4})= 105\quad 
\bar r(\lieq_{5})= 106\quad 
\bar r(\lieq_{6})= 98\phantom{0}\quad  $\\$
\bar r(\lieq_{7})= 78\phantom{0}\quad
\bar r(\lieq_{8})= 92$\Bs
 \end{tabular}
\\ 			\hline
$F_4$&
\FDYNK  &
\begin{tabular}{l  }
$ \delta =  2\alpha _1 + 3\alpha _2 + 4 \alpha _3 + 2\alpha _4 $\Ts \\
$ \delta' =  \alpha _1 + 3\alpha _2 + 4 \alpha _3 + 2\alpha _4 $\Ms
\\
$\bar r(\lieq_{1})= 15 \quad
\bar r(\lieq_{2})= 20 \quad 
\bar r(\lieq_{3})= 20 \quad  
\bar r(\lieq_{4})= 15$\Bs
 \end{tabular}
\\ 			\hline
$G_2$&
\GDYNK  &
\begin{tabular}{l   }
$ \delta = 2 \alpha _1 + 3\alpha _2 $\Ts\Ms \quad \quad \quad
$ \delta'  =  \alpha _1 + 3\alpha _2 $
\\
$\bar r(\lieq_{1})= 5 \phantom{0}\quad
\bar r(\lieq_{2})= 5\phantom{0}$\Bs
\end{tabular}
\\ 			\hline
\end{tabular}
\end{center}
\label{default}
\end{table}
}

\def\BDTABEL{

\begin{table}[h]
\footnotesize
\def\scales{.7}
\def\ssizl{14em}
\caption{Simple roots and positive roots for the root systems $B_\ell$ and $D_\ell$ relative to  the basis $\{\eps_i\}$ of $(\R^\ell)^*$ dual to the standard basis $\{e_i\}$ of  $\R^\ell$.  (c.f. \cite[Appendix C]{MR1920389})}\label{table:1}\label{table:poop}
\begin{center}
\begin{tabular}{|c|m{90pt} |l|}
\hline
& \begin{tabular} {c} Dynkin diagram and  \\simple roots  in\\  standard presentation  \end{tabular} 
&   \begin{tabular} {c}Positive roots \end{tabular} 

\\
\hline
$B_\ell$&
 \begin{tabular}{l}
\BDYNK
\Ms\\  $\alpha_i = \eps_i - \eps_{i+1},$    \\
         {\footnotesize $  1\le i \le \ell-1$};\\     $ \alpha_\ell = \eps_\ell$\Bs
 \end{tabular}
&
\begin{tabular}{m{\ssizl} l  }
 $\alpha_i +  \dots + \alpha _{k} = \eps_i - \eps_ {k+1} $ &  {\footnotesize $1\le i \le k \le \ell-1$}\Ts \\
$\alpha_i +  \dots + \alpha _{\ell} = \eps_i   $&  {\footnotesize$1\le i \le \ell $}\\
$\alpha_i +\dots  +\alpha _{k} +    2 \alpha _{k+1} + \dots  $&  {\footnotesize $1\le i \le k < \ell$}\\
\hfill $   + 2 \alpha _\ell   = \eps_i + \eps_ {k+1}$ \Bs
 \end{tabular}
\\

\hline

 \begin{tabular}{c}
$D_\ell$\\
{\footnotesize$\ell\ge 4$}
\end{tabular} &
 \begin{tabular}{l}
\DDYNKup \Ms\\
           $\alpha_i = \eps_i - \eps_{i+1},$  
    \\
      {\footnotesize $  1\le i \le \ell-1$};\\
        $ \alpha_\ell = \eps_{\ell-1} + \eps_\ell$ \Bs
 \end{tabular}
&
%
%
%
\begin{tabular}{m{\ssizl} l  }
\Ts $\alpha_i$ &$1\le i\le \ell$\\
$\alpha_i  + \dots + \alpha _{k}= \eps_i - \eps_ {k+1}$ & {\footnotesize $1\le i < k \le \ell-2$}\\
$\alpha_i +   \dots + \alpha _{\ell-2}+ \alpha _{\ell-1} = \eps_i - \eps_ {\ell}$ & {\footnotesize $1\le i  \le \ell-2$}\\
$\alpha_i +   \dots + \alpha _{\ell-2}+ \alpha _\ell = \eps_i + \eps_ {\ell}$ & {\footnotesize $1\le i  \le \ell-2$}\\
$\alpha_i +   \dots + \alpha _{\ell-1}+ \alpha _\ell = \eps_i + \eps_ {\ell-1}$ & {\footnotesize $1\le i  \le \ell-2$}\\
$\alpha_i +\dots  +\alpha _{k}  +    2 \alpha _{k+1} +\dots$  & {\footnotesize  $1\le i \le  k <\ell-2$}\\
\hfill$  
   + 2 \alpha _{\ell-2}+ \alpha _{\ell-1} + \alpha _\ell $\\
\hfill $ = \eps_i + \eps_ {k+1}$\Bs
\end{tabular}
\\
 \hline

\end{tabular}
\end{center}

\end{table}%
}

\maketitle

\let\oldtocsection=\tocsection
\let\oldtocsubsection=\tocsubsection 
\let\oldtocsubsubsection=\tocsubsubsection
 
\renewcommand{\tocsection}[2]{\hspace{0em}\oldtocsection{#1}{#2}}
\renewcommand{\tocsubsection}[2]{\hspace{1em}\oldtocsubsection{#1}{#2}}
\renewcommand{\tocsubsubsection}[2]{\hspace{2em}\oldtocsubsubsection{#1}{#2}}
\setcounter{tocdepth}{2}

\begin{abstract}
We study smooth actions by lattices   in higher-rank simple Lie groups.  Assuming one element of the action acts with positive topological entropy, we prove a number of new rigidity results.  
For lattices in $\mathrm{SL}(n,\Rbb)$ acting on $n$-manifolds, if the action has  positive topological entropy we show the lattice must be commensurable with $\mathrm{SL}(n,\Zbb)$.  Moreover, such actions admit an absolutely continuous probability measure with positive metric entropy; adapting arguments by Katok and Rodriguez Hertz, we show such actions are measurably conjugate to affine actions on  (infra-)tori.  

In a main technical argument, we study families of probability measures invariant under sub-actions of the induced action by the ambient Lie group on an associated fiber bundle.  
To control entropy properties of such measures when passing to limits, in the appendix we establish certain upper semicontinuity of fiber entropy under weak-$*$ convergence, adapting classical results of Yomdin and Newhouse.  

\end{abstract}

\tableofcontents
\section{Introduction}\label{sec:intro}

It is well known that (irreducible) lattice subgroups $\Gamma$ in  higher-rank (semi-)simple Lie groups $G$ exhibit very strong rigidity properties with respect to linear representations $\pi\colon \Gamma\to \Gl(d,\R)$.  
The \emph{Zimmer program} is a collection of questions and conjectures that, roughly, aim to establish analogous rigidity results for smooth (perhaps volume-preserving) actions.  Very recently, substantial progress in the Zimmer program has been made, especially in the following directions: 
\begin{enumerate}
\item ({\it Isometric and trivial actions.}) Establishing {\it Zimmer's conjecture}: actions by lattices in $\Sl(n,\R)$ for $n\ge 3$
(and in many other higher-rank simple Lie groups) on low-dimensional manifolds are isometric or finite; see especially 
the recent works \cite{MR4502593,Brown:2020aa,Brown:2021aa,BDZ,2008.10687,ABZ}.
\item ({\it Hyperbolic actions.}) Showing Anosov actions by higher-rank lattices on tori and nilmanifolds are smoothly conjugate to affine actions; see especially \cite{HL:MR4768583,HL:Dom,BRHW17}, building on many earlier works including \cite{MQ01,KLZ,KL96,MR1643954}.
\item ({\it Projective actions.}) Showing global and local rigidity of projective actions: Actions by  lattices  in $\Sl(n,\R)$ on $(n-1)$-dimensional manifolds (for $n\ge 3$) are either finite or are smoothly conjugate to standard projective actions on $S^{n-1}$ or $\R P^{n-1}$; see \cite{BRHWbdy}. 
 Moreover, the standard projective actions (on grassmannians, flag varieties, generalized flag varieties, etc.) by higher-rank lattices 
are locally rigid; see \cite{BRHWbdy} building on earlier results including \cite{KS97,MR1421873}.  See also \cite{2303.00543} for $C^0$-local rigidity results of  boundary actions by higher-rank  cocompact lattices  building on earlier work including \cite{MR4468857} in the rank-one setting.   See also recent results including \cite{MR4693950} which classifies  projective (relative to a connection) actions 
and \cite{MR4794146} which classifies $\Sl(n,\R$)-actions on $n$-manifolds as those induced from a projective action.
\end{enumerate}

This paper continues the second theme  of studying  {hyperbolic actions} by higher-rank lattices.   
Prior results in this direction are typically of the following flavor: Let $\Gamma$ be a higher-rank lattice, let $\alpha\colon \Gamma\to \Diff^\infty(M)$ be an action, and suppose there exists $\gamma_0\in \Gamma$ such that the diffeomorphism $\alpha(\gamma_0)\colon M\to M$ is   Anosov   (or perhaps satisfies a weaker notion of hyperbolicity such as being partially hyperbolic or admitting a dominated splitting).  Under such dynamical hypotheses, if the manifold is assumed to be a torus or nilmanifold, one can often classify the action as smoothly conjugate to an affine action; see for example the main results of \cite{BRHW17}.  Under additional dimension assumptions, one can often classify the topology of the manifold; see especially the results in \cite{HL:Dom}.

In this paper, rather than imposing any (uniform) hyperbolicity assumption on the action,  we consider actions satisfying a much weaker dynamical property: we assume   there exists one element $\gamma_0\in \Gamma$ such that the diffeomorphism $\alpha(\gamma_0)\colon M\to M$ has positive topological entropy, $h_{\top}(\alpha(\gamma_{0}))>0$.    

We note that Anosov diffeomorphisms always have positive topological entropy, though there are many diffeomorphisms with positive topological entropy that are not Anosov.  On the other hand, via the variational principle (see  \cite[Chap.\  20]{KHbook}) and the Margulis--Ruelle inequality (see \Cref{thm:MRineq} below), positivity of topological entropy ensures the existence of an $\alpha(\gamma_0)$-invariant Borel probability measure with non-zero Lyapunov exponents; thus our distinguished element of the action $\alpha(\gamma_0)$ exhibits some nonuniform  (partial) hyperbolicity. 

To the best of our knowledge, this paper is the first paper to study rigidity properties of smooth actions by higher-rank lattice under the assumption that the action admits elements with positive topological entropy.  Under this mild dynamical assumption (and  further dimension constrains on $M$) we  prove a number of surprising new rigidity results.  
In the remainder of this introduction, we formulate a number of corollaries and consequences of our main results, primarily formulated for actions by lattices in $\Sl(n,\R)$, as well as one of our  main technical theorems, \cref{thm:SLnZintro}.  
  We will formulate the remainder of our  main technical theorems  in \cref{mainresultssection} after a review of  terminology and standard constructions.

We remark that our main results are stated only for $C^\infty$ actions.  This is primarily because we frequently appeal to certain upper semicontinuity properties of metric entropy.  Specifically, on the suspension space $M^\alpha$ (see \cref{mainresultssection}), we assert that fiber metric entropy is upper semicontinuous when restricted to certain classes of invariant measures (with certain quantitative decay of mass near $\infty$); see \cref{prop:uscfiberentropy} below.    

Upper semicontinuity of (standard) metric entropy (for $C^\infty$ diffeomorphisms of compact manifolds) is established in the classical results of Newhouse \cite{MR986792} following the work of  Yomdin \cite{MR0889979,MR889980}; however, to the best of our knowledge, no formulation of analogous results for fiber metric entropy appears in the literature. 
In \cref{app:yomdinnewhouse}, 
we present an abstract formulation of upper semicontinuity of fiber metric entropy.    See especially \cref{lem:locentzeromeansUSC,prop:Yomdin} for statements of results.   We particularly note that since we do not assume our lattices our cocompact, our fibered actions (the  action by translations on the suspension space $M^\alpha$) do not occur on compact manifolds.  We  thus formulate our upper semicontinuity results without any compactness assumption on the total space.  To accommodate for the lack of compactness (and consequence unboundedness of the fiber dynamics) in such generality, it seems necessary to impose some uniform integrability on the family of measures considered; see \S\ref{unifint} for definitions.  

\subsection{Main results for actions by lattices in $\Sl(n,\R)$}  
The main technical theorems of the paper are \cref{Ainv,thm:AtoG} presented in \cref{mainresultssection} and \cref{thm:SLnZintro} below.  In the remainder of this introduction, we outline a number of motivating corollaries of our main technical theorems,  focusing especially on results for actions by lattices in $\Sl(n,\R)$.  

\subsubsection{Motivation: affine $\Sl(n,\Z)$-actions and their deformations}
To motivate the results enumerated below,  recall  that  $\Gamma= \Sl(n,\Z)$ is a lattice subgroup of $\Sl(n,\R)$ and induces an action $\rho\colon \Gamma\to \Aut(\T^n)$ by (affine) toral automorphisms. For any  element $\gamma\in \Sl(n,\R)$ with at least one eigenvalue outside the unit circle, the map $\rho(\gamma)$ has positive topological entropy.  Thus the action $\rho$ admits many elements acting with positive topological entropy (and in fact many Anosov elements).  Moreover, the action $\rho(\Gamma)$ preserves an absolutely continuous invariant probability measure, the Haar measure on $\T^n$.

Let $\Gamma$ be a finite-index subgroup of $\Sl(n,\Z)$.  The affine action $\rho\colon\Gamma \to \Aut(\T^n)$ can be deformed in a number of ways. First, one may replace fixed points or finite orbits for the action with $(n-1)$-dimensional spheres; then, one may either (1) quotient boundary spheres by  antipodal maps (i.e.\ blowup points), (2) glue  closed $n$-balls to a boundary spheres, or (3) glue two spheres in  the same torus or in different tori by tubes.  Such constructions can be performed so that the affine $\Gamma$-action outside the finite orbits extends to a real-analytic $\Gamma$-action; moreover, some of the above constructions can be made to preserve an ergodic, no-where vanishing smooth density.  Some of these deformations were first described in \cite{KL96}; further deformations (and more detailed treatments of those from \cite{KL96}) only appeared quite recently in  \cite{MR4794146}.  

Fix  $n\ge 3$.  Let $G= \Sl(n,\R)$ and let $\Gamma$ be a lattice in $G$.  As conjectured in  \cite[Conj.\ 3.6]{MR4794146}, any  $\Gamma$-action on an  $n$-dimensional manifold $M$ either \begin{enumerate}
\item factors through the action of a finite group on $M$,
\item extends to a $G$-action (and is then classified as in \cite{MR4794146})  on $M$, or
\item $\Gamma$ is commensurable with $\Sl(n,\Z)$ and the action is built from the above modifications of the affine action $\rho\colon \Sl(n,\Z) \to \Aut(\T^n)$.
\end{enumerate}
We note that for any action satisfying the first or second outcomes of the above trichotomy, every element of the corresponding action acts with zero  topological entropy.  

The  results of our paper provide evidence for the above conjecture and we expect will be useful to obtain a full classification of actions on $n$-manifolds.  Namely, for $n\ge 3$,  for any lattice $\Gamma$ in $\Sl(n,\R)$, and for any $n$-manifold $M$, we consider an action $\alpha\colon \Gamma\to \diff^\infty (M)$ such that $\alpha(\gamma_0)$ has positive topological entropy for some $\gamma_0\in \Gamma$.  We then obtain a weak version of outcome (3) in the above trichotomy: First, we show lattice $\Gamma$ must be commensurable with $\Sl(n,\Z)$.  Second, we recover the structure of an affine  $\Gamma$-action  on $\T^n$ outside of possible deformations by showing (for $n\ge 4$) the action preserves an absolutely continuous probability measure; moreover, relative to this measure, the action is measurably conjugate to an affine action on finitely many copies of  $\T^n$ (or  $\T^n/\{\pm1\}$).   This measurable conjugacy extends to a continuous map from an open set of $M$ to the complement of a finite set which gives topological restrictions on the manifold $M$; see  discussion following \cref{KRHtopol} below.

\subsubsection{Classification of lattices acting with positive entropy}
Our first surprising result asserts that any lattice $\Gamma\subset \Sl(n, \R)$ acting on a $n$-manifold $M$ with positive topological entropy must, in fact,  be  commensurable to $\Sl(n,\Z)$.


\begin{corollary}\label[corollary]{coro:SLn}
For $n\ge 3$, let $\Gamma$ be a lattice in $\SL(n,\Rbb)$.  
Let $M$ be a closed manifold with $\dim M= n$ and let $\alpha\colon\Gamma\to\diff^\infty(M)$ be an action such that $h_{\top}(\alpha(\gamma_{0}))>0$  for some $\gamma_{0}\in \Gamma$. Then $\Gamma$ is   commensurable  {(up to conjugation)} with $\SL(n,\Zbb)$.  That is, there exists a finite index subgroup $\Gamma_{0}<\Gamma$  {that is conjugate to} a finite index subgroup of $\SL(n,\Zbb)$.
\end{corollary}
For instance, if $\Gamma\subset \Sl(n,\R)$ is a cocompact lattice   (or a lattice with $\rank_\Q(\Gamma)<
\rank_\Q(\SL(n,\Z))= n-1$), \cref{coro:SLn} implies that for any $n$-manifold and any action $\alpha\colon \Gamma\to \Diff^\infty(M)$, we have  $h_{\top}(\alpha(\gamma))=0$ for every element $\gamma\in \Gamma$.

\cref{coro:SLn} will follow directly from \cref{thm:SLnZintro} below as explained in \cref{sss:msblclass}.

We note that when $n\ge 3$, every lattice $\Gamma< \Sl(n,\R)$ is arithmetic and thus has a well defined $\Q$-rank.
This motivates the following question for possible future investigation:
\begin{question}
For $n\ge 3$, let $\Gamma$ be a lattice in $\Sl(n,\R)$ with $\rank_\Q(\Gamma)= q\le n-2$.  Find the smallest dimension $d$ (in terms of $q$ or the arithmetic structure of $\Gamma$) 
such that there exists a $d$-dimensional manifold  and an action $\alpha\colon \Gamma\to \diff^\infty (M)$  with $h_\top(\alpha(\gamma_0))>0$ for some $\gamma_0\in \Gamma$.  
\end{question}


\subsubsection{Actions with positive entropy admit  absolutely continuous invariant measures}
In the proof of \cref{coro:SLn}, we will first establish (for $n\ge 4$) that any action $\alpha\colon \Gamma\to \Diff^\infty(M)$ of a lattice $\Gamma\subset \Sl(n, \R)$  on a $n$-manifold $M$ with positive topological entropy admits an $\alpha(\Gamma)$-invariant probability measure $\nu$.  Moreover, the measure $\nu$ will be absolutely continuous and have positive (metric) entropy for some element of the action.

\begin{corollary}\label[corollary]{coro:invmsrSLNR}
For $n\ge 4$, let $\Gamma$ be a lattice in $\SL(n,\Rbb)$. 
Let $M$ be a closed manifold with $\dim M= n$ and let $\alpha\colon\Gamma\to\diff^\infty(M)$ be an action such that $h_{\top}(\alpha(\gamma_{0}))>0$  for some $\gamma_{0}\in \Gamma$. 

Then, there exists an $\alpha(\Gamma)$-invariant Borel probability measure $\nu$ on $M$; moreover $\nu$ is absolutely continuous (with respect to any smooth density on $M$)  and $h_\nu(\alpha(\gamma_0))>0$ for some $\gamma_0\in \Gamma$.  
\end{corollary}

  \cref{coro:invmsrSLNR}   follows from  Theorems \ref{Ainv}, \ref{thm:AtoG}, and \ref{thm:abs} which we formulate in \cref{mainresultssection}.
  We note that \cref{coro:invmsrSLNR} is rather surprising as $\Gamma$ is non-amenable and thus   abstract continuous $\Gamma$-actions need not admit any invariant probability measure. 
In \cref{thm:SLnZintro} below, we show further that the action $\alpha\colon\Gamma\to\diff^\infty(M)$ is measurably (relative to the measure guaranteed in the conclusion of  \cref{coro:invmsrSLNR}) conjugate to an affine action on a torus or infra-torus.

When $n=3$, our methods do not establish \cref{coro:invmsrSLNR} when $\Gamma$ is commensurable with $\Sl(n,\Z)$; however, the conclusion of  \cref{coro:invmsrSLNR}  still holds for actions  with positive topological entropy by lattices in $\Sl(3,\R)$ that are not  commensurable with $\Sl(n,\Z)$.    
Using \cref{thm:SLnZintro} below, we arrive at a contradiction for lattices in $\Sl(3,\R)$ with $\Q$-rank strictly smaller than 2,  establishing \cref{coro:SLn} for lattices in $\Sl(3,\R)$.  
  

\subsubsection{Measurable classification of actions with positive entropy; proof of \cref{coro:SLn}}\label{sss:msblclass}
Relative to the $\alpha(\Gamma)$-invariant Borel probability measure on $M$ guaranteed by \cref{coro:invmsrSLNR}, we  classify positive entropy actions $\alpha\colon \Gamma\to \Diff^\infty(M)$ up to measurable conjugacy.
Here, the model actions for classification are the affine actions on the $n$-torus $\T^n$ or infra-torus $\Tbb^{n}_{\pm}:=\T^n/\{\pm 1\}$ induced by a representation $\Gamma\to \Gl(n,\Z)$.    
Recall that a linear map $A\in \Gl(n,\Z)$ induces an automorphism of the torus $\T^n$ which descends to an invertible affine map on the orbifold $\T^n_{\pm};$ similarly, a representation  $\rho\colon\Gamma\to \Gl(n,\Z)$ induces actions $\rho\colon\Gamma\to \Aut(\T^n)$ and $\hat \rho\colon\Gamma\to \Aff(\T^n_{\pm})$ by affine orbifold transformations.

Let $L$ denote either  either the torus $\T^n$ or infra-torus $\Tbb^{n}_{\pm}$ and let $\leb$ denote the normalized Haar measure on $L$.  


\begin{theorem}\label{thm:SLnZintro}
Let $G=\SL(n, \Rbb)$ with $n\ge 3$ and let $\Gamma$ be a lattice in $G$.

Let $M$ be a closed smooth $n$-manifold and let $\alpha\colon\Gamma\to \Diff^{r}(M)$ for $r>1$.   Assume there exists an ergodic,  $\alpha(\Gamma)$-invariant  Borel probability measure $\nu$ on $M$. Further assume  there exists $\gamma_{0}\in \Gamma$ such that $h_{\nu}(\alpha(\gamma_{0}))>0$.

Then, 
\begin{enumerate}
\item $\nu$ is absolutely continuous with respect to the Lebesgue measure on $M$.
\end{enumerate}
Moreover, there exist
\begin{enumerate}[resume]
\item a finite-index subgroup $\Gamma'$ in $\Gamma$, 
\item a measurable, $\Gamma'$-invariant subset $U_0\subset M$ with $\nu(U_0 )>0$,
\end{enumerate}
such that, writing $\nu_0 =\frac{1}{\nu(U_0)}\restrict{\nu}{U_0}$, there exist
\begin{enumerate}[resume]
\item a measurable isomorphism $h\colon(L, \leb)\to (M,\nu_0 )$ (where $L$ is either $\T^n$ or $\Tbb^{n}_{\pm}$), and 
\item a  representation $\rho\colon\Gamma'\to  \Sl(n,\Z)$ 
\end{enumerate} 
 such that  if $\hat \rho\colon \Gamma'\to \Aff(L)$ is the action on $L$ induced by $\rho$ then   for all $\gamma\in\Gamma'$
\[\alpha(\gamma)\circ h =h\circ \hat \rho(\gamma).\] 
\end{theorem}

\cref{coro:SLn} now follows directly from \cref{thm:SLnZintro} and \cref{coro:invmsrSLNR}.  Indeed, 
we apply \cref{coro:invmsrSLNR} to obtain a $\alpha(\Gamma)$-invariant probability measure with $h_{\nu}(\alpha(\gamma_{0}))>0$ for some $\gamma_{0}\in \Gamma$. We then obtain a measurable conjugacy $h$ from  \cref{thm:SLnZintro}.  Following notation in  \cref{thm:SLnZintro}, the action $\widehat{ \rho}\colon\Gamma'\to \Diff^\infty(L)$ has elements of positive $\leb$-entropy and thus $\rho\colon \Gamma'\to \Sl(n,\Z)$ has unbounded image in $\Sl(n,\R)$.  By Margulis's superrigidity theorem \cite{Mar91}, it follows that $\rho( \Gamma')$ is a lattice in $\Sl(n,\R)$ and   thus $\rho( \Gamma')\subset \Sl(n,\Z)$ must be of finite index in $\Sl(n,\Z)$.  


\subsection{Consequences for Anosov actions}
The starting point for this collaboration was a hypothesis in the paper   \cite{HL:Dom} by the second author.  In \cite{HL:Dom}, the second author establishes  the following: \begin{theorem*}[{\cite[Cor.\ 1.1]{HL:Dom}}]{ For $n\ge 3$, let  $\Gamma$ be a lattice in $\Sl(n,\R)$.  Let $M$ be an $n$-manifold $M$ and let $\alpha\colon \Gamma\to \Diff^\infty(M)$ be an action such that $\alpha(\gamma)$ is an Anosov diffeomorphism.  If the action $\alpha(\Gamma)$ preserves a volume form, then the manifold is a flat torus and the action is smoothly conjugate to an affine action.}\end{theorem*}

Since Anosov diffeomorphisms automatically have positive topological entropy and since absolutely continuous probability measures invariant under Anosov diffeomorphisms are always smooth (by a standard Livsic argument),  \cref{coro:invmsrSLNR}  allows us to remove the volume-preserving hypothesis from (most cases of) the main results in \cite{HL:Dom} for lattices in $\SL(n,\R)$ and $\Sp(2n,\R)$.    
Combining \cref{coro:invmsrSLNR} (and its extension in \cref{thm:abs})  and \cite[Thm.\ 1.5]{HL:Dom} gives the following.  
\begin{corollary}\label[corollary]{coro:Anosov}
Suppose one of the following:
\begin{enumerate}
\item $G$ is $\SL(n,\Rbb)$ with $n\ge 4$ and $\dim M=n$, or
\item $G$ is $\Sp(2n,\Rbb)$ with $n\ge 3$ and $\dim M=2n$.
\end{enumerate}
Let $\Gamma$ be a lattice in $G$, let $M$ be a closed manifold, and let $\alpha\colon\Gamma\to \Diff^{\infty}(M)$ be an action.  

Assume  there exists $\gamma_{0}\in \Gamma$ such that $\alpha(\gamma_{0})$ is an Anosov diffeomorphism. Then $M$ is diffeomorphic to a (infra-)torus and the action $\alpha$ is smoothly conjugate to an Anosov affine action.
\end{corollary}
\begin{remark}
In the proof of \cref{coro:Anosov}, the main result of \cite{HL:Dom} asserts that $\alpha$ is smoothly conjugate to an affine action on an infra-torus  or infra-nilmanifold.  However,  the affine Anosov  action of $\Gamma$ induces a representation  on the commutator subgroup of the ambient nilpotent Lie group.  Dimension counting and appealing to Margulis superrigidity shows this representation must be isometric.   Since the action is Anosov, the commutator subgroup must be trivial  and thus $M$ is a (infra-)torus.  \end{remark}
 
\begin{remark}[Anosov actions of $\Gamma<\Sl(n,\R)$ for $n=3$]
Let $n=3$ and suppose $\Gamma$ is a lattice in $\Sl(n,\R)$.   Every Anosov diffeomorphism on a $3$-manifold $M$ is codimension-1 and thus the Franks--Newhouse theorem implies $M$ is homeomorphic to the torus $\T^3$.  In this case, the conclusion of \cref{coro:Anosov} holds if one can verify the action $\alpha \colon \Gamma\to \Homeo (\T^3)$ lifts to an action  $\wtd \alpha \colon \Gamma_0\to \Homeo (\R^3)$ on a finite index subgroup $\Gamma_0$; see main results of \cite{BRHW17}.  When $n\ge 4$, we  use the existence of an $\alpha(\Gamma)$-invariant measure to ensure this lifting property.  Although we expect the lifting property  in  \cref{coro:Anosov}  holds when $n=3$ (and, in fact, hold for all Anosov actions on tori and nilmanifolds), it is not   established in the literature.
\end{remark}

\subsection{Topological consequences}
After establishing 
our main technical theorems, \cref{Ainv,thm:AtoG,thm:abs} below,
the  proof of \cref{thm:SLnZintro}
 follows from revisiting and extending the main arguments and constructions from \cite{KRHarith}.  We thus obtain similar  topological corollaries (see especially  \cite[Cor.\ 7]{KRHarith})  as in   \cite{KRHarith} including the following:
 
\begin{corollary}\label[corollary]{KRHtopol}
For $n\ge 5$ odd, let $\Gamma$ be a lattice in $G=\SL(n, \Rbb)$. Let $M$ be a closed manifold with $\dim (M) = n$ and let $\alpha\colon\Gamma\to\diff^\infty(M)$ be an action such that $h_{\top}(\alpha(\gamma_{0}))>0$  for some $\gamma_{0}\in \Gamma$.

Then  $\pi_{1}(M)$ contains a subgroup   isomorphic to {$\Zbb^{n}$}. 
\end{corollary}
Further topological corollaries can be deduced from \cite[Thm.\ 2]{KRHarith} which in our setting,  roughly, states there is a finite index subgroup $\Gamma_0<\Gamma$, an   $\alpha(\Gamma_0)$-invariant open set $U\subset M$, and a continuous surjection $h\colon U\to L\sm F$, where $L$ is either $\T^n$ or $\T^n/\{\pm \id\}$ and $F$ is a finite set, intertwining the action $\alpha\colon \Gamma_0\to \Diff^\infty(U)$ and the action on $L\sm F$ induced by a representation $\rho\colon \Gamma_0\to \Sl(n,\Z)$.    
We note that  \cite{KRHarith} only considers actions by $\Z^{n-1}$ on $n$-manifolds.  As we consider actions by   much larger groups $\Gamma<\Sl(n,\R)$, it is possible one can substantially improve the classification results and topological corollaries of \cite{KRHarith}.  Since our main results leading to \Cref{coro:SLn}, \Cref{coro:invmsrSLNR}, and \cref{thm:SLnZintro} are already quite involved, we do not pursue this investigation in this paper.

\subsection{Actions by lattices in split orthogonal groups} 
Let $G= \SO(n,n)$.  
(Algebraic) Anosov  actions of lattices in $\So(n,n)$ first appear in dimension $2n$.  In  \cite{HL:Dom}, the volume-preserving Anosov actions on $2n$-dimensional manifolds are fully classified.  Following the notation of  \cref{sec:numerology}, this dimension  $2n=n(G) $ is the dimension of the defining representation, whereas $v(G)$, the dimension of the smallest non-trivial projective action, is $v(G)=2n-2= n(G)-2$.  
For technical reasons arising in the proof (see \cref{remk:why+1} below), our generalization of \cref{coro:invmsrSLNR}  in  \cref{Ainv,thm:AtoG} below only holds for actions on manifolds of dimension $(v(G)+1).$  
In particular, we are (so far) unable to remove the volume-preserving assumption for Anosov actions by lattices in $\So(n,n)$, $(n\ge 4)$,  in  \cite{HL:Dom}. 

In a similar direction, the appropriate version of Zimmer's conjecture for actions by lattices in $G=\So(n,n)$ or $G=\So(n,n+1)$ asserts that  all volume-preserving actions on manifolds of dimension $(v(G) +1)$ are isometric. (Note in the notation of  \cref{sec:numerology} that $v(G)<d(G) = v(G)+ 1< n(G) = v(G) +2$ and certain lattices in $G$   admit infinite-image representations into $\So(n(G))$, inducing infinite-image isometric actions on the sphere $S^{d(g)}$.)  The results of \cite{MR4502593,Brown:2020aa,Brown:2021aa}  establish that volume-preserving actions by lattices in $G$ are isometric (and thus finite) on manifolds of dimension $v(G)$.  The results of \cite{BRHWbdy} show that any infinite-image action in dimension $v(G)$ is smoothly conjugate (on an index-$2$ subgroup) to the projective action on (possibly a  cover of) the space of isotropic lines for a quadratic form of signature $(n,n)$ or $(n, n+1)$.  
While a full classification of lattices in $G$ acting on  $(v(G)+1)$-dimensional manifolds seems still out of reach, 
our techniques provide evidence for Zimmer's conjecture and an approach towards such a classification in this setting by showing that all such actions have no positive entropy elements.  

\begin{corollary}\label[corollary]{coro:SO} Let $G$ be 
either $\SO(n,n)$ with $n\ge 4$ or $\SO(n,n+1)$ with $n\ge 3$ and let $\Gamma$ be lattice subgroup in $G$. Let $M$ be a closed manifold with $\dim M \le \left(v(G)+1\right)$ and let $\alpha\colon\Gamma\to\diff^{\infty}(M)$ be an action. 

Then, for every $\gamma\in \Gamma$,  we have 
$h_{\top}(\alpha(\gamma))=0$. 
\end{corollary}

\begin{proof}[Proof of \Cref{coro:SO}] Let $G$ be a Lie group as in \Cref{coro:SO}. Note that there is no nontrivial homomorphism from $G$ to $\GL(v(G)+1,\Rbb)$ since $v(G)+1<n(G)$ (see   \cref{sec:numerology}).
Assume that there is $\gamma \in \Gamma$ with $h_{\top}(\alpha(\gamma))>0$. By \Cref{thm:AtoG,Ainv} below, there is an ergodic, $\alpha(\Gamma)$-invariant  probability measure $\nu$ such that $h_{\nu}(\alpha(\gamma_{0}))>0$ for some $\gamma_{0}\in\Gamma$.
Using Zimmer's cocycle superrigidity theorem,  \Cref{thm:ZCSRGamma}, since there is no non-trivial homomorphism from $G$ to $\GL(v(G)+1,\Rbb)$, the derivative cocycle $(\gamma,x)\mapsto D_{x}\alpha(\gamma)$ is measurably cohomologous to a compact group valued cocycle.  In particular, for every $\gamma\in \Gamma$ the top Lyapunov exponent of $\alpha(\gamma)$ vanishes at $\nu$-almost every $x\in M$.  
By the Margulis--Ruelle inequality (\Cref{thm:MRineq} below), $h_{\nu}(\alpha(\gamma))=0$ for all $\gamma\in\Gamma$ which contradicts the positivity of $h_{\nu}(\alpha(\gamma_{0}))$.
\end{proof}

\begin{example}\label{ex:cricleext}
One can build non-trivial examples of actions of lattices in $\Gamma$ in $G= \So(n,n)$ or $G= \So(n,n+1)$ on manifolds of dimension $v(G)+1$ as follows: fix a quadratic form of the appropriate signature and let $Q<G$ be the stabilizer of an isotropic line.  The group   $G$ acts transitively on all such lines and $\dim (G/Q)= v(G)$.  The natural (projective) left action of $\Gamma$ on $G/Q$ extends to an action on $G/Q\times S^1$ acting by the identity on $S^1$.  

More generally, let $Y$ be a nontrivial vector field on $S^1$ and let $\phi_Y^t$ denote the induced flow on $S^1$.  Let $\psi\colon Q\to \R$ be a non-trivial homomorphism.  On $G\times S^1$ build left-$G$-actions and right-$Q$-actions by 
$$h\cdot (g,x) = (hg, x), \quad \quad (g, x)\cdot q = (gq, \phi_Y^{\psi (q\inv)}(x)).$$
Let $M= (G\times S^1)/Q$ be the quotient manifold.  The $G$-action descends to a $G$-action on $M$ which we may restrict to a $\Gamma$-action.  

Note that the above examples have the standard projective action of $\Gamma$ on $G/Q$ as a smooth factor.  Also note (by \cref{coro:SO}) that  every element acts with zero topological entropy.  
\end{example}

\begin{question}
Let $\Gamma $ be a lattice in  $G= \So(n,n)$ or $G= \So(n,n+1)$.  Let $M$ be a closed manifold of dimension $(v(G)+1)$ and let $\alpha\colon \Gamma\to \diff^\infty(M)$ be an action that is not isometric (for any choice of $C^0$ Riemannian metric on $M$).  

Let $Q<G$ be the stabilizer of an isotropic line for a quadratic form of the appropriate signature.  Does the standard projective action of $\Gamma$ on $G/Q$ occur as an equivariant factor of $\alpha\colon \Gamma\to  \diff^\infty(M)$?  That is, is there a ($C^0$ surjection, $C^k$ submersion) $p\colon M\to G/Q$ such that $$p\circ \alpha(\gamma) (x)= \gamma\cdot p(x)$$
for all $\gamma$ in (a finite index subgroup of) $ \Gamma$ and all $x$ in $ M$?

We remark that in this setting, a measurable projective factor is constructed in \cite{MR4502594}.  One approach towards this question would be to upgrade the measurable factor to a continuous/smooth factor.  
\end{question}

\subsection*{Acknowledgements} Among many others, the authors would like to thank David Fisher  and Dave Witte Morris for useful discussion and encouragement.  The authors also thank Sami Douba and David Fisher for helpful comments. A.\ B.\ was partially supported by the  National Science Foundation under Grant Nos.\ DMS-2020013 and DMS-2400191.  H.\ L.\ was supported by an AMS-Simons Travel Grant.  The authors also benefitted from hospitality of Institut Henri Poincar\'e while working on this project (UAR 839 CNRS-Sorbonne Universit\'e, ANR-10-LABX-59-01).

\section{Suspension space, two themes, and formulation of main technical results}\label{mainresultssection}

\subsection{Standing hypotheses on $G$ and $\Gamma$ and reductions}\label{reduct}
Throughout, we assume  the following standing hypotheses on $G$ and $\Gamma$:
\begin{hypot}\label[hypot]{hypstand}$G$ is a connected, simple Lie group with finite center and real rank at least 2.   $\Gamma<G$ is a lattice subgroup; we moreover assume that $\Gamma$ is virtually torsion free.  
\end{hypot}
We note that $\Gamma$ is always virtually torsion free whenever $G$ is linear.  

All results in this paper hold if they are verified for the restriction of an action to a finite-index subgroup of $\Gamma$.   Thus, we may first assume that $\Gamma$ is a torsion-free; we may then assume  $G$ is a center free, linear algebraic group.  
Replacing $G$ with its algebraically simply connected cover $\wtd G$ and lifting $\Gamma$ to a lattice subgroup $\wtd \Gamma$ of $\wtd G$ we may induce an action of $\wtd \Gamma$ through $\wtd \Gamma\to \Gamma$.  Any result announced above for $\Gamma$-actions holds if it is established for the induced $\wtd \Gamma$-action.  Thus,   it is with no further loss of generality to also assume the following reduction: 
\begin{hypot}\label[hypot]{hypstand2}
$G=\bfG(\R)^\circ$ where $\bfG$ is a connected, simple,  algebraically simply connected, linear algebraic group defined over $\R$.  The $\R$-rank of $G$ is at least 2.   $\Gamma$ is a lattice subgroup of $G$.   
\end{hypot}

We note that many of the preliminaries here hold when $G$ is semisimple and $\Gamma$ is irreducible; however, for our main applications and in our main theorems, we will require that $G$ be simple.

\subsection{Suspension, induced $G$-action, and fiber entropy}\label{sss:susp}
Let $G$ and $\Gamma$ be as in \cref{hypstand}.  
We follow a well-known construction (used previously in e.g.\ \cite{MR4502594,BRHW17,MR4502593,Brown:2021aa,ABZ,BDZ} among others) which allows us to relate various properties of a $\Gamma$-action $\alpha\colon \Gamma\to \Diff(M)$ with properties of a $G$-action on an associated  bundle $M^\alpha$ over $G/\Gamma$.

\subsubsection{Suspension space and induced $G$-action}
On the product $G\times M$ consider the right $\Gamma$-action and the left $G$-action
\begin{equation}\label{eq:susp} (g,x)\cdot \gamma= (g\gamma , \alpha(\gamma\inv)(x)),\quad \quad \quad a\cdot (g,x) = (ag, x).\end{equation}
Define  the quotient manifold $M^\alpha:= G\times M/\Gamma $.  
Given $(g,x)\in G\times M$, denote by $[g,x]$ 
the corresponding equivalence class in $M^\alpha.$
As the  $G$-action on $G\times M$ commutes with the $\Gamma$-action, we have an induced left $G$-action  on $M^\alpha$.  For $g\in G$ and $x\in M^\alpha$ we denote the action of $g$ on $x$ by $\wtd \alpha(g)(x)$.  Then $\wtd \alpha (g)\colon M^\alpha\to M^\alpha$ is a diffeomorphism.  

We write $p \colon M^\alpha\to G/\Gamma$ for the natural projection map; then  $M^\alpha$ has the structure of a fiber bundle over $G/\Gamma$ induced by the map $p$ with fibers diffeomorphic to $M$. The $G$-action preserves the fibers of $M^\alpha$.
Let 
\begin{equation}\scrF(x):= p\inv (p(x))\end{equation}denote the fiber of $M^\alpha$ through $x$ and let 
\begin{equation}F:= \ker Dp\end{equation}
denote the fiberwise tangent bundle of $M^\alpha$ (so that $T_x \scrF(x) = F(x)$).  
Given $x\in M^\alpha$  and $g\in G$, write $$D^F_x\wtd \alpha(g)\colon F(x)\to F(\wtd \alpha( x))$$ for the fiberwise derivative cocycle.   
		
\subsubsection{Fiber entropy and conditional measures}
A key concept in our arguments is the fiber entropy of translation by $g\in G$ on $M^\alpha$.  
We write $\Fsc$  for the partition into level sets of $p \colon M^\alpha\to G/\Gamma$ whose atom through $x$ is $\Fsc(x)$.  
Let $\mu$ be a $g$-invariant Borel probability measure on $M^\alpha$.  
We define $h_\mu(g\mid \Fsc)=h_\mu(\wtd \alpha(g)\mid \Fsc)$ to be the \emph{fiber metric entropy} of the map $\alpha(g)$ with respect to $\mu$.  See \cref{fiberentsec} for full definition in a more general setting.  

Given a probability measure $\mu$ on $M^\alpha$,  we let $\{\mu_x^\Fsc\}_{x\in M^\alpha}$ denote a family of conditional measures relative to the partition into atoms of $\Fsc$.  We refer to $\mu_x^\Fsc$ as the \emph{fiberwise conditional measure} through $x$.  If $\pi\colon M^\alpha \to G/\Gamma$ is the canonical projection, we also write $\{\mu_{g\Gamma}^\Fsc\}$, $\mu_{\pi(x) }^\Fsc=\mu_{x }^\Fsc$, for the collection of fiberwise conditional measures indexed by  $g\Gamma\in G/\Gamma$.

\subsubsection{Dictionary between invariant measures}\label{sss:dictionary measures}

Given an $\alpha(\Gamma)$-invariant probability measure $\nu$ on $M$, a standard construction induces a unique $G$-invariant probability measure $\mu$ on the bundle $M^\alpha$.  Moreover, $\nu$ is ergodic if and only if $\mu$ is ergodic.  

Similarly, given a $G$-invariant probability measure $\mu$ on the bundle $M^\alpha$, there is a unique $\alpha(\Gamma)$-invariant probability measure $\nu$ on $M$ such that $\mu$ is induced by $\nu$.  One can view $\nu$ as $\mu_{\1\Gamma}^\Fsc$, the {fiberwise conditional measure} over the identity coset in $G/\Gamma$.  (Note that since $\mu$ is $G$-invariant, by modifying the collection  $\{\mu_{g\Gamma}^\Fsc\}$ on a set of measure zero, one may assume $\mu_{g\Gamma}^\Fsc$ is defined and $G$-equivariant for every $g\Gamma\in G/\Gamma$.)
The measure $\nu$ is absolutely continuous on $M$ if and only if the {fiberwise conditional measure} $\mu_{g\Gamma}^\Fsc$ is absolutely continuous for almost every $g\Gamma\in G/\Gamma$.

\subsection{Iwasawa decomposition and maximal $\R$-split Cartan subgroups}
We let $$G= KAN$$ denote a choice of Iwasawa decomposition of the Lie group $G$.  Then $\Ad(A)$ is a maximal, connected, $\R$-diagonalizable subgroup of $\Ad(G)$ isomorphic to $\R^k$ where $k$ is the real rank of $G$.   We refer to such a choice of $A$ a maximal $\R$-split Cartan subgroup.  

Frequently, we restrict  the dynamics of $G$ on the suspension space $M^\alpha$ to the action of $A$.  
For the remainder, we fix a choice of maximal $\R$-split Cartan subgroup $A$ in $G$.  At times, we may further specify a choice of $A$ coherent with a choice of restricted roots defined over $\R$ or $\Q$.  We note that since all choices of $A$ are conjugate in $G$, all results stated below are independent of the choice of $A$.

\subsection{Two common themes in recent approaches to the Zimmer program}\label{section:2themes}
Much of the  of the recent progress in the Zimmer program follows an outline that, roughly, is summarized by two principles.  Versions of both these themes feature heavily in previous collaborations of the first author including \cite{Brown:2020aa,Brown:2021aa,MR4502593,BRHW17,BRHWbdy,ABZ}.  

First, associated to certain  dynamical properties of the action $\alpha(\Gamma)$, one constructs a probability measure on  the suspension space with related dynamical properties:
\begin{principle}\label{theme1}
 Dynamical properties of an action $\alpha\colon \Gamma\to \diff (M)$ are mimicked by (fiberwise) dynamical properties of $A$-invariant Borel probability measures on $M^\alpha$ projecting to the Haar measure on $G/\Gamma$. \end{principle}
For instance, in \cite{Brown:2020aa,Brown:2021aa,MR4502593}, the defect in an action being isometric is witnessed by an  $A$-invariant Borel probability measures on $M^\alpha$ projecting to the Haar measure on $G/\Gamma$ with a non-zero (fiber) Lyapunov exponent.  

Second, using that $A$ is a higher-rank abelian group, one expects that Borel probability measures invariant under $A$ have extra regularity and invariance: 
\begin{principle}\label{theme2}
 Under dimension or dynamical constraints, $A$-invariant  Borel probability measures on $M^\alpha$ projecting to the Haar measure on $G/\Gamma$ often exhibit extra invariance and extra regularity.  
  \end{principle}
  For instance, in \cite{Brown:2020aa,Brown:2021aa,MR4502593}, dimension constraints on $M$ force such $A$-invariant Borel probability measures to be automatically $G$-invariant.  
 

This paper presents  version of  \cref{theme1} (described in the next subsection), that does not appear in any prior literature. 
We expect this perspective will be useful in the future.  
We also give a version of \cref{theme2} that follows from similar perspectives taken in \cite{BRHWbdy,ABZ}.

\subsection{Topological entropy is mimicked by  a measure on the suspension space with positive fiber entropy}
Recall $M^{\alpha}= (G\times M)/\Gamma$ denotes the suspension space with induced $G$-action and that  $\Fsc$ denotes the partition into fibers of $p\colon M^\alpha\to G/\Gamma$.  
Our first main technical theorem realizes \cref{theme1} of \cref{section:2themes} by translating   positivity of topological entropy for an element $\gamma_0\in \Gamma$ into positivity of fiber metric entropy for an $A$-invariant  Borel probability measure on $M^\alpha$ projecting to the Haar measure on $G/\Gamma$.  
\begin{theorem}\label{Ainv}
Let $G$ and $\Gamma$ be as in \cref{hypstand}.  Let $M$ be a compact manifold, and let $\alpha\colon \Gamma\to \diff^\infty(M)$ be an action such that $h_\top(\alpha(\gamma_0))>0$ for some $\gamma_0\in \Gamma$.

Suppose further that $G$ and $\Gamma$ satisfy one of the following: 
\begin{enumerate}[label={(\alph*)}, ref={(\alph*)}]
	\item\label{AinvSL3} $G=\SL(3,\Rbb)$ and $\Gamma$ is $\Qbb$-rank $1$,
	\item\label{Ainvuniform} $\rank_{\Rbb}(G)\ge 2$ and $\Gamma$ is cocompact, or
	\item\label{Ainvrank3} $\rank_{\Rbb}(G)\ge 3.$
\end{enumerate}
Then there exists a Borel probability measure $\mu$ on $M^{\alpha}$ such that 
\begin{enumerate}
	\item $\mu$ is $A$-invariant,
	\item $\mu$ projects to the Haar measure on $G/\Gamma$, and 
	\item $h_{\mu}(a\mid \Fsc)>0$ for some $a\in A$.
\end{enumerate}
\end{theorem}
Note  that we do not make any assumptions on $M$ in the statement of \cref{Ainv} (other than being a closed $C^\infty$ manifold).  In particular, we do not impose any dimension constraints on $M$ in  \Cref{Ainv}. 

We further note (since $A$ acts ergodically on $G/\Gamma$ and since entropy is affine) that we may replace the measure in the conclusion of  \Cref{Ainv} with an $A$-ergodic component $\mu'$ of $\mu$ with the same properties.  It is thus with no loss of generality to assume $\mu$ is  $A$-ergodic.  

Although we expect the result should hold for latices in $\Sl(3, \R)$ with $\Q$-rank 2 (i.e.\  commensurable to $\Sl(3, \Z)$) and for nonuniform  lattices in $\Sp(4, \R)$, for technical reasons arising in the proof, our general arguments require that $\rank_\R(G)\ge 3$ or that $\Gamma$ be cocompact.  In the special case of $\Q$-rank-1 lattices in $\Sl(3,\R)$, we use a separate  argument in \cref{sec:sl3qrank1} specific to the classification of $\Q$-rank-1 lattices in $\Sl(3,\R)$ in order to  prove \Cref{Ainv}.

The proof of \Cref{Ainv} will occupy \Cref{sec:AinvI,sec:AinvII}.

\subsection{Numerology associated with semisimple Lie groups}\label{sec:numerology}
Given a  connected semisimple Lie group $G$ with Lie algebra $\Lie(G) = \lieg$, we recall the numbers  $v(G)=v(\lieg)$ and $n(G)=n(\lieg)$: 
We define $v(\lieg) = \min \{\dim (\lieg/\lieq)\}$ where the minimum is taken over all proper parabolic subalgebras $\lieq\subset \lieg$. 
We define $n(\lieg)$ to be the the minimal dimension of a vector space admitting a non-trivial representation of $\lieg$.  

For classicial $\R$-split groups of interest in this paper we have the following:
\begin{enumerate}
\item $\lieg = \sl(n,\R)$: $v(\lieg) = n-1$ and $n(\lieg) = n$.
\item $\lieg = \sp(2n,\R)$: $v(\lieg) = 2n-1$ and $n(\lieg) = 2n$.
\item $\lieg = \so(n,n)$: $v(\lieg) = 2n-2$ and $n(\lieg) = 2n$.
\item $\lieg = \so(n,n+1)$: $v(\lieg) = 2n-1$ and $n(\lieg) = 2n+1$.
\end{enumerate}

\subsection{Measure rigidity  and extra invariance of measures arising in \cref{Ainv}}
 Our second and third main technical theorems realize \cref{theme2} of \cref{section:2themes}.  Assuming certain constraints on the dimension of $M$, we show that the $A$-invariant  Borel probability measure on $M^\alpha$ projecting to the Haar measure on $G/\Gamma$ arising in \cref{Ainv} has extra invariance and extra regularity.

When $\dim(M) \le v(G)+1$, we show any measure $\mu$ satisfying the conclusion of \cref{Ainv} is $G$-invariant.  Again, we note that as  $G$ is non-amenable, there is no {\it a priori} reason for any such $G$-invariant Borel probability measure to exists.  

\begin{theorem}\label{thm:AtoG}
Let $G$ and $\Gamma$ be as in \cref{hypstand} and assume $G$ is $\Rbb$-split.  
Let $M$ be a closed $\left(v(G)+1\right)$-dimensional smooth manifold and for $r>1$, let $\alpha\colon \Gamma\to \Diff^{r}(M)$ be an action. 
Suppose   there exists an ergodic, $A$-invariant Borel probability measure $\mu$ on the suspension space $M^\alpha$ such that \begin{enumerate}[label={(\alph*)}]
\item $\mu$ projects to Haar measure on $G/\Gamma$, and
\item there exists $a\in A$ with $h_{\mu}(a \mid \Fsc)>0$.
\end{enumerate}
Then $\mu$ is $G$-invariant.  

\end{theorem}

Next when $G$ is isogenous to either $\SL(n,\Rbb)$ with $n\ge 3$ or $G=\Sp(2n,\Rbb)$ with $n\ge 2$, we show the  $A$-invariant Borel probability measure (which is $G$-invariant by \cref{thm:AtoG}) on $M^\alpha$ guaranteed by \cref{Ainv} has extra regularity.

\begin{theorem}\label{thm:abs}
Let $G$ and $\Gamma$ be as in \cref{hypstand} with  $\Lie(G)=\mathfrak{sl}(n,\Rbb)$ with $n\ge 3$ or $\Lie(G)=\mathfrak{sp}(2n,\Rbb)$ with $n\ge 2$.   
Let $M$ be a closed $n(G)$-dimensional smooth manifold and for $r>1$, let $\alpha\colon \Gamma\to \Diff^{r}(M)$ be an action. 
Assume  there exists  an ergodic, $G$-invariant Borel probability measure $\mu$ on the suspension  space $M^\alpha$ such that $$h_{\mu}(a_{0} \mid \Fsc)>0$$ for some $a_{0}\in A$.

Then, for $\mu$-almost every $x$, the fiberwise conditional measure $\mu^{\Fsc}_{x}$ is absolutely continuous (with respect to any smooth density    on $M$). 

Moreover, there exists an ergodic, absolutely continuous $\alpha(\Gamma)$-invariant Borel probability measure $\nu$ on $M$ (such that $\mu$ is the measure on $M^\alpha$ induced by $\nu$) such that $$h_{\nu}(\gamma_{0})>0$$  for some $\gamma_{0}\in\Gamma$.   
\end{theorem}
\begin{remark}\label[remark]{remk:why+1}
The proof of \cref{thm:abs} also applies in the case that   $\glie=\mathfrak{so}(n,n+1)$ with $n\ge 3$ or $\glie=\mathfrak{so}(n,n)$ with $n\ge 4$.  However, in the proof  we require the dimension constraint $\dim (M)\le \left(v(\glie)+1\right)$ in \cref{thm:AtoG} and since for these groups, 
$v(\glie)+1< n(\glie) = v(\glie)+2$, there are no natural examples of actions admitting measures satisfying the conclusions of  \cref{thm:abs}.   

In our proof of \cref{thm:abs}, we heavily use for $\lieg = \sl(n,\R)$ or $\lieg = \sp(2n,\R)$ that $n(\lieg) = v(\lieg)+1$ and thus 
$\dim (M)= v(\lieg)+1$.  

When $\lieg=  \so(n,n)$ or $\lieg=  \so(n,n+1)$, an analogue of \cref{thm:abs} may still hold when $\dim (M)=n(\lieg) = v(\lieg) +2$ for $\glie=\mathfrak{so}(n,n)$.  In this setting, our proof fails as there could be exactly two negatively proportional fiberwise Lyapunov exponents contributing to positive fiber entropy  $h_{\mu}(a_{0} \mid \Fsc)>0$, with  neither   positively proportional to a root.  In this case, our argument to obtain extra entropy along the orbit of a root group $U^\beta$ in \cref{nonat} (and the consequential extra invariance using  \cref{thm:highent}) may fail.  For this reason, we only state \cref{thm:abs} for groups isogenous to $\Sl(n,\R)$ and $\Sp(2n,\R)$.
\end{remark}

\begin{remark}
{ In \Cref{thm:abs}, the last assertion about $\Gamma$-actions can be directly deduced by the statement about the suspension space. Indeed, Zimmer's cocycle superrigidity theorem (see \cref{superrig} below) says that, after choosing a measurable trivialization $TM^{\alpha}\simeq M\times (\glie\oplus \Rbb^{v(G)+1})$, the  fiberwise derivative cocycle $D^{F}\colon G\times M^{\alpha}\to \GL(v(G)+1,\Rbb)$, $D^{F}(g,x)=\restrict{D_{x}\wtd{\alpha}(g)}{F}$ is measurable cohomologous to $\rho\cdot \kappa$ where $\rho$ is a representation $\rho\colon G\to \GL(v(G)+1,\Rbb)$ and $\kappa\colon G\times M^{\alpha}\to K$ is a cocycle valued in compact subgroup $K$ in $\GL(v(G)+1,\Rbb)$. Since we assume the $A$-action has  positive fiber entropy, we know that $\pi$ has unbounded image. 
Furthermore, since almost every fiberwise measure $\mu^{\Fsc}_{x}$ is absolutely continuous, we know that $\nu$ is absolutely continuous. Since the fiberwise derivative cocycle $D^{F}$ is induced by the  $\alpha(\Gamma)$-action, after choosing a measurable trivialization $TM\simeq M\times \Rbb^{v(G)+1}$ with respect to $\nu$, the derivative cocycle $D\colon\Gamma\times M\to \GL(v(G)+1,\Rbb)$, $D(\gamma,x)=D_{x}\alpha(\gamma)$ is measurable cohomologous to $\restrict{\rho}{\Gamma}\cdot \kappa'$ where $\kappa'\colon\Gamma\times M\to K'$ is a measurable cocycle valued in a compact subgroup $K'$ in $\GL(v(G)+1,\Rbb)$. Since $\rho$ is non-trivial, $\restrict{\rho}{\Gamma}$ has also unbounded image (by Margulis' superrigidity \cite{Mar91}). In particular, there exists $\gamma_0 \in \Gamma$ such that $\alpha(\gamma_0 )$ has a positive top Lyapunov exponent with respect to $\nu$. Since $\nu$ is absolutely continuous,   Pesin's entropy formula (\Cref{thm:Pesinent} below) implies that $h_{\nu}(\gamma_0 )>0$.}
\end{remark}

\section{Preliminaries,   definitions, and auxiliary results}\label{sec:pre}
Throughout this section, we follow the notation in \cref{sec:intro,mainresultssection}. 
 We follow the reductions in \cref{reduct} and assume $G$ and $\Gamma$ are as in \cref{hypstand2}.  
In particular, we assume that $\Gamma$ is a lattice in $G$ where $G=\bfG(\Rbb)^{\circ}$ is the connected  (with respect to Hausdorff topology) component of the $\Rbb$-points of a connected, algebraically simply connected, simple algebraic $k$-group $\bfG$ with $\rank_{\Rbb}(\bfG)\ge 2$ where $k=\Q$ or $k=\R$.

Let $M$ be a smooth closed manifold. Given a  $\Gamma$-action  $\alpha\colon\Gamma\to \Diff^{r}(M)$ for $r>1$, the induced action on the suspension $M^{\alpha}$ is denoted by $\wtd{\alpha}\colon G\to \Diff^{r}(M^{\alpha})$.  We note that only in  \cref{sec:semicontientropy,sec:ratnerpre}, we need to assume the action is $C^{\infty}$ in order to conclude upper semicontinuity of fiber entropy. 

\subsection{Terminology in Lie theory}

\subsubsection{Restricted $k$-roots} 
For  $k=\Q$ or $k=\R$, 
let $\bfS$ be a maximal $k$-split torus in $\bfG$ and write $S = \bfS(\R)^\circ$.
We write $\Phi(S,G)=\Phi_k(S,G)=\Phi_k(\bfS,\bfG)$ for the collection of (restricted) $k$-roots of $G$ with respect to the choice of  $\bfS$.
For each $k$-root $\alpha\in \Phi(S,G)$, such that $\frac{1}{2}\alpha \not\in \Phi(S,G)$, there is a unique connected unipotent $k$-subgroup~$U^{[\alpha]}$ with Lie algebra $\lieg^{\alpha}$ or $\lieg^\alpha \oplus \lieg^{2\alpha}$.
Given a choice of order on the abstract roots system $\Phi_k(S,G)$, we also write $\Delta_k(S,G)$ for the associated collection of simple positive roots determined by this order.  

\subsubsection{Standard rank-1 subgroups and diagonal elements determined by $k$-roots}
Let  $k=\Q$ or $k=\R$ and let $\alpha\in \Phi_k(S,G)$ be a $k$-root.  Then $-\alpha\in \Phi_k(S,G)$ and we obtain connected unipotent $k$-subgroups~$U^{[\alpha]}$ and $U^{[-\alpha]}$
 The following construction is well known, but there does not seem to any established terminology for referring to this subgroup.
 \begin{definition}\label{StandardRk1SubgrpDefn}
\

\begin{enumerate}
\item For a $k$-root $\alpha \in \Phi_k(S,G)$ such that $\frac{1}{2}\alpha \not\in \Phi_k(S,G)$,
the subgroup~$H_\alpha$ of $G$ generated by $U^{[\alpha]}$ and~$U^{[-\alpha]}$ is called the \emph{standard $k$-rank-1 subgroup} generated by $\alpha$.   
\item Fix a $k$-root $\alpha \in \Phi_k(S,G)$ such that $\frac{1}{2}\alpha \not\in \Phi_k(S,G)$.  Then $S\cap H_\alpha$ is a connected $1$-parameter subgroup of $S$, called \emph{the diagonal subgroup of $\alpha$ in $S$.} 

Given a non-zero $X\in \liea$ with $\exp(X)\in S\cap H_\alpha$ and $\alpha(X)<0$, we say $d^1_\alpha=\exp(X)$ is a \emph{diagonal element} of $\alpha$ in $S$.  

We often write $\{d_\alpha^\R\}$ for the diagonal $1$-parameter subgroup of $\alpha$ in $S$.  

\item We say a subgroup $S'\subset S$ is \emph{$k$-standard} if there is a collection $\theta \subset \Delta_k(S,G)$ such that $$S'= \exp\left(\bigcap_{\beta\in\theta} \ker \beta\right).$$

\end{enumerate}

\end{definition}




\subsection{Choice of norms, height function, control of mass at $\infty$}

\subsubsection{Fundamental sets and  adapted norms} 
\label{sec:norms}
When $\Gamma$ is cocompact in $G$, all choices of Riemannian metrics on $TM^\alpha$ (or the bundle $F\to M^\alpha$) are equivalent.  
When $\Gamma$ is nonuniform , $M^\alpha$ is not compact and the local dynamics along orbits need not be uniformly bounded.  To employ tools from smooth ergodic theory,  more care is needed in specifying norms on the fibers of $M^\alpha$.  
In this case, we follow either \cite[\S 5.4]{Brown:2021aa} or \cite[\S 2]{MR4502594}; we summarize these results below.  

When $\Gamma$ is nonuniform , we may assume $\Ad(G)$ is $\Q$-algebraic and that $\Ad(\Gamma)$ is commensurable with the $\Z$-points in $\Ad(G)$.  We may define Siegel sets and Siegel fundamental sets in $\Ad(G)$ (and thus $G$) relative to any choice of Cartan involution $\theta$ on $G$ and a minimal $\Q$-parabolic subgroup in $\Ad(G)$.  Using the Borel-Serre partial compactification (of $G$ for which $G/\Gamma$ is an open dense set in a compact manifold with corners), we equip the bundle $G\times TM \to G\times M$ with a $C^\infty$ metric with the following properties: Write $\langle \cdot, \cdot \rangle_{g,x}$ for the inner product on the fiber over $(g,x)$.   Then
\begin{enumerate}
\item $\Gamma$ acts by isometries on $G\times TM$.
\item (uniform comparability): There exists a fundamental set $D$ for $\Gamma$ in $G$ (namely, any choice Siegel fundamental set $D\subset G$ relative to a choice of Cartan involution $\theta$ and a minimal $\Q$-parabolic subgroup) 
and $C>1$ such that for all $g,g'\in D$,
$$
\frac 1 C\langle \cdot, \cdot \rangle_{g,x}\le 
\langle \cdot, \cdot \rangle_{g',x}
\le C\langle \cdot, \cdot \rangle_{g,x}.
$$ 
\end{enumerate}
The metric on $G\times TM$ then descends to a $C^\infty$ metric on the bundle $F\to M^\alpha$. For the remainder,  given $x\in M^\alpha$, we denote by $\|\cdot \|_x^F$ the induced norm on the fiber of $F$ through $x$.  

We similarly equip $G$ with any right-invariant metric and equip $G\times M$ with the associated Riemannian metric that makes $G$-orbits orthogonal to fibers, restricts to the right-invariant metric on $G$-orbits, and restricts to the above metric on every fiber.  This induces a metric on $TM^\alpha$.

\subsubsection{Quasi-isometry properties}
We also recall the following fundamental result of Lubotzky-Mozes-Raghunathan.  Relative to any fixed choice of generating set for $\Gamma$, write $|\gamma|$ for the word length of $\gamma$ in $\Gamma$.   Equip $G$ with any right-$G$-invariant, left-$K$-invariant metric $d_G$.  
When $G$ has finite center, the following is the main result of \cite{LMR00}.  
\begin{theorem}[\cite{LMR00}]\label{LMR}
The word metric and Riemannian metric on $  \Gamma$ are quasi-isometric: there are $A_0,B_0$ such that for all $\gamma,\gamma'\in   \Gamma $,
$$\frac 1 {A_0 }d_G(\gamma,\gamma') -B_0 \le |\gamma\inv\gamma'| \le {A_0 }d_G(\gamma,\gamma') +B_0.$$
\end{theorem}

As discussed above, when $\Gamma$ is nonuniform , given $a\in A$, the fiberwise dynamics of $a$ on $M^\alpha$ need not be bounded (in the $C^1$ topology.)  However, for any $a$-invariant probability measure on $M^\alpha$ projecting to the Haar measure on $G/\Gamma$, the degeneracy is subexponential along orbits.  We summarize this below 
in \cref{prop:integrability}.

\def\fund{{D_F}}

Let $D\subset G$ be a fundamental set on which the fiberwise metrics are uniformly comparable and fix 
 a Borel fundamental domain~$\fund$ contained in $D$ for the right $\Gamma$-action on $G$.  Let  $\beta_\fund\colon G \times G/\Gamma \rightarrow \Gamma$ be the \emph{return cocycle}: given $\hat x\in G/\Gamma$, take $\wtd x$ to be the unique
lift of $\hat x$ in $\fund$ and define $b(g,\hat x) = \beta_\fund(g,\hat x)$ to be the unique $\gamma\in \Gamma$ such that
$g\wtd{x}\gamma\inv \in \fund$.
One verifies that $\beta_\fund$ is a Borel-measurable cocycle and a second choice of fundamental domain 
defines a cohomologous cocycle.

%

Given a diffeomorphism $g\colon M\to M$, let $\|g\|_{C^k}$ denote the $C^k$ norm of $g$ (say, relative to some choice of embedding of $M$ into some Euclidean space $\R^N$.)
Given $g\in G$ and $x\in \fund$, let  $$\psi_k(g,x) = \|\alpha (\beta_{\fund}(g,x))\|_{C^k}.$$

Let $|\gamma|$ denote the word length of $\gamma$ relative to a fixed symmetric generating set.  
For each $k$, we may find $C_k$ (depending only on the action $\alpha$, the choice of generating set, and $k$) such that $$\psi_k(g,x) \le C_k^{k|\beta_{\fund}(g,x)|}$$
(see e.g.\ \cite[Lem.\ 7.7]{MR4502593}.)
From \cref{LMR}, we have 
\begin{equation}\label{tempered}
\log (\psi_k(g,x) )\le A_k( d_G(g,\1) + d_G(x, \1 \Gamma)) + B_k\end{equation}
for some $A_k>1$ and $B_k>0$ (depending only on the action $\alpha$, the choice of generating set, and $k$).

Let $m_{G/\Gamma}$ denote the normalized Haar measure on $G/\Gamma$.  Using \cref{LMR} and standard properties of Siegel sets, we obtain the following: 

\begin{proposition}\label[proposition]{prop:integrability}
For any $k$, any  $1\le q<\infty$, and any compact $B\subset G$, the map $$x\mapsto \sup_{g\in B} \log \psi_k (g,x)$$ is $L^q(m_{G/\Gamma})$ on $G/\Gamma$.  
In particular, for $m_{G/\Gamma}$-a.e.\ $x\in G/\Gamma$ and any $g\in G$,
\begin{equation}\label{eq:integrabl}
\lim_{n\to \infty} \frac 1 n \log^+( \psi_k (g,g^n\cdot x) ) = 0.
\end{equation}
\end{proposition}

\subsubsection{Height function and control of mass at $\infty$}
We describe a strong control on the decay of mass near $\infty$ for probability measures on $G/\Gamma$ when $\Gamma$ is nonuniform.  Such quantitative tightness was heavily used in \cite{Brown:2020aa, Brown:2021aa}

Let $h\colon G/\Gamma\to [0,\infty)$ be the distance  function $h(g\Gamma) = d(g\Gamma, \Gamma)$ where $d$ is the distance on $G/\Gamma$ induced by any choice of left-$K$-invariant, right-$G$-invariant metric on $G$.  
We extend $h$ to $h\colon M^\alpha\to [0,\infty)$ by precomposing with the projection $p\colon M^\alpha\to G/\Gamma$.  

\begin{definition}[{\cite[\S3.2]{Brown:2020aa}}]\label{def:masscusp}
We say that a Borel probability measure $\mu$ on $G/\Gamma$ or $M^\alpha$ has \emph{exponentially small mass at $\infty$} if there is $\tau_\mu>0$ such that for all $0<\tau<\tau_\mu$,
\begin{equation}\label{eq:SmallMassIneq}\int_X e^{\tau h(x)} \ d \mu(x) <\infty.\end{equation}
We say that a collection $\scrM$ of Borel probability measures on $G/\Gamma$ or $M^\alpha$ has \emph{uniformly exponentially small mass at $\infty$} if there is $\tau_0>0$ such that for all $0<\tau<\tau_0$,
\begin{equation}\label{eq:SmallMassIneq2}\sup_{\mu\in \scrM} \left\{\int e^{\tau h(x)} \ d \mu(x)\right\} <\infty.\end{equation}
\end{definition}

We note that if a collection $\scrM$ of probability measures on $G/\Gamma$ or $M^\alpha$ has {uniformly exponentially small mass at $\infty$}, then the collection is uniformly tight.  In particular, such a family is precompact in the space of probability measures equipped with the weak-$*$ topology (dual to bounded continuous functions) and thus no sequence in $\scrM$ witnesses  escape of mass.

\subsection{Superrigidity}\label{superrig}

The following adapts the more general statement of Zimmer's cocycle superrigidity (see \cite{FM03}) to our setting.  Recall that we took $G=\bfG(\R)^\circ$ where $\bfG$ is algebraically simply connected and so that the results of \cite{FM03} apply directly. 
Also, the $\log$-integrability of the measurable cocycle needed in  \cite{FM03} holds automatically in our setting since $M$ is compact.

\begin{theorem}\label{thm:ZCSRGamma}  
Let $\nu$ be an ergodic  $\alpha(\Gamma)$-invariant  probability measure on $M$. For the derivative cocycle $D\colon \Gamma\times M\to \GL(\dim M,\Rbb)$,  there exist a linear representation $\pi\colon G\to \GL(\dim M,\Rbb)$, a compact group $K<\GL(\dim M, \Rbb)$, a $K$-valued cocycle $\kappa\colon \Gamma\times M\to K$, and a measurable framing  $\left\{\psi_{x}\colon T_{x}M^{\alpha}\to \Rbb^{\dim M} \right\}$ defined for $\nu$-a.e.\ $x$ such that  \[\psi_{\alpha(\gamma)(x)}\circ D_{x}\alpha(\gamma) \circ\left(\psi_{x}\right)^{-1}=\pi(\gamma)\kappa(\gamma,x),\] for all $\gamma\in \Gamma$ and for $\nu$-almost every $x$. Moreover, $K$ commutes with $\pi(G)$.  
\end{theorem}

For the fiberwise derivative cocycle on the suspension space $M^\alpha$, we similarly have the following:
\begin{theorem}\label{thm:ZCSRG}
Let $\nu$ be an ergodic $\alpha(\Gamma)$-invariant   probability measure on $M$ and let $\mu$ be the ergodic $\wtd{\alpha}(G)$-invariant  measure on $M^{\alpha}$ induced by $\nu$. For the fiberwise derivative cocycle $D^{F}\colon G\times M\to \GL(\dim M,\Rbb)$, there exist a compact group $K'<\GL(\dim M, \Rbb)$, a $K'$-valued cocycle $\kappa'\colon G\times M^{\alpha}\to K'$, and a measurable framing  $\left\{\psi^F_{x}\colon T_{x}^{F}M^{\alpha}\to \Rbb^{\dim M} \right\}$ defined for $\mu$-a.e.\ $x$ such that  \[\psi^F_{\widetilde{\alpha}(g)(x)}\circ D_{x}^{F}\widetilde{\alpha}(g) \circ\left(\psi^F_{x}\right)^{-1}=\pi(g)\kappa'(g,x),\] for all $g\in G$ and for $\mu$-almost every $x$. 
Moreover, $K'$ commutes with $\pi(G)$.  Here $\pi$ is  (up to conjugation) the same representation as appeared in \Cref{thm:ZCSRGamma}
\end{theorem}

\subsection{Lyapunov exponents and Lyapunov manifolds}

\subsubsection{Higher-rank Oseledec's theorem}
Let $A\subset G$ be a maximal $\R$-split Cartan subgroup and let $\liea=\Lie(A)$.   We identify $A$ with its Lie algebra $\liea$ via the exponential map and extend  linear functionals $\liea^*$ on $\liea$ to functions on $A$.  
 Equip  $\liea$ with a choice of norm $|\cdot|$.  

\begin{proposition}\label[proposition]{thm:higherrankMET}
Let $\mu$ be an ergodic, $A$-invariant Borel probability measure on $M^\alpha$ with exponentially small mass at $\infty$.
Let $E\subset TM^\alpha$ be a continuous, $\restrict{D\wtd \alpha}{A}$-invariant   subbundle.  
  Then (relative to the choice of norms in \cref{sec:norms})
there are
	\begin{enumerate}
	\item an $A$-invariant subset $\Lambda_0\subset X$ with $\mu(\Lambda_0)=1$;

  \item 
   linear functionals $\lambda_i\colon A\to \R$ for $1\le i\le p$;  
	\item   and splittings   $E(x)= \bigoplus _{i=1}^p E^{\lambda_i}(x)$ 
	into families of mutually transverse,  $\mu$-measurable  subbundles $E^{\lambda_i}(x)\subset E(x)$ defined  for $x\in \Lambda_0$
	\end{enumerate}
such that
\begin{enumerate}	[label={ (\alph*)}]
	\item $D_x\wtd \alpha (s) E^{\lambda_i}(x)= E^{\lambda_i}(D\wtd \alpha(s)( x))$ for all $s\in A$ and
	\item\label{lemma:partb} $\displaystyle \lim_{|s|\to \infty} \frac { \log |  D_x\wtd \alpha (s) (v)| - \lambda_i(s)}{|s|}=0$
\end{enumerate}	
	for all $x\in \Lambda_0$ and all $ v\in  E^{\lambda_i}(x)\sm \{0\}$.  
 \end{proposition}

We have three distinguished subbundles $E\subset TM^\alpha$ of interest in the sequel:
\begin{enumerate}
\item $E(x) = T_xM^\alpha$, in which case we refer to the linear functionals $\lambda_i$ in \cref{thm:higherrankMET} as \emph{total Lyapunov exponents}.
\item $E(x) =  F(x)= \ker D_x p$ (where $p \colon M^\alpha\to G/\Gamma$), in which case we refer to the linear functionals $\lambda_i$ in \cref{thm:higherrankMET} as \emph{fiberwise Lyapunov exponents}.
\item $E(x)$ is tangent to the $G$-orbit through $x\in M^\alpha$, in which case we refer to the linear functionals $\lambda_i$ in \cref{thm:higherrankMET} as \emph{base Lyapunov exponents}.
\end{enumerate}
We write  $\Lcal^{\widetilde{\alpha}}(\mu)\subset \alie^{*}$ for the set of total Lyapunov functionals for the action $\restrict{\widetilde{\alpha}}{A}$ on $(M^{\alpha},\mu)$. 
Write $\Lcal^{\widetilde{\alpha}, F}(\mu)$ for the fiberwise Lyapunov exponents and $\Lcal^{\widetilde{\alpha},B}(\mu)$ for the base Lyapunov exponents.  By direct computation, $ \Lcal^{\widetilde{\alpha},B}(\mu)$ coincide with the restricted $\R$-roots $\Phi(A,G)$ of $G$ relative to $A$ and are thus independent of the choice of measure $\mu$.  

Given $\lambda\in \Lcal^{\widetilde{\alpha}, F}(\mu)$, we write $E^{\lambda, F}(x) = E^\lambda(x)\cap F(x)$ for the associated \emph{fiberwise Lyapunov subspace}.

\subsubsection{Coarse  Lyapunov exponents and fiberwise Lyapunov manifolds}
Suppose $\alpha\colon \Gamma\to \Diff^r(M)$ for $r>1$ and let $M^\alpha$ be the suspension space with induced $G$-action.  
Let $\mu$ be an ergodic, $A$-invariant Borel probability measure on $M^\alpha$.  
It is natural to group together  Lyapunov exponents in $\Lcal^{\widetilde{\alpha}}(\mu)$. 
that are positively proportional (as they can not be qualitatively distinguished by the dynamics of $A$).  
A \emph{coarse Lyapunov exponent} is an equivalence class $\Lcal^{\widetilde{\alpha}}(\mu)$. 
We similarly defined \emph{coarse fiberwise Lyapunov exponents}  and \emph {coarse} (restricted) \emph{roots}.  

Given a linear functional $\lambda\in \Lcal^{\widetilde{\alpha}}(\mu)$, we  write $[\lambda]$ for the equivalence class of $\lambda$.
We sometimes write $\what \Lcal^{\widetilde{\alpha}}(\mu)$ and $\what \Lcal^{\widetilde{\alpha}, F}(\mu)$ for the collections of equivalence classes of  coarse Lyapunov exponents and coarse  fiberwise Lyapunov exponents, respectively.  
Given $\chi \in \what \Lcal^{\widetilde{\alpha}}(\mu)$, we write 
$$E^{\chi}(x):=\bigoplus_{\lambda\in \chi}E^{\lambda, F}(x)$$
and given  $\chi \in \what \Lcal^{\widetilde{\alpha}, F}(\mu)$,  write 
$$E^{\chi, F}(x):=\bigoplus_{\lambda\in \chi}E^{\lambda, F}(x).$$

Given $a\in A$, write $$E^{s, F}_a(x) = \bigoplus _{\substack{\lambda\in \Lcal^{\widetilde{\alpha}, F}(\mu)\\ \lambda (a)<0}} E^{\lambda, F}(x).$$

\begin{proposition}
Let $\mu$ be an ergodic, $A$-invariant Borel probability measure on $M^\alpha$ with exponentially small mass at $\infty$.  
Let $\chi\in \what  \Lcal^{\widetilde{\alpha}, F}(\mu)$ be a coarse fiberwise Lyapunov exponent and let $a\in A$.  

Then for $\mu$-a.e.\ $x\in M^\alpha$ there exists  path-connected, injectively immersed $C^r$ submanifolds $W^{\chi, F}(x)$ and $W^{s, F}_a(x)$ contained in $\scrF(x):= p\inv (p(x))$ with the following properties:
\begin{enumerate}
\item $T_xW^{\chi, F}(x)=
E^{\chi, F}(x)$ and $T_xW^{s, F}_a(x)=E^{s, F}_a(x)$.
\item $\wtd \alpha(b)(W^{\chi, F}(x)) = W^{\chi, F}(\wtd \alpha(b) (x)) $ and  $\wtd \alpha(b)(W^{s, F}_a(x)) = W^{s, F}_a(\wtd \alpha(b) (x)) $
for all $b\in A$ and $\mu$-a.e.\ $x$.
\item For $\mu$-a.e.\ $x$,  $$W^{s, F}_a(x)=\{ y\in \scrF(x): \limsup \frac 1 n \log d(\wtd \alpha(a^n)(x),\wtd \alpha(a^n)(y)) <0\}.$$
\item There exist $\{a_1,\dots, a_k\}\subset A$ such that $$\chi= \{\lambda\in \Lcal^{\widetilde{\alpha}, F}(\mu):\lambda(a_i)<0 \text{ for all } 1\le i \le k\}.$$  
Then $$
E^{\chi, F}(x)=\bigcap_{1\le i \le k} E^{s, F}_{a_i}(x)
$$
and $W^{\chi, F}(x)$ is the path-connected component of $ \bigcap_{1\le i \le k} W^{s, F}_{a_i}(x)$.
\end{enumerate}
\end{proposition}
The manifold $W^{\chi, F}(x)$ is the \emph{fiberwise coarse Lyapunov manifold} associated to $\chi$ through $x$ and the manifold $W^{s, F}_a(x)$  is the \emph{fiberwise stable manifold}  for the dynamics $\wtd \alpha(a)$ through $x$.  The 
collection of leaves $\{W^{\chi, F}(x)\}$ forms a partition of a full-measure subset of $(M^\alpha, \mu)$; we denote the associate measurable lamination of $M^\alpha$ by $\Wcal^{\chi, F}(x).$

Given a (total) coarse Lyapunov exponent $\chi \in \what  \Lcal^{\widetilde{\alpha}}(\mu)$, we similarly define (total) coarse Lyapunov manifold $W^{\chi}(x)$ through $\mu$-a.e.\ $x$.  Of particular interest in the sequel will be the following: $\lambda\in \Lcal^{\widetilde{\alpha}, F}(\mu)$ is a fiberwise Lyapunov exponent and $\beta\in \Phi(A,G)$ is a root satisfying the following:
\begin{enumerate}
\item no other fiberwise Lyapunov exponent is positively proportional to $\lambda$;
\item $2\beta$ and $\frac 1 2 \beta$ are not roots;
\item $\beta$ and $\lambda$ are positively proportional.
\end{enumerate}
Then $\chi = \{\lambda\}\cup\{\beta\}$ is the associated (total) coarse Lyapunov exponent.  Since the associated fiberwise coarse Lyapunov exponent $\chi\cap  \Lcal^{\widetilde{\alpha}, F}(\mu)$ is a singleton $\{\lambda\}$, in this case we write 
$W^{\lambda, F}(x):=W^{\chi, F}(x)$ for the associated coarse fiberwise Lyapunov manifold.  
Then, for $\mu$-a.e.\ $x$, the (total) coarse Lyapunov manifold $W^{\chi}(x)$ through $x$ is then the $U^\beta$-orbit of $W^{\lambda, F}(x)$.

We will be espeically interested in the above when  $\dim E^{\chi, F}(x)=1$  and $\dim \lieg ^\beta =1$ and so $W^\chi(x)$ is 2-dimensional and bifoliated by transverse $C^r$ curves, the $U^\beta$-orbits and images of $W^{\lambda, F}(x)$ under translation by elements of $U^\beta$.

\subsection{Leafwise measures}
For this subsection, we assume that there exists $A$-invariant, $A$-ergodic measure $\mu$ on $M^{\alpha}$ such that $\mu$ has exponentially small mass at $\infty$.  Note that for a coarse root $[\beta]\subset\Phi(A,G)$, we have a partition $\Ucal^{[\beta]}$ of $M^\alpha$ by $U^{[\beta]}$-orbits. 
 Let $\Tcal$ denote one of the following laminations:  $\Wcal^{\chi}$ for $\chi\in \what \Lcal^{\widetilde{\alpha}, F}(\mu)$,  $\Wcal^{\chi, F}$ for $\chi\in \what \Lcal^{\widetilde{\alpha}}(\mu)$, or $\Ucal^{[\beta]}$.  
\begin{definition}
A measurable partition $\eta$ is subordinate to $\Tcal$ if for $\mu$-almost every $x$, \begin{enumerate}
\item $\eta(x)\subset \Tcal(x)$,
\item $\eta(x)$ contains an open neighborhood of $x$ (in the immersed topology) in $\Tcal$, and
\item $\eta(x)$ is precompact in $\Tcal(x)$ (in the immersed topology).
\end{enumerate}
\end{definition}
A measurable partition $\eta$ induces a system of conditional measures, denoted  $\{\mu_{x}^{\eta}\}$, defined for $\mu$-almost every $x$. We patch together conditional measures with respect to subordinate measurable partitions in order to obtain locally finite (possibly infinite) measures on leaves of $\Tcal$---uniquely defined up to normalization---with the following properties.  
\begin{proposition}
There exists a measurable family of locally finite measures $\mu_{x}^{\Tcal}$ (called the \emph{leafwise measures})---defined for $\mu$-almost every $x\in M^{\alpha}$---with the following properties:
\begin{enumerate}
\item $\mu_{x}^{\Tcal}$ is a  Radon measure on $\Tcal (x)$ which is well defined up to normalization.
\item For $a\in A$ and $\mu$-almost every $x$, $\wtd{\alpha}(a)_{*}\mu_{x}^{\Tcal}\propto \mu_{\wtd{\alpha}(a)(x)}^{\Tcal}$.
\item For $\mu$-almost every $x$ and $\mu_{x}^{\Tcal}$-almost every $y$, $\mu_{x}^{\Tcal}\propto \mu_{y}^{\Tcal}$.
\item For any measurable partition $\eta$   subordinate to $\Tcal$ and for $\mu$-almost every $x$, the conditional measure 
$\mu_{x}^{\eta}$ associated to the atom $\eta(x)$ is given by 
 \[\mu_{x}^{\eta}=\frac{\restrict{\mu^{\Tcal}_{x}}{\eta_x}}{\mu^{\Tcal}_{x}(\eta(x))}. \] 
\end{enumerate} 
Moreover, such a system of leafwise measures is unique up to null sets and normalization.
\end{proposition} 
Above, given two radon measures $\nu_1$ and $\nu_2$ on a manifold $N$ we write $\nu_1\propto \nu_2$ if $\nu_2 = c \nu_1$ for some scalar $c>0$.

When $\Tcal=\Wcal^{\chi}, \Wcal^{\chi, F}$, or $\Ucal^{[\beta]}$ we denote the system of leafwise measure by $\mu^{\chi}$, $\mu^{\chi, F}$, and $\mu^{U^{[\beta]}}$, respectively. 
Given a coarse fiberwise Lyapunov exponent $\chi\in \what \Lcal^{\widetilde{\alpha}}(\mu)$ that is a singleton $\chi=\{\lambda\}$, we also write 
$\mu^{\lambda, F}=\mu^{\chi, F}$ for the associated system of leafwise measures.  

\subsection{Normal forms for conformal contractions}
Let $\alpha\colon \Gamma\to \diff^r(M)$ be an action.  
Let $A\subset G$ be a maximal $\R$-split Cartan subgroup and let $\mu$ be an ergodic, $A$-invariant  Borel probability measure on $M^\alpha$.

Let $\lambda$ be a fiberwise Lyapunov exponent functional.  
Although many of the following results and constructions hold under more general hypothesis, we will focus only the case that 
$\dim E^{\lambda, F}_x=1$ for a.e.\ $x$ and that no other fiberwise Lyapunov exponent is positively proportional to $\lambda$.  In particular, the coarse fiberwise Lyapunov exponent $\chi=[\lambda]\in \what \Lcal^{\widetilde{\alpha}}(\mu)$ containing $\lambda$ satisfies  $[\lambda]=\{\lambda\}$ and $\dim E^{\chi, F}(x)=1$ for a.e.\ $x$.   

In this  situation, one may construct normal form coordinates on each  leaf of the Lyapunov foliation $\Wcal^{\lambda, F}(x)$ as in \cite[Thm.\ 4]{KRHarith}.  We summarize the properties of here: 
\begin{lemma}\label[lemma]{normalitem} 
Suppose $\lambda$ is a non-zero fiberwise Lyapunov exponent with $\dim E^{\lambda, F}_x=1$ that is not positively proportional to any other fiberwise Lyapunov exponent functional. 
Then for $\mu$-a.e.\ $x\in M^\alpha$, there exists a unique $C^r$ diffeomorphism $$\Phi_{x}^{\lambda, F}\colon E^{\lambda, F}(x)\to W^{\lambda, F}(x)$$ such that the following hold: 
\begin{enumerate}
\item $\Phi_{x}^{\lambda, F}(0)=x$ and $D_0\Phi_{x}^{\lambda, F}= \id$. 
\item $\{\Phi_{x}^{\lambda, F}\}$ varies measurably for $x$.  
\item 
For  $b\in A$ and a.e.\ $x$,  the map 
$$\left(\Phi_{\widetilde{\alpha}(b)(x)}^{\lambda, F}\right)^{-1}\circ \widetilde{\alpha}(b)\circ\Phi_{x}^{\lambda, F} \colon E^{\lambda, F}(x)\to E^{\lambda, F}(\wtd \alpha(b)(x))$$ coincides with the restriction of the derivative 
$$\restrict{D^F\wtd \alpha(b)}{ E^{\lambda, F}(x)}\colon  E^{\lambda, F}(x)\to  E^{\lambda, F}(\wtd \alpha(b)(x)).$$

\item Moreover, the above uniquely determine the family of parametrizations $\{\Phi_{x}^{\lambda, F}\}$. 
\end{enumerate}
\end{lemma}

\subsubsection{Further properties of normal forms on 1-dimensional fiberwise Lyapunov manifolds}
As above, suppose that $\lambda$ is a fiberwise Lyapunov exponent for an ergodic, $A$-invariant probability measure $\mu$ with  $\dim E^{\lambda, F}(x) =1$ for a.e.\ $x$ and such that no other fiberwise Lyapunov exponent is positively proportional to $\lambda$.  
Let $\chi$ denote the (total) coarse Lyapunov exponent containing $\lambda$.  Then either $\chi = \{\lambda\}$ or $\chi = \{\lambda\}\cup [\beta]$ where $\beta$ is a root that is positively proportional to $\lambda$.

We suppose that $\beta$ is a root positively proportional to $\lambda$ so that $\chi = \{\lambda\}\cup [\beta]$.  For almost every $x$, the manifold $W^\chi(x)$ is $C^r$ subfoliated by $U^{[\beta]}$-orbits.  Given $y = \wtd\alpha(u)(x)$ for $u\in U^{[\beta]}$, write $W^{\lambda, F}(y) =\wtd\alpha(u)(W^{\lambda, F}(x))$ and for $y'\in W^{\lambda, F}(y) $ write 
$E^{\lambda, F}(y') = T_{y'}W^{\lambda, F}(y')$.  Then (for a.e.\ $x$),  $E^{\lambda, F}(\bullet)$ is a H\"older continuous, everywhere defined, orientable,  1-dimensional bundle on $W^{\chi}(x)$.

Fix $a_0\in A$ with $\lambda(a_0)<0$.  Because $x'\mapsto \|\restrict{D_{x'}\wtd \alpha(a_0) }{E^{\lambda, F}(x')}\|$ is H\"older continuous when restricted to $W^\chi(x)$ and because   $d(\wtd \alpha(a_0^n)(x), \wtd \alpha(a_0^n)(x'))$ approaches 0 exponentially fast for $\mu$-a.e.\ $x$ and every $x'\in W^\chi(x)$, it follows for $\mu$-a.e.\ $x$ and every $x'\in W^{\chi}(x)$ that 
the limit $$c_{x,x'}(a_0):=\lim _{n\to \infty} \frac { \|\restrict{D_{x'}\wtd \alpha(a^n) }{E^{\lambda, F}(x')}\|}{\restrict{ \|D_x\wtd \alpha( a^n)}{E^{\lambda, F}(x)}\|}$$
converges and is non-zero.   
For all such $x$ and $x'$, defining the linear map $H_{x,x'}\colon E^{\lambda, F}(x) \to E^{\lambda, F}(x')$ to be the unique linear map preserving either choice of orientation on the bundle $E^{\lambda, F}(\bullet)$ with $\|H_{x,x'}\| = c_{x,x'}(a_0)$.
Given arbitrary $x', x''\in  W^{\chi}(x)$ we define
$$H_{x',x''}:= H_{x,x''}\circ H_{x,x'}\inv.$$
  
By construction of the map $H_{x,x'}$, for every $b\in A$, $u\in U^{[\beta]}$,    a.e.\ $x$, and every $x'\in W^{\lambda, F}(x)$ the following hold:
\begin{align}
\restrict{D_{x'} \wtd \alpha (b)}{E^{\lambda, F}(x')}
 \circ H_{x,x'} &=  H_{\wtd \alpha (x),\wtd \alpha (x')}  \circ  \restrict{D_x \wtd \alpha (b)}{E^{\lambda, F}(x)}\label{eqequiv} \\
\restrict{D_{x'} \wtd \alpha (u)}{E^{\lambda, F}(x')}& = H_{x', \wtd \alpha (u)(x')}.\label{holonom}
\end{align}
Indeed, \eqref{eqequiv} follows from definition.  For \eqref{holonom}, we have 
\begin{align*}
\|\restrict{D_{x'}\wtd \alpha(u)}{E^{\lambda, F} (\wtd\alpha ( u )(x'))}
&\|\cdot \|\restrict{D_{\wtd\alpha ( u )( x')}\wtd \alpha(a^n)}{E^{\lambda, F}(x')}\| \\&=
\|\restrict{D_x' (\wtd\alpha (a_0^n  u)) }{E^{\lambda, F}(x')}\| 
\\&=\|\restrict{D_x' (\wtd\alpha ( ( a_0^n u a_0^{-n})a_0^n)) }{E^{\lambda, F}(x')}\| 
\\&=\|\restrict{ D_{\wtd\alpha (  a_0^n)( x')} \wtd\alpha (  a_0^n u a_0^{-n})}{E^{\lambda, F}(\wtd\alpha ( a_0^n )(x'))}\| \cdot  \|\restrict{D_x' (\wtd\alpha (a_0^n))}{E^{\lambda, F}(x')}\| 
\end{align*}
and the claim follows since $ a_0^n u a_0^{-n}$ converges to the identity in $U^\beta$ as $n\to \infty$ and so $\| D_{\wtd\alpha (  a_0^n)( x')} \wtd\alpha (  a_0^n u a_0^{-n})\|\to 1$ (perhaps passing to a subsequence if needed).

Now, consider $x\in M^\alpha$ for which $\Phi_x^{\lambda, F}$ is defined.  Take $x'\in W^{\lambda, F}(x)$ and  $w\in E^{\lambda, F}(x)$ such that $x'= \Phi_x^{\lambda, F}(w)$.  
Under the canonical identification of $T_0E^{\lambda, F}(x)$ (the domain of $H_{x,x'}$) with $T_wE^{\lambda, F}(x)$
(the domain of $D_w \Phi^{\lambda, F}_x$), we claim that \begin{equation}D_w \Phi^{\lambda, F}_x= H_{x,x'}.\end{equation}
Indeed,  for $n\ge 0$ we have 
\begin{align*}
\restrict{D_{x'}\wtd \alpha(a_0^n) }{E^{\lambda, F}(x')}\circ D_w \Phi_x^{\lambda, F}
&= \restrict{ D_w (\wtd \alpha(a_0^n)}{E^{\lambda, F}(x)}\circ \Phi_x^{\lambda, F})
\\&=  D_w ( \Phi_{\wtd \alpha(a_0^n)(x)} ^{\lambda, F} \circ \restrict{D_x\wtd \alpha(a_0^n))}{E^{\lambda, F}(\alpha(a_0^n)(x))}
\\&=  D_{D_x\wtd \alpha(a_0^n) w}  \Phi_{\wtd \alpha(a_0^n)(x)} ^{\lambda, F} \circ\restrict{ D_x\wtd \alpha(a_0^n)}{E^{\lambda, F}(x)}. 
\end{align*}
As $n\to \infty$, we have that $D_x\wtd \alpha(a_0^n) w\to 0$ and so $$\|\restrict{D_{D_x\wtd \alpha(a_0^n) w}  \Phi_{\wtd \alpha(a^n)(x)} ^{\lambda, F}}{E^{\lambda, F}(\alpha(a_0^n)(x))}\|\to 1$$ along subsequences of Poincar\'e recurrence to sets on  which  $\bullet \mapsto \|D_{\bullet}  \Phi_{\wtd \alpha(a_0^n)(x)} ^{\lambda, F}\|$ is uniformly continuous.  This  shows that $\|D_w \Phi_x^{\lambda, F}\| = c_{x,x'}(a_0)$.

With the above observations, we derive that the $U^{[\beta]}$-action is affine relative to the coordinates  $\Phi^{\lambda, F}_\bullet$.  
\begin{lemma}\label[lemma]{holonomy is affine}
Suppose that $\beta\in \Phi(A,G)$ is a root that is positively proportional to $\lambda$.  Let $\chi=[\lambda]=[\beta]$ denote the associated (total) coarse Lyapunov exponent.  
For $\mu$-a.e. $x$ and $\mu_x^\chi$-a.e.\ $y\in W^\chi(x)$, writing $y = u\cdot x'$ where $u\in U^{[\beta]}$ and $x'\in W^{\lambda, F}(x)$, the map
$$\left(\Phi_y^{\lambda, F} \right) \inv \circ\wtd \alpha(u)\circ \Phi_x^{\lambda, F}\colon E^{\lambda, F}(x)\to E^{\lambda, F}(y)$$
is an affine map.
\end{lemma}
\begin{proof}
Let $x'=\Phi^{\lambda, F}_x(w)$. Consider the map 
$E^{\lambda, F}(y)\to E^{\lambda, F}(y) $,
$$\Psi\colon v\mapsto \wtd \alpha (u) \circ \Phi^{\lambda, F}_x (w + H_{x',x} (D_y \wtd \alpha(u \inv) v)).$$
Fix $v\in E^{\lambda, F}(y)$ and let $\wtd v =  w + H_{x',x} (D_y \wtd \alpha(u \inv) v)$,   $\wtd x = \Phi^{\lambda, F}_x (\wtd v)$, and $\wtd y = \wtd \alpha(u) (\wtd x)$. 
We have
\begin{align*}D_v\Psi &= 
\restrict{D_{\wtd x}\wtd \alpha (u)}{E^{\lambda, F}(\wtd x)} \circ   D_{\wtd v} \Phi^{\lambda, F}_x \circ  H_{x',x} \circ \restrict{ D_y \wtd \alpha(u \inv) }{E^{\lambda, F}(y)}\\
&= H_{\wtd x, \wtd y}
 \circ  H_{x,\wtd x} \circ  H_{x',x} \circ H_{y, x'}\\
&=H_{y,\wtd y}.
\end{align*}
We have $\Phi^{\lambda, F}_y(0) = y =\Psi(0)$ and $D_v\Phi^{\lambda, F}_y = D_v\Psi$ for every $v\in  E^{\lambda, F}(y)$ and thus we conclude that $\Psi =\Phi^{\lambda, F}_y$.  
Thus, $v= (\Phi^{\lambda, F}_y)\inv (\wtd \alpha (u) \circ \Phi^{\lambda, F}_x (w + H_{x',x} (D_y \wtd \alpha(u \inv) v))$ and so 
$(\Phi^{\lambda, F}_y)\inv \wtd \alpha (u) \circ \Phi^{\lambda, F}_x $ coincides with the affine map \[v'\mapsto D_{x'}\wtd \alpha(u)H_{x,x'}(v'-w)
= H_{x,y}(v'-w)
.\qedhere\]
\end{proof}
%
%
%

\subsubsection{Normal forms relative to a measurable framing}
In the sequel, we   often   a measurable choice of basis of $E^{\lambda, F}(x)$.  This induces a measurable choice framing $\psi_x^{\lambda, F}\colon \Rbb\to E^{\lambda, F}(x)$.
Relative to such a choice of framing, we write $\Psi_x^{\lambda, F}\colon \R\to W^{\lambda, F}(x)$, 
$$\Psi_x^{\lambda, F}(t) = \Phi_x^{\lambda, F}\circ \psi_x^{\lambda, F}(t).$$

\subsection{Mechanisms for extra invariance}

We recall two sufficient conditions to obtain additional invariance of a probability measure on $M^\alpha$.   The first mechanism is the {\it high entropy method}.  We note that this holds for more general spaces admitting $G$-actions,  but we only formulate the results for the induced $G$-action on the  suspension space $M^\alpha$. 
 \begin{theorem}[High entropy method, \cite{MR2191228} ]\label{thm:highent}

 Let $G$ be a group as in \Cref{thm:AtoG}. Let $\Gamma$ be a lattice in $G$.  For $r>1$, let $\alpha\colon\Gamma\to \diff^{r}(M)$ be an  action on $M$.  Let $M^{\alpha}$ be the suspension $G$-space. Let $\mu$ be an $A$-invariant, $A$-ergodic probability measure on $M^{\alpha}$. 

Let $\beta_{1},\beta_{2}\in \Phi(A,G)$ be roots such that $\delta=\beta_{1}+\beta_{2}\in \Phi(A,G)$. Assume that for $\mu$-almost every $x$, $\mu_{x}^{U^{\beta_{1}}}$ and $\mu_{x}^{U^{\beta_{2}}}$ are non-atomic. Then $\mu$ is $U^{\delta}$-invariant.

 \end{theorem}

Recall that $\Lcal^{\widetilde{\alpha}, F}(\mu)\subset \liea^{*}$ denotes the set of fiberwise Lyapunov functionals for the action $\widetilde{\alpha}|_{A}$ on $M^{\alpha}$ relative to the ergodic, $A$-invariant measure $\mu$.  
We say that a root $\beta\in \Phi(A,G)$ is \emph{non-resonant} with fiberwise Lyapunov functionals if, for every fiberwise Lyapunov functional $\lambda\in \Lcal^{\widetilde{\alpha}, F}(\mu)$, $\beta$ is not positively proportional to $\lambda$. 
\begin{theorem}[\cite{MR4502594} Proposition 5.1]\label{thm:nonres}
Let $G$ be a group as in \Cref{thm:AtoG}. Let $\Gamma$ be a lattice in $G$ and for $r>1$, let $\alpha\colon\Gamma\to \diff^{r}(M)$ be an action on $M$. Let $M^{\alpha}$ be the suspension $G$-space. Let $\mu$ be an ergodic, $A$-invariant Borel probability measure on $M^{\alpha}$ that projects to the Haar measure on  $G/\Gamma$.   

Let $\beta\in \Phi(A,G)$ be a non-resonant root. Then  the measure $\mu$ is $U^{\beta}$-invariant for the action $\widetilde{\alpha}$.
\end{theorem}

\subsection{Properties of fiber entropy}%
We still assume the setting and   notation from \Cref{mainresultssection}.  Recall that $\Fsc$ is the measurable partition  fibers of $p\colon M^{\alpha}\to G/\Gamma$. For notational simplicity, given $\gamma\in \Gamma$ or $g\in G$,  an $\alpha(\gamma)$-invariant measure $\mu_0$ on $M$ or a $\wtd \alpha (g)$-invariant measure $\mu$ on $M^\alpha$, we often write 
$h_{\mu_0}(\gamma ):= h_{\mu_0}( \alpha(\gamma))$ or 
$h_\mu(g\mid \Fsc):= h_\mu(\wtd \alpha(g)\mid \Fsc)$, respectively, for the metric entropy or fiber metric entropy of the corresponding transformation.  
 \subsubsection{Entropy formula}
We often make use of the following relations between Lyapunov exponents and entropy. Often, given an absolutely continuous $\Gamma$- or $G$-invariant measure, we use the following two facts to
relate the existence a positive entropy element of the action with  properties of the homomorphism arising in \cref{thm:ZCSRGamma,thm:ZCSRG}.

\begin{theorem}[Margulis--Ruelle inequality, \cite{MR1314494,MR516310}]\label{thm:MRineq}
For any $\alpha(\gamma)$-invariant probability measure $\mu_{0}$ on $M$, we have \[h_{\mu_0}(\gamma)\le \int_{M}\sum_{\lambda(x):\lambda(x)>0}\lambda(x)\, d\mu_{0}(x).\] On the suspension, for any $\wtd{\alpha}(g)$-invariant measure $\mu$ with exponentially small mass at $\infty$, we have 
\[h_{\mu}(g\mid\Fsc)\le \int_{M^{\alpha}}\sum_{\lambda^{F}(x):\lambda^{F}(x)>0}\lambda^{F}(x)\, d\mu(x).\] 
\end{theorem}
On the suspension, we note that when $\mu$ has exponentially small mass at $\infty$, the  $\log$-integrability of $C^1$-norms (relative to choice of norms in \Cref{sec:norms}) needed to apply  \cite{MR1314494} holds.

When the invariant measure is absolutely continuous, the inequality \cref{thm:MRineq} becomes an equality.  
\begin{theorem}[Pesin's entropy formula, {\cite[Section 10.4]{NUHbook}}]\label{thm:Pesinent}
Let $\mu_{0}$ be an $\alpha(\gamma)$-invariant absolutely continuous (with respect to the Lebesgue measure class) probability measure on $M$. Then \[h_{\mu_{0}}(\gamma)=\int_{M}\sum_{\lambda(x):\lambda(x)>0} \lambda(x)\,  d\mu_{0}(x).\]  
\end{theorem}

\subsubsection{Product structure of entropy}
 Let $\mu$ be an ergodic, $A$-invariant  probability measure on $M^\alpha$ with exponentially small mass at $\infty$.  
The following adaptation of   \cite[Thm.\  13.1]{MR4599404}  to our notation and setting will be used frequently.
 \begin{theorem}[\cite{MR4599404}]\label{thm:coent}  
 For any $a\in A$, we have \begin{equation}\label{eq:coent} h_{\mu}(a\mid \Fsc)=\sum_{i:\chi_{i}^{F}(a)>0}h_{\mu}(a\mid  \Wcal^{\chi_{i}^{F}, F}).\end{equation} \end{theorem}
Here, $\{\chi_i^F\}$ denotes the fiberwise coarse Lyapunov exponents and $h_{\mu}(a\mid  \Wcal^{\chi_{i}^{F}, F})$ denotes the contribution to fiber entropy from the lamination associated with $\chi_i^F$.  See \cite{MR4599404} for definitions.

 \cref{thm:coent} implies that the (fiberwise) entropy $a\mapsto h_{\mu}(a\mid \Fsc)$ is semi-norm on $A$ and thus, if non-zero, it vanishes on at most a hyperplane.  \cref{thm:coent}  also implies  subadditivity of fiber entropy:  
  
\begin{theorem}[Subadditivity of fiber entropy, see \cite{MR4599404,MR1213080}]\label{thm:subadditive}
For any $a_{1},a_{2}\in A$, \[h_{\mu}(a_{1}a_{2}\mid \Fsc) \le h_{\mu}(a_{1}\mid \Fsc)+h_{\mu}(a_{2}\mid\Fsc).\]
\end{theorem}

\subsubsection{Semicontinuity of entropy}\label{sec:semicontientropy}
For the following, we heavily use that the action $\alpha\colon \Gamma\to \Diff^\infty(M)$
inducing the suspension space $M^\alpha$  (and the  induced $G$-action) is  $C^\infty$.  This is because we appeal to a fibered version of the classical  results of Newhouse and Yomdin, explicated in \cref{app:yomdinnewhouse}.

\begin{proposition}\label[proposition]{prop:uscfiberentropy}
Fix $g\in G$ and let $\scrM$ be a collection of $g$-invariant Borel probability measures on $M^\alpha$ with {uniformly exponentially small mass at $\infty$}.  Then for any sequence $\{\mu_j\}\subset \scrM$ and any subsequential limit point $\mu_\infty$ of $\mu_j$ we have 
\begin{equation}
h_{\mu_\infty}(g\mid \Fsc)\ge \limsup_{j\to \infty}h_{\mu_j}(g\mid \Fsc).  
\end{equation}
\end{proposition}

\cref{prop:uscfiberentropy} follows from 
\cref{lem:locentzeromeansUSC} and \cref{prop:Yomdin} in \cref{app:yomdinnewhouse}.  
To apply the results of \cref{app:yomdinnewhouse},  we take $Z= M^\alpha$,  $Y= G/\Gamma$, and $F\colon Z\to Z$ to be the action  of translation by $g\in G$.    We let $I\colon G/\Gamma\times M\to Z$,  be the Borel trivialization  $(g,x)\mapsto [g,x]$ where $g\in D$ is the unique representative of $g\Gamma$ in a choice of Siegel fundamental domain $D\subset G$ for $\Gamma$.  We may moreover choose the Borel domain $D$ such that $\mu_\infty(\partial D)=0$.   We recall   the family of fiber metrics whose properties are outlined in \cref{sss:susp} are uniformly comparable over Siegel sets.

The uniform integrability of the family of $y\mapsto \log R_{y,k}$ required to apply \cref{prop:Yomdin} follows from \eqref{tempered} and the assumption the collection $\scrM$ has uniformly exponentially small mass at $\infty$.  
Indeed, let $\phi\colon G/\Gamma\to [1, \infty)$ be  the function $$\phi(y) = \phi_k(g,y) = R_{y,k}.$$
By \eqref{tempered}, there are $A, B>0$ such that  $\log \phi(y)\le A d(y, \1\Gamma) + B$ and thus there exists $\tau>0$ such that 
$$L:=  \sup_{\mu\in \scrM} \int ( \phi(y) )^\tau \, d\mu(y)= \sup_{\mu\in \scrM} \int e^{\tau \log( \phi(y) )} \, d\mu(y)<\infty.$$  
This exponential moment on $\log \phi$ then yields uniform integrability with respect to $\scrM$.  Indeed, 
for $\mu\in \scrM$ we have
$$\mu\bigl(\{y\in G/\Gamma: \log \phi(y) \ge T\}\bigr) \le L e^{-\tau T}$$ 
and so 
\begin{align*}
\sup_{\mu\in \scrM} &\int_{\log \phi(z)\ge K} \log \phi(z) \, d\mu(z) 
\\&\le \sup_{\mu\in \scrM} \left( K \mu\bigl(\{y\in G/\Gamma: \log \phi(y) \ge K\}\bigr)
+
\int_{T\ge  K} \mu\bigl(\{y\in G/\Gamma: \log \phi(y) \ge T\}\bigr)  \, d T\right)
\\&\le KLe^{-\tau K}+  \int_{T\ge  K} L e^{-\tau T} \, d T
\end{align*}
which approaches $0$ as $K\to \infty$.

	\subsection{Consequences of Ratner's measure classification and equidistribution theorems}\label{sec:ratnerpre}

	For the dynamics on the base $G/\Gamma$, we often use various consequences of  Ratner's measure classification in order to produce an $\wtd{\alpha}(A)$-invariant measure on $M^{\alpha}$   with extra invariance.    

	First, the following gives sufficient conditions for a measure to be invariant under an opposite root.  
\begin{proposition}[{ \cite[Prop.\  2.1]{MR1135878}}]\label{thm:opproot}
Let $A$ be a split Cartan subgroup of~$G$, let $\beta \in \Phi(A,G)$, and let $\mu$ be a Borel probability measure on $G/\Gamma$. If $\mu$ is invariant under both the coarse root group~$U^{[\beta]}$  and the diagonal subgroup $\{d_\alpha^\R\}$ of $\alpha$ in $A$, then $\mu$ is also invariant under the coarse root group~$U^{[-\beta]}$.
\end{proposition}

For the second result, we recall that when averaging an $A$-invariant Borel probability measure along a \Folner sequence in a subgroup that centralizes (or normalizes)~$A$, any weak-$*$ limit point is an $A$-invariant Borel probability measure. Ratner's equidistribution theorem implies that, on the homogeneous space $G/\Gamma$, {$A$-invariance} is preserved when passing to limits of averages by unipotent subgroups normalized by $A$.

 \begin{proposition}[{\cite[Prop.\ 6.2(b)]{Brown:2020aa}, \cite[Prop.\ 4.10]{Brown:2021aa}}]\label{thm:averaginghomo2}

Let $A$ be a split Cartan subgroup of~$G$,
let $\mu$ be an $A$-invariant Borel probability measure on $G/\Gamma$,
let $U$ be a unipotent subgroup that is normalized by $A$,
and
let $\{F_j\}$ be a \Folner sequence of centered intervals in~$U$.
Then
\begin{enumerate}
\item the family $\{F_j\ast \mu\}$ is uniformly tight, and
\item\label{jlkjlkj} every weak-$*$ subsequential limit of $\{F_j\ast \mu\}$ is $A$-invariant.
\end{enumerate}
\end{proposition}
See \cite[Def.\ 4.5]{Brown:2021aa} for the definition of a centered interval.  
 We remark that conclusion \eqref{jlkjlkj} employs the equidistribution theorem of Ratner and its extension by Shah; see \cite{MR1291701,MR1106945}.

 We also prove the following corollary of Ratner's measure classification theorem which will be useful to obtain invariance under $\sl(2)$-triples.

\begin{proposition}\label[proposition]{SL2}
Let $G$ be a connected real algebraic Lie group and $\Gamma$ be a lattice in $G$. Let $H$ be a closed connected subgroup of $G$ with  $\Lie(H)\simeq \mathfrak{sl}_{2}(\Rbb)$. Denote by  $K$ and $U$ the  closed connected subgroups  of $H$ with $\Lie(K)\simeq \mathfrak{so}(2)$ and $\Lie(U)\simeq \left\{\begin{bmatrix} 0 &t \\ 0 &0 \end{bmatrix}:t\in\Rbb\right\}$.

Let $\mu$ be a $K$-invariant probability measure on $G/\Gamma$. Let $U\ast \mu$ denote the weak-$*$ limit of \[\left\{\frac{1}{T}\int_{0}^{T} (u_{t})_{*}\mu dt : T\in \Rbb\right\}.\] Then $U*\mu$ is $H$-invariant.  
\end{proposition}

\begin{proof}[Proof of \Cref{SL2}]
We recall the following consequences of Ratner's measure classification theorem \cite{MR1135878}.   Let $\Acal$ denote the set of closed connected subgroups $F$ of $G$ such that 
\begin{enumerate}
\item $F\cap \Gamma$ is a lattice in $F$, and 
\item there is an unipotent element $u\in F$ such that $u$ acts on $F/(F\cap \Gamma)$ ergodically.
\end{enumerate}
We have that $\Acal$ is countable; see \cite{MR2648693}. 

For each $g\Gamma \in G/\Gamma$, there exists a unique subgroup $F_{g}\in \Acal$ and a Borel probability measure $m_{g\Gamma}^{U}$ such that 
	\begin{enumerate}
\item $U\subset gF_{g}g^{-1}$,   
	\item the orbit closure  $\overline{Ug\Gamma}$ is the translate of the  closed $F_g$-orbit $g \cdot F_{g}\Gamma$
\item $m_{g\Gamma}^{U}$ is a normalized $gF_{g}g^{-1}$-invariant Haar measure on the orbit $g \cdot F_{g}\Gamma= \left(gF_{g}g^{-1}\right)g\Gamma,  $ and
\item  $\lim _{T\to \infty} \frac{1}{2T}\int_{-T}^{T} (u_{t})_{*} \delta_{g\Gamma} \, dt = m_{g\Gamma}^{U}$.
	\end{enumerate}

Let $m_{K}$ be the normalized Haar measure on $K$. By  $K$-invariance of $\mu$, we have
\[U*\mu=\int_{G/\Gamma} m_{g\Gamma}^{U} \, d\mu (g\Gamma)=\int_{G/\Gamma}\int_{K} m_{kg\Gamma}^{U} \, dk\,  d\mu(g\Gamma).\]  


As $\Lie(H) = \sl(2,\R)$, the only subgroups of $H$ generated by unipotent elements are 
conjugates of $U$ or all of $H$.  Moreover, the centralizer of $U$ in $K$ coincides with center of $H$. 
Thus, for $k,k'\in K$,   $k^{-1}Uk$ and $k'^{-1}Uk'$ generate $H$ unless 
$k\inv k'$ is central.  

Fix $g\Gamma \in G/\Gamma$.  Since  $\Acal$ is countable and $K$ is uncountable, there is at least one $F\in \Acal$ such that $F_{kg}=F$ for a
$m_{K}$-positive measure set of $k\in K$.

Take $k,k'\in K$ with $k\inv k'$ non-central such that $F:=F_{kg}=F_{k'g}$.  
Since  $U\subset kgFg^{-1}k^{-1}$ and $U\subset k'gFg^{-1}k'^{-1}$, it follows that $H\subset gFg^{-1}$ and thus,   
for $m_{K}$-almost every $k\in K$, $m_{kg\Gamma}^{U}$ is $H$-invariant and thus $U*\mu$ is also $H$-invariant.
\end{proof}


We will frequently use  \cref{SL2} in the following form to upgrade the invariance of certain measures without losing any dynamical properties.    

\begin{corollary}\label[corollary]{SL2coro}
Let $\Gamma$ be a lattice in a simple real algebraic Lie group $G$. Let $\alpha\colon \Gamma\to \Diff^{\infty}(M)$ be an action 
and let $\mu$ be a probability measure on the suspension space $M^{\alpha}$ with 
exponentially small mass at $\infty$.  

Suppose   there exists a restricted root $\delta\in { \Phi(A,G)}$ and a subgroup $A_0\subset A$ such that
\begin{enumerate}
\item $A_0\subset  \ker \delta$, 
\item $\mu$ is $A_0$-invariant, and
\item there exists $g\in A_0$ such that $h_{\mu}(g\mid \Fsc) >0$.
\end{enumerate}

Let $H$ be the standard $\R$-rank-$1$ subgroup generated by $\delta$. Then there exists a probability measure  $\mu'$ on $M^\alpha$ such that 
\begin{enumerate}
\item $\mu'$ has exponentially small mass at $\infty$; 
\item $\mu'$ is invariant under the diagonal of $\delta$ in $A$;  
\item $\mu'$ is $A_0$-invariant; 
\item $p_{*}\mu':=\overline{\mu'}$ is an $H$-invariant probability measure on $G/\Gamma$;
\item $h_{\mu'}(g\mid \Fsc)>0$.
\end{enumerate}
\end{corollary}
\begin{proof}
Let $A'\subset A\cap H$ be a $\R$-split torus in $H$ contained in $A$.    Note that $A'$ is the diagonal of $\delta$ in $A$.  
Fix a $\mathfrak{sl}_{2}$ triple $\lieh'$ in $\lieg$ containing non-zero vectors in both $\lieg^\delta$ and $\Lie(A')$.     
 Let $H'$ be the connected closed subgroup if $G$ with $\Lie(H) = \lieh'$ and fix an      
 Iwasaw decomposition   $H'= K'A'U'$ of $H'$ with  $\R$-split torus   $A'$.     
 
 We first average $\mu$ over $K'$ to obtain a $K'$-invariant  probability measure $\mu_{1}$. Since $K'$ is compact, $\mu_{0}$ has exponentially small mass at $\infty$. 

We may average $\mu_0$ over a F\"olner sequence of centered intervals in $U'$
 and take any subsequential limit point to obtain a probability measure $\mu_1$ on $M^\alpha$; from \cref{thm:BFHunip} below, $\mu_1$ has exponentially small mass at $\infty$ and from 
 \Cref{SL2}, the projection $p_*\mu_1$ of $\mu_1$ to $G/\Gamma$ is $H'$-invariant.  
 We may thus average over a F\"olner sequence in $A'$
 and take any subsequential limit point to obtain an $A'$-invariant probability measure $\mu_2$ on $M^\alpha$ with 
  $p_*\mu_2=  p_*\mu_1$.

Let  $H= KA'N$ be an Iwasawa decomposition of $H$ with $\R$-split torus   $A'$.  
We may average $\mu_2$ over a F\"olner sequence of centered intervals in $N$
 and take any subsequential limit point to obtain a probability measure $\mu_3$ on $M^\alpha$;
 again,  \cref{thm:BFHunip} below implies $\mu_3$ has exponentially small mass at $\infty$.  
 Since $\mu_2$ is $A'$-invariant, the measure $p_{*}\mu_{3} $ on $G/\Gamma$ remains $A'$-invariant.  { By \cref{thm:opproot},  the measure $p_{*}\mu_{3} $  is $H$-invariant.  }
Finally, we again average over a F\"olner sequence of centered intervals in $A'$
 and take any subsequential limit point to obtain an $A'$-invariant probability measure $\mu_4$ on $M^\alpha$  such that 
 $p_{*}\mu_{4}= p_{*}\mu_{3}$ is  $H$-invariant and has exponentially small mass at $\infty$.

Finally, since  $H$ (and thus $H'$) commutes with $A_0$,  the measures $\mu_{0}, \dots, \mu_4$ are all $A_0$-invariant and by \cref{prop:uscfiberentropy} satisfy 
$$h_{\mu_{j}}(g\mid \Fsc)= h_{\mu}(g\mid \Fsc)\ge h_{\mu}(g\mid \Fsc)  >0$$ 
for every $j\in \{0, \dots, 4\}.$
 \end{proof}

\subsection{Averaging operations on measures and fiber entropy}

Given the $G$-action on $M^\alpha$, 
we frequently average certain measures on $M^\alpha$ over  (centered) F\"olner sequences in certain amenable subgroups in $G$.  First, we recall the following consequence of quantitative non-divergence of unipotent averages.  See \cite{Brown:2021aa} for the definition of  centered intervals in a unipotent subgroup.

\begin{lemma}[{\cite[Lem.\  4.8]{Brown:2021aa}}]\label{thm:BFHunip}

Suppose
	\begin{enumerate}
	\item $\{\mu_n\}$ is a sequence of probability measures on~$G/\Gamma$ with uniformly exponentially small mass at $\infty$ and
	\item $\{\interval{U}_n\}$ is a sequence of centered intervals (relative to a fixed regular basis $\calB$) in a unipotent subgroup~$U$ of~$G$.
	\end{enumerate}
Then the family of measures $\{\interval{U}_n * \mu_n\}$ has uniformly exponentially small mass at $\infty$.
\end{lemma}

We summarize a number of averaging operations that we frequently perform throughout the paper.  
\begin{proposition}\label[proposition]{averagingunipotent}
Let $A$ be a maximal $\R$-split Cartan subgroup of $G$.  Let $ I\subset \Phi(A,G)$ be a collection of roots that is positive with respect to some choice of simple roots $\Delta(A,G)$ and let $U^{[I]}$ be the unipotent subgroup generated by $\{\lieg^{[\alpha]}: \alpha\in I\}$.  
Let $A_0\subset A$ and $R\subset G$ be closed subgroups with  $R\subset C_G(U^{[I]})$ and $A_0\subset  N_G(R)$.

Let $\mu_0$ be an $A_0$-invariant Borel probability measure on $M^\alpha$ such that  
\begin{enumerate}[label=(\alph*)]
\item $p_* \mu_0$ has exponentially small mass at $\infty$, 
\item $p_* \mu_0$ is $R$-invariant 
\item there is $a_0\in A_0 \cap C_G(U^{[I]})$ such that $h_{\mu_0} (a_0 \mid \Fsc)>0$.
\end{enumerate}
Then there exists a Borel probability measure $\mu_1$ on $M^\alpha$ such that 
\begin {enumerate}
\item $\mu_1$ is $A_0$-invariant, 
\item $p_*\mu_1$ is $A_0$-invariant, $R$-invariant, and $U^{[I]}$-invariant,
\item $p_*\mu_1$ has exponentially small mass at $\infty$, and 
\item $h_{\mu_1} (a_0 \mid \Fsc)>0$.
\end{enumerate}
\end{proposition}
\begin{proof}
Take a  F\"olner sequence of centered intervals $\{F_n\}$ in $U^{[I]}$ and let 
$$\mu_0^n = F_n\ast \mu_0 := \frac 1 {|F_n|} \int_{F_n} h_* \mu_0\ d h.$$
By \cref{thm:BFHunip}, $\{p_*\mu_0^n\}$ has uniformly exponentially small mass at $\infty$  and, in particular, is uniformly tight.  Let $\td \mu_0$ be any subsequential limit point of $\{\mu_0^n\}$.  
Then 
$\td \mu_0$ is  $U^{[I]}$-invariant and remains $a_0$-invariant as $a_0$ centralizes $U^{[I]}$.  Moreover, 
$p_*\td\mu_0$ has exponentially small mass at $\infty$ and, by \cref{thm:averaginghomo2}, $p_*\td\mu_0$ is $A_0$-invariant.  Moreover since $R$ centralizes $U^{[I]}$, $p_*\td\mu_0$ remains  $R$-invariant.  
By \cref{prop:uscfiberentropy}, we have 
$h_{\td\mu_0} (a_0 \mid \Fsc)\ge h_{\mu_0} (a_0 \mid \Fsc)>0.$

We now take a  F\"olner sequence   $\{F_n\}$ in $A_0$ and let 
$$\td\mu_0^n = F_n\ast \mu_0 := \frac 1 {|F_n|} \int_{F_n} h_* \td \mu_0\ d h.$$
Since $p_*\td \mu_0$ is $A_0$-invariant, we have $p_*\td\mu_0^n= p_*\td\mu_0$ thus $\{p_*\td\mu_0^n\}$ has exponentially small mass at $\infty$.  
Let $\mu_1$ be any subsequential limit point of $\{\td \mu_0^n\}$.  
Then $p_*\mu_1= p_* \td\mu_0$ is invariant under $A_0$, $R$, and $U^{[I]}$.  Moreover, $\mu_1$ is $A_0$-invariant and 
by \cref{prop:uscfiberentropy}, we have 
$h_{\mu_1} (a_0 \mid \Fsc)\ge h_{\td \mu_0} (a_0 \mid \Fsc)>0.$
\end{proof}

The second paragraph in the above proof of \cref{averagingunipotent} also establishes the following variant.  
\begin{proposition}\label[proposition]{averagingA}
Let $A$ be a maximal $\R$-split Cartan subgroup of $G$.   Let $H\subset G$ be a subgroup and let $A_0$ be a closed subgroup of $A\cap H$.  

Suppose there is $a\in A_0$ and an $a$-invariant Borel probability measure  $\mu_0$ on $M^\alpha$ such that  
\begin{enumerate}[label=(\alph*)]
\item $p_* \mu_0$ is $H$-invariant and has exponentially small mass at $\infty$, and 
\item $h_{\mu_0} (a \mid \Fsc)>0$.
\end{enumerate}

Then there exists a Borel probability measure $\mu_1$ on $M^\alpha$ such that 
\begin {enumerate}
\item $\mu_1$ is $A_0$-invariant,
\item $p_*\mu_1= p_*\mu_0$ 
\item $h_{\mu_1} (a \mid \Fsc)>0$.
\end{enumerate}
\end{proposition}

 Finally, we will use the following averaging operation in our setting. The proof is the same as in the averaging process in \cite[\S 6.3--6.5]{MR4502593}, replacing upper semicontinuity fiberwise Lyapunov exponents with upper semicontinuity of fiber entropy. 
\begin{proposition}\label[proposition]{prop:cocompact}
Let $\mu$ be an $A$-invariant probability measure on $M^{\alpha}$ with exponentially small mass at $\infty$ and $h_{\mu}(a\mid \Fsc)>0$ for some $a\in A$. 

Then, there exists an $A$-invariant probability measure $\mu'$ on $M^{\alpha}$ such that 
\begin{enumerate}
\item $h_{\mu'}(a'\mid\Fsc)>0$ for some $a'\in A$ and
\item $\mu'$ projects to the Haar measure on $G/\Gamma$.
\end{enumerate}

\end{proposition}

Even though in \cite{MR4502593} it is assumed that $\Gamma$ is uniform so that $M^{\alpha}$ is compact, we can follow the same proof since we assume $A$-invariance and exponentially small mass at $\infty$ of the measure $\mu$. Indeed, averaging $\mu$ along a  \Folner sequence of centered intervals  in a unipotent subgroup $U$ that is normalized by $A$ and passing to a subsequential limit point, one obtains (by \cref{thm:BFHunip} above) a probability measure $\mu'$ with exponentially small mass at $\infty$.  Moreover, since $p_* \mu$ is $A$-invariant, the measure $p_*\mu'$ is $A$-invariant (by \cref{thm:averaginghomo2} above).  
 Since $p_*\mu'$ is $A$-invariant, averaging $\mu'$ along a  \Folner sequence of centered intervals in $A$ gives a family of measures with uniformly exponentially small mass at $\infty$; again one can a subsequential limit point.  
 At each step of averaging, one uses the upper semicontinuity of entropy in \cref{prop:uscfiberentropy} rather than upper semicontinuity of Lyapunov exponents used in  \cite{MR4502593}.  
 Thus, the averaging procedure employed in \cite[\S 6.3--6.5]{MR4502593} to upgrade an $A$-invariant probability measure $\mu$ to an $A$-invariant probability measure that projects to the Haar measure on $G/\Gamma$ can be used nearly verbatim, with the above modifications to control escape of mass and semicontinuity of entropy.

		\section{{Proof of \texorpdfstring{\Cref{Ainv}}{Theorem 2.1}, Case I: 
		\texorpdfstring{$\rank_{\Rbb}(G)\ge 3$}{R-rank at least 3} and  $\Gamma$ nonuniform}}
		\label{sec:AinvI}


We follow the notation in previous sections.  The aim of this section is proving \Cref{Ainv} under the assumptions that $\rank_{\Rbb}(G)\ge 3$ and that $\Gamma$ is nonuniform  which we summarize as follows:
\begin{theorem}\label{thm:Ainvrank3}
Let $G$ and $\Gamma$ be as in \cref{hypstand2}.  
Suppose $\rank_{\Rbb}(G)\ge 3$ and $\Gamma$ is  nonuniform.

 Let $M$ be a compact smooth manifold and let $\alpha\colon \Gamma\to \diff^\infty(M)$ be an action such that $h_\top(\alpha(\gamma_0))>0$ for some $\gamma_0\in \Gamma$.
Then there exists a Borel probability measure $\mu$ on $M^{\alpha}$  such that 
\begin{enumerate}
	\item $\mu$ is $A$-invariant,
	\item $\mu$ projects to the Haar measure on $G/\Gamma$, and 
	\item $h_\mu(a\mid \Fsc)>0$ for some $a\in A$.
\end{enumerate}
\end{theorem}

To construct the measure in the above theorem, we repeatedly average various measures on $M^\alpha$ along \Folner sequences in certain subgroups of $G$.   When  averaging, we use either \Cref{thm:BFHunip} or \Cref{SL2coro} to avoid escape of mass.  
By averaging along subgroups that (1) commute with an element for which the fiber entropy is positive and (2) preserve quantitative decay of  measures at $\infty$, using \Cref{prop:uscfiberentropy} we maintain  positive fiber entropy when passing to limits. 

	\subsection{Reduction to a semisimple  element}
	We start with the following proposition which asserts that we may assume that the positive entropy element in $\Gamma$ is semisimple. Note that the following proposition is true  not only under assumption \ref{Ainvrank3} but also under  assumptions \ref{AinvSL3} and \ref{Ainvuniform} of \Cref{Ainv}.
\begin{proposition}\label[proposition]{ss}
Let $G$ and $\Gamma$ be as in \cref{hypstand}. 
Let $M$ be a  compact manifold and let $\alpha\colon \Gamma\to \Diff^\infty (M)$ be an action.  Suppose there exists $\gamma_0\in \Gamma$ such that  
$$h_\top (\alpha(\gamma_0))>0.$$
Then there exists a semisimple $\gamma_1\in \Gamma $ such that 
$$h_\top (\alpha(\gamma_1))>0.$$
\end{proposition}

\begin{proof}[Proof of \Cref{ss}]
By Margulis's arithmeticity theorem and by replacing $\Gamma$ with a subgroup of finite index if necessary, we may assume the center of $G$ is trivial and that there is a $\Q$-simple $\Q$-group $\bfF$, a  surjective continuous homomorphism $\sigma\colon \bfF(\R)^\circ \to G$, and a finite index subgroup $\Lambda$ of $\bfF(\Z)\cap \bfF(\R)^\circ $ such that $\sigma \colon \Lambda\to \Gamma$ is an isomorphism. 

Let $\lambda=\sigma\inv (\gamma_0)$. We recall the Jordan-Chevalley  decomposition $\lambda= \lambda_s\lambda_u$ where $\lambda_s$ is semisimple, $\lambda_u$ is unipotent, and $\lambda_s$ and $\lambda_u$ commute; moreover,  $\lambda_s, \lambda_u\in  \bfF(\Q)$.  
Since $\lambda_u$ is unipotent, there is $n\in \N$ such that $\lambda_u^n\in \Lambda$.  It follows that $\lambda_s^n=\lambda^n(\lambda_u^n)\inv\in \Lambda$.

Let $\gamma_1 = \sigma (\lambda_s^n).$  We claim $$h_\top (\alpha(\gamma_1))>0.$$ Indeed, unipotent elements are distorted in $\Lambda$ (see,  \cite{LMR00, MR2219247}) and so $\alpha(\sigma(\lambda_u^n ))$ is a distortion element in $\Diff^{\infty}(M)$.  Since $h_\top (f^n)=n h_\top(f)$, if $f$ is a distortion element then $h_\top(f)=0$ and so $h_\top \left(\alpha(\sigma (\lambda_u^n))\right)=0$.  
By  \cite[Thm.\  B]{MR1213080},
\[0<h_{\top}(\alpha(\gamma_{0}^{n}))\le h_\top (\alpha(\sigma(\lambda_s^n))) + h_\top(\alpha(\sigma(\lambda_u^n))) =h_{\top}(\alpha(\sigma(\lambda_s^n))).  \qedhere\]
\end{proof}

\subsection{Reduction to an diagonal element}
We prove the following proposition which will serve as the base case for induction in our proof of \Cref{thm:Ainvrank3} later in \cref{sec:induction}.  
\begin{proposition}\label[proposition]{prop:diag}
Let $G$, $M$, $\Gamma$, $\alpha$ be as in 
\Cref{thm:Ainvrank3}.   
Then there is 
\begin{enumerate}
\item a probability measure $\mu_{2}$ on $M^{\alpha}$,
\item a maximal $\Rbb$-split torus $A$ in $G$,
\item a restricted (simple) root $\delta_{0}$ in the restricted root system $\Phi(A,G)$, and
\item a one parameter subgroup $d_{0}^{\Rbb}=\{d_{0}^{t}\}_{t\in \Rbb}\subset A$   that   is the diagonal  of  $\delta_{0}$  in $A$ 
\end{enumerate} 
such that
\begin{enumerate}[resume]
\item $\mu_{2}$ has exponentially small mass at $\infty$,
\item $\mu_{2}$ is $d_{0}^{\Rbb}$-invariant, and
\item $h_{\mu_{2}}(d_{0}^{1}\mid \Fsc)>0$.
\end{enumerate}
\end{proposition}

\subsubsection{Arithmetic reductions}
We return to the setting of  \cref{hypstand} with  $\Gamma$  assumed nonuniform.  Up to passing to a finite-index subgroup of $\Gamma$ and inducing an action on finitely many copies of $M$ (see discussion in 
\cite[p.\, 1002]{Brown:2020aa}), the following  standard reductions hold:
\begin{hypot}
$G=\bfG(\Rbb)$ for some algebraically simply connected $\Qbb$-simple algebraic group $\bfG$ defined over $\Qbb$. Let $\Gamma=\bfG(\Zbb)$ be a lattice in $\bfG(\Rbb)$. Assume that $\rank_{\Rbb}(\bfG)\ge 3$ and $\rank_{\Qbb}(\bfG)\ge 1$.
\end{hypot}
Indeed,  $\rank_{\Qbb}(\bfG)=0$ if and only if $\Gamma=\bfG(\Zbb)$ is uniform.  Since we assumed $\Gamma$ is nonuniform  lattice, we have that $\rank_{\Qbb}(\bfG)\ge 1$.  Moreover, since $\Gamma$ is assumed nonuniform, one does not need to pass to a compact extension when applying Margulis' Arithmeticity Theorem (see \cite[Cor.\ 5.3.2]{MorrisArith}) and so   $G=\bfG(\Rbb)$ does not have  compact simple factors.

Recall that we assume $\bfG$ is defined over $\Qbb$ and from \Cref{ss}
 that  $\gamma_{0}\in\bfG(\Zbb)$ is semisimple.  It follows that $\gamma_0$ is contained in a maximal  $\Q$-torus $\Tbf<\bfG$.  
Taking a power of $\gamma_{0}$ if necessary, 
we may assume that $\gamma_{0}\in \Tbf(\Zbb)$.  
 The $\Q$-torus $\Tbf$ splits uniquely as an almost direct product of $\Qbb$-tori $\Tbf=\Tbf_{s}^{\Qbb}\cdot\Tbf_{a}^{\Qbb}$ where  $\Tbf_{s}^{\Qbb}$  is $\Qbb$-split and $\Tbf_{a}^{\Qbb}$ is $\Qbb$-anisotropic. 
 We have $\Tbf_{s}^{\Qbb}(\Rbb)\cap \Gamma=\Tbf_{s}^{\Qbb}(\Rbb)\cap \Tbf(\Zbb)$ is finite and thus, after taking a power of $\gamma_{0}$ if necessary, we may assume further that $\gamma_{0}\in \Tbf_{a}^{\Qbb}(\Rbb)$.
Finally,  passing to another  power of $\gamma_0$ if needed, we may assume $\gamma_0 = g^1$ for some 1-parameter subgroup $g^{\Rbb}=\{g^t\}$ in $T=\Tbf_{a}^{\Qbb}(\Rbb)$. 
 

\subsubsection{A seed measure with positive fiber entropy}\label{sec:seed}
 Recall $\gamma_0= g^1$ for a 1-parameter subgroup $\{g^t\}$ in $G$.  Consider the $g_t$-orbit of the fiber in $M^\alpha$ over the identity coset $\1\Gamma$.  This orbit is compact, projects to a closed curve in $G/\Gamma$, and coincides with the suspension flow induced by $\alpha(\gamma_0)$.  In particular,  this collection of fibers is $g_t$-invariant and the  $g_t$-flow has positive topological entropy.
The variational principle \cite[Chap.\ 20]{KHbook}
applied to the $g_t$-flow on this collection of fibers then gives the following.  
\begin{claim}\label[claim]{claim:seed} 
 There exists a probability  measure $\mu_{0}$ on $M^{\alpha}$  such that
 \begin{enumerate}
 \item $\mu_{0}$ is $g^t$-invariant, 
 \item  $h_{\mu_{0}}(g^1\mid \Fsc)>0$, and
 \item  $\mu_{0}$ projects to the Haar measure on the closed orbit $\{g^t \cdot \Gamma:  t\in \R \}.$
 \end{enumerate}
 In particular, $\mu_0$ is compactly supported.  
 \end{claim}

\subsubsection{Averaging over the  $\Qbb$-anisotropic torus and reduction to a $\R$-split element}\label{sec:aniave}

We have that $\Tbf_{a}^{\Qbb}$ is defined over $\R$ and so further splits as an almost direct product of a $\R$-split torus and a $\R$-anisotropic (i.e.\ compact) torus.   As $\Tbf_{a}^{\Qbb}(\Rbb)$ is abelian, fix a \Folner sequence $F_{n}$ in $\Tbf_{a}^{\Qbb}(\Rbb)$ and take \[\mu_{0}^{n}=F_{n}*\mu_{0}:=\frac{1}{|F_{n}|}\int_{F_{n}} h_{*}\mu_{0} \, dh.\]
Since $T= \Tbf_{a}^{\Qbb}(\Rbb)/\Tbf_{a}^{\Qbb}(\Zbb)$ is a compact torus on $G/\Gamma$, the family $\{\mu_0^n\}$ is uniformly tight.
Let $\mu_{1}$ be a weak-$*$ limit of the sequence $\{\mu_0^n\}$. Then, $\mu_1$ is $\Tbf_{a}^{\Qbb}(\Rbb)$ invariant.
By pre-compactness of the family $\{\mu_0^n\}$ and since $\Tbf_{a}^{\Qbb}(\Rbb)$ is abelian, by \cref{prop:uscfiberentropy} we have $ h_{\mu_{1}}(g^1\mid \Fsc)>0$.  




Write $g^1= a\cdot k$ where $a\in \Tbf_{a}^{\Qbb}(\Rbb)$ is  $\R$-split and $k\in \Tbf_{a}^{\Qbb}(\Rbb)$ is $\R$-anisotropic. 
Since $h_{\mu_{1}}(g^1\mid \Fsc)>0$, $a$ is non-trivial and we may write $a= a^{1}$ for a 1-parameter $\R$-split subgroup $\{a^t\}$ in $\Tbf^{\Qbb}_{a}(\Rbb)$.  
Since the  closure of $\{k^n:n\in \Z\}$ in $\Tbf^{\Qbb}_{a}(\Rbb)$ is compact and commutes with $g^1$, we have 
$$h_{\mu_{1}}(a^1\mid \Fsc)=h_{\mu_{1}}(g^1\mid \Fsc)>0.$$

In summary, we found a compactly supported probability measure $\mu_1$ on $M^{\alpha}$ such that $\mu_1$ is $\Tbf_{a}^{\Qbb}(\Rbb)$-invariant and $h_{\mu_{1}}(a^{1}\mid \Fsc)>0$ for a $1$-parameter $\R$-split subgroup $\{a^{t}\}\subset \Tbf_{a}^{\Qbb}(\Rbb)$.

\subsubsection{Finding a diagonal element}
Recall that $\Tbf_{s}^{\Qbb}$ is the $\Q$-split  part of a maximal $\Q$-torus $\Tbf$ but may not be a maximal $\Q$-split torus in $\bfG$.  
 Take a maximal $\Qbb$-split torus $\Dbf_{1}$  containing $\Tbf_s$ and a collection of $\Q$-roots $\Phi(\Dbf_{1}, \bfG)_{\Q}$ and a choice of  collection of simple  $\Q$-roots $ \Delta(\Dbf_{1}, \bfG)_{\Qbb}$.  
     As $\Dbf_{1}$ is also $\Rbb$-split, we can find a maximal $\Rbb$-split  $\R$-torus $\Dbf_{2}$ containing $\Dbf_{1}$
     with 
     $\R$-root system $\Phi(\Dbf_{2}, \bfG)_{\R}$ and a choice of  collection of simple roots $ \Delta(\Dbf_{1}, \bfG)_{\R}$.  
     The restriction map $j\colon \Phi(\Dbf_{2}, \bfG)_{\Rbb}\to \Phi(\Dbf_{1}, \bfG)_{\Rbb}$ is   a surjective map $j$ between the root systems.       As in \cite[21.8]{Borel}, we can make the order  coherent so that  the restriction map $j$ takes simple roots to simple roots; that is, 
\[j\colon\Delta(\Dbf_{2}, \bfG)_{\Rbb} \to \Delta(\Dbf_{1}, \bfG)_{\Rbb}\cup\{0\}= \Delta(\Dbf_{1}, \bfG)_{\Qbb}\cup\{0\}.\]
(Note the equality holds as the adjoint representation is defined over $\Qbb$.)

From  \cite[Prop.\  20.4]{Borel}, the centralizer of $\Tbf_{s}^{\Qbb}$ in $\bfG$ is a Levi component of a $\Q$-parabolic subgroup.  
Since $\Tbf_{s}^{\Qbb}$ was assumed the maximal $\Q$-split torus in the maximal $\Qbb$-torus $\Tbf$ it follows (see \cite[Prop.\ 1.2]{MR1289055} and referenced discussion in \cite[\S 20--21]{Borel})
 that there is a collection of simple $\Q$-roots $I'\subset \Delta(\Dbf_{1}, \bfG)_{\Qbb}$ such that $\Tbf_{s}^{\Qbb}\subset \Dbf_1$ also coincides with $$\Tbf_{s}^{\Qbb}= \bigcap_{\alpha\in I'} \ker \alpha.$$
 Let $I= j\inv (I'\cup\{0\})\subset \Delta(\Dbf_{2}, \bfG)_{\Rbb} $.  Then $\Tbf_{s}^{\Qbb}\subset \Dbf_2$ coincides with 
 $$\Tbf_{s}^{\Qbb}= \bigcap_{\alpha\in I} \ker \alpha.$$
  It follows that 
	\[Z_{\bfG}(\Tbf_{s}^{\Qbb})=\Tbf_{s}^{\Qbb}\cdot\Tbf_{0}\cdot\Hbf \] for some $\Rbb$-torus $\Tbf_{0}$ and semisimple $\Rbb$-subgroup $\Hbf$ of $\bfG$. 
Moreover, $$\dim_{\Rbb}(\Tbf_{s}^{\Qbb})+\rank_{\Rbb}(\Hbf)=\dim_{\Rbb}(\Dbf_2)=\rank_{\Rbb}(\bfG)$$ and so $\Tbf_0$ is  $\Rbb$-anisotropic.

	Since every maximal  torus has the the same dimension and since $\Tbf$ is maximal  torus in $\bfG$, there is a maximal $\Rbb$-torus $\Tbf_H$ in $\Hbf$ such that $\Tbf_{0}\cdot\Tbf_H=\Tbf_{a}^{\Qbb}$. Let $\Sbf$ be the $\Rbb$-split part of the maximal torus $\Tbf_H$ of $\Hbf$. Note that $\Tbf_{s}\cdot \Sbf$ is a $\Rbb$-split torus and that  $\{a_t\}\subset \Sbf(\R)$.  	

	Let $\Dbf_3$ be a maximal $\R$-split torus of $\Hbf$ containing $\Sbf$.  We may further assume $\Dbf_3\cdot \Tbf_s = \Dbf_2$ and thus view roots $\Phi (\Dbf_3 , \Hbf)_\R$ as the restriction of $\Delta(\Dbf_{2}, \bfG)_{\Rbb}$ to $\Dbf_3$.

Applying 	  \cite[Prop.\  20.4]{Borel} and discussion in  \cite[Prop.\  21.11]{Borel} or  \cite[Prop.\  1.2]{MR1289055} again, we have 	
\[Z_{\Hbf}(\Sbf)=\Sbf\cdot\Sbf_{0}\cdot\Hbf'\] 
for some $\Rbb$-semisimple group $\Hbf'$ and $\Rbb$-anisotropic torus $\Sbf_{0}$. Since $\Sbf$ is the $\Rbb$-split part of a maximal torus $\Tbf_H$ in $\Hbf$, $\Hbf'$ has an $\Rbb$-anisotropic maximal torus.  By  \cite[Prop.\  8.5.2]{MR0573070},  each $\Rbb$-simple factor $H_{1}',\dots, H_{r}'$ of $H'=\Hbf'(\Rbb)$ is of inner type.  Moreover the inner type $\Rbb$-simple groups are completely classified; see \cite[Tables I, II]{MR1289055} or \cite[\S 8.5]{MR0573070}.  
In particular, all such groups have irreducible restricted root systems of the type $$B_\ell, C_\ell, BC_\ell,  \text{$D_\ell$ for $\ell $ even}, E_7, E_8, F_4, G_2.$$
These abstract root systems admit a pairwise orthogonal collection of roots  of cardinality is the rank of the root system; see discussion in \cite[\S2]{MR1682805}.

Let $S=\mathbf{S}(\Rbb)$. Let $A_{H'}$ and $A_{H_{i}'}$ be the maximal $\Rbb$-split torus in $H'$ and $H_{i}'$ in $\Dbf_3(\R)$, respectively, for $i=1,\dots,r$, so that $A_{H'}=A_{H_{1}'}\cdot\dots\cdot A_{H_{r}'}$. Note that $A_{H}=S\cdot A_{H'}= \Dbf_3(\R)$.
As $H'_j$ is inner, for each $1\le j\le r$, we may find a collection of roots $\{\alpha_{1}, \dots, \alpha_{\ell}\}$ in $\Phi(H'_j,A_{H'_j})$ where $\ell_j= \rank_\Rbb({H'_j}) = \dim _\Rbb( A_{H'_j})$
such that if $H^{\alpha_i}$ is the  standard $\R$-rank-1 subgroup of $H'_j$ generated by  $U^{[\alpha_i]}$ and~$U^{[-\alpha_i]}$ then 
\begin{enumerate}
\item $H^{\alpha_i}$ and $H^{\alpha_k}$ commute for $1\le i\neq k\le \ell_j$, and 
\item the group generated by  $\{d_{\alpha_i}^{\Rbb}, 1\le i\le \ell_j\}$, the diagonals of $\alpha_i$ in $A_{H'_j}$, 
is all of $A_{H'_j}$.  
\end{enumerate}

We complete the proof of   \Cref{prop:diag}. Note that $A_{H'}\simeq \Rbb^{k}$ for some $k\ge 0$. 
If $k=0$, we have that  $\Sbf$ is a maximal $\Rbb$-split torus in $\Hbf$ and thus  $S=A_{H}$. Since \begin{enumerate}
\item there is a one parameter subgroup $a^{\Rbb}\subset S$ such that $h_{\mu_{1}}(a^{1}\mid\Fsc)>0$, 
\item $\mu_{1}$ is invariant under $A_{H}$, and 
\item the fiber entropy is a non-zero semi-norm, 
\end{enumerate} 
we can find a root $\delta_{0}$ in $\Phi(H,A_{H})$ such that the diagonal $d_{0}^{\Rbb}\subset S= A_H$ of $\delta_{0}$ in $A_H$ has positive fiber entropy, $h_{\mu_{1}}(d_{0}^{1}\mid \Fsc)>0$. In this case we can choose $\mu_{2}=\mu_{1}$.

 Otherwise, we have $k>0$.  Applying  \Cref{SL2coro} for all $\alpha_{i}$, $i=1,\dots,\ell_j$ and each $j=1,\dots, r$, we can find a measure $\mu_{1}'$ such that 
\begin{enumerate}
\item $\mu_{1}'$ has exponentially small mass at $\infty$,
\item $\mu_1'$  is invariant under $S$,
\item $h_{\mu_{1}'}(a^{1}\mid \Fsc)>0$, and
\item\label{itemdiag} the image $p_{*}(\mu_{1}')$ of $\mu_1'$ in $G/\Gamma$ is invariant under $A_{H'}$.
\end{enumerate}
Above, conclusion \eqref{itemdiag} comes from the fact that for each $j=1,\dots,r$, the diagonals of $\alpha_{i}$ in $A_{H'_j}$, $i=1,\dots, \ell_{j}$ generate $A_{H_{i}'}$.   
 Fix a \Folner sequence $\{A_{n}\}$ in $A_{H'}$.   Since $p_{*}\mu_{1}$ is $A_{H'}$-invariant, $\left\{A_{n}*\mu_{1}'\right\}$ has uniformly   exponentially small mass at $\infty$. 
Take $\mu_2$ to be any subsequential limit point of  $\left\{A_{n}*\mu_{1}'\right\}$.  
Then   $\mu_{2}$ is invariant under $A_{H'}$ by construction and under $S$ since $S$ commutes with $A_{H'}$. 
  Thus,  $\mu_{2}$ is invariant under $A_{H}$. Finally, using \Cref{prop:uscfiberentropy} and that $a^{\Rbb}$ commutes with $A_{H'}$, we have $h_{\mu_{2}}(a^{1}\mid\Fsc)>0$.   
  Again, using the fact that the fiber entropy is a non-zero semi-norm, there is a root $\delta_{0}\in \Phi(H,A_{H})$ such that the diagonal $d_{0}^{\Rbb}\subset A_{H}\subset A$ of $\delta_{0}$ in $A_H$ has a positive entropy, $h_{\mu_{2}}(d_{0}^{1}\mid \Fsc)>0$.
  
  Finally we recall that we view roots $\delta_0\in \Phi(A_H, H)$ as restrictions of roots in $\Phi(\Dbf_2 , \bfG)_\Rbb$ to $A_H\subset \Dbf_2(\R)$, completing the proof of \Cref{prop:diag}.

\subsection{Inducting on the number of simple roots}\label{sec:induction}
Starting from \Cref{prop:diag}, we will prove the following by induction on the number of simple roots to deduce  \Cref{thm:Ainvrank3}.

\begin{proposition}\label[proposition]{prop:oneroot}
Let $G$, $\Gamma$, $M$, $\alpha$ be as in \Cref{thm:Ainvrank3}. Assume further that $G$ has finite center. Assume that there exists  
\begin{enumerate}[label=(\alph*)]
\item a maximal $\Rbb$-split torus $A$ in $G$,
\item a restricted root $\delta_{1}\in \Phi(A,G)$,
\item a $1$-parameter subgroup $d_{1}^{\Rbb}=\{d_{0}^{t}\}_{t\in\Rbb}$ that is the diagonal of $\delta_{1}$ in $A$, and 
\item a probability measure $\mu_{2}$ on $M^{\alpha}$ with exponentially small mass at $\infty$
\end{enumerate}
such that 
\begin{enumerate}[label=(\alph*), resume] 
\item $\mu_{2}$ is $\{d_{1}^{t}\}_{t\in\Rbb}$-invariant, and
\item $h_{\mu_{2}}(d_{1}^{1}\mid\Fsc)>0$.
\end{enumerate}

Then there exists a probability measure $\mu_{\infty}$ on $M^{\alpha}$ such that 
\begin{enumerate}
\item $p_{*}(\mu_{\infty})$ is Haar measure on $G/\Gamma$, 
\item $\mu_{\infty}$ is $A$-invariant, and
\item $h_{\mu_{\infty}}(a\mid \Fsc)>0$ for some $a\in A$.
\end{enumerate}
\end{proposition}

\subsubsection{Invariance under the rank-$3$ (or $4$) group generated by semi-adjacent root(s)}
Consider a  measure  $\mu_2$ satisfying the hypotheses of \Cref{prop:oneroot}. 

Since the diagonals of $\delta_1$ and $2\delta_1$ in $A$ coincide, we may assume $\frac 1 2 \delta_1$ is not a root.  
Acting by the Weyl group, we may select a system of simple roots $\Delta(A,G)$ for $\Phi(G, A)$ for which $\delta_1$ is the left-most root (if the root system is of type $A_n,  D_n, E_6, E_7,$ or $ E_8$) or either the 
left-most or right-most root (if the root system is of type $ B_n, C_n,$ or $ BC_n$).   
The only exception is the following: if the root system if of type $F_4$, we will assume $\delta_1$ is either the right-most  root or the root second from left.    
We follow the choice of  left-to-right orientation for  Dynkin diagrams associated to systems of simple roots for the irreducible roots systems with  rank at least 3 in   \cref{table:dynk}.

 {\renewcommand{\arraystretch}{1.0}  
  \newcolumntype{C}[1]{>{\centering\arraybackslash} m{#1} }

\def\scales{.5}
\def\bdboxA{\pgfresetboundingbox \useasboundingbox (0,-.3) rectangle (4,.5)   }
\def\bdboxB{\pgfresetboundingbox \useasboundingbox (0,-.3) rectangle (4,.5)   }
\def\bdboxD{\pgfresetboundingbox \useasboundingbox (0,-.7) rectangle (4,1.)   }
\def\bdboxE{\pgfresetboundingbox \useasboundingbox (0,-.3) rectangle (5,1.5)}
\def\bdboxF{\pgfresetboundingbox  \useasboundingbox (0,-.3) rectangle (3,.5)}

\begin{table}[h]
\footnotesize
\centering
\begin{tabular}{ | C{1.4cm}  | C{3.2cm} || C{1.4cm}| C{3.5cm}| } \hline
 root system  & Dynkin Diagram &  root system  & Dynkin Diagram \\
 \hline
$A_n$  &
\begin{tikzpicture}[scale=\scales]
\bdboxA;

\node[dnode](1) at (0,0) {};
\node[dnode](2) at (1,0) {};
\node[dnode](4) at (3,0) {};
\node[dnode](5) at (4,0) {};
    \path (1) edge[sedge] (2)
(4) edge[sedge] (5)
          (2) edge[sedge,dashed, dash phase=1.4pt] (4);
 \end{tikzpicture}
 &
$D_n$  &
\begin{tikzpicture}[scale=\scales]
\bdboxD;

\node[dnode](1) at (0,0) {};
\node[dnode](2) at (1,0) {};
\node[dnode](5) at (3,0) {};
    \node[dnode] (6) at (4,.5) {};
      \node[dnode] (7) at (4,-.5) {};
    \path (1) edge[sedge] (2)
(2) edge[sedge,dashed, dash phase=1.4pt] (5)
(4) edge[sedge] (5)
(5) edge[sedge] (6)
(5) edge[sedge] (7);          
 \end{tikzpicture}
\\\hline
$B_n, BC_{n}$ 
&
 \begin{tikzpicture}[scale=\scales]
\bdboxB;
\node[dnode] (1) at (4,0) {};
\node[dnode] (2) at (3,0) {};
\node[dnode] (5) at (1,0) {};
\node[dnode] (6) at (0,0) {};

    \path (2) edge[dedge] (1)

          (5) edge[sedge] (6)
          (2) edge[sedge,dashed, dash phase=1.4pt] (5)
          ;
\end{tikzpicture}
&
{$E_6$, $E_7$, $E_8$  }&
\begin{tikzpicture}[scale=\scales]

\bdboxE;

\node[dnode](1) at (0,0) {};
\node[dnode](2) at (1,0) {};
\node[dnode](3) at (2,0) {};
\node[dnode](4) at (2,1) {};
    \node[dnode] (5) at (3,0) {};
      \node[dnode] (6) at (5,0) {};
    \path (1) edge[sedge] (2)
    (2) edge[sedge] (3)
    (3) edge[sedge] (4)
    (3) edge[sedge] (5)
(5) edge[sedge,dashed, dash phase=1.4pt] (6);
 \end{tikzpicture} 
\\ \hline
$C_n$ 
&
 \begin{tikzpicture}[scale=\scales]
\bdboxB;
\node[dnode] (1) at (4,0) {};
\node[dnode] (2) at (3,0) {};
\node[dnode] (5) at (1,0) {};
\node[dnode] (6) at (0,0) {};

    \path (1) edge[dedge] (2)

          (5) edge[sedge] (6)
          (2) edge[sedge,dashed, dash phase=1.4pt] (5)
          ;
\end{tikzpicture}
 &
  { $F_{4}$} &  
  \begin{tikzpicture}[scale=\scales]
  
\bdboxF;

 \node[dnode](1) at (0,0) {};
 \node[dnode] (2) at (1,0) {};
 \node[dnode] (3) at (2,0) {};
 \node[dnode](4) at (3,0) {};
    \path (1) edge[sedge] (2)
	(2) edge[dedge] (3)
	(3) edge[sedge] (4);
 \end{tikzpicture} 
\\ \hline
\end{tabular}
\caption{Dynkin diagrams for irreducible root systems with rank $\ge  3$}\label{table:dynk}
\end{table}
}

In the Dynkin diagram associated to $ \Delta(A,G)$, let $\delta_2$ be the root adjacent to $\delta_1$ and let $\delta_3$ be root commuting with $\delta_1$ adjacent to $\delta_2$.  There is a unique such choice of $\delta_2$ and $\delta_3$ for all Dynkin diagrams except for diagrams of type $D_4$. If the Dynkin diagram is of type $D_4$, we take  $\delta_3$ and $\delta_4$ to be adjacent to $\delta_2$ and commuting with $\delta_1$.  

Since $\delta_3$ (and $\delta_4$) commute with $\delta_1$, we have 
$\{d_1^{t}\}\subset \ker\delta_{3}\cap \ker \delta_4$.  We can thus apply \Cref{SL2coro} with $\delta_{3}$ and $\mu_{2}$. As a result, we obtain the following. 


	\begin{claim}\label[claim]{prop:std} 
Retain the notation and assumptions  above.   Let $H_3$ be the standard $\R$-rank-$1$ subgroup generated  by $\delta_{3}$.  

There is a probability measure $\mu_{2}'$ on $M^{\alpha}$ with exponentially small mass at $\infty$  such that 
\begin{enumerate}
\item $\overline{\mu_{2}'}=p_{*}(\mu_{2}')$ is $H_{3}$-invariant and $\{d_1^{\Rbb}\}$-invariant,
\item $\mu_{2}'$ is invariant under $d_1^{\Rbb}$ and the diagonal of $\delta_{3}$ in $A$, and
\item $h_{\mu_{2}'}(d_1^{1}\mid \Fsc)>0$.
\end{enumerate}
\end{claim}

Starting from \Cref{prop:std}, we proceed with our first step of induction by considering the rank-3 subgroup generated by $\delta_{1}$, $\delta_{2}$, and $\delta_{3}$ (or by $\delta_{1}$, $\delta_{2}$,   $\delta_{3}$, and  $\delta_4$ if $\Phi(A,G)$ is of type $D_4$).   Note that the sub-Dynkin diagram containing $\delta_{1}$, $\delta_{2}$, and $\delta_{3}$ is connected and is thus of type 
$A_{3}$, $B_{3}$, $C_{3}$, or $BC_{3}$.  Write $L_{3}$ for the closed connected $\R$-rank-$3$ subgroup of $G$ which is generated by the root subgroups of $\pm\delta_{1}$, $\pm\delta_{2}$, and $\pm\delta_{3}$.  Write  $A_{L_{3}}:= A\cap{L_{3}}.$

\begin{proposition}\label[proposition]{prop:firstind} There is a probability measure $\mu_{3}$ on $M^{\alpha}$ such that 
\begin{enumerate}
\item $\mu_{3}$ is $A_{L_{3}}$-invariant and has exponentially small mass at $\infty$,
\item $\overline{\mu_{3}}=p_{*}(\mu_{3})$ is $L_{3}$-invariant, and
\item there is $a\in A_{L_{3}}$ such that $h_{\mu_{3}}(a\mid \Fsc)>0$.   
\end{enumerate}
Moreover, we may assume $a$ is in either the diagonal of $\delta_1$ or $\delta_3$ in $A$.

Moreover, if $\Phi(A,G)$ is of type $D_4$, we may assume $\mu_3$ is $A$-invariant and $\overline{\mu_{3}}=p_{*}(\mu_{3})$ is $G$-invariant.  
\end{proposition}
\begin{proof}[Proof of \Cref{prop:firstind}] We start with the measure $\mu_{2}'$ as in \Cref{prop:std}.  We consider separately the case that the restricted root system of $L_3$ is of type $A_{3}$ or $C_3$ versus when $L_3$ is of type $B_{3}$ or $BC_3$ versus when $G$ is of type $D_4$.

\subsubsection*{\it  $L_{3}$ is of type $A_{3}$ or $C_3$:} 
When $L_{3}$ is of type $A_{3}$, we can describe the set of roots $\Phi(L_{3},A_{L_{3}})$ as\footnote{Here, $\delta_{i}$ is $\delta_{1}=e_{1}-e_{2}, \delta_{2}=e_{2}-e_{3}, \delta_{3}=e_{3}-e_{4}$ in the notation of \cite[Appendix C]{KnappLie}.} \[\Phi(L_{3},A_{L_{3}})=\{\pm\delta_{i}:i=1,2,3\}\cup\{\pm(\delta_{1}+\delta_{2}),\pm(\delta_{1}+\delta_{2}+\delta_{3}),\pm(\delta_{2}+\delta_{3})\}. \]
When $L_{3}$ has type $C_{3}$ and $\delta_{1}$ is the right-most (long) root, we can describe the space of  
$\Phi(L_{3},A_{L_{3}})$ as\footnote{Here, $\delta_{i}$ as $\delta_{1}=2e_{3}, \delta_{2}=e_{2}-e_{3}, \delta_{3}=e_{1}-e_{2}$ in the notation of \cite[Appendix C]{KnappLie}.}
\begin{eqnarray*} 
\Phi(L_{3},A_{L_{3}})&=&\{\pm\delta_{1},\pm\delta_{2},\pm\delta_{3}\}\cup\{\pm(\delta_{1}+\delta_{2}),\pm(\delta_{1}+\delta_{2}+\delta_{3}), \pm(\delta_{2}+\delta_{3})\}\\ 
&&\cup\{\pm(\delta_{1}+2\delta_{2}), \pm(\delta_{1}+2\delta_{2}+\delta_{3}), \pm(\delta_{1}+2\delta_{2}+2\delta_{3})\}.
\end{eqnarray*}
If $\delta_1$ is the left-most (short) root, we have $
\pm(\delta_{1}+2\delta_{2}), \pm(\delta_{1}+2\delta_{2}+2\delta_{3})\notin \Phi(L_{3},A_{L_{3}})$ and 
instead have 
$$\pm(2\delta_{2}+\delta_{3}),  \pm(2\delta_{1}+2\delta_{2}+\delta_{3})\in \Phi(L_{3},A_{L_{3}}).$$
In both cases, we check that the subgroup of $A$ generated by diagonals of $\delta_{1}$  and $\delta_{3}$ in $A$ is the same as the subgroup  of $A$ generated by $\ker \delta_{2} \cap \ker (\delta_{1}+\delta_{2}+\delta_{3})$ and $\ker(\delta_{1}+\delta_{2})\cap \ker(\delta_{2}+\delta_{3})$.  
Since $\mu_{2}'$ is invariant under the diagonal of $\delta_{3}$ and the diagonal of $\delta_{1}$,   $\mu_{2}'$ is invariant under $\ker \delta_{2} \cap \ker (\delta_{1}+\delta_{2}+\delta_{3})$ and $\ker(\delta_{1}+\delta_{2})\cap \ker(\delta_{2}+\delta_{3})$. 
Since entropy $h_{\mu_{2}'}(\cdot\mid\Fsc)$ is a non-zero semi-norm on this group, either $\ker \delta_{2} \cap \ker (\delta_{1}+\delta_{2}+\delta_{3})$ or $\ker(\delta_{1}+\delta_{2})\cap \ker(\delta_{2}+\delta_{3})$ contains an element $a$ such that $h_{\mu_{2}'}(a\mid\Fsc)>0$. 

Assume first that that there is $a\in \ker \delta_{2}\cap \ker (\delta_{1}+\delta_{2}+\delta_{3})$ such that $h_{\mu_{2}'}(a\mid \Fsc)>0$. Note that $U^{\delta_{2}}$, $U^{-(\delta_{1}+\delta_{2}+\delta_{3})}$, and $U^{-\delta_3}$ commute.
Applying \cref{averagingunipotent} with $I= \{\delta_{2},-(\delta_{1}+\delta_{2}+\delta_{3})\}$,  $A_0\subset  A_{L_{3}}$ the subgroup generated by the diagonal of $\delta_1$ and $\delta_{3}$, and $R= U^{-\delta_{3}}$,
we obtain a measure $\mu_{2}''$ such that 
\begin{enumerate}
\item $\mu_{2}''$ is invariant under  $A_{0}$, $U^{\delta_{2}}$, and $U^{-(\delta_{1}+\delta_{2}+\delta_{3})}$ 
\item $p_{*}(\mu_{2}'')$ is invariant under $U^{\delta_{2}}$, $U^{-\delta_{3}}$, $A_0$, 
and $U^{-(\delta_{1}+\delta_{2}+\delta_{3})}$,  and has exponentially small mass at $\infty$, and
\item $h_{\mu_{2}''}(a\mid \Fsc)>0$.
\end{enumerate}

Let $p_{*}(\mu_{2}'')=\overline{\mu_{2}''}$.  By \cref{thm:opproot}, $\overline{\mu_{2}''}$ is invariant under $U^{\delta_{3}}$ since $\overline{\mu_{2}''}$ is invariant under the diagonal of $\delta_{3}$ and $U^{-\delta_{3}}$. As $\overline{\mu_{2}''}$ is invariant under $U^{\delta_{3}}$, $U^{\delta_{2}}$, and $U^{-(\delta_{1}+\delta_{2}+\delta_{3})}$, we can deduce that $\overline{\mu_{2}''}$ is also invariant under $U^{-(\delta_{1}+\delta_{2})}$, $U^{-\delta_{1}}$ and $U^{\delta_{2}+\delta_{3}}$. In addition, $\overline{\mu_{2}''}$ is invariant under the diagonal of $\delta_{1}$ and $U^{-\delta_{1}}$ and so we can deduce that $\overline{\mu_{2}''}$ is also invariant under $U^{\delta_{1}}$. In summary, $\overline{\mu_{2}''}$ is invariant under $U^{\delta_{1}}$, $U^{-\delta_{1}}$, $U^{\delta_{2}}$, $U^{\delta_{3}}$, $U^{-\delta_{3}}$, and $U^{-(\delta_{1}+\delta_{2}+\delta_{3})}$. This implies that $\overline{\mu_{2}''}$ is indeed invariant under $U^{-\delta_{2}}$, and this is enough to see that $\overline{\mu_{2}''}$ is $L_{3}$-invariant since $L_{3}$ is generated by $U^{\pm \delta_{1}}$, $U^{\pm \delta_{2}}$, and $U^{\pm \delta_{3}}$. 

 Finally, since $\overline{\mu_{2}''}$ is $L_{3}$-invariant, by applying \cref{averagingA} with $H=L_{3}$ and $A_0= A_{L_{3}} $,  
 we obtain a measure $\mu_{3}$ such that 
\begin{enumerate}
\item $\mu_{3}$ is $A_{L_{3}}$-invariant 
\item $p_{*}(\mu_{3})$ is $L_{3}$-invariant and has exponentially small mass at $\infty$,   and
\item $h_{\mu_{3}}(a\mid\Fsc)>0$.
\end{enumerate}

 If  $h_{\mu_{2}'}(a\mid \Fsc)=0$ for all $a\in \ker \delta_{2}\cap \ker (\delta_{1}+\delta_{2}+\delta_{3})$, then we may select $a'\in \ker(\delta_{1}+\delta_{2})\cap \ker(\delta_{2}+\delta_{3})$ with $h_{\mu_{2}'}(a'\mid \Fsc)>0$. In this case, we apply \cref{averagingunipotent} with $I= \{\delta_{1}+\delta_{2},-(\delta_{2}+\delta_{3})\}$,  $A_0\subset  A_{L_{3}}$ the subgroup generated by the diagonal of $\delta_1$ and $\delta_{3}$, and $R= U^{-\delta_{3}}$ to obtain a measure $\mu_{2}''$ such that 
\begin{enumerate}
\item $\mu_{2}''$ is invariant under $U^{(\delta_{1}+\delta_{2})}$, $A_0$, and $U^{-(\delta_{2}+\delta_{3})}$ and has exponentially small mass at $\infty$, 
\item $p_{*}(\mu_{2}'')$ is invariant under $U^{(\delta_{1}+\delta_{2})}$, $U^{-(\delta_{2}+\delta_{3})}$, $A_0$, 
and $U^{-\delta_{3}}$, and 
\item $h_{\mu_{2}''}(a'\mid \Fsc)>0$.
\end{enumerate}

Again, let $p_{*}(\mu_{2}'')=\overline{\mu_{2}''}$. Since $\overline{\mu_{2}''}$ is invariant under $U^{-\delta_{3}}$ and the diagonal of $\delta_{3}$ in $A$, by \cref{thm:opproot}, $\overline{\mu_{2}''}$ is invariant under $U^{\delta_{3}}$.  
Since $\overline{\mu_{2}''}$ is invariant under $U^{\delta_{3}}$ and $U^{-(\delta_{2}+\delta_{3})}$, $\overline{\mu_{2}''}$ is also invariant under $U^{-\delta_{2}}$. Combined with the fact that $\overline{\mu_{2}''}$ is invariant under $U^{(\delta_{1}+\delta_{2})}$, we can deduce that $\overline{\mu_{2}''}$ is invariant under $U^{\delta_{1}}$. As $\overline{\mu_{2}''}$ is invariant under the diagonal of $\delta_{1}$, $\overline{\mu_{2}''}$ is invariant under $U^{-\delta_{1}}$ by  \cref{thm:opproot}. Finally, since $\overline{\mu_{2}''}$ is invariant under $U^{(\delta_{1}+\delta_{2})}$ and $U^{-\delta_{1}}$, $\overline{\mu_{2}''}$ is invariant under $U^{\delta_{2}}$. In summary, $\overline{\mu_{2}''}$ is invariant under $U^{\pm\delta_{1}}$, $U^{\pm \delta_{2}}$, and $U^{\pm\delta_{3}}$. This implies that $\overline{\mu_{2}''}$ is $L_{3}$-invariant.

 As above,  applying \cref{averagingA} with $H=L_{3}$ and $A_0= A_{L_{3}} $,  
 we obtain a measure $\mu_{3}$  that satisfies all conclusions in \Cref{prop:firstind}. 

Moreover, $a$ and $a'$ are contained in the rank-2 subgroup of $A_{L_3}$ generated by the diagonals of $\delta_1$ and $\delta_3$.  Since entropy is a non-zero seminorm on this subspace, we may assume $h_{\mu_{3}}(a\mid \Fsc)>0$ for $a$ in the diagonal of $\delta_1$ in $A$ or the diagonal of $\delta_3$ in $A$.

\subsubsection*{\it  $L_{3}$ is of type $B_{3}$ or $BC_3$ and $\delta_1$ is the left-most (long) root:} In this case, we check that the subgroup of $A$ generated by diagonals of $\delta_{1}$  and $\delta_{3}$ in $A$ is the same as the subgroup  of $A$ generated by $\ker \delta_{2} \cap \ker (\delta_{1}+\delta_{2}+2\delta_{3})$ and $\ker(\delta_{1}+\delta_{2})\cap \ker(\delta_{2}+2\delta_{3})$.  Moreover, the group $U^{[-\delta_3]}$ (with Lie algebra $\lieg ^{-\delta_3}\oplus \lieg^{-2\delta_3}$) commutes with the pair $U^{\delta_{2}}$ and $ U^{-(\delta_{1}+\delta_{2}+2\delta_{3})}$, and with 
the pair $U^{\delta_{1}+\delta_{2}}$ and $U^{-(\delta_{2}+2\delta_{3})}.$

The arguments are almost verbatim to the above; we leave  the details to the reader.  

\subsubsection*{\it  $L_{3}$ is of type $B_{3}$ or $BC_3$ and $\delta_1$ is the right-most (short) root:} In this case, we check that the subgroup of $A$ generated by diagonals of $\delta_{1}$  and $\delta_{3}$ in $A$ is the same as the subgroup  of $A$ generated by $\ker \delta_{2} \cap \ker (2\delta_{1}+\delta_{2}+\delta_{3})$ and $\ker(2\delta_{1}+\delta_{2})\cap \ker(\delta_{2}+\delta_{3})$.    Moreover, the group $U^{-\delta_3}$ (with Lie algebra $\lieg ^{-\delta_3}$) commutes with the pair $U^{\delta_{2}}$ and  $U^{-(2\delta_{1}+\delta_{2}+\delta_{3})}$, and with 
the pair $U^{2\delta_{1}+\delta_{2}}$ and $U^{-(\delta_{2}+\delta_{3})}.$

The arguments are almost verbatim to the above; we leave  the details to the reader.  

\subsubsection*{\it  $L_{4}$ is of type $D_{4}$:} In this case, we check that the subgroup of $A$ generated by diagonals of $\delta_{1}$  and $\delta_{3}$ in $A$ is the same as the subgroup  of $A$ generated by 
$\ker \delta_{2} \cap \ker (\delta_{1}+\delta_{2}+\delta_{3}+\delta_{4})\cap \ker (\delta_4)$ and $\ker(\delta_{1}+\delta_{2})\cap \ker(\delta_{2}+\delta_{3}+\delta_4)\cap \ker \delta_4$.   

If there is $a\in \ker \delta_{2} \cap \ker (\delta_{1}+\delta_{2}+\delta_{3}+\delta_{4})\cap \ker (\delta_4)$ with  
$h_{\mu_{2}'}(a\mid \Fsc)>0$, we apply \cref{averagingunipotent} with $I= \{\delta_{2}, -(\delta_{1}+\delta_{2}+\delta_{3}+\delta_{4}), \delta_4\}$,  $A_0\subset  A_{L_3}$ the subgroup generated by the diagonal of $\delta_1$ and $\delta_3$, and $R= U^{-\delta_{3}}$,
we obtain a measure $\mu_{2}''$ such that 
\begin{enumerate}
\item $\mu_{2}''$ is invariant under  $A_{0}$, 
\item $p_{*}(\mu_{2}'')$ is invariant under $U^{\delta_{2}}$, $U^{-(\delta_{1}+\delta_{2}+\delta_{3}+\delta_{4})}$, $U^{\delta_4}$, $U^{-\delta_3}$, and $A_0$, 
and has exponentially small mass at $\infty$, and
\item $h_{\mu_{2}''}(a\mid \Fsc)>0$.
\end{enumerate}
Since $A_0$ contains the diagonal of $\pm \delta_3$, we have that $p_{*}(\mu_{2}'')$ is $U^{\delta_3}$-invariant.  We can construct $-\delta_1$ as a positive combination of $\delta_{2}, -(\delta_{1}+\delta_{2}+\delta_{3}+\delta_{4}), \delta_4,$ and $ -\delta_3$ and similarly conclude that $p_{*}(\mu_{2}'')$ is $U^{\pm \delta_1}$-invariant.  Finally, 
construct $-\delta_4$ as a positive combination of $ -(\delta_{1}+\delta_{2}+\delta_{3}+\delta_{4}), \delta_1,\delta_{2},$ and $ \delta_3$ and 
construct $-\delta_2$ as a positive combination of $ -(\delta_{1}+\delta_{2}+\delta_{3}+\delta_{4}), \delta_1, \delta_3,$ and $ \delta_4$.  We 
similarly conclude that $p_{*}(\mu_{2}'')$ is $U^{\pm \delta_4}$-invariant and it follows that   $p_{*}(\mu_{2}'')$ is $G$-invariant.   We finish by applying \cref{averagingA} by applying \cref{averagingA} as above.  

Similarly, if there is $a'\in \ker(\delta_{1}+\delta_{2})\cap \ker(\delta_{2}+\delta_{3}+\delta_4)\cap \ker \delta_4$ with  
$h_{\mu_{2}'}(a'\mid \Fsc)>0$, we apply \cref{averagingunipotent} with $I= \{\delta_{1}+\delta_{2},-( \delta_{2}+\delta_{3} +\delta_4), \delta_4\}$,  $A_0\subset  A_{L_3}$ the subgroup generated by the diagonal of $\delta_1$ and $\delta_3$ in $A$, and $R= U^{-\delta_{3}}$,
we obtain a measure $\mu_{2}''$ such that 
\begin{enumerate}
\item $\mu_{2}''$ is invariant under  $A_{0}$, 
\item $p_{*}(\mu_{2}'')$ is invariant under $U^{\delta_{1}+\delta_{2}}$, $U^{-( \delta_{2}+\delta_{3} +\delta_4)}$, $U^{\delta_4}$, $U^{-\delta_3}$, and $A_0$, 
and has exponentially small mass at $\infty$, and
\item $h_{\mu_{2}''}(a'\mid \Fsc)>0$.
\end{enumerate}
Since $A_0$ contains the diagonal of $\pm \delta_3$, we have that $p_{*}(\mu_{2}'')$ is $U^{\delta_3}$-invariant.  We can construct $\delta_1$ as a positive combination of $\delta_{1}+\delta_{2}, -( \delta_{2}+\delta_{3} +\delta_4), \delta_4,$ and $ \delta_3$ and thus   conclude that $p_{*}(\mu_{2}'')$ is $U^{\pm \delta_1}$-invariant.  
We similarly construct $\pm \delta_2$ from $-( \delta_{2}+\delta_{3} +\delta_4), \delta_3, \delta_4$ and $\delta_1+ \delta_2, -\delta_1$ and then construct $- \delta_4$ from $-( \delta_{2}+\delta_{3} +\delta_4), \delta_2, \delta_3$.  It follows that   $p_{*}(\mu_{2}'')$ is $G$-invariant and we finish by applying \cref{averagingA} as above.  
\end{proof}

\subsubsection{Base case of  induction, indexing of simple roots}

When  $\Phi(A,G)$ is of type $A_3$, $  B_3$, $ C_3$, $ BC_3$, or $D_4$, \cref{prop:oneroot} follows from \cref{prop:firstind}.  Otherwise, we induct on the number nodes of in the Dynkin diagram, producing extra invariance of the factor measure $p_*\mu_j$ in $G/\Gamma$ by averaging at each step of induction.       

Before setting up our induction, recall the measure guaranteed by  \cref{prop:firstind} is $A_{L_3}$-invariant and satisfies $h_{\mu_{3}}(a\mid \Fsc)>0$ where $a$ is in the diagonal of $\delta_1$ or $\delta_3$ in $A$.  


When $\Phi(A,G)$ is not of type $F_4$, $\delta_1$ is either the left-most or right-most root in \cref{table:dynk}.
For all such root systems, we further specify the indexing $\{\delta_{i}\}_{i=1,\dots, n}$ on the simple roots $\Delta(A,G)$   so that $\delta_{i-1}$ and $\delta_{i}$ are adjacent for all $i$ except for the following case: if the Dynkin diagram branches (i.e.\ is trivalent) at the simple root $\delta_{k}$ (in root systems of type $D_\ell, E_6, E_7$, or $E_8$)  we allow the order to  satisfy that
 \begin{enumerate}
 \item  $\delta_{k}$ is connected with $\delta_{k+1}$, $\delta_{k-1}$, and $\delta_{k+2}$ 
 \item  and $\delta_{k+1}$ is connected only to $\delta_{k}$.

\begin{table}[h]
\footnotesize
\centering

\def\scales{.5}
\def\bdbox{\useasboundingbox (-.5,-.5) rectangle (4.5,1.6)   }

\begin{tikzpicture}[scale=\scales]
\bdbox;

 \node (1) at (-.5,0) {};
    \node[dnode,label=270:$\delta_{k-1}$\ \ \ ] (2) at (1,0) {};
    \node[dnode,label=270:$\delta_{k}$] (3) at (2,0) {};
    \node[dnode,label=above:$\delta_{k+1}$] (4) at (2,1) {};
    \node[dnode,label=270:\ \ \ $\delta_{k+2}$] (5) at (3,0) {};
        \node[] (6) at (4.5,0) {};
    \path (1) edge[sedge,dashed, dash phase=1.4pt] (2)
(2) edge[sedge] (3)
(3) edge[sedge] (4)
          (3) edge[sedge] (5)
                (5) edge[sedge,dashed, dash phase=1.4pt] (6);
 \end{tikzpicture}
\caption{Conventions for indexing in  trivalent Dynkin diagrams}
\end{table}
\end{enumerate}

\subsubsection{Action of Weyl group and  conventions for induction on the number of roots}
In our induction below, we will specify a root $\delta_* \in \Phi(L_3, A_{L_3})$ such that $h_{\mu_{3}}(a\mid \Fsc)>0$ for some $a$ in the diagonal of $\delta_*$ in $A$.
When $ \Phi(G, A)$ is not of type $F_4$, taking $\delta_*=\delta_1$ to be the left- or right-most root following the above conventions simplifies the inductive steps.  We thus argue that, after acting by and element of the Weyl group, there is a system of simple roots $\what \Delta(G, A)$ with an indexing of simple roots following the conventions above for which we may assume $\delta_* = \hat \delta_1$.  

We may assume the rank $n$ of the Dynkin diagram is at least $4$.  We first observe that the diagonals of the short roots as well as the diagonals of all long roots in $\Phi({L_3},A_{L_3})$ generate $A_{L_3}$.  If $L_3$ is of type $A_3$, we may assume for $\delta_*=\delta_1$ or $\delta_*=\delta_3$, there is an element $a$ that is in the diagonal of $\delta_*$ in $A_{L_3}$ such that  $h_{\mu_{2}}(a\mid \Fsc)>0$.
 If $L_3$ is of type $B_3$ (or $BC_3$) we necessarily have that $\delta_1$ is the right-most (short) root; we may assume for one of the short roots $\delta_*=\delta_1$,  $\delta_*=\delta_2+\delta_1$, or $\delta_*=\delta_3+\delta_2+\delta_1$ that there is an element $a$ that is in the diagonal of $\delta_*$ in $A_{L_3}$ such that  $h_{\mu_{2}}(a\mid \Fsc)>0$.
 If $L_3$ is of type $C_3$ we necessarily have that $\delta_1$ is the right-most (long) root; we may assume for one of the long roots $\delta_*=\delta_1$,  $\delta_*=2\delta_2+\delta_1$, or $\delta_*=2\delta_3+2\delta_2+\delta_1$ that there is an element $a$ that is in the diagonal of $\delta_*$ in $A_{L_3}$ such that  $h_{\mu_{2}}(a\mid \Fsc)>0$.

\begin{claim}\label[claim]{claim:weyl2}
Let $\Phi(A,G)$ be of rank at least $4$ and not of type $F_4$ or $D_4$ and let $\Delta(A,G)=\{\delta_1, \delta_2, \dots,\}$ be the above choice of simple roots.   
There is a system of positive roots $\what \Delta (A,G)= \{\hat \delta_1, \hat \delta_2, \dots,\}$ such that 
\begin{enumerate}
\item the subgroup  $\what L_3$ generated by $U^{\hat \delta_1}, U^{\hat \delta_2}, U^{\hat \delta_3}$ coincides with $L_3$, and 
\item $ \delta_*$ is the left-most root (in root systems of type $A_n, D_n, E_6, E_7, E_8$), or the left-most root or right-most root (in root systems  $B_n, C_n, BC_n$).  
\end{enumerate}
\end{claim}

When $\delta_*\neq \delta_1$, we act by explicit elements of the Weyl group to produce the new  system of simple roots $\what \Delta(G, A)$ with the above properties.  See \cref{table:weyl}.

 {
 \begin{table}[h]
 \renewcommand{\arraystretch}{1.0}  
  \newcolumntype{C}[1]{>{\centering\arraybackslash} m{#1} }

\def\scales{.5}
\def\bdbox{\pgfresetboundingbox \useasboundingbox (0,.-.2) rectangle (4.0,1.2)   }
\def\bdboxALT{\pgfresetboundingbox \useasboundingbox (0,-1.3) rectangle (4.0,1.3)   }
\def\bdboxALTo{\pgfresetboundingbox \useasboundingbox (0,.-.2) rectangle (3.0,1.2)   }
\def\bdboxALTb{\pgfresetboundingbox \useasboundingbox (0,.-.2) rectangle (4,1.2)   }

\footnotesize
\centering
\begin{tabular}{ | C{.9cm}  | C{1.6cm} | C{1.5cm}| C{3.5cm}| C{3.1cm} | } \hline
$L_{3}$ root system  & $\delta_* $ & Weyl group element & 
Dynkin diagram assoc.\ to colls.\ of simple roots 
$\{\delta_1, \delta_2, \ldots\}$ and $\{\hat \delta_1, \hat\delta_2, \ldots\} $ 
& Relationship between simple roots $\{\delta_1, \delta_2, \ldots\}$ and $ \{\hat \delta_1, \hat\delta_2, \ldots\} $ \\
 \hline

$A_3$ 
&$\delta_3$ 
&$ s_{\delta_1+\delta_2+\delta_3}\circ  s_{\delta_2}$
&

\begin{tikzpicture}[scale=\scales]
\bdboxALTb; 

 \node[dnode,label=above:$\delta_1$] (1) at (0,0) {};
    \node[dnode,label=above:$\delta_2$] (2) at (1,0) {};
    \node[dnode,label=above:$\delta_{3}$,] (3) at (2,0) {};
    \node[dnode,label=above:$\delta_{4}$] (4) at (3,0) {};
        \node[] (5) at (4.5,0) {};
    \path (1) edge[sedge] (2)
(2) edge[sedge] (3)
(3) edge[sedge] (4)
          (4) edge[sedge,dashed, dash phase=1.4pt] (5);
 \end{tikzpicture}

 \begin{tikzpicture}[scale=\scales]
\bdboxALTb;
    \node[dnode,label=above:$\hat \delta_1$ ] (1) at (0,0) {};
    \node[dnode,label=above:$\hat \delta_2$ ] (2) at (1,0) {};
    \node[dnode,label=above:$\hat \delta_{3}$ ] (3) at (2,0) {};
    \node[dnode,label=above:$\hat \delta_{4}$ ] (4) at (3,0) {};
        \node[] (5) at (4.5,0) {};
    \path (1) edge[sedge] (2)
(2) edge[sedge] (3)
(3) edge[sedge] (4)
          (4) edge[sedge,dashed, dash phase=1.4pt] (5);
 \end{tikzpicture}
 &

 $\hat \delta_1=-\delta_3$, 
 
 $\hat \delta_2=-\delta_2$,\quad 
 $\hat \delta_{3}= -\delta_1$, 
 
 $\hat \delta_{4} = \delta_1+\delta_2+\delta_3+\delta_4$, 

$\hat \delta_k =\delta_k$, \ $k\ge 5$
\\ \hline

$A_3$ 
&$\delta_3$ 
&$ s_{\delta_1+\delta_2+\delta_3}\circ  s_{\delta_2}$
&

 \begin{tikzpicture}[scale=\scales]
\bdboxALTo;

 \node[dnode,label=above:$\delta_1$] (1) at (0,0) {};
    \node[dnode,label=above:$\delta_2$] (2) at (1,0) {};
    \node[dnode,label=above:$\delta_{3}$] (3) at (2,0) {};
    \node[dnode,label=above:$\delta_{4}$] (4) at (3,0) {};
    \node[] (6) at (5,0) {};
    \path (1) edge[sedge] (2)
(2) edge[sedge] (3)
(3) edge[dedge] (4);
 \end{tikzpicture}

  \begin{tikzpicture}[scale=\scales]
\bdboxALTo;

 \node[dnode,label=above:$\hat \delta_1$] (1) at (0,0) {};
    \node[dnode,label=above:$\hat \delta_2$] (2) at (1,0) {};
    \node[dnode,label=above:$\hat \delta_{3}$] (3) at (2,0) {};
    \node[dnode,label=above:$\hat  \delta_{4}$] (4) at (3,0) {};
    \node[] (6) at (5,0) {};
    \path (1) edge[sedge] (2)
(2) edge[sedge] (3)
(3) edge[dedge] (4);
 \end{tikzpicture}
 &

 $\hat \delta_1=-\delta_3$, 
 
 $\hat \delta_2=-\delta_2$,  \quad 
 $\hat \delta_{3}= -\delta_1$, 
 
 $\hat \delta_{4} = \delta_1+\delta_2+\delta_3+\delta_4$
\\ \hline

$A_3$ 
&$\delta_3$ 
&$ s_{\delta_1+\delta_2+\delta_3}\circ  s_{\delta_2}$
&

 \begin{tikzpicture}[scale=\scales]
\bdboxALTo;

 \node[dnode,label=above:$\delta_1$] (1) at (0,0) {};
    \node[dnode,label=above:$\delta_2$] (2) at (1,0) {};
    \node[dnode,label=above:$\delta_{3}$] (3) at (2,0) {};
    \node[dnode,label=above:$\delta_{4}$] (4) at (3,0) {};
    \node[] (6) at (5,0) {};
    \path (1) edge[sedge] (2)
(2) edge[sedge] (3)
(4) edge[dedge] (3);
 \end{tikzpicture}

  \begin{tikzpicture}[scale=\scales]
\bdboxALTo;

 \node[dnode,label=above:$\hat \delta_1$] (1) at (0,0) {};
    \node[dnode,label=above:$\hat \delta_2$] (2) at (1,0) {};
    \node[dnode,label=above:$\hat \delta_{3}$] (3) at (2,0) {};
    \node[dnode,label=above:$\hat  \delta_{4}$] (4) at (3,0) {};
    \node[] (6) at (5,0) {};
    \path (1) edge[sedge] (2)
(2) edge[sedge] (3)
(4) edge[dedge] (3);
 \end{tikzpicture}
 &

 $\hat \delta_1=-\delta_3$, 
 
 $\hat \delta_2=-\delta_2$,  \quad 
 $\hat \delta_{3}= -\delta_1$, 
 
 $\hat \delta_{4} = 2\delta_1+2\delta_2+2\delta_3+\delta_4$ 
\\ \hline

$A_3$ 
&$\delta_3$ 
&$ s_{\delta_1+\delta_2+\delta_3}\circ  s_{\delta_2}$
&

\begin{tikzpicture}[scale=\scales]
\bdboxALT;

 \node[dnode,label=below:$\delta_1$] (1) at (0,0) {};
    \node[dnode,label=below:$\delta_2$] (2) at (1,0) {};
    \node[dnode,label=below:$\delta_{3}$] (3) at (2,0) {};
      \node [dnode,label={[label distance=-3pt]10:$\delta_{4}$}]   
    (4) at (3,.5) {};
      \node [dnode,label={[label distance=-3pt]350:$\delta_{5}$}]   (5) at (3,-.5) {};
        \node[] (6) at (4.5,-.5) {};
    \path (1) edge[sedge] (2)
(2) edge[sedge] (3)
(3) edge[sedge] (4)
          (3) edge[sedge] (5)
          (5) edge[sedge,dashed, dash phase=1.4pt] (6);
 \end{tikzpicture}

\begin{tikzpicture}[scale=\scales]
\bdboxALT;
    \node[dnode,label=below:$\hat \delta_1$ ] (1) at (0,0) {};
    \node[dnode,label=below:$\hat \delta_2$ ] (2) at (1,0) {};
    \node[dnode,label=below:$\hat \delta_{3}$ ] (3) at (2,0) {};
      \node [dnode,label={[label distance=-3pt]4:$\hat \delta_{4}$}] (4) at (3,.5) {};
      \node [dnode,label={[label distance=-3pt]350:$\hat \delta_{5}$}]    (5) at (3,-.5) {};
        \node[] (6) at (4.5,-.5) {};
    \path (1) edge[sedge] (2)
(2) edge[sedge] (3)
(3) edge[sedge] (4)
          (3) edge[sedge] (5)
          (5) edge[sedge,dashed, dash phase=1.4pt] (6);
 \end{tikzpicture}
 &

 $\hat \delta_1=-\delta_3$, 
 
 $\hat \delta_2=-\delta_2$,  
 
 $\hat \delta_{3}= -\delta_1$, 
 
 $\hat \delta_{4} = \delta_1+\delta_2+\delta_3+\delta_4$, 
 
 $\hat \delta_{5} = \delta_1+\delta_2+\delta_3+\delta_5$, 

 $\hat \delta_k =\delta_k$, \ $k\ge 6$
\\ \hline

 $B_3$ 
&$\delta_1+\delta_2$ 
&$  s_{\delta_2}$
&

 \begin{tikzpicture}[scale=\scales]
\bdbox;
    \node[dnode,label=above:$\delta_1$] (1) at (4,0) {};
    \node[dnode,label=above:$\delta_2$] (2) at (3,0) {};
    \node[dnode,label=above:$\delta_3$] (3) at (2,0) {};
        \node[dnode,label=above:$\delta_4$] (4) at (1,0) {};
    \node[] (5) at (-.5,0) {};

    \path (2) edge[dedge] (1)

          (2) edge[sedge] (3)
          (2) edge[sedge] (3)
          (3) edge[sedge] (4)
          (4) edge[sedge,dashed, dash phase=1.4pt] (5)
          ;
\end{tikzpicture}

 \begin{tikzpicture}[scale=\scales]
\bdbox;
    \node[dnode,label=above:$\hat \delta_1$] (1) at (4,0) {};
    \node[dnode,label=above:$\hat \delta_2$] (2) at (3,0) {};
    \node[dnode,label=above:$\hat \delta_3$] (3) at (2,0) {};
        \node[dnode,label=above:$\hat \delta_4$] (4) at (1,0) {};
    \node[] (5) at (-.5,0) {};

    \path (2) edge[dedge] (1)

          (2) edge[sedge] (3)
          (2) edge[sedge] (3)
          (3) edge[sedge] (4)
          (4) edge[sedge,dashed, dash phase=1.4pt] (5)
          ;
\end{tikzpicture}
&

$\hat \delta_1=\delta_1+ \delta_2$, 

 $\hat \delta_2=-\delta_2$, 

 $\hat \delta_{3}= \delta_3+\delta_2$, 
 \quad
 $\hat \delta_{4} = \delta_4$, 
 
 $\hat \delta_k =\delta_k$, \ $k\ge 5$
\\ \hline

$B_3$
&$\delta_1+\delta_2+\delta_3$ 
&$  s_{\delta_3}\circ s_{\delta_2}$
&
 \begin{tikzpicture}[scale=\scales]
\bdbox;
     \node[dnode,label=above:$\delta_1$] (1) at (4,0) {};
    \node[dnode,label=above:$\delta_2$] (2) at (3,0) {};
    \node[dnode,label=above:$\delta_3$] (3) at (2,0) {};
        \node[dnode,label=above:$\delta_4$] (4) at (1,0) {};
    \node[] (5) at (-.5,0) {};

    \path (2) edge[dedge] (1)

          (2) edge[sedge] (3)
          (2) edge[sedge] (3)
          (3) edge[sedge] (4)
          (4) edge[sedge,dashed, dash phase=1.4pt] (5)
          ;
\end{tikzpicture}

 \begin{tikzpicture}[scale=\scales]
\bdbox;
    \node[dnode,label=above:$\hat \delta_1$] (1) at (4,0) {};
    \node[dnode,label=above:$\hat \delta_2$] (2) at (3,0) {};
    \node[dnode,label=above:$\hat \delta_3$] (3) at (2,0) {};
        \node[dnode,label=above:$\hat \delta_4$] (4) at (1,0) {};
    \node[] (5) at (-.5,0) {};

    \path (2) edge[dedge] (1)

          (2) edge[sedge] (3)
          (2) edge[sedge] (3)
          (3) edge[sedge] (4)
          (4) edge[sedge,dashed, dash phase=1.4pt] (5)
          ;
\end{tikzpicture}
&

$\hat \delta_1=\delta_1+ \delta_2+\delta_3$, 

 $\hat \delta_2=-\delta_2-\delta_3$,  
 
 $\hat \delta_{3}= \delta_2$, \quad 
 $\hat \delta_{4} = \delta_3+\delta_4$, 

 $\hat \delta_k =\delta_k$, \ $k\ge 5$
\\ \hline

$C_3$ 
&$\delta_1+2\delta_2$ 
&$  s_{\delta_2}$
&

 \begin{tikzpicture}[scale=\scales]
\bdbox;
      \node[dnode,label=above:$\delta_1$] (1) at (4,0) {};
    \node[dnode,label=above:$\delta_2$] (2) at (3,0) {};
    \node[dnode,label=above:$\delta_3$] (3) at (2,0) {};
        \node[dnode,label=above:$\delta_4$] (4) at (1,0) {};
    \node[] (5) at (-.5,0) {};

    \path (1) edge[dedge] (2)

          (2) edge[sedge] (3)
          (2) edge[sedge] (3)
          (3) edge[sedge] (4)
          (4) edge[sedge,dashed, dash phase=1.4pt] (5)
          ;
\end{tikzpicture}

 \begin{tikzpicture}[scale=\scales]
\bdbox;
       \node[dnode,label=above:$\hat \delta_1$] (1) at (4,0) {};
    \node[dnode,label=above:$\hat \delta_2$] (2) at (3,0) {};
    \node[dnode,label=above:$\hat \delta_3$] (3) at (2,0) {};
        \node[dnode,label=above:$\hat \delta_4$] (4) at (1,0) {};
    \node[] (5) at (-.5,0) {};

    \path (1) edge[dedge] (2)

          (2) edge[sedge] (3)
          (2) edge[sedge] (3)
          (3) edge[sedge] (4)
          (4) edge[sedge,dashed, dash phase=1.4pt] (5)
          ;
\end{tikzpicture}
&
$\hat \delta_1=\delta_1+ 2 \delta_2$, 

 $\hat \delta_2=-\delta_2$, 
 
 $\hat \delta_{3}= \delta_3+\delta_2$,
 \quad  
 $\hat \delta_{4} = \delta_4$, 

$\hat \delta_k =\delta_k$, \ $k\ge 5$
\\ \hline

$C_3$
&$\delta_1+2\delta_2+2\delta_3$ 
&$  s_{\delta_3}\circ s_{\delta_2}$
&
 \begin{tikzpicture}[scale=\scales]
\bdbox;
    \node[dnode,label=above:$\delta_1$] (1) at (4,0) {};
    \node[dnode,label=above:$\delta_2$] (2) at (3,0) {};
    \node[dnode,label=above:$\delta_3$] (3) at (2,0) {};
        \node[dnode,label=above:$\delta_4$] (4) at (1,0) {};
    \node[] (5) at (-.5,0) {};

    \path (1) edge[dedge] (2)

          (2) edge[sedge] (3)
          (2) edge[sedge] (3)
          (3) edge[sedge] (4)
          (4) edge[sedge,dashed, dash phase=1.4pt] (5)
          ;
\end{tikzpicture}

 \begin{tikzpicture}[scale=\scales]
\bdbox;
     \node[dnode,label=above:$\hat \delta_1$] (1) at (4,0) {};
    \node[dnode,label=above:$\hat \delta_2$] (2) at (3,0) {};
    \node[dnode,label=above:$\hat \delta_3$] (3) at (2,0) {};
        \node[dnode,label=above:$\hat \delta_4$] (4) at (1,0) {};
    \node[] (5) at (-.5,0) {};

    \path (1) edge[dedge] (2)

          (2) edge[sedge] (3)
          (2) edge[sedge] (3)
          (3) edge[sedge] (4)
          (4) edge[sedge,dashed, dash phase=1.4pt] (5)
          ;
\end{tikzpicture}
&

$\hat \delta_1=\delta_1+ 2\delta_2+2\delta_3$, 

 $\hat \delta_2=-\delta_2-\delta_3$,  
 
 $\hat \delta_{3}= \delta_2$,  \quad 
 $\hat \delta_{4} = \delta_3+\delta_4$, 

 $\hat \delta_k =\delta_k$, \ $k\ge 5$
 \\ \hline
\end{tabular}
\caption{Action by certain Weyl group elements on systems of simple roots;  in the third column, $s_{\alpha}$ denotes the Weyl group element obtained by reflection relative to the root $\alpha$.}\label{table:weyl}\label{table:144}
\end{table}
}

By  replacing $\Delta(A,G)$ with $\what \Delta(A,G)$ from \cref{claim:weyl2} if needed, when $\Phi(A,G)$ is not of type $F_4$  we will further assume the  system of simple roots $\Delta(A,G)$ and indexing satisfies the following:
\begin{enumerate}
\item $\delta_* = \delta_1$ is left-most  root for  the Dynkin diagram if $\Phi(A,G)$ is of type $A_n$, $D_n$, $E_6$, $ E_7$, or $E_8$. 
\item $\delta_* = \delta_1$ is either the left-most or right-most  root for  the Dynkin diagram if $\Phi(A,G)$ is of type $B_n, C_n, $ or $BC_n$.  
\end{enumerate}
 When $\Phi(A,G)$ is of type $F_4$, we have $L_3$ is generated by the 3 right-most roots.  We re-index the Dynkin diagram if needed from right to left  so that the roots  $\{\delta_1, \delta_2, \delta_3\}$ generating a rank-3 subgroup of type $C_3$.  We thus have either
  \begin{enumerate}[resume]
 \item $\delta_*=\delta_1$ or $\delta_*=\delta_3$  if $\Phi(A,G)$ is of type $F_4$.  
\end{enumerate}

\subsubsection{Inductive step in proof of  \Cref{prop:oneroot}}
Let $n= \rank_\R(G)$.  For any root $\delta$, the root subgroup corresponding to  $\delta$ is denoted by $U^{\delta}$. For $0\le j\le n-1$, let $L_{j}$ be a connected closed subgroup of $G$ which is generated by the simple root subgroups $U^{\pm\delta_{1}},\dots, U^{\pm\delta_{j}}$.
 \Cref{prop:firstind} provides the base case for the induction and completes the proof of  \Cref{prop:oneroot} if $n=\rank(G) = 3$ (or if  $\Phi(A,G)$ is  of type $D_4$).  
 
 We thus assume $n\ge 4$ (and that $\Phi(A,G)$ is not of type $D_4$).  
With the  indexing conventions in \cref{table:144},  \Cref{prop:oneroot} then follows with $\mu_{\infty}=\mu_{n}$ after we establish the following inductive hypothesis: 

\begin{proposition}\label[proposition]{prop:induction}
For $3\le j\le (n-1)$, suppose  there exists a probability measure $\mu_{j}$ on $M^{\alpha}$ such that 
\begin{enumerate}[label=(\alph*)]
\item $\mu_{j}$ is $A_{L_{j}}$-invariant, 
\item $\overline{\mu_{j}}=p_{*}(\mu_{j})$ is $L_{j}$-invariant and has exponentially small mass at $\infty$, and
\item there is $a_*\in A_{L_{j}}$ which is diagonal for $\delta_*$ in $A$ such that $h_{\mu_{j}}(a_*\mid \Fsc)>0$.
\end{enumerate}
Then, there is a probability measure $\mu_{j+1}$ on $M^{\alpha}$ such that
\begin{enumerate}
\item $\mu_{j+1}$ is $A_{L_{j+1}}$-invariant,  
\item $\overline{\mu_{j+1}}=p_{*}(\mu_{j+1})$ is $L_{j+1}$-invariant and has exponentially small mass at $\infty$, and
\item  $h_{\mu_{j+1}}(a_*\mid \Fsc)>0$.
\end{enumerate}
\end{proposition}

\begin{proof}
Fix $j\ge 3$.  
With the indexing conventions in \cref{table:144},  $\Phi(L_{j+1}, A_{L_{j+1}})$ is connected.  
We claim there exists a  collection of roots $I\subset \Phi(L_{j+1}, A_{L_{j+1}})$ such that 
\begin{enumerate}
\item $I$ is closed in $\Phi(L_{j+1}, A_{L_{j+1}})$ and $U^I$ is a unipotent subgroup;
\item $ \delta_*$ commutes with $U^I$;
\item for every $1\le k\le j$,   either $\delta_k$ or $-\delta_k$  centralizes $ U^I$;
\item $\pm \delta_{j+1}$ is a linear combination of elements of $I$ and $\pm \delta_1,\dots, \pm \delta_j$ with non-negative integer coefficients.  
\end{enumerate}
The choice of such $I$ for all possible root systems   $ \Phi(L_{j+1}, A_{L_{j+1}})$ and choice of $\delta_*$ that can arise  are listed in \cref{table:induction}.
{\renewcommand{\arraystretch}{1.3}  
  \newcolumntype{C}[1]{>{\centering\arraybackslash} m{#1} }
\begin{table}[h]
\footnotesize
\centering
\def\scales{.5}
\def\bdbox{\useasboundingbox (-.5,-.2) rectangle (4,1.2)   }
\def\bdboxAA{\useasboundingbox (-.5,-.5) rectangle (4,1.5)   }
\def\bdboxBB{\useasboundingbox (0,-.2) rectangle (3,1.2)   }
 
 \begin{tabular}{ | C{2cm}  | C{2cm} |   C{2.8cm}| C{4cm} |} \hline
Root system of $L_{j+1}$ & $\delta_* $ & $\Delta(L_{j+1}, A_{L_{j+1}})$&   $I$  \\
 \hline
  $A_{j+1}$,  $E_6$, $E_7$, $E_8$
&  $\delta_1$, left-most 
	& 
	\begin{tikzpicture}[scale=\scales]
\bdbox;
 \node[dnode,label=above:$\delta_1$] (1) at (0,0) {};
    \node[dnode,label=above:$\delta_2$] (2) at (1,0) {};
    \node[dnode,label=above:$\delta_{j}$] (3) at (2.5,0) {};
    \node[dnode,label=above:$\ \ \ \ \delta_{j+1}$] (4) at (3.5,0) {};
    \path (1) edge[sedge] (2)
	(3) edge[sedge] (4)
          (2) edge[sedge,dashed, dash phase=1.4pt] (3);
 \end{tikzpicture}
	&
$  \delta_{j},$ $ \delta_{j}+ \delta_{j+1},$ $ - \delta_{j+1}
 $
\\ \hline
 $B_{j+1},$ $C_{j+1}$,  $BC_{j+1}$& 

 $\delta_1$, right-most	&
  	\begin{tikzpicture}[scale=\scales]
\bdbox;
 \node[dnode,label=above:$\delta_{j+1} \ \ \ \ $] (1) at (0,0) {};
    \node[dnode,label=above:$\delta_{j}$] (2) at (1,0) {};
    \node[dnode,label=above:$\delta_{2}$] (3) at (2.5,0) {};
    \node[dnode,label=above:$ \delta_{1}$] (4) at (3.5,0) {};
    \path (1) edge[sedge] (2)
	(3) edge[dedgeplain] (4)
          (2) edge[sedge,dashed, dash phase=1.4pt] (3);
 \end{tikzpicture}  
  &
$  \delta_{j},$ $ \delta_{j}+ \delta_{j+1}, $$- \delta_{j+1}
 $
\\ \hline
  $ B_{j+1}$ &   $\delta_1$, left-most (long)	&
 \begin{tikzpicture}[scale=\scales]
\bdbox;
 \node[dnode,label=above:$\delta_1$] (1) at (0,0) {};
    \node[dnode,label=above:$\delta_2$] (2) at (1,0) {};
    \node[dnode,label=above:$\delta_{j}$] (3) at (2.5,0) {};
    \node[dnode,label=above:$\ \ \ \ \delta_{j+1}$] (4) at (3.5,0) {};
        \node[] (6) at (5,0) {};
    \path (1) edge[sedge] (2)
	(3) edge[dedge] (4)
          (2) edge[sedge,dashed, dash phase=1.4pt] (3);
 \end{tikzpicture}
  &
$  \delta_{j},$ $ \delta_{j}+ \delta_{j+1}, $ $- \delta_{j+1}
 $
\\ \hline
 { $C_{j+1}$} &   $\delta_1$,  left-most (short) 	&  
 \begin{tikzpicture}[scale=\scales]
\bdbox;
 \node[dnode,label=above:$\delta_1$] (1) at (0,0) {};
    \node[dnode,label=above:$\delta_2$] (2) at (1,0) {};
    \node[dnode,label=above:$\delta_{j}$] (3) at (2.5,0) {};
    \node[dnode,label=above:$ \ \ \ \  \delta_{j+1}$] (4) at (3.5,0) {};
        \node[] (6) at (5,0) {};
    \path (1) edge[sedge] (2)
	(4) edge[dedge] (3)
          (2) edge[sedge,dashed, dash phase=1.4pt] (3);
 \end{tikzpicture}
 &
$   2\delta_{j}+ \delta_{j+1}, $$- \delta_{j+1}
 $

\\ \hline
 { $BC_{j+1}$} &   $\delta_1$,  left-most (long) 	&   
 \begin{tikzpicture}[scale=\scales]
\bdbox;
 \node[dnode,label=above:$\delta_1$] (1) at (0,0) {};
    \node[dnode,label=above:$\delta_2$] (2) at (1,0) {};
    \node[dnode,label=above:$\delta_{j}$] (3) at (2.5,0) {};
    \node[dnode,label=above:$\ \ \ \ \delta_{j+1}$] (4) at (3.5,0) {};
        \node[] (6) at (5,0) {};
    \path (1) edge[sedge] (2)
	(3) edge[dedge] (4)
          (2) edge[sedge,dashed, dash phase=1.4pt] (3);
 \end{tikzpicture} 
   &
$ 2\delta_{j}+ 2\delta_{j+1},$
$ - \delta_{j+1}, $

\\
 \hline
  { $D_{j+1}$} &$\delta_1$, left-most &  
\begin{tikzpicture}[scale=\scales]
\bdboxAA;

    \node[dnode,label=above:$\delta_1$] (1) at (0,0) {};
    \node[dnode,label=above:$\delta_2$] (2) at (1,0) {};
    \node[dnode,label=above:$\delta_{j-1}$] (4) at (2.5,0) {};
  \node [dnode,label={[label distance=-3pt]60:$\delta_{j}$\ \ }] (5) at (3.5,0.5) {};

  \node [dnode,label={[label distance=-3pt]300:$\delta_{j+1}$\ \ }] (6) at (3.5,-0.5) {};

    \path (1) edge[sedge] (2)
          (2) edge[sedge,dashed, dash phase=1.4pt] (4)
          (4) edge[sedge] (5)
              edge[sedge] (6);
\end{tikzpicture}

&$  \delta_{j-1},$ $ \delta_{j-1}+\delta_{j+1},$ $ -\delta_{j+1}$
\\ \hline
 { $F_{4}$} &  $\delta_3$, second from left	&  
 \begin{tikzpicture}[scale=\scales]
\bdboxBB;
 \node[dnode,label=above:$\delta_4$] (1) at (0,0) {};
    \node[dnode,label=above:$\delta_3$] (2) at (1,0) {};
    \node[dnode,label=above:$\delta_{2}$] (3) at (2,0) {};
    \node[dnode,label=above:$ \delta_{1}$] (4) at (3,0) {};
    \path (1) edge[sedge] (2)
	(2) edge[dedge] (3)
	(3) edge[sedge] (4);
 \end{tikzpicture}
   &
$-\delta_4-\delta_3-\delta_2-\delta_1, $
$\delta_3+2\delta_2, $
$\delta_{4}+2\delta_{3}+3\delta_{2}+\delta_{1}
$
\\ \hline
 { $F_{4}$} & $\delta_1$,  right-most	&  
  \begin{tikzpicture}[scale=\scales]
\bdboxBB;
 \node[dnode,label=above:$\delta_4$] (1) at (0,0) {};
    \node[dnode,label=above:$\delta_3$] (2) at (1,0) {};
    \node[dnode,label=above:$\delta_{2}$] (3) at (2,0) {};
    \node[dnode,label=above:$ \delta_{1}$] (4) at (3,0) {};
    \path (1) edge[sedge] (2)
	(2) edge[dedge] (3)
	(3) edge[sedge] (4);
 \end{tikzpicture}
  &
$  \delta_4 +\delta_3$, $\delta_3$, $-\delta_4$
\\ \hline


\end{tabular}
\caption{Choice of $I$  in induction step, \cref{prop:induction}}\label{table:induction}\label{table:144}
\end{table}
}

Applying \cref{averagingunipotent} to the measure $\mu_{j}$ with $I$ as in \cref{table:induction},  $A_0$ the diagonal of $\delta_*$ in $A$, and $R$ the subgroup generated by the choices of $U^{\pm \delta_{k}}$ normalizing $U^I$ for $1\le k\le j$, 
we obtain a measure $\mu_{j}'$ such that 
\begin{enumerate}
\item $\mu_{j}'$ is invariant under  $A_{0}$, $U^{I}$; 
\item $p_{*}(\mu_{j}')$ is invariant under $A_{L_j}$, $U^{I}$, $U^{\pm \delta_{k}}$ for $1\le k \le j$, 
and has exponentially small mass at $\infty$, and
\item $h_{\mu_{j}'}(a_*\mid \Fsc)>0$.
\end{enumerate}
Recall that $\pm \delta_{j+1}$ is a linear combination of elements in $I$ and $\pm\delta_{1},\dots, \pm\delta_{j}$ with non-negative integer coefficients.
We thus conclude that $p_{*}(\mu_{j}')$ is invariant under $U^{\pm \delta_i}$ for $1\le i \le j+1$ and thus invariant under $L_{j+1}$.

We may thus apply \cref{averagingA} to $\mu_j'$ with $H=L_{3}$ and $A_0= A_{L_{j+1}} $,  
 to obtain a measure $\mu_{j+1}$ such that 
\begin{enumerate}
\item $\mu_{j+1}$ is $A_{L_{j+1}}$-invariant 
\item $p_{*}(\mu_{j+1})= p_{*}(\mu_{j}')$ is $L_{j+1}$-invariant and has exponentially small mass at $\infty$,   and
\item $h_{\mu_{j+1}}(a_*\mid\Fsc)>0$.
\end{enumerate}
 \cref{prop:induction} then follows.
\end{proof}

\cref{prop:oneroot} now follows directly from \cref{prop:induction}.

\section{{Proof of \texorpdfstring{\Cref{Ainv}}{Theorem 2.1}, Case II: Assumptions \texorpdfstring{
\ref{Ainvuniform} or \ref{AinvSL3}
}{(a) or (b)} }}\label{sec:AinvII}
We follow the notation in previous sections and assuming either \ref{Ainvuniform} or \ref{AinvSL3} in \Cref{Ainv} in this section.

\subsection{Assumption \texorpdfstring{\ref{Ainvuniform}: $\rank_{\Rbb}(G)\ge 2$}{(b): R-rank at least 2} and   $\Gamma$ cocompact}\label{sec:coco} 
When $\Gamma$ is a cocompact lattice in $G$, $M^{\alpha}$ is compact. As   discussed above, \Cref{ss} is still valid in this case, so we may assume that $\gamma_{0}$ is semisimple and satisfies $\gamma_0=g_1$ for some $1$-parameter subgroup $\{g_t\}$ in $G$.  
We  apply \Cref{claim:seed} to obtain a probability measure $\mu_0$ such that 
 \begin{enumerate}
 \item $\mu_{0}$ is $g_t$-invariant, 
 \item  $h_{\mu_{0}}(g_1\mid \Fsc)>0$, and
 \item  $\mu_{0}$ projects to the Haar measure on the orbit $\{g_t \cdot \Gamma:  t\in \R \}.$
 \end{enumerate}
 Using the real Jordan decomposition, there are two commuting elements $g^{s},g^{e}\in G$ such that
 \begin{enumerate}
\item $g^{s}$ is semisimple over $\Rbb$,
\item $g^{e}$ is contained in a compact subgroup $K_{0}<G$, and
\item $\gamma_{0}=g_1=g^{s}g^{e}$. 
\end{enumerate}

Since $K_{0}$ is compact, we can average along $K_{0}$ to obtain a measure $\mu_{1}'$ such that $\mu_{1}'$ is $K_{0}$-invariant and $\gamma_{0}$-invariant. This implies that $\mu_{1}'$ is $g^{s}$-invariant. We claim that $h_{\mu_{1}'}(g^{s}\mid \Fsc)>0$. Indeed, since $g^{e}$ is contained in the compact subgroup $K_{0}$, $h_{\mu_{1}'}(g^{e}\mid \Fsc)=0$. Then   
 $h_{\mu_{1}'}(g^{s}\mid \Fsc)>0$ follows from  \Cref{thm:subadditive}.  
 Since $g^{s}$ is semisimple over $\Rbb$, there is a  maximal $\Rbb$-split torus $A<G$ containing $g^{s}$.  Since $A$ is abelian, we can average along a \Folner sequence $\{U_n\}$ in $A$ and take \[\mu_{1}^{n}=U_{n}*\mu_{1}':=\frac{1}{|U_{n}|}\int_{U_{n}} a_{*}\mu_{1}' \, da.\] 
 Let $\mu_2$ be any weak-$*$ limit point of $\{\mu_{1}^{n}\}$. Then $\mu_{2}$ is a probability measure on $M^\alpha$ since $M^\alpha$ is compact;  moreover, $\mu_2$ is $A$-invariant and by \Cref{prop:uscfiberentropy}, we have $h_{\mu_{2}}(g^{s}\mid \Fsc)>0$. 

Finally, we apply \cref{prop:cocompact} to $\mu_2$ to obtain  an $A$-invariant probability measure $\mu$ on $M^{\alpha}$ such that 
\begin{enumerate}
\item $p_{*}\mu$ is the normalized Haar measure on $G/\Gamma$, and
\item $h_{\mu}(a\mid \Fsc)>0$ for some $a\in A$.
\end{enumerate}
\Cref{Ainv} then follows.

\subsection{Assumption \texorpdfstring{\ref{AinvSL3}: $\Qbb$-rank $1$}{(a): $\Qbb$-rank-$1$} lattices in $\SL(3, \R)$}\label{sec:sl3qrank1}
We consider the case that $G=\SL(3, \Rbb)$ and that  $\Gamma$ is a $\Qbb$-rank-$1$ lattice in $G$.
 We can describe all commensurability classes of $\Gamma$ explicitly as follows. See, for instance,  \cite[\S 6.6]{MorrisArith}. 
\begin{proposition}[Classification of $\Qbb$-rank-$1$ lattices in $\SL(3, \Rbb)$]
Let $\Gamma$ be a nonuniform  lattice in $\SL(3, \Rbb)$ that is not commensurable with $\SL(3, \Zbb)$. Then, $\Gamma$ has $\Qbb$-rank $1$. Furthermore, there exists a square free positive integer $r\ge 2$ such that, after replacing  $\Gamma$ by {$h\Gamma h^{-1}$ for some $h\in G$},  $\Gamma$ is commensurable with \[\Gamma_{r}=\left\{g\in \SL\left(3,\Zbb\left[\sqrt{r}\right]\right):\sigma(g^{tr})Jg=J\right\}.\] Above, $\sigma$ is the map on matrix entries induced by  Galois conjugation  $\sqrt{r}\mapsto -\sqrt{r}$ and \[J=\begin{bmatrix} &&1 \\ &1&\\1&& \end{bmatrix}.\] 
\end{proposition}

More explicitly, the $\Q$-vector space \[V_{\Qbb}=\left\{(a,b,c,\overline{c},\overline{b},\overline{a}) : a, b, c \in \Qbb(\sqrt{r})\right\} \] defines a $\Qbb$-form on $V=\Rbb^{6}$ with $\Z$-lattice $\mathcal L=\left\{(a,b,c,\overline{c},\overline{b},\overline{a}): a, b, c  \in \Z[\sqrt{r}]\right\}$.

Write $\R^6=\R^3\oplus \R^3$.  
Under the homomorphism $\rho\colon \SL(3, \R )\to \SL(6, \R )$,  $\rho(A)(u,v):=(Au,(A^{tr})^{-1}v)$, the image $\rho(\SL(3, \R ))$ is defined over $\Qbb$ (with respect to the $\Qbb$-form $V_{\Qbb}$). 
In particular, there is  an algebraic group $\bfG\simeq \SL(3, \subset \mathbf {GL}(V)$ defined over $\Qbb$  (with respect to the $\Qbb$-form $V_{\Qbb}$) such that $\rho(\SL(3, \R )) = \bfG(\R)^\circ$.  Furthermore, $\rho(\Gamma_{r})=\bfG(\Zbb)$ relative to the $\Z$-lattice $\mathcal L$.  See \cite[Prop.\ (6.6.1)]{MorrisArith}.   


Consider an action  $\alpha\colon\Gamma\to \Diff^{\infty}(M)$ and let $\gamma_{0}\in \Gamma$ be an element with $h_{\top}(\alpha(\gamma_{0}))>0$.
Using \Cref{prop:cocompact}, we directly deduce  \Cref{Ainv} under assumption \ref{AinvSL3} from the following.  
\begin{lemma}\label[lemma]{lem:AmsrSL3}
There exists an $A$-invariant probability measure $\mu'$ on $M^{\alpha}$ with exponentially small mass at $\infty$ such that $h_{\mu'}(a\mid \Fsc)>0$ for some $a\in A$.
\end{lemma}

\begin{proof}[Proof of \Cref{lem:AmsrSL3}] 
We identify $G$ with $\rho( G) = \bfG(\R)^\circ$ and $\Gamma_r$ with $\rho(\Gamma_r)= \bfG(\Z)$; we thus  view $\Gamma\subset G\subset \SL(6, \R)$ as  a subgroup commensurable with $\bfG(\Z)$. 
Using \Cref{ss} and after taking powers of $\gamma_{0}$ if necessary, we may assume that $\gamma_{0}$ is a semisimple element contained  $\Gamma_r= \bfG(\Zbb)$.   Take a maximal torus $\Tbf$ defined over $\Qbb$ containing $\gamma_{0}$.  
Since we assumed $\rank_{\Qbb}(\bfG)=1$, either $\rank_{\Qbb}(\Tbf)=0$ or $1$. We divide the remainder of the proof into these two cases.


\subsubsection*{Case 1: $\rank_{\Qbb}(\Tbf)=1$}
We follow the discussion in the proof of \cite[Prop.\ (6.6.1)]{MorrisArith}.  Suppose $\rank_{\Qbb}(\Tbf)=1$. 
We can decompose $\Tbf$ into an almost direct product $\Tbf=\Tbf_{a}^{\Qbb}\cdot \Tbf_{s}^{\Q}$ of a $\Qbb$-anisotropic torus $\Tbf_{a}^{\Qbb}$ and a $\Qbb$-split torus $\Tbf_{s}^{\Q}$.  Since $\rank_{\Qbb}(\Tbf)=1$, $\dim(\Tbf_{s}^{\Q}(\R))=1$.  
Since $\Tbf_{s}^{\Q}(\Rbb)\cap \bfG(\Zbb)$ is finite, after taking a power if necessary, we may assume that $\gamma_{0}\in \Tbf_{a}^{\Qbb}(\R)^\circ$.  Also, since  $h_{\top}(\alpha(\gamma_{0}))>0$, the element $\gamma_0$ is unbounded and thus $  \Tbf_{a}^{\Qbb}(\R)$ is not compact.
 In particular, we conclude that $\rank_{\Rbb}(\Tbf)=2$.


Using the explicit $\Qbb$-form on $\bfG$, we verify that the group \[\left\{\rho\left(
\diag(\ell, 1, \ell\inv)
\right): \ell\in \Qbb^\times\right\}\] forms  the $\Qbb$-points of a $\Qbb$-split torus in $\bfG$. Since $\rank_{\Qbb}(\bfG)=1$, all $\Q$-split tori are conjugate in $\bfG(\Q)$ and thus there is  $h\in \bfG(\Qbb)$ such that
\begin{equation}\label{eq:split}h\Tbf_{s}^{\Q}(\Qbb)h^{-1}=\left\{\rho\left(
\diag(\ell, 1, \ell\inv)
\right): \ell\in \Qbb^\times\right\}.\end{equation} 
Since $h\in \bfG(\Q)$, $h^{-1}\Gamma_{r}h$ is commensurable with $\Gamma_{r}$ and so $h^{-1}\Gamma h$ is commensurable with both $\Gamma_r$ and $\Gamma$.  Replacing $\gamma_0$ with a power if needed,  we may assume $h\gamma_0h^{-1} \in \Gamma_r\cap \Gamma$. 
Thus, after replacing $\Tbf$ with $h\Tbf h^{-1}$, it is with no loss of generally to assume $\Tbf_{s}^{\Q}(\Q)$ is of the form in \cref{eq:split} and that $\gamma_0\in \Tbf^\Q_{a}(\R)^\circ \cap \Gamma_r\cap \Gamma$.  We fix these choice for the remainder.

The subgroup  $\Tbf_{a}^{\Qbb}(\Qbb)$ commutes with $\Tbf_{s}^{\Q}(\Qbb)$.   Direct calculation shows that the only matrices in $\rho(\SL(3,\Rbb))$ that commute with $\{\rho(\diag(\ell,1,\ell^{-1})):\ell\in \Q^\times\}$   are diagonal. Thus 
$\Tbf_{a}^{\Qbb}(\Qbb)$ consists of matrices  of the form 
\[\diag(p,q,r,p^{-1},q^{-1},r^{-1})\] with $pqr=1$; preservation of the $\Q$-form $V_\Q$, implies  that $p,q,r \in \Q(\sqrt{r})^\times $ satisfy \begin{eqnarray}\label{zcondition} q\sigma(q)=p\sigma(r)=1.\end{eqnarray}
Let $k=  \Q(\sqrt{r})$ and let $N_{k/\Q}(\cdot)$ denote the norm of the field extension $k\to \Q$.   The map  $\diag(p,q,r,p^{-1},q^{-1},r^{-1})\mapsto N_{k/\Q}(p)$ is $\Q$-{character} (as it is invariant under $\sigma$) on $\Tbf$ and thus takes value $1$ on $\Tbf_{a}^{\Qbb}$.  
Since $r = \frac 1{\sigma (p)}= \frac p{p\sigma (p)}$, we have  $r=  p$; using that $pqr=1$, we have $q = p^{-2}$.  
We thus have
\[\Tbf_{a}^{\Qbb}(\Qbb)=\left\{\rho(\diag(p, p^{-2}, p)): p \in \Q(\sqrt{r})^\times\right\}.\] 
Having previously replaced $\gamma_0$ with a power  ensuring  $\gamma_0\in \Tbf_{a}^{\Qbb}(\Q)$, there is a unit $\omega\in \Zbb[\sqrt{r}]$ with $\omega \sigma(\omega)=1$  such that 
$\gamma_{0}=\rho(\diag(\omega,\omega^{-2},\omega))$.


%
%
%

%

Let $S= \Tbf(\R)^\circ$; then $$S=\{\rho\left(\diag \left(\exp(e_1),\exp(e_2),\exp(e_3)\right)\right): e_1+ e_2+ e_3=0\}$$
is a maximal $\Rbb$-split torus in $G=\bfG(\R)^\circ$.  
If $\Delta(G,S)=\{e_1-e_2, e_2-e_3, e_1-e_3\}$ denotes the standard collection of positive roots of $\Phi(G,S)$, the element $\rho(\diag(\omega,\omega^{-2},\omega))$ is contained in $\ker (e_{1}-e_{3})$.  
 Let $\{b^{t}\}\subset \Tbf_a^\Q(\R)^\circ $ denote the one-parameter subgroup  parameterized so that  $b^{1}=\gamma_{0}$.  
 Then $b^{\Rbb}\subset \ker(e_{1}-e_{3})$. From  \Cref{claim:seed}, we can find a compactly supported probability measure 
 $\mu_{0}$ on $M^\alpha$ such that 
\begin{enumerate} 
\item $\mu_{0}$ is $b^{\R}$-invariant,  
\item $h_{\mu_{0}}(b^{1}\mid\Fsc)>0$, and
\item $\mu_{0}$ projects to the Haar measure on the closed orbit $\{b^{t}\cdot \Gamma\}$ in $G/\Gamma$. 
\end{enumerate}

Next, we apply \Cref{SL2coro} to $\mu_{0}$ with $g=b^{1}$ and $\delta=e_{1}-e_{3}$ and obtain a probability measure $\mu'$ on $M^\alpha$ such that 
\begin{enumerate}
\item $\mu'$ is invariant under $b^{\Rbb}$,
\item $\mu'$ is invariant under the diagonal of $\delta$ in $S$, 
\item   $\mu'$ has exponentially small mass at $\infty$, and 
\item $h_{\mu'}(b^{1}\mid\Fsc)>0$.
\end{enumerate}
In particular,  since  $b^{\Rbb}$ and the diagonal of $\delta$ in $S$ generate $S$, $\mu'$ is $S$-invariant.   
Note that $S$ and $A$ are both maximal $\R$-split tori and thus are $\bfG(\R)^\circ$-conjugate.  
Hence, after  translating $\mu'$ by an element of $G$ if necessary, we obtain a probability measure $\mu'$ on $M^\alpha$ that is $A$-invariant, has exponentially small mass at $\infty$, and $h_{\mu'}(a\mid\Fsc)>0$ for some $a\in A$. This proves \Cref{lem:AmsrSL3} when $\rank_{\Qbb}(\Tbf)=1$.

\subsubsection*{Case 2: $\rank_{\Qbb}(\Tbf)=0$}
Now suppose $\rank_{\Qbb}(\Tbf)=0$. In this case $\Tbf$ is $\Qbb$-anisotropic and thus  $\Tbf(\Rbb)/\Tbf(\Zbb)$ is a compact subset of $G/\Gamma$. As before, by passing to a power, we may assume $\gamma_0 \in \Tbf (\Rbb)^\circ$.
As in   \Cref{sec:aniave}, we can find a probability measure $\mu_{1}$ on $M^\alpha$ such that 
\begin{enumerate} 
\item $\mu_{1}$ is $\Tbf(\Rbb)$-invariant,
\item $h_{\mu_{1}}(\gamma_{0}\mid\Fsc)>0$, and
\item $\mu_{1}$ projects to the Haar measure on $\Tbf(\Rbb)/\Tbf(\Zbb)$.
\end{enumerate}

Since  $h_{\top}(\alpha(\gamma_{0}))>0$, $\gamma_0\in\Tbf(\Rbb)$ is unbounded in $\Tbf(\Rbb)$
and thus $\rank_{\Rbb}\Tbf$ is either $1$ or $2$. If $\rank_{\Rbb}\Tbf=2$ then   $\Tbf(\Rbb)^\circ\simeq \Rbb^{2}$ is a maximal $\Rbb$-split torus in $G$.  Thus $\Tbf(\Rbb)^\circ$ and $A$ are conjugate and, after translating $\mu_1$ by an element of $G$ if necessary, $\mu'=\mu_{1}$ satisfies the  conclusion of \Cref{lem:AmsrSL3}.

Now suppose   $\rank_{\Rbb}(\Tbf)=1$.
Let $\Tbf=\Tbf^{\Rbb}_{s}\Tbf^{\Rbb}_{a}$ be the (non-trivial) decomposition of $\Tbf$ into a $\Rbb$-split  and $\Rbb$-anisotropic tori. 
Fix a collection $\Delta_\R(A,G)$ of simple $\R$-roots  in $\Phi_\R(A,G)$. 
There is $g\in G$ such that  $g\Tbf^{\Rbb}_{s}g^{-1}$  is a standard (with respect to $\Delta_\R(A,G)$) $\Rbb$-split torus ( \cite[Prop.\  1.2]{MR1289055}). 
Let $\{a^t\}=g\Tbf^{\Rbb}_{s}(\R)^\circ g^{-1} \subset A$ be the one-parameter subgroup parameterized such that $a^1= g\gamma_0g^{-1}.$
Let $\mu_1'= g_* \mu_1$.  
Then 
\begin{enumerate}
\item $\mu_{1}'$ is $a^{\R}$-invariant, compactly supported,  and 
\item $h_{\mu_{1}'}(a^{1}\mid \Fsc)>0$.
\end{enumerate} 
Since $\rank_{\Rbb}(\bfG)=2$ and since $\{a^t\}\subset A$ is a standard $\Rbb$-split torus, there is a simple  root $\delta_{0}\in \Delta_\R(A,G)$ with 
such that $\Tbf_{s}^{\Rbb}(\Rbb)^\circ= \ker \delta_{0}$. 
We apply \Cref{SL2coro} to $\mu_{1}'$ and obtain a measure $\mu'$ that satisfies conclusions of \Cref{lem:AmsrSL3} when $\rank_{\Qbb}(\Tbf)=0$. 
\end{proof}

\section{Measure rigidity and proofs of  \texorpdfstring{\Cref{thm:AtoG,thm:abs}}{Theorems 2.2 and 2.3}}\label{sec:msrrigidity}

 Throughout this section, we assume  $G$ is a connected, $\R$-split simple Lie group with finite center as in \Cref{thm:AtoG}.  Let $\lieg = \Lie (G)$.  
We let $\Gamma\subset G$ be a lattice subgroup and for $r>1$ take $\alpha\colon \Gamma\to \Diff^r(M)$ a $C^r$ action.   
We let $M^{\alpha}$ denote the suspension space with induced $G$-action. We also write \begin{equation}\label{v+1} s(G)=v(G)+1. \end{equation}
 We also fix a maximal, $\R$-split  Cartan subgroup $A$ in $G$ and let $\liea$ be Lie algebra of $A$.  

\subsection{Reformulation of \texorpdfstring{\Cref{thm:AtoG}}{Theorem 2.2}}. 

 \Cref{thm:AtoG} follows immediately from the following reformulation.  
\begin{proposition}\label[proposition]{combmsr}  
Let $\mu$ be an ergodic, $A$-invariant Borel probability measure on $M^{\alpha}$ projecting to the Haar measure on $G/\Gamma$. Let $H=\textrm{Stab}_{G}(\mu)$ be subgroup preserving $\mu$ and let $\lieh=\textrm{Lie}(H)$. Suppose that \begin{enumerate}
\item there exists $a\in A$ such that $h_{\mu}(a\mid\Fsc)>0$, and
\item $\#\left\{\beta\in\Phi(A,G):\glie^{\beta}\not\subset\hlie\right\}\le s(G)=v(G)+1.$
\end{enumerate}
Then $H=G$.
\end{proposition}

\begin{remark}We remark that the conclusion of \cref{combmsr} may be false without the positive entropy assumption $h_{\mu}(a\mid\Fsc)>0$.  Indeed, let $G=\Sl(n,\R)$.  Then $v(G) = n-1$.  For the standard projective action of $\Gamma\subset \Sl(n,\R)$  on $\R P^{n-1}$ or for any circle-bundle extension (see discussion in \cref{ex:cricleext}) on the $s(G)$-dimensional manifold $\R P^{n-1}\times S^1$, there exist an ergodic Borel probability measures $\mu$ on $M^\alpha$ whose stabilizer $H$ is the $v(G)$-dimensional parabolic subgroup stabilizing a line.  Thus $\#\left\{\beta\in\Phi(A,G):\glie^{\beta}\not\subset\hlie\right\}=v(G) <s(G).$
However, as there are no $\alpha(\Gamma)$-invariant Borel probability measures on $\R P^{n-1}$, there are no $G$-invariant Borel probability measures $\mu$ on $M^\alpha$.   
\end{remark}

\Cref{thm:AtoG} can be deduced directly from  \Cref{combmsr,thm:nonres}.
\begin{proof}[Proof of \Cref{thm:AtoG}] 
Let $\mu$ be as in \Cref{thm:AtoG} and let $H= \textrm{Stab}_{G}(\mu)$.  We have $A\subset H$ and, as we assume $\mu$ projects to the Haar measure on $G/\Gamma$, may apply \cref{thm:nonres}.  We thus obtain 
 \[\#\left\{\beta\in \Phi(A,G): \glie^{\beta}\not\subset \lieh \right\} \le \dim(M)\le s(G).\]
  By \Cref{combmsr}, $G= H=\textrm{Stab}_{G}(\mu)$. \end{proof}

\subsection{Entropy considerations}
The following consequence of positivity of fiber entropy with be used frequently in the proofs of  \Cref{thm:AtoG,thm:abs}: 
 \begin{lemma}\label[lemma]{lem:twoent}
 Let $\mu$ be an ergodic, $A$-invariant, probability measure on $M^{\alpha}$. Let $\chi_{1}^{F},\dots,\chi_{k}^{F}$ be the distinct fiberwise coarse Lyapunov functionals for $A$ action.

 Assume that $h_{\mu}(a_{0}\mid \Fsc)>0$ for some $a_{0}\in A$. Then, there exists $i\neq j$ so that \[h_{\mu}(a_{0}\mid\Wcal^{\chi_{i}^{F}, F})>0 \textrm{ and } h_{\mu}(a_{0}^{-1}\mid\Wcal^{\chi_{j}^{F}, F})>0.\] 
 \end{lemma}
 \begin{proof}[Proof of \Cref{lem:twoent}]
We know that there is $a_{0}\in A$ with $h_{\mu}(a_{0}|\Fsc)>0$.  Applying \Cref{eq:coent} to $a_{0}$ and $a_{0}^{-1}$, we have \[h_{\mu}(a_{0}\mid\Fsc)=\sum_{j:\chi_{j}^{F}(a_{0})>0}h_{\mu}(a_{0}\mid\Wcal^{\chi_{j}^{F}, F})\] and \[h_{\mu}(a_{0}^{-1}\mid\Fsc)=\sum_{j:\chi_{j}^{F}(a_{0})<0}h_{\mu}(a_{0}^{-1}\mid\Wcal^{\chi_{j}^{F}, F}).\]

Since $\Fsc$ is  $A$-invariant, $h_{\mu}(a_{0}\mid\Fsc)=h_{\mu}(a_{0}^{-1}\mid\Fsc)>0$. Hence, there exists $i$ and $j$ such that $\chi_{i}^{F}(a_{0})>0$, $\chi_{j}^{F}(a_{0})<0$, (hence, $i\neq j$) and \[h_{\mu}(a_{0}\mid\Wcal^{\chi_{i}^{F}, F})>0 \textrm{ and } h_{\mu}(a_{0}^{-1}\mid\Wcal^{\chi_{j}^{F}, F})>0.\] This proves \Cref{lem:twoent}.
\end{proof}

In \Cref{sec:Ginv2}, we use the high-entropy method, \Cref{thm:highent},
to show \Cref{combmsr}  assuming the following proposition, whose proof we present in \Cref{sec:Ginv3}.  
\begin{proposition}\label[proposition]{nonat}

Assume that $\dim M=v(G)+1$. Let $M^{\alpha}$ denote the suspension space with induced $G$-action.  Let $\mu$ be an ergodic, $A$-invariant, probability measure on $M^{\alpha}$ that projects to the normalized Haar measure on $G/\Gamma$.

 Assume further that there exists a fiberwise Lyapunov functional $\lambda$ and a root $\beta$ satisfying the following: 
\begin{enumerate}
\item $\dim E^{\lambda, F}=1$ and no other fiberwise Lyapunov functional is positively proportional to $\lambda$, 
\item 
 $h_{\mu}(a\mid\Wcal^{\lambda, F})>0$ for some $a\in A$, and 
\item $\beta$ is is positively proportional to $\lambda$.
\end{enumerate}
Then for $\mu$-almost every $x\in M^{\alpha}$, the leafwise measure $\mu_{x}^{U^{\beta}}$ is non-atomic. 
\end{proposition}

\subsection{Proof of \Cref{combmsr} assuming \Cref{nonat}}\label{sec:Ginv2}

In this subsection, we prove \Cref{combmsr}.

\subsubsection{Classification of subalgebras with codimension at most $s(\lieg) $}
We start with a classification of possible subgroups $H=\textrm{Stab}_{G}(\mu)$ arising in  \Cref{combmsr}.  
\begin{proposition}\label[proposition]{claim:class}
Let $\lieh\subset \lieg$ be a subalgebra such that $\liea\subset \lieh$ and  $\lieg^\beta\subset \lieh$ for all roots $\beta\in \Phi(A,G)\sm S$ where $S$ is a subset of $\Phi(A,G)$ of cardinality at most $s(\lieg)= v(\lieg)+1$.  If $\lieh\neq \lieg$ then  either  $\lieh$ is contained in a maximal parabolic subalgebra of $\lieg$ or $\#S = v(\lieg) + 1$ and 
\begin{enumerate}
\item $\lieg= \so(n,n+1)$ and $\lieh= \so(n,n)$ is the subalgebra generated by all long root spaces;
\item $\lieg$ is of type $G_2$ and $\lieh\simeq \sl(3,\R)$ is the subalgebra generated by all long root spaces;
\item $\lieg$ is of type $F_4$ and $\lieh\simeq \so(4,5)$ is the subalgebra generated by all long root spaces and one short root space.  
\end{enumerate}
\end{proposition}
\begin{proof}
First we note that  if $\#S\le v(\lieg )$, then by \cite[Lem.\  3.7]{MR4502593}, $H$ is either $G$ or a maximal parabolic. We thus assume that $\#S=v(\lieg)+1$. 
Second, we observe that if $S$ contains a long root, it follows exactly as in the proof of \cite[Prop.\ 3.5]{MR4502593} that $\lieh\subset \lieq$ for some parabolic subalgebra.

We thus consider the case that $S$ contains only short roots; in this case, we deduce the 3 exceptions enumerated above. 
 
   If $\glie$ is of type $C_n$ for $n\ge 3$, then suppose $\lieh$ contains all root spaces associated with long roots.  We have $v(\lieg)+1 = 2n$.  If $n\ge 3$ then $4(n-1)>2n$.  For fixed $1\le i_0\le n$, there are $4(n-1)$ short roots of the form $\beta=\pm e_{i_0} \pm e_j$ for $j\neq i_0$.  Thus the root space for at least one such root is contained in $\lieh$.  Taking brackets with all long roots $\pm 2 e_j$ for all $j$, it follows the root space associated to every short root $\pm e_{i} \pm e_j, i\neq j$ is contained in $\lieh$, whence $\lieh= \lieg$.

Consider  $\lieg$  of type $B_n$, $G_2$, or $F_4$.  We suppose that $S$ contains only short roots.   From (the proof of) \cite[Lem.\ 3.6]{MR4502593}, if $\#S\le v(\lieg)$ it follows that $\lieh = \lieg$.  We thus have $\#S= v(\lieg)+1$.  
If $\lieg$ is of type $B_n$ or $G_2$, there are exactly $v(\lieg)+1$ short roots and one may check the subspace of $\lieg$ spanned by   long root spaces is a subalgebra $\lieh$ of the type asserted in the proposition.  
 If $\lieg$ is of type $F_4$ there are 48 roots, with 24 long and 24 short.  Also $v(\lieg) +1 = 16$.   One may also check the abstract root system generated by all  root spaces associated to all long roots and one short root is $B_4$ and generates a subalgebra isomorphic to $\so(4,5)$ with has codimension 16.
\end{proof}

\subsubsection{{Proof of \cref{combmsr}}}
We begin the proof of \cref{combmsr}.  For the sake of contradiction, assume that $H\neq G$. We start by asserting that every coarse fiberwise Lyapunov exponent has multiplicity 1.  
\begin{claim}\label[claim]{claim:msrhypot}
In \cref{combmsr}, suppose that $\mu$ is not $G$-invariant.  Then $\dim(M)= s(\lieg) = v(\lieg)+1$ and there are $s(\lieg)$ distinct coarse fiberwise Lyapunov exponents $\chi_i^F$.  Consequently $\dim E^{\chi_i^F}(x)=1$ for a.e.\ $x$.  

\end{claim}

\begin{proof}[Proof of \Cref{claim:msrhypot}] Recall that $H=\textrm{Stab}_{G}(\mu)$ and  write $\hlie=\textrm{Lie}(H)$. 
Let \[ 
S=\left\{\beta\in\Phi(A,G):\glie^{\beta}\not\subset\hlie\right\}.\]
Let $k$ denote the number of distinct coarse fiberwise Lyapunov exponents.  By  \Cref{thm:nonres} and dimension count, we have $\#S\le k\le \dim  M$.

If  $k<v(\lieg)$ then $ \#S<v(G)$ and  $H=G$, contradicting our hypothesis.   If $k=v(G)$ then $\hlie$ is a maximal parabolic subalgebra.  Moreover, by \Cref{thm:nonres}, every $\beta_i\in S$ is positively proportional to a fiberwise exponent and thus there are $v(G)$ fiberwise Lyapunov functionals $\lambda_1^F ,\dots, \lambda_{v(G)}^F$ such that each $\lambda_i^F$ is positively proportional to an element $\beta\in S$. 
However, there exists $a_0\in A$ with $\beta(a_0)<0$ for every $\beta\in S$
and thus $\lambda_i^F(a_{0})<0$ for all $i=1,\dots, v(G)$.  Since every fiberwise Lyapunov exponent is negative, a fiberwise version of Margulis--Ruelle's inequality (\Cref{thm:MRineq}) implies $h_{\mu}(a\mid\Fsc)=0$ for an open set of $a$, contradicting hypothesis.  
Thus, we conclude that $k=v(G)+1$.  \end{proof}

By \Cref{claim:msrhypot}, we  have $v(\lieg)+1$ distinct fiberwise Lyapunov functionals
$\lambda_1^{F},\dots, \lambda_{v(G)+1}^{F}$,
each with $\dim E^{\lambda_i^F}=1$.   Furthermore, by \Cref{thm:nonres} we necessarily have that at least $v(G)$ fiberwise Lyapunov exponents   are positively proportional to elements $\beta\in S$.  After  reindexing, if needed, we will assume $\lambda_1^F, \dots,\lambda_{v(G)}^F$ are 
positively proportional to elements $\beta_i\in S$.

We divide possible $\hlie$ into two cases below; $\hlie\subset \lieq$ for some maximal parabolic $\lieq$ versus the other cases (items (1) to (3) in \Cref{claim:class})


\subsubsection*{Case 1:  $\hlie\subset\lieq$ for some maximal parabolic subalgebra $\lieq$}
First, we consider the case that $\hlie\subset \lieq$ for some parabolic subalgebra. 
\begin{claim}\label[claim]{hinQ} Suppose $\hlie\subset \lieq$ for some   parabolic subalgebra. Let $\Pi$ be a collection of simple roots inducing an order on roots such that $\lieq=\lieq_\Delta$ for some $\Delta\subset \Pi$. Let $\Sigma_\qlie=\left\{\beta\in\Phi(A,G):\glie^{\beta}\subset \qlie\right\}$. 
Then, for $\mu$-almost every $x$, \begin{enumerate}
\item for all positive roots $\beta^+$, $\mu_{x}^{U^{\beta^+}}$ is non-atomic and 
\item for at least one negative root $\gamma_{-}\in \Sigma\sm \Sigma_\lieq$, $\mu_{x}^{U^{\gamma_{-}}}$ is non-atomic.
\end{enumerate}
\end{claim}
\begin{proof}[Proof of \Cref{hinQ}] 

We first claim $\#\Delta=1$ and that the codimension of $\lieq$ is $v(\lieg)$.    
  Indeed, otherwise, 
  by dimension count, the codimension of $\lieh$ is at least $v(\lieg)+1$ and thus every fiberwise Lyapunov exponent is positively proportional to some $\beta\in \Sigma\sm \Sigma_\lieq$; 
it would then follow there is $a\in A$ such that $\lambda_i^F(a)<0$ for all fiberwise Lyapunov exponents, contradicting the assumption that $\mu$ has positive fiber entropy. It follows that $\lieq= \lieq_{\alpha_{j}}$ for some simple root $\alpha_j\in \Pi$.  Again by dimension counting,  we have either $\lieq=\lieh$ or $\lieq= \lieh\oplus \lieg^{\beta_0}$  for some positive root $\beta_0\in S$.  

We thus have that $\lieq$ is a maximal parabolic.  Since $\lieh\subset \lieq$,  every root $\beta\in \Sigma\sm \Sigma_\lieq$ is positively proportional to some Lyapunov exponent; up to reindexing, we assume $ \lambda_1^F, \dots,\lambda_{v(G)}^F$ are proportional to the collection of roots $\beta\in \Sigma\sm \Sigma_\lieq$.  
If $\lieq=\lieh$ then for all positive roots $\beta^+$, $\mu_{x}^{U^{\beta^+}}$ is the Haar measure and thus is non-atomic.  
If $\lieq= \lieh\oplus \lieg^{\beta_0}$  for some positive root $\beta_0\in S$, then   we have that $\lambda_{v(\lieg)+1}^F$ is positively proportional to $\beta_0$ and   $\mu_{x}^{U^{\beta^+}}$ is the Haar measure and thus   non-atomic for all positive roots $\beta^+\neq \beta_0$.  
In both  cases, we may find $a_0\in A$ with 
\begin{enumerate}
\item $h_{\mu}(a_{0}\mid \Fsc)>0$;
\item $\beta(a_0)<0$ for all $\beta\in \Sigma\sm \Sigma_\lieq$;
\item $\lambda_{v(\lieg)+1}(a_0)>0$.
\end{enumerate}
By \cref{lem:twoent}, we have $h_{\mu}(a_{0}\mid\Wcal^{\lambda_{k}^{F}, F})>0$
for $k= v(\lieg)+1$ and at least one $1\le k\le v(\lieg)$.  
It follows from  \Cref{nonat} that $\mu_{x}^{U^{\beta}}$ is non-atomic for the negative root $\beta\in \Sigma\sm \Sigma_\lieq$ that is positively proportional to the fiberwise exponent $\lambda_k^F$ for $1\le k\le v(\lieg)$ such that $h_{\mu}(a_{0}\mid\Wcal^{\lambda_{k}^{F}, F})>0$.  
Moreover, if $\lambda_{v(\lieg)+1}^F$ is positively proportional to $\beta_0$ it follows from  \Cref{nonat}   $\mu_{x}^{U^{\beta_0}}$ is non-atomic.
\end{proof}

To conclude  \cref{combmsr}, we apply high entropy method, \Cref{thm:highent} and derive a contradiction.  Let $\gamma_-$ be as in \cref{hinQ}.  
Suppose $\gamma_-= -\alpha_j$ is the unique simple negative root omitted from $\Sigma_\lieq$.
Then there is a simple positive root $\delta$ adjacent to $\alpha_j$ in the Dynkin digram and thus $\beta= \gamma_--\delta$.  Otherwise, there exists a simple positive root $\delta$ such that $\beta=\gamma_-+\delta$ is a negative root.  In either  case, $\beta\in \Sigma\sm \Sigma_\lieq$.

 
 Applying the high entropy method \Cref{thm:highent}, 
  it follows that $\mu$ is invariant under $U^{\beta}$,   
  contradicting that $\lieh\subset \lieq$. This shows \Cref{combmsr} when $\hlie\subset \qlie$ for some (maximal) parabolic subalgebra.

\subsubsection*{Case 2: Exceptional cases in \Cref{claim:class}} 
The remaining cases to consider in the proof of  \Cref{combmsr} are when that $\#S=v(\lieg)+1$ and $\lieh$ is one of the three exceptional types in \cref{claim:class}.  Since the codimension of $\lieh$ is $v(\lieg)+1$, every fiberwise Lyapunov functional is positively proportional to a root in $S$.  
Since there is $a\in A$ such that $h_{\mu}(a\mid\Fsc)>0$, by \cref{lem:twoent} there are at least two fiberwise Lyapunov functionals $\lambda_i^F$ and $\lambda_j^F$, $i\neq j$, such that $h_{\mu}(a\mid \Wcal^{\lambda_i^F})>0$ and $h_{\mu}(a^{-1} \mid \Wcal^{\lambda_j^F})>0$.   Therefore, it again follows from \Cref{nonat} that there are at least two short roots $\beta_i, \beta_j\in S$ such that $\mu_x^{U^{\beta_i}}$ and $\mu_x^{U^{\beta_j}}$ are non-atomic.

 If $\lieg$ is of type $B_n$ or $G_2$, the sums of a short root with all long roots generates all short roots.  Again, from the high-entropy method, it follows that $\lieh=\lieg$  which contradicts to the earlier assumption $H\neq G$ we made.

 If $\lieg$ is of type $F_4$, there are 3 subcollections of 8 short roots invariant under taking brackets by all long roots.  By dimension count, one such collection is contained in $\lieh$.  The high-entropy method \Cref{thm:highent} applied to the roots $\beta_j$ and $\beta_i$ (whose root subgroups are not contained in $\lieh$) implies that at least one other subcollection is contained in $\lieh$.  As in the proof of \cite[Lem.\ 3.6]{MR4502593}, this implies $\lieh=\lieg$ which contradicts our assumption that $H\neq G$ again.

This completes the proof of   \Cref{combmsr}.

 \subsection{Proof of \Cref{nonat}}\label{sec:Ginv3}
 The proof of \Cref{nonat} will occupy the entire subsection.

\subsubsection{Normal form parametrization}
We consider the action of $A$ on $M^{\alpha}$, $\restrict{\widetilde{\alpha}}{A}$.  Let $\chi=[\beta]$ denote the coarse Lyapunov exponent  containing $\beta$.   By assumption, $\chi$ contains a unique root $\beta$ and fiberwise Lyapunov exponent $\lambda$ with $\dim E^{\lambda, F}(x) = 1$ and with no other fiberwise Lyapunov exponents positively proportional to $\lambda$.  Let $\Wcal^{\chi}$ denote the associated (total) coarse Lyapunov foliation for $\chi=[\beta]=[\lambda]$.   The leaves of $\Wcal^{\chi}$ are 2-dimensional and subfoliated by $U^\beta$-orbits and leaves of the fiberwise foliation $\Wcal^{\lambda, F}$.

 Let $\Phi^{\lambda, F}_x
\colon   E^{\lambda, F}(x)\to W^{\lambda, F}(x)$ be the normal forms along leaves of the fiberwise Lyapunov manifolds $\Wcal^{\lambda, F}$ in \cref{normalitem}.  

Extend $\Phi^F_{x}$ to a parametrization of a.e.\  leaf of the  (total) coarse Lyapunov foliation $W^{\chi}(x)$ as follows:   
Let $\Phi^\chi_x\colon \lieg^{\beta}\times  E^{\lambda, F}(x)\to W^{\chi}(x)$ be 
$$\Phi^\chi_x(X,v) = \exp_\lieg(X) \cdot \Phi^{\lambda}_x(v).$$
Then for $\mu$-almost every $x$,   $\Phi^\chi_x$ is a well-defined $C^{r}$ diffeomorphism (where $\alpha$ is a $C^r$ action of $\Gamma$ for $r>1$), depends measurably on $x$, and satisfies the following: 
\begin{enumerate} 
\item$\Phi^\chi_x(0,0)=x$ and $D_{(0,0)}\Phi^\chi_x=\Id$.
\item For every $b\in A$ and a.e.\ $x\in M^\alpha$,
  \[(\Phi^\chi_{\tilde {\alpha}(b) {x} })^{-1}\circ \tilde{\alpha}(b)({x})\circ \Phi^\chi_x=(\Ad(b)(X), D_x\tilde \alpha(b))= (e^{\beta(b)}X, D_x\tilde \alpha(b)).\]
\end{enumerate}

We fix an identification $\psi^\beta\colon  \R\to \lieg^\beta$ and a choice of measurable framing $\psi_x^{\lambda}\colon \R\to E^{\lambda, F}(x)$.  Let $\Psi_x^\chi\colon \R^2\to W^\chi(x)$ denote 
$$\Psi^\chi_x(s,t) := \Phi_x^\chi(\psi^\beta(s), \psi_x^{\lambda}(t)).$$
Also write $\Psi_x^{\lambda, F}:= \Phi_x^{\lambda, F} \circ \psi_x^{\lambda}.$
We note that $\Psi^{\chi}_x$ takes a horizontal line $\R\times \{t\}$ to the $U^\beta$-orbit of $\Psi_x^{\lambda, F}(t)$.  

For $\mu$-almost every $x\in M^{\alpha}$, let $\mu_{x}^{\chi}$ and $\mu_{x}^{\lambda, F}$ denote the leafwise measures of $\mu$  along the leaves  $W^{\chi}(x)$ and $W^{\lambda, F}(x)$, respectively.   
We fix a normalization so that $\mu_{x}^{\chi}$ and $\mu_{x}^{\lambda, F}$ are normalized on the image of the unit balls under $\Psi_x^{\chi}$ and $\Psi_x^{\lambda, F}$, respectively.  
From the assumption $h_{\mu}(a_{0}\mid\Wcal^{\lambda, F})>0$ for some $a_{0}\in A$, we have the following: 
\begin{claim}\label[claim]{fibernonat} {For $\mu$-almost $x$, $\mu_{x}^{\lambda, F}$ (and thus  $\mu_{x}^{\chi}$) is non-atomic.} \end{claim}

Let $H$ be a subgroup of affine transformations in $\Rbb^{2}$ 
of the form \[H=\left\{\phi_{r,p,q}:\Rbb^{2}\to\Rbb^{2}|\phi_{r,p,q}(s,t)=(s+r,pt+q), r,q\in\Rbb, p\in\Rbb^{*}\right\}\simeq \Rbb\times ( \Rbb^{*}\ltimes \Rbb).\] 
We note that $H$ is a closed subgroup in $\Diff^{1}(\Rbb^{2},\Rbb^{2})$. From the construction and from \cref{holonomy is affine}, 
for $\mu$-almost every $x$, and $\mu^\chi_x$-a.e.\ $y\in W^\chi(x)$, the change of coordinates $(\Psi_y^{\lambda, F})^{-1}\circ \Psi_x^{\lambda, F}$ is an element of $ H$ for $\mu^{\chi}_{x}$ almost every $y\in W^{\chi}(x)$.

 Let $A'=\ker\beta$ be the kernel of the root $\beta$ in $A$. Let $\Ecal$ be the $A'$-ergodic decomposition of $\mu$. Since $A$ is abelian, $\Ecal$ is an $A$-invariant measurable partition on $M^{\alpha}$. We denote by $\{\mu_{x}^{\Ecal}\}$ a system of conditional measures with respect to the parition $\Ecal$.
The following is adaptation of \cite[Lem.\  5.9]{BRHWbdy} and \cite[Lem\ 5.1]{ABZ} to our setting. It guarantees that the  foliation  $\Wcal^{\chi}$ still contributes entropy  (for elements $a\in A$ with $\chi(z)>0$) when conditioned on $A'$-ergodic component $\Ecal$.  
 \begin{lemma}[\cite{BRHWbdy,ABZ}]\label{entropyonerg}
 For $a\in A$ with $\chi(a)>0$, 
 \begin{eqnarray}
 \beta(a)=h_{\mu}\left(a\mid \Ecal\vee \Wcal^{\chi}\right)-h_{\mu}\left(a\mid \Ecal\vee \Wcal^{\lambda, F}\right)
 \end{eqnarray}
 \end{lemma}
 
 For $\mu$-almost every $x$,  let $\mu_{x}^{\chi,\Ecal}$ be a family of  of leafwise measures along $\Wcal^{\chi}$ for $\mu^{\Ecal}_{x}$. Since $h_{\mu}\left(a\mid \Ecal\vee \Wcal^{\chi}\right)>0$ for $a$ with $\chi(a)>0$, we have that that $\mu_{x}^{\chi,\Ecal}$ is non-atomic for $\mu$-almost every $x$.
 
Fix an $A'$-ergodic component $\mu'$ of $\mu$.  We can find $b\in \ker\beta$ such that $\mu'$ is $\widetilde{\alpha}(b)$-ergodic by \cite{MR283174}. We fix such $b\in\ker\beta$. Under the above notation and settings, we have the following lemma and the corollary. The proof of \Cref{bddlusin,bdd} can be found in  \cite[Lem.\ 5.7]{ABZ} (see also  \cite{BRHWbdy}).
\begin{lemma}\label[lemma]{bddlusin} 
If $\mu^{U^{\beta}}_{x}$ is atomic for $\mu$-almost every $x$ then, for every $\delta>0$, there exists $C_{\delta}>1$ and a subset $K\subset M^{\alpha}$ with $\mu'(K)>1-\delta$ such that for every $x\in K$ and every $n\in \Zbb$ with $\widetilde{\alpha}(b^{n})(x)\in K$, 
\[\frac{1}{C_{\delta}}\le \|D\widetilde{\alpha}(b)^{n}|_{E^{\lambda, F}}\|\le C_{\delta}.\]
\end{lemma}
\begin{corollary}\label[corollary]{bdd}
If $\mu^{U^{\beta}}_{x}$ is atomic for $\mu$-almost every $x$ then, for $\mu'$ almost every $x$ and every $\delta>0$, there is $C_{x,\delta}\ge 1$ such that 
\[ \liminf_{N\to \infty}\frac{1}{N}\# \left\{ 0\le n\le N : \frac{1}{C_{x,\delta}}\le \|D\widetilde{\alpha}(b)^{n}|_{E^{\lambda, F}}\|\le C_{x,\delta}\right\} \ge 1-\delta. \] 
\end{corollary}
 Using \cref{bdd} to control distortion, a standard  measure rigidity argument such as those in \cite{KK07} imply the following.  See \cite[\S5.4]{ABZ} for a detailed argument.  

\begin{lemma}\label[lemma]{keylemmsr}
Suppose that $\mu^{U^{\beta}}_{x}$ is atomic for  $\mu$-almost every $x\in M^\alpha$.  
 Then for $\mu$-almost every $x$, 
there is a closed subgroup $H_x\subset H$ such that
\begin{enumerate}
\item the measure $((\Psi_x^{\chi})^{-1})_{*}\mu^{\chi,\Ecal}_{x}$ is supported on the orbit $H_x\cdot (0,0)$, and 
\item for every $h\in H_x$, $$h_*((\Psi_x^{\chi})^{-1})_{*}\mu^{\chi,\Ecal}_{x}\propto ((\Psi_x^{\chi})^{-1})_{*}\mu^{\chi,\Ecal}_{x}.$$
\end{enumerate}

\end{lemma}

\subsubsection{Completion of the proof of \Cref{nonat}}

We finish the proof of \Cref{nonat} using \cref{keylemmsr}.  

Suppose for $\mu$-almost every $x\in M^\alpha$ that $\mu^{U^{\beta}}_{x}$ is atomic and that $\mu^{\chi}_{x}$ is non-atomic. 
Let $H_x$ be as in \cref{keylemmsr} and let $H_x^\circ$ denote the identity component of $H_x$.  Note that $H_x$ contains at most countably many components.  
Then, the restriction of the measure $((\Psi_x^{\chi})^{-1})_{*}\mu^{\chi,\Ecal}_{x}$ to the orbit $H_x^\circ \cdot (0,0)$ is in the Lebesgue class on the orbit $H_x^\circ \cdot (0,0)$.
By the entropy considerations in \cref{entropyonerg}, the orbit $H_x^\circ \cdot (0,0)$ can not be supported on the vertical axes $\{0\}\times \R$ and, in particular, the orbit $H_x^\circ \cdot (0,0)$ can not be 0-dimensional.  Similarly, by the assumption that the leafwise measure $\mu^{U^{\beta}}_{x}$ is atomic, the orbit $H_x^\circ \cdot (0,0)$ can not be 2-dimensional.  Thus the orbit  $H_x^\circ \cdot (0,0)$ is 1-dimensional.  Moreover, the image of $H_x^\circ$ under the projection to the horizontal axis is the group of all translations; by classifying all subgroups of $H$ with the above properties, one can show the orbit  $H_x^\circ \cdot (0,0)$ is closed.  

In particular, 
\begin{enumerate}
\item $H_x^\circ \cdot (0,0)$ is an embedded $C^\infty$ curve that intersects each horizontal line and each vertical line in at most one point;
\item the orbit $H_x^\circ \cdot (0,0)$ has non-zero  $((\Psi_x^{\chi})^{-1})_{*}\mu^{\chi,\Ecal}_{x}$-measure; moreover the restriction of $((\Psi_x^{\chi})^{-1})_{*}\mu^{\chi,\Ecal}_{x}$ to this orbit is in the Lebesgue class on this orbit.  
\end{enumerate}
Using that the coordinate changes $(\Psi_{x}^{\lambda, F})^{-1}\circ \Psi_{x'}^{\lambda, F}$ are affine and send horizontal lines to horizontal lines and vertical lines to vertical lines, for $\mu^{\chi}_{x}$-a.e.\ $x'\in W^{\chi^F}(x)$ 
the measure 
$$((\Psi_x^{\chi})^{-1})_{*}\mu^{\chi,\Ecal}_{x'}$$
is in the Lebesgue class on countably many embedded $C^\infty$ curves, each of which intersects each horizontal line and each vertical line in at most one point.

Let $\xi^\chi$ be a measurable partition subordinate to $W^\chi$-manifolds.   Then for a.e.\ $x$, the conditional measures $\mu^{\xi^\chi}_x$ and $\mu^{\xi^\chi\vee \calE}_x$ are given by $$\mu^{\xi^\chi}_x =\frac {1 }{\mu^{\chi}_{x}(\xi^\chi(x))} \restrict{\mu^{\chi}_{x}}{\xi^\chi(x)}$$ and 
$$\mu^{\xi^\chi\vee\calE}_x =\frac {1 }{\mu^{\chi,\calE}_{x}(\xi^\chi(x))} \restrict{\mu^{\chi,\calE}_{x}}{\xi^\chi(x)},$$
respectively.  

For $x\in M^\alpha$ write $\nu_x^\chi:= ((\Psi_x^{\chi})^{-1})_{*}\mu^{\xi^\chi}_x.$
For $\mu$-a.e.\ $x$ there exists a  compact set $X_x\subset \R^2$ 
and, for every $y\in X_x$, an embedded $C^\infty $ curve $\gamma_{y}$ containing $y$ which intersects each horizontal line and every vertical line  in at most one point such that the following hold:
\begin{enumerate}
\item  $0<\nu_x^\chi(X_x)<\infty$ and $(0,0)$ is a density point of $\restrict{\nu_x^\chi}{X_x}$.
\item $X_x = \bigcup_{y'\in X_x } \gamma_{y'}$  and for $y',y''\in X_x$, either $\gamma_{y'}= \gamma_{y''}$ or $\gamma_{y'}\cap \gamma_{y''}=\emptyset$.
\item The map $y'\mapsto \gamma_{y'}$ is continuous from $X_x$ to the space of $C^1$-embedded curves.
\item  The partition $\{ \gamma_{y'}\}$ of $X_x$ is measurable and each conditional measure $m_{y'}$ relative to this partition is in the Lebesgue class on the curve $\gamma_{y'}$.
\end{enumerate}
Since the family of curves $y\mapsto \gamma_{y}$ varies continuously on $X_x$, there is $\epsilon_x>0$ such that for all  $y \in (\{0\}\times \R)\cap X_x$ sufficiently close to $(0,0)$, the curve $\gamma_{y}$ intersects the horizontal line $\R\times \{t\}$ for all $|t|<\epsilon_x$.

Recall we assume  $h_{\mu}(a\mid\Wcal^{\lambda, F})=h_{\mu}(a\mid \Wcal^{\chi} \vee \Fsc) >0$.  Thus $\mu^{\lambda, F}_{x}$ is non-atomic $\mu$ for a.e.\  $x$.  Since $(0,0)$ is a density point of $X_x$, this implies the measure $\restrict{\nu_x^\chi}{X_x \cap(\R\times (-\epsilon_x,\epsilon_x) }$ is not supported on an embedded curve for $\mu$-a.e.\ $x$.

On the other hand, since we assume $\mu^{U^{\beta}}_{x}$ is atomic for $\mu$-a.e.\ $x$, for a.e.\ $x$ there is a subset $G_x\subset \R^2$ with 
$${\nu_x^\chi}(\R^2\sm G_x)=0$$
and such that $G_x\cap (\R\times \{t\})$ has cardinality at most $1$ for every $t\in \R$.  
Let $$Y_x=G_x\cap X_x\cap( \R\times (-\epsilon_x,\epsilon_x)).$$  For $\nu_x^\chi$-a.e.\ $y'\in Y_x$,  $m_{y'} ( \R\times (-\epsilon_x,\epsilon_x)) \sm Y_x)) =0$.  Since $(0,0)$ is a density point of $X_x$, we may find $y',y''\in X_x$ such that $\gamma_{y'}\cap \gamma_{y''} =\emptyset$ and such that 
$$m_{y'} (( \R\times (-\epsilon_x,\epsilon_x))\sm Y_x) =0= m_{y''} (( \R\times (-\epsilon_x,\epsilon_x))\sm Y_x).$$  
Since the horizontal foliation is smooth, the horizontal holonomy from $(\gamma_{y'},m_{y'})$ to $(\gamma_{y''},m_{y''})$ is absolutely continuous.  In particular, for $m_{y'}$-a.e. $(s,t)\in Y_x \cap \gamma_{y'}$, we have $$ (\R\times \{t\}) \cap \gamma_{y''}\in Y_x$$ contradicting the assumptions on $G_x$.  

This contradiction finishes the proof of \Cref{nonat}.

 \subsection{Proof of \texorpdfstring{\Cref{thm:abs}}{Theorem 2.3}}\label{sec:abs}
 Starting from the ergodic $G$-invariant Borel probability measure $\mu$ on $M^\alpha$ guaranteed by \cref{thm:AtoG}, when $G$ is isogenous to either $\SL(n,\R)$ or $\Sp(n,\R)$, we show that the fiberwise conditional measures $\mu_x^\Fsc$ are absolutely continuous along a.e.\ fiber of $M^\alpha$.

 Recall that we denote the $G$-action on the suspension $M^{\alpha}$ by $\widetilde{\alpha}$.  Fix $V=\Rbb^{n(G)}$.
Applying Zimmer's cocycle superrigidity theorem, \Cref{thm:ZCSRG}, to the fiberwise derivative cocycle $D^{F}\wtd \alpha (\cdot)$ we deduce the following. 

\begin{corollary}\label[corollary]{thm:ZCSRtoabs}
With the assumptions in \Cref{thm:abs}, there exists a homomorphism $\pi\colon G\to \SL(V)$, a compact group $K<\GL(V)$, a compact group valued cocycle $\kappa\colon G\times M^{\alpha}\to K$, and a measurable framing  $\left\{\psi^F_{x}\colon T_{x}^{F}M^{\alpha}\to V \right\}$ defined for $\mu$-a.e.\ $x$ such that  \[\psi^F_{\widetilde{\alpha}(g)(x)}\circ D_{x}^{F}\widetilde{\alpha}(g) \circ\left(\psi^F_{x}\right)^{-1}=\pi(g)\kappa(g,x),\] for all $g\in G$ and for $\mu$-almost every $x$.  

Moreover, $K$ commutes with $\pi(G)$. 

\end{corollary}


Recall that we write  $\Lcal^{\widetilde{\alpha}}(\mu)\subset \alie^{*}$ for  Lyapunov exponents for the action $\widetilde{\alpha}|_{A}$ on $M^{\alpha}$. Then, we have $\Lcal^{\widetilde{\alpha}}(\mu)=\Lcal^{\widetilde{\alpha}, F}(\mu)\cup \Lcal^{\widetilde{\alpha},B}(\mu)$ where $\Lcal^{\widetilde{\alpha},B}(\mu)$ is the set of Lyapunov exponents for the $A$-action on base $G/\Gamma$  and $\Lcal^{\widetilde{\alpha}, F}(\mu)$ is the set of fiberwise Lyapunov exponents.  Recall that the set $\Lcal^{\widetilde{\alpha},B}(\mu)$ coincides with the roots $\Phi(A,G)$ (and is thus independent of $\mu$).  A standard argument relates the representation $\pi$ in \Cref{thm:ZCSRtoabs}, to the fiberwise Lyapunov exponents.  
\begin{corollary}\label[corollary]{thm:lyapZCSR}
Each fiberwise Lyapunov functional $\lambda^{F}\in \Lcal^{\widetilde{\alpha}, F}(\mu)$ is a weight of the representation $\pi$ in \Cref{thm:ZCSRtoabs}.    
\end{corollary}

We assumed that $h_{\mu}(\widetilde{\alpha}(a_{0})\mid\Fsc)>0$ for some $a_{0}\in A$. Thus, $\pi$ cannot be trivial due to Margulis--Ruelle inequality (\Cref{thm:MRineq}).   
By the definition of $n(G)$, the representation $\pi$ is, up to conjugation, either the defining representation or the dual of defining representation (or one of the three triality representations if $G=\SO(4,4)$). In particular, up to conjugacy, the compact subgroup of the centralizer $Z_{\GL(n(G),\Rbb)}(\pi(G))$ of $\pi(G)$ is either $\{\Id\}$ or $\{\pm \Id\}$. Thus, we may assume $K\subset \{\pm  \Id \}$.   

We summarize consequences of the above discussion in the following claim.  

\begin{claim}\label[claim]{claim:abs1}
Fix $G$ as in \Cref{thm:abs}.  
Fix a vector space $V=\Rbb^{n(\lieg)}$ with the standard inner product with a orthonormal basis and a non-trivial representation  
$\pi\colon G\to \SL(V)$ as in \Cref{thm:ZCSRtoabs}.  

\begin{enumerate}
\item  There are $n(\lieg)$ distinct fiberwise Lyapunov exponents $\lambda_i^F\in \Lcal^{\widetilde{\alpha}, F}(\mu)$ for the $A$-action on $(M^{\alpha}, \mu)$, each of which coincides with a weight of $\pi$.   In particular, no two distinct fiberwise Lyapunov functionals are positively proportional.  
\item\label{derivinframe} There exists a measurable framing  $\left\{\psi^F_{x}\colon T_{x}^{F}M^{\alpha}\to V \right\}$ defined for $\mu$-a.e.\ $x$ such that  \[\psi^F_{\widetilde{\alpha}(g)(x)}\circ D_{x}^{F}\widetilde{\alpha}(g) \circ\left(\psi^F_{x}\right)^{-1}=\pm\pi(g),\] for all $g\in G$ and for $\mu$-almost every $x$. 

\item If $V^{\lambda_j^F}$ denotes the weight space of $\pi$ then for $\mu$-a.e.\ $x$,  $\psi_x^F(V^{\lambda_j^F})= E^{\lambda_{i}^{F}, F}(x)$ is  corresponding fiberwise Lyapunov subspace.  In particular, all coarse fiberwise Lyapunov subspaces are   $1$-dimensional space.  
\item There is $a\in A$ such that $\lambda^F_j(a)\neq 0$ for every weight $\lambda^F_j$ of $\pi$.  

\end{enumerate}
\end{claim}

Fix a
 fiberwise Lyapunov exponents $\lambda^F\in \Lcal^{\widetilde{\alpha}, F}(\mu)$.
Let
$\Phi^{\lambda^F, F}_x
\colon   E^{\lambda^F, F}(x)\to W^{\lambda^F, F}(x)$ be the normal forms along leaves of the fiberwise Lyapunov manifolds $\Wcal^{\lambda^F, F}$ in \cref{normalitem}.  
Write $\Psi^{\lambda^F, F}_x
\colon   V^{\lambda^F}\to W^{\lambda^F, F}(x)$ for 
$$\Psi^{\lambda^F, F}_x(v) = \Phi^{\lambda^F, F}_x\circ \psi_x^F(v).$$
Relative to the coordinates $\Psi^{\lambda^F, F}_x$, for $b\in A$ and a.e.\ $x$ and $v\in V^{\lambda^F}$, we have 
\[(\Psi^{\lambda^F, F}_{\tilde {\alpha}(b) {x} })^{-1}\circ \tilde{\alpha}(b)({x})\circ  \Psi^{\lambda^F, F}_x(v) = \pm e^{\lambda^F(b)}v.\]
Let $\mu_{x}^{\lambda_{i}^{F}, F}$ denote the leafwise measure on $W^{\lambda_{i}^{F}, F}(x)$ normalized on the image of the unit ball relative to the coordinates  $\Psi^{\lambda^F, F}_x$.

We will finish the proof of \Cref{thm:abs} assuming the following, \Cref{claim:abs2}. The proof of \Cref{claim:abs2} will be presented in the next subsection, \Cref{sec:abs2}. 
\begin{proposition}\label[proposition]{claim:abs2} Under the assumption in \Cref{thm:abs}, for every fiberwise Lyapunov functional $\lambda_{i}^F$ we have $h_{\mu}(a\mid\Wcal^{\lambda_i^F})>0$ for some $a\in A$.
\end{proposition}

 Fix a fiberwise Lyapunov exponent $\lambda_i^F$ and fix a non-identity $b\in A$ with $\lambda^F(b) = 0$.  Since $\mu$ is $G$-invariant, we may apply Moore's ergodicity theorem ( \cite[Thm.\  2.2.6]{Zimbook}) and conclude that $\wtd\alpha(b)\colon (M^\alpha, \mu)\to (M^\alpha, \mu)$ is ergodic. 
Exactly as in \cite{KK07},  using that $\wtd\alpha(b)$ is isometric relative to the coordinates $\Psi_{x}^{\lambda^F, F}$, \cref{claim:abs2}, and ergodicity of $\wtd \alpha(b)$ implies the following:

\begin{claim}\label[claim]{fibleb}
 $\left((\Psi_{x}^{\lambda^F_i, F})^{-1}\right)_{*}\mu_{x}^{\lambda_i^F, F}$ is equivalent to the Lebesgue measure on $V^{\lambda^F_i}$.
 \end{claim} 
 (In fact, one can show that $ ((\Psi_{x}^{\lambda^F, F})^{-1} )_{*}\mu_{x}^{\lambda_i^F, F}$ coincides with the Lebesgue measure on $V^{\lambda^F_i}$, up to the choice of normalization, but this will not be used.) 
  Indeed, the proof of \Cref{fibleb} follows from the standard measure rigidity argument (See, for instance, in  \cite[Prop.\  7.2 and 7.3]{BRHWnormal}.  By \Cref{fibleb}, $\mu_{x}^{\lambda_i^F, F}$ is absolutely continuous with respect to the (leafwise) volume on $W^{\lambda_i^F, F}_{x}$ for every $i=1,\dots, n(\lieg)$. 

Since the fiberwise conditional measures $\mu_x^{\Fsc}$ are absolutely continuous along every fiberwise Lyapunov foliation $\Wcal^{\lambda_i^F, F}$ and since there is $a'\in A$ such that $\lambda_i^F(a')\neq 0$ for every fiberwise Lyapunov exponent $\lambda_i^F$, it follows the measure are fiberwise hyperbolic (i.e., no fiberwise Lyapunov exponent is zero).  Exactly as in  \cite[Thm.\  3.1]{KKRH11},  it follows the fiberwise conditional measures $\mu_x^{\Fsc}$ are absolutely continuous.   Indeed, as in \cite{KKRH11}, we can deduce that $\mu_x^{\Fsc}$ is absolutely continuous along leaves of stable and unstable laminations. As in \cite[Cor.\  H]{MR0819557}, we conclude $\mu_{x}^{\Fsc}$ is absolutely continuous.

 \subsection{Proof of \texorpdfstring{\Cref{claim:abs2}}{Proposition 6.16}}\label{sec:abs2}
 To finish the proof of  \Cref{thm:abs}, it remains to establish \Cref{claim:abs2}. We follow the same notation as in \Cref{sec:abs}. 
Recall that each  fiberwise Lyapunov functional $\lambda_{i}^{F}$ coincides with a weight of a nontrivial representiation $\pi\colon G\to \Sl(V)$ (the defining or its dual).  Thus each fiberwise Lyapunov subspace is 1-dimensional and no pair of fiberwise Lyapunov functionals are positively proportional.   In particular, each coarse fiberwise Lyapunov exponent consists of a single linear functional, $\chi_{i}^{F}=\left[\lambda_{i}^{F}\right]$.

Since we assumed that $h_{\mu}(a_0 \mid\Fsc)>0$ for some $a_0 \in A$, as in \Cref{lem:twoent}, there exists at least two $i\neq j$ such that \[h_{\mu}(a_{0}\mid\Wcal^{\chi_{i}^{F}, F})>0 \textrm{ and } h_{\mu}(a_{0}^{-1}\mid\Wcal^{\chi_{j}^{F}, F})>0\] with $\chi_{i}(a_{0})>0$ and $\chi_{j}(a_{0})<0$. 

Up to reindexing, let $\chi_{1}^{F}=\left[\lambda_{1}^{F}\right]$  and  $\chi_{2}^{F}=\left[\lambda_{2}^{F}\right]$ satisfy $h(a_{0}\mid\Wcal^{\chi_{1}^{F}, F})>0$ and  $h(a_{0}^{-1}\mid\Wcal^{\chi_{2}^{F}, F})>0$. 
Then, by \Cref{fibleb}, we can deduce the following.
 \begin{claim}\label[claim]{k1abs} For $\mu$-almost every $x$, the leafwise measures  $\mu_{x}^{\lambda_{1}^{F}, F}$ and  $\mu_{x}^{\lambda_{2}^{F}, F}$ along the fiberwise Lyapunov foliations  $\Wcal^{\lambda_1^F, F}$ and  $\Wcal^{\lambda_2^F, F}$, respectively, are non-atomic and in the Lebesgue class.\end{claim}

In order to prove \Cref{claim:abs2}, it is enough to show the following claim: 
\begin{claim}\label[claim]{nonnegcase}
For every fiberwise Lyapunov functional $\lambda_{k}^{F}$, $\mu_{x}^{\lambda_{k}^{F}, F}$ is non-atomic for $\mu$-almost every $x$.
\end{claim}
When $k=1$ or $k=2$, \cref{nonnegcase} follows from \Cref{k1abs}. 
Let $3\le k \le n(\lieg)$ be arbitrary.   Since $\lambda_{1}^{F}$ and $\lambda_{2}^{F}$ are not positively proportional, there exists $j\in \{1,2\}$ such that $\lambda_{k}^{F}$ is not negatively proportional to $\lambda_{j}^{F}$.  Again, up to reindexing, it is with no loss of generality to assume $j=1$.  
We thus suppose that $\lambda_{k}^{F}$ is not negatively proportional to $\lambda_{1}^{F}$ and is distinct from $\lambda_{1}^{F}$.

 For the sake of contradiction, we assume that $h_{\mu}(a\mid\Wcal^{\chi_{k}^{F}, F})=0$. Then, for $\mu$-almost every $x$, the leafwise measure $\mu_{x}^{\lambda_{k}^{F}, F}$ is atomic.

Recall that the set of fiberwise Lyapunov functionals is same as the set of weights of the defining representation or its dual.   As the weights of $\pi$ are and roots of $\lieg$ are explicit, 
since we assumed that $\lambda_{k}^{F}\neq -\lambda_{1}^{F}$ and $\lambda_{k}^{F}\neq \lambda_{1}^{F}$, 
 by direct computation  we have \begin{equation}\label{inroot} \beta=\lambda_{k}^{F}-\lambda_{1}^{F}\in \Phi(A,G).\end{equation}  is a root of $\lieg$.  

Again, by explicit presentation of the weights of the representation $\pi$, we obtain the following:

There exist $a_{1},a_{2}\in \ker \beta \subset A$ such that 
\begin{enumerate}
\item $\lambda_{1}^{F}(a_{1})<\lambda_{k}^{F}(a_{1})<0$, $\lambda_{k}^{F}(a_{2})<\lambda_{1}^{F}(a_{2})<0$, and
\item for all $l$ with $l\neq 1$ and $l\neq k$, either $\lambda_{l}^{F}(a_{1})\ge0$ or $\lambda_{l}^{F}(a_{2})\ge0$.

\end{enumerate}

Moreover, for $u\in U^{\beta}$
\begin{enumerate}[resume]
\item $\pi(u)( V^{\lambda_1^F}\oplus V^{\lambda_k^F})=  V^{\lambda_1^F}\oplus V^{\lambda_k^F}$.  
\item $\pi(u)( V^{\lambda_k^F})= V^{\lambda_k^F}$, and 
\item $\pi(u)(V^{\lambda_1^F})\cap  V^{\lambda_1^F}=\{0\}$ if $u\neq \Id$.  
\end{enumerate}

As an intersection of fiberwise stable foliations of $\widetilde{\alpha}(a_{1})$ and $\widetilde{\alpha}(a_{2})$, $E^{\lambda_{1}^{F}, F}\oplus E^{\lambda_{k}^{F}, F}$ integrates to a measurable lamination which we denote it by $\Wcal^{\lambda_{1}^{F}\oplus \lambda_{k}^{F}, F}$.  
Also, since $\ker(\lambda_{k}^{F}-\lambda_{1}^{F})\subset A$ commutes with $U^{\beta}$, the measurable lamination $\Wcal^{\lambda_{1}^{F}\oplus \lambda_{k}^{F}, F}$ is $\widetilde{\alpha}\left(U^{\beta}\right)$-equivariant; that is, \[\widetilde{\alpha}(u)\left(W^{\lambda_{1}^{F}\oplus \lambda_{k}^{F}, F}(x)\right)=W^{\lambda_{1}^{F}\oplus \lambda_{k}^{F}, F}(\widetilde{\alpha}(u)(x)),\] for all $u\in U^{\beta}$ and for $\mu$-almost every $x$.

Let $\mu_{x}^{\lambda_{1}^{F}\oplus \lambda_{k}^{F}, F}$ denote the leafwise measure (with some choice of normalization)  on $W^{\lambda_{1}^{F}\oplus \lambda_{k}^{F}, F}(x)$ for $\mu$-almost every $x$. 
Adapting the main result of \cite{MR2818693} (for the dynamics of $\wtd \alpha (a_1)$ inside the leaves of the lamination $\Wcal^{\lambda_{1}^{F}\oplus \lambda_{k}^{F}, F}$), it follows that the leafwise measure  
$\mu_{x}^{\lambda_{1}^{F}\oplus \lambda_{k}^{F}, F}$ is supported on $W^{\lambda_{1}^{F}, F}(x)$ inside of $W^{\lambda_{1}^{F}\oplus \lambda_{k}^{F}, F}$.  
In particular, 
for $\mu$-almost every $x$, the leafwise measure  $\mu_{x}^{\lambda_{1}^{F}\oplus \lambda_{k}^{F}, F}$  on the leaf  $W^{\lambda_{1}^{F}\oplus \lambda_{k}^{F}, F}(x)$, is in the Lebesgue class on the smooth embedded curve $W^{\lambda_{1}^{F}, F}(x)$ in $W^{\lambda_{1}^{F}\oplus \lambda_{k}^F, F}(x)$.

To derive a contradiction, since the measure $\mu$ and the lamination $\Wcal^{\lambda_{1}^{F}\oplus \lambda_{k}^{F}, F}$ are $\widetilde{\alpha}(U^{\beta})$-invariant,  
for every $u\in U^{\beta}$ we have the following equivariance of leafwise measures: for $\mu$-a.e.\ $x$,
\begin{equation}\label{uaction} 
\widetilde{\alpha}(u)_{*}\left(\mu_{x}^{\lambda_{1}^{F}\oplus \lambda_{k}^{F}, F}\right) \propto \mu_{\widetilde{\alpha}(u)(x)}^{\lambda_{1}^{F}\oplus \lambda_{k}^{F}, F}.
\end{equation}
Moreover, we have 
\begin{equation}\label{uaction2} 
 \mu_{\widetilde{\alpha}(u)(x)}^{\lambda_{1}^{F}\oplus \lambda_{k}^{F}, F}\propto \mu_{\widetilde{\alpha}(u)(x)}^{\lambda_{1}^{F}, F},\quad \quad   \mu_{x}^{\lambda_{1}^{F}\oplus \lambda_{k}^{F}, F}\propto \mu_{x}^{\lambda_{1}^{F}, F}.
\end{equation}
We also know that, for $\mu$-almost every $x$, \[ 
 D_{x}\widetilde{\alpha}(u)\left(E^{\lambda_{1}^{F}, F}_{x}\right)=T_{x}\left(\widetilde{\alpha}(u)\left(W^{\lambda_{1}^{F}, F}(x)\right)\right).\]
We view $E^{\lambda_{1}^{F}, F}(x)$ as tangent to $\supp\left(\mu_{x}^{\lambda_{1}^{F}, F}\right)$ at $x$.  By \eqref{uaction}, for every $u\in U^{\beta}$ and $\mu$-almost every $x$, 
$D_{x}\widetilde{\alpha}(u)\left(E^{\lambda_{1}^{F}, F}(x)\right)$ is  tangent to $\supp\left(\mu_{\widetilde{\alpha}(u)(x)}^{\lambda_{1}^{F}, F}\right)$ at $\widetilde{\alpha}(u)(x)$.  
Combined with \eqref{uaction} and \eqref{uaction2}, it follows that that \begin{equation}\label{6.4}D_{x}^{F}\widetilde{\alpha}(u)\left(E^{\lambda_{1}^{F}, F}({x})\right)=E^{\lambda_{1}^{F}, F}({\widetilde{\alpha}(u)(x)}).\end{equation}

On the other hand, recall that $V^{\lambda_{1}^{F}, F}$ denotes the weight space 
 weight $\lambda_{1}^{F}$ of the representation $\pi$ in \Cref{claim:abs1}.  
 By \eqref{derivinframe} of  \Cref{claim:abs1},
  Using the item (3) in 
 \Cref{claim:abs1}, the restriction of derivative on $E^{\lambda_{1}^{F}, F}\oplus E^{\lambda_{k}^{F}, F}$ can be written as, for all $u\in U^{\beta}$,
\begin{equation}\label{rep1k}
\psi_{\widetilde{\alpha}(u)(x)}^F \circ D_{x}^{F}\left(\widetilde{\alpha}(u)\right)\circ \left(\psi_{x}^F\right)^{-1}(V^{\lambda_{1}^F}) ={\pi(u)}(V^{\lambda_{1}^{F}, F}).\end{equation} 
If $u\neq \Id$ then 
$${\pi(u)}(V^{\lambda_{1}^{F}, F})\cap V^{\lambda_{1}^{F}, F}=\cap\{0\}.$$
Since $E^{\lambda_{1}^{F}, F}({x})=\left(\psi_{x}^F\right)^{-1}(V^{\lambda_{1}^F})$ and $E^{\lambda_{1}^{F}, F}({\widetilde{\alpha}(u)(x)})=(\psi_{\widetilde{\alpha}(u)(x)}^F )^{-1}(V^{\lambda_{1}^{F}, F})$, we have
 \begin{equation}\label{6.4}D_{x}^{F}\widetilde{\alpha}(u)\left(E^{\lambda_{1}^{F}, F}({x})\right)\cap E^{\lambda_{1}^{F}, F}({\widetilde{\alpha}(u)(x)})=\{0\}\end{equation}
 contradicting \eqref{6.4}.  
 This contradiction finishes the proof of  \Cref{nonnegcase}, and thus, \Cref{claim:abs2}.

\section{Proof of \texorpdfstring{\Cref{thm:SLnZintro}}{Theorem 1.4}: Measurable conjugacy to an affine action} 
Throughout this section, let $\Tbb^{n}\simeq \Rbb^{n}/\Zbb^{n}$ denote the standard torus and let $\leb$ be the normalized Haar measure on $\Tbb^{n}$.  We also write $\Tbb^{n}_{\pm}$ for the infratorus and $Leb_{\pm}$ be the normalized Haar measure on $\Tbb^{n}_{\pm}$.  In this section, we prove the measurable classification theorem \Cref{thm:SLnZintro}.

Throughout this section we fix the following: 
\begin{enumerate}
\item We retain notation and assumptions in \Cref{thm:SLnZintro}.  In particular $G= \SL(n,\R)$ and $\Gamma$ is a lattice in $G$.  
\item Let $\Phi(A,G)$ be the set of roots of $G$ with respect to $A$.
\item Let $\mu$ be the $\widetilde{\alpha}(G)$-invariant ergodic probability measure on $M^{\alpha}$ which is induced by $\nu$.  By \cref{thm:abs}, $\nu$ is necessarily absolutely continuous.  
\item Fix a vector space $V\simeq \Rbb^{n}$ with a standard inner product with orthonormal basis.  
\item Fix a Lebesgue measure $m_V$ on $V$.

\end{enumerate}

In order to prove \Cref{thm:SLnZintro}, we adapt the proof of the main result in \cite{KRHarith}.  Adapting the main arguments in \cite{KRHarith} provides  $A$-equivariant affine structures and  homoclinic groups (candidates for group of deck transformation on $V$ viewed as a covering space of our affine model) along the fibers of the suspension $M^{\alpha}$ at almost every point.  Using Zimmer's cocycle superrigidity theorem and uniqueness of normal forms, we show such affine structures and homoclinic groups are, in fact, $G$-equivariant.  From the $G$-equivariance of such structures, we deduce  \Cref{thm:SLnZintro} in \Cref{sec:hgroup}.

\subsection{Preliminaries} \label{sec:affine}
In this subsection, we adapt several facts from \cite{KRHarith} to the induced action on the suspension space $M^{\alpha}$. Recall that we denote the $G$-action on the suspension $M^{\alpha}$ by $\widetilde{\alpha}$. We denote by $\mu$ the ergodic,  $\widetilde{\alpha}(G)$-invariant  measure  on $M^{\alpha}$ induced by $\nu$.   
We also denote the fiberwise derivative cocycle by $D^{F}\colon (x,g)\mapsto D_{x}^{F}\widetilde{\alpha}(g)$ as before. We adapt Zimmer's cocycle superrigidity  theorem, \Cref{thm:ZCSRG}, to the fiberwise derivative cocycle $D^{F}$.
\begin{theorem}\label{thm:ZCSRtotori}
Retain all notation from \Cref{thm:SLnZintro}.  There exists a homomorphism $\pi\colon G\to \SL(V)$, a compact group $K<\GL(V)$, compact group valued cocycles $\kappa\colon G\times M^{\alpha}\to K$, and a  measurable family of framing $\left\{{\psi}_{x}: T_{x}^{F}M^{\alpha}\to V \right\}$ so that \[{\psi}_{\widetilde{\alpha}(g)(x)}\circ D_{x}^{F}\widetilde{\alpha}(g) \circ\left({\psi}_{x}\right)^{-1}=\pi(g)\kappa(g,x),\] for all $g\in G$ and for $\mu$-almost every $x$.  Moreover, $K$ commutes with $\pi(G)$. 
\end{theorem}

Recall we assumed that $h_{\nu}(\alpha(\gamma))>0$ for some $\gamma$.  Thus, the representation $\pi$ is non-trivial by the Margulis--Ruelle inequality (\cref{thm:MRineq}).
Since $\dim M=\dim V=n$, it follows that, up to conjugacy, $\pi$ is  either the defining representation or its dual as in \Cref{sec:abs}. In either case, as $K$ commutes with $\pi(G)$ we have $K=\{\pm I_{V}\}$. 

Recall that we set $\Lcal^{\widetilde{\alpha}}(\mu)\subset \alie^{*}$ to be the  Lyapunov exponent functionals for the action $\restrict{\widetilde{\alpha}}{A}$ on $M^{\alpha}$. Then, we have $\Lcal^{\widetilde{\alpha}}(\mu)=\Lcal^{\widetilde{\alpha}, F}(\mu)\cup \Lcal^{\widetilde{\alpha},B}(\mu)$ where $\Lcal^{\widetilde{\alpha},B}(\mu)$ is the set of Lyapunov functionals (i.e.\ roots) for the $A$-action in the base $G/\Gamma$  and $\Lcal^{\widetilde{\alpha}, F}(\mu)$ is the set of fiberwise Lyapunov functionals.  

 For $\lambda^{F}\in \Lcal^{\widetilde{\alpha}, F}(\mu)$, let $E^{\lambda^{F}, F}$ denote the corresponding fiberwise Lyapunov distribution in $TM^{\alpha}$.  
For $\restrict{\widetilde{\alpha}}{A}$, each fiberwise Lyapunov functional $\lambda^{F}$ is a weight of the representation $\pi$ in \Cref{thm:ZCSRtotori}.   In particular, each associated fiberwise Lyapunov distribution $E^{\lambda_{i}^{F}, F}$ is $1$-dimensional. Furthermore, there are no two $i\neq j$ such that $\lambda_{i}^{F}$ is positively proportional to $\lambda_{j}^{F}$. 

For $j=1,\dots,n(G)$, let $V^{\lambda_{j}^{F}, F}$ denote the weight space of $\pi\colon G\to \SL(V)$ 
with weight $\lambda_{j}^{F}$ with respect to $A$;  that is, 
\[V^{\lambda_{j}^{F}, F}=\left\{v\in V: \pi(b)(v)=\lambda_j^F (b)v\textrm{ for all }b\in A\right\}.\] 
We may assume that there exists an orthonormal basis \begin{equation}\label{basis}\Bsc=\{\widehat{v}_{1},\dots, \widehat{v}_{n}\}\end{equation} of $V$ such that $V^{\lambda_i^F, F}=\Rbb \widehat{v}_i$. 
Finally,  for each fiberwise Lyapunov functional $\lambda_i^F$ we denote by $\Wcal^{\lambda_i^F, F}$ the corresponding Lyapunov measurable lamination and denote by $W^{\lambda_i^F}(x)$ the leaf through $x$ which is a $C^{r}$ injectively immersed submanifold. Since  no pair of  Lyapunov functionals in $\Lcal^{\widetilde{\alpha}}(\mu)$ is positively proportional, each coarse fiberwise Lyapunov functional $\chi_{i}^{F}$ consists of a single simple fiberwise Lyapunov functionals $\chi_{i}^{F}=\left[\lambda_{i}^{F}\right]$ for $i=1,\dots, n$.

We define a (open) \emph{Weyl chamber} to be a connected component in $\liea\sm \left(\bigcup_{i=1}^n\ker \lambda_i^F\right)$.  
We note that our Weyl chambers are defined relative to the weights of the representation $\pi\colon \Sl(n,\R)\to \Gl(V)$ (rather than the relative to the weights of the adjoint representation).  
For each Weyl chamber $\Ccal$ of the representation $\pi$, define  \[V_{\Ccal}^{s, F}=\bigoplus_{i:\lambda_{i}^F (b)<0, b\in\Ccal} V^{\lambda_{i}^F, F }\quad\textrm{and}\quad
V_{\Ccal}^{u, F}=\bigoplus_{i:\lambda_{i}^F (b)>0, b\in\Ccal} V^{\lambda_{i}^F, F }.\] 
Similarly, for each Weyl chamber $\Ccal$, we denote    
$E^{s, F}_{\Ccal}(x)$, $E^{u, F}_{\Ccal}(x)$ be fiberwise stable and unstable subspaces, respectively, for $\mu$-almost every $x\in M^{\alpha}$, that is,
\[E^{s, F}_{\Ccal}(x)=\bigoplus_{i:\lambda_i^F (b)<0, b\in \Ccal} E^{\lambda_i^F, F}(x)\quad\textrm{and}\quad E^{u, F}_{\Ccal}(x)=\bigoplus_{i:\lambda_i^F (b)>0, b\in \Ccal}E^{\lambda_i^F, F}(x).\] 
For each Weyl chamber $\Ccal$, we denote by $\Wcal_{\Ccal}^{s, F}$ and $\Wcal_{\Ccal}^{u, F}$ the fiberwise stable and fiberwise unstable laminations, respectively, for the action by elements in $\Ccal$. We also denote by $W_{\Ccal}^{s, F}(x)$ and $W_{\Ccal}^{u, F}(x)$ the leaf of $\Wcal_{\Ccal}^{s, F}$ and $\Wcal_{\Ccal}^{u, F}$ through $x$, respectively.

We summarize some properties of the above discussion
\begin{enumerate}
\item  Each  fiberwise Lyapunov distribution $E^{\lambda_{i}^{F}, F}$ is $1$-dimensional. Furthermore, there are no two $i\neq j$ such that $\lambda_{i}^{F}$ is positively proportional to $\lambda_{j}^{F}$. 

\item\label{it1} There exists a measurable family of framings $\left\{{\psi}_{x}: T_{x}^{F}M^{\alpha}\to V \right\}$ such that \[{\psi}_{\widetilde{\alpha}(g)(x)}\circ D_{x}^{F}\widetilde{\alpha}(g) \circ\left({\psi}_{x}\right)^{-1}=\pm\pi(g),\] for all $g\in G$ and for $\mu$-almost every $x$.

Moreover, for all $i=1,\dots, n$, 
$E^{\lambda_{i}^{F}, F}=\psi_{x}^{-1}\left(V^{\lambda_{i}^{F}, F}\right)$ 
 and 
for  $v\in V^{\lambda_i^F, F}$,
\[\psi_{\widetilde{\alpha}(x)}\circ D_{x}^{F}\widetilde{\alpha}|_{E^{\lambda_{i}^{F}, F}}(b)\circ{\left( \psi_x\right)}^{-1}(v) =\pm e^{\lambda_i^F (b)} v.\]   
\end{enumerate}

\begin{enumerate}[resume]  
\item\label{it2} There are exactly $2^{n}-2$ Weyl chambers.  For each non-empty, proper subset $\sigma\subset \{1,\dots, n\}$,  there exists a Weyl chamber $\Ccal_{\sigma}$ such that  for all $a\in \Ccal_{\sigma}$,  $\lambda_j^F (a)<0$ for all $i \in \sigma$ and  $\lambda_i^F (a)>0$ for all $i\notin \sigma$.

\begin{enumerate}[resume]
\item Let $k=\mathrm{card}{(\sigma)}$.    There exists $a_0\in A$ such that 
 \begin{equation}\label{conformal}
\text{$\lambda^F_i(a_0) = \frac{-1}{k}$ for every  $i\in \sigma$ and $\lambda^F_j(a_0) = \frac{1}{n-k}$ for every  $j\notin \sigma$.  }\end{equation}

\item { For each $i\in \sigma$, there exists $a_i\in \Ccal_\sigma$ such that $\lambda_i^F(a_i)<\lambda_j^F(a_i)<0$ for all $j\in \sigma\sm\{i\}$}
\end{enumerate}
\item For each non-empty, proper subset $\sigma\subset \{1,\dots, n\}$,  define the following:
\begin{enumerate}
\item $V^{s, F}_\sigma = V^{s, F}_{\Ccal_\sigma, F}$ and $V^u_\sigma = V^{s, F}_{-\Ccal_\sigma}=V^{s, F}_{\{1,\dots, n\}\sm \sigma}$

\item $E^{s, F}_{\sigma}(x)= E^{s, F}_{\Ccal_\sigma}(x)$
\item   The  measurable lamination $\Wcal^{s, F}_{\sigma}=\Wcal_{\Ccal_\sigma}^{s, F}$ whose leaves are tangent to $E^{s, F}_{\sigma}(x)$.  

\end{enumerate}

\end{enumerate}

\subsection{Affine structures on Lyapunov manifolds}

Recall \cref{normalitem} asserts the existence of   normal form coordinates $\Phi_{x}^{i, F}$ along almost every leaf of the lamination by (the 1-dimensional) $W^{i, F}$-leaves.
Recall that given any nonempty proper  subset $\sigma\subset \{1,\dots, n\}$, there is a (open) Weyl chamber $\Ccal_\sigma$ such that $\lambda^F_i(a)<0$ for every $a\in \Ccal_\sigma $ and every  $i\in \sigma$ and 
$\lambda^F_j(a)>0$ for every $a\in \Ccal_\sigma $ and every  $j\notin \sigma$.  
For any such $\sigma$, we have the following fibered version of \cite[Prop.\ 3.1]{KRHarith} which asserts that the coordinates in \cref{normalitem} assemble to give affine coordinates on the associated leaves of the lamination $\Wcal^{s, F}_\sigma(x)$.  The proof follows exactly as in  \cite[Prop.\ 3.1, 3.2]{KRHarith} with only minor notational changes.

\begin{proposition}[{\cite[Prop.\ 3.1, 3.2]{KRHarith}}]\label[proposition]{prop:hol1}  For any  nonempty, proper  subset $\sigma\subset \{1,\dots, n\}$, there exists  full $\mu$-measure subset $R_\sigma \subset M$ such for every $x\in R_\sigma$ there is a unique $C^r$ diffeomorphism 
$$\Phi^{\sigma, F}_x\colon E^{s, F}_\sigma(x)\to W^{s, F}_\sigma(x)$$
with the following properties:
\begin{enumerate}
\item $\Phi^{\sigma, F}_x(0) = 0$ and $D_0\Phi^{\sigma, F}_x = \id$.
\item The family $\{\Phi^{\sigma, F}_x\}$ is measurable in $x$.  
\item $\restrict{\Phi^{\sigma, F}_x}{E^{\lambda_i^F}(x)} = \Phi_{x}^{i, F}$.  
 \item For any $b\in A$, we have  
$$
\left(\Phi_{\widetilde{\alpha}(b)(x)}^{\sigma, F}\right)^{-1}\circ \widetilde{\alpha}(b)\circ\Phi_{x}^{\sigma, F}=D_{x}^{F}\widetilde{\alpha}|_{E_\sigma^{s, F}(x)}(b).$$
\end{enumerate}
Moreover, for $y\in W^{s, F}_\sigma(x)\cap R_\sigma$,
\begin{enumerate}[resume]
\item the map $\left(\Phi^{\sigma, F}_y\right)\inv \circ \Phi^{\sigma, F}_x\colon E^{s, F}_\sigma(x)\to E^{s, F}_\sigma(y)$ is affine with diagonal linear part.  
\end{enumerate}
\end{proposition}

Consider any $a\in \Ccal_\sigma  $ and $g\in C_G(a)$.  Since $\wtd \alpha(g)$ commutes with $\wtd \alpha(a)$, it follows that $\wtd \alpha(g)$ intertwines the fiberwise stable and unstable manifolds for $\wtd \alpha(a)$; that is, for $\mu$-a.e.\ $x$, 
$$\wtd \alpha(g) (W^{s, F}_\sigma(x) )=  W^{s, F}_\sigma(\wtd \alpha(g)(x))$$ 
and thus 
$$D_x\wtd \alpha(g) (E^{s, F}_\sigma(x) )=  E^{s, F}_\sigma(\wtd \alpha(g)(x)).$$ 
By uniqueness of the normal forms, a similar property holds relative to the family $\Phi_{x}^{\sigma, F}$ when $a= a_0$ is as in \eqref{conformal}.

\begin{lemma}\label[lemma]{lem:commute}
Let $\sigma$ be as in \cref{prop:hol1} and suppose there are $t>0$ and  $b\in \Ccal_\sigma $   such that $\lambda_i^F(b) =-t$ for every $i\in \sigma$.  Then for every 
$g\in  C_G(b)$ and a.e.\ $x$, the map $$\left(\Phi_{\widetilde{\alpha}(g)(x)}^{\sigma, F}\right)^{-1}\circ \widetilde{\alpha}(g)\circ\Phi_{x}^{\sigma, F}\colon E^{s, F}_\sigma(x)\to E^{s, F}_\sigma(\wtd \alpha(g)(x))$$
coincides with $\restrict{D_x^F\wtd \alpha(g)}{E^{s, F}_\sigma(x)}.$ 
\end{lemma}

\begin{proof}
\def\what{\widehat}
Consider the  measurable family of maps
$$\what \Phi_{x}^{\sigma, F}:=\widetilde{\alpha}(g\inv )\circ \Phi_{\wtd \alpha(g)(x)}^{\sigma, F}\circ D_x\widetilde{\alpha}(g).$$
Since $b$ and $g$ commute, for $\mu$-a.e.\ $x$ 
\begin{enumerate}
\item $\what \Phi_{x}^{\sigma, F}\colon  E^{s, F}_\sigma(x)\to W^{s, F}_\sigma(x)$ is a $C^r$ diffeomorphism, and 
\item $\what \Phi_{x}^{\sigma, F}(0) = x$ and $D_0\what \Phi_{x}^{\sigma, F} = \id$. 
\end{enumerate}
Moreover, for $\mu$-a.e.\ $x$  we directly verify that 
\begin{enumerate}[resume]
\item $\left(\what \Phi_{\wtd \alpha(b)(x)}^{\sigma, F}\right)\inv \circ \wtd \alpha(b)\circ \what \Phi_{x}^{\sigma, F}=  D_{\widetilde{\alpha}(g)(x)} \widetilde{\alpha}(b)$.
\end{enumerate}
Indeed, using that $b$ and $g$ commute
\begin{align*}
\left(\what \Phi_{\wtd \alpha(b)(x)}^{\sigma, F}\right)\inv & \circ \wtd \alpha(b)\circ \what \Phi_{x}^{\sigma, F}
\\
&=
\left(\what \Phi_{\wtd \alpha(b)(x)}^{\sigma, F}\right)\inv  \circ \wtd \alpha(b)\circ
\widetilde{\alpha}(g\inv )\circ \Phi_{\wtd \alpha(g)(x)}^{\sigma, F}\circ D_x\widetilde{\alpha}(g)
\\
&=
\left(\what \Phi_{\wtd \alpha(b)(x)}^{\sigma, F}\right)\inv  \circ \wtd \alpha(g\inv)\circ
\Phi_{\wtd \alpha(b g)(x)}^{\sigma, F}\circ  D_{\widetilde{\alpha}(g)(x)} \widetilde{\alpha}(b)\circ  D_x\widetilde{\alpha}(g)
\\
&=
\left(D_x\widetilde{\alpha}(g)\right)\inv
\circ 
\left(\Phi_{\wtd \alpha(gb)(x)}^{\sigma, F}\right)\inv\circ 
\Phi_{\wtd \alpha(b g)(x)}^{\sigma, F}\circ  D_{\widetilde{\alpha}(g)(x)} \widetilde{\alpha}(b)\circ  D_x\widetilde{\alpha}(g)
\\
&=
\left(D_x\widetilde{\alpha}(g)\right)\inv\circ
 D_{\widetilde{\alpha}(g)(x)} \widetilde{\alpha}(b)\circ  D_x\widetilde{\alpha}(g)
\\
&=
 D_{\widetilde{\alpha}(g)(x)} \widetilde{\alpha}(b).
\end{align*}
Since the Lyapunov exponents of $\restrict{D\wtd \alpha(b)}{E^{s, F}_\sigma(x)}$ coincide, it follows from \cite[Thm.\ 4]{KRHarith} that $$ \left(\what  \Phi_{x}^{\sigma, F}\right)\inv\circ   \Phi_{x}^{\sigma, F}\colon E^{s, F}_\sigma(x)\to E^{s, F}_\sigma(x)$$ is a linear map;  since $D_0\what  \Phi_{x}^{\sigma, F}= \id=D_0  \Phi_{x}^{\sigma, F}$, we conclude that $$\what  \Phi_{x}^{\sigma, F}=   \Phi_{x}^{\sigma, F}$$
for almost every $x$ and the conclusion follows.  
\end{proof}

 For $1\le i\le n$, write  $$\Psi^{i, F}_x:=  \Phi_{x}^{i, F}\circ  \psi_x\colon V^{\lambda_i^F, F}\to W^{\chi_i^F, F}(x)$$ where $ \psi_x$ is as in \cref{thm:ZCSRtotori}.  
Then for $b\in A$ and $\mu$-a.e.\ $x$ there exists $\epsilon \in \{0, 1\}$ such that for every $i\in \{1, \dots, n\}$ and $v_i\in V^{\lambda_i^F}$,   
 \begin{equation}\label{eq:scale}\left(\Psi_{\widetilde{\alpha}(b)(x)}^{i, F}\right)^{-1}\circ \widetilde{\alpha}(b)\circ\Psi_{x}^{i, F} ( {v}_{i})= (-1)^\epsilon e^{\lambda_i^F(b)}  {v}_{i} .\end{equation}
Similarly, we write $$\Psi_{x}^{\sigma, F}= \Phi_{x}^{\sigma, F}\circ \restrict{ \psi_x} {V^s_{\sigma}}\colon V^{s, F}_\sigma\to W^{s, F}_\sigma(x).$$  
From \cref{thm:ZCSRtotori,prop:hol1,lem:commute} the dynamics  of elements in $A$ relative to the coordinates $\Psi_{x}^{\sigma, F}$ is of particularly nice form.  Moreover, given an element  $b\in A$   acting conformally on $E^{s, F}_\sigma$, elements  in the centralizer of $b$ also take a nice form relative to these coordinates.  
\begin{corollary}\label[corollary]{cor:summary}The coordinates $\Psi_{x}^{\sigma, F}$ satisfy the following:
\begin{enumerate}
\item For $b\in A$ and $\mu$-a.e.\ $x$, there exists $\epsilon \in \{0, 1\}$ such that for every  nonempty proper  subset $\sigma\subset \{1,\dots, n\}$ and $v\in V_\sigma^s$,  
\begin{equation*}\label{eq:scale} 
\left(\Psi_{\widetilde{\alpha}(b)(x)}^{\sigma, F}\right)^{-1}\circ \widetilde{\alpha}(b)\circ\Psi_{x}^{\sigma, F} (v) = (-1)^\epsilon \pi(b) {v}.\end{equation*}
In particular, with respect to the restriction of the basis $\Bsc$ in \eqref{basis} to $V_{\sigma}^{s}$
the matrix of $\left(\Psi_{\widetilde{\alpha}(b)(x)}^{\sigma, F}\right)^{-1}\circ \widetilde{\alpha}(b)\circ\Psi_{x}^{\sigma, F}$  is a diagonal.  
\item Let $s,t>0$  and  $b\in \Ccal_\sigma$ be such that  $\lambda_i^F(b) = -t$ for all $i\in \sigma$ and $\lambda_j^F(b) = s$ for all $j\notin \sigma$.  Then for $g\in C_G(b)$ and $\mu$-a.e.\ $x$, there exists $\epsilon \in \{0, 1\}$ such that writing $\widehat \sigma= \{1,\dots, n \}\sm \sigma$, for $v\in V^{s, F}_\sigma$ and $w\in V^{s, F}_{\widehat\sigma}=V^{u, F}_\sigma$, 

$$\left(\Psi_{\widetilde{\alpha}(b)(x)}^{\sigma, F}\right)^{-1}\circ \widetilde{\alpha}(g)\circ\Psi_{x}^{\sigma, F} (v) = (-1)^\epsilon \pi(g) {v}$$
and 
$$\left(\Psi_{\widetilde{\alpha}(b)(x)}^{\widehat \sigma, F}\right)^{-1}\circ \widetilde{\alpha}(g)\circ\Psi_{x}^{\widehat \sigma, F} (v) = (-1)^\epsilon \pi(g) {w}.$$
\end{enumerate}
\end{corollary}

\subsection{The holonomy coordinates and development map}
Given a   nonempty proper  subset $\sigma\subset \{1,\dots, n\}$ and writing $\widehat \sigma= \{1,\dots, n \}\sm \sigma$, we use unstable holonomies to assemble the coordinates $\Psi_{x}^{\sigma, F}$ and $\Psi_{x}^{\widehat \sigma, F}$ along leaves of fiberwise laminations into coordinates defined on a positive measure subset in almost every fiber of $M^\alpha$.

The following summarizes \cite[Prop.\ 3.5]{KRHarith} and nearby discussion.  
\begin{proposition}[{\cite[Prop.\ 3.5]{KRHarith}}]\label[proposition]{prop:hol2}  Fix  a  nonempty, proper  subset $\sigma\subset \{1,\dots, n\}$.  There is a full measure subset $R\subset  R_\sigma \cap R_{\{1,\dots, n\}\sm \sigma}$ such that for every $x\in R$ and 
every $y\in W^{u, F}_\sigma(x)\cap R$,  there is a unique measurable function $\Hol^u_{x,y,\sigma}\colon  W^{s, F}_\sigma(x)\to  W^{s, F}_\sigma(y)$, defined for Lebesgue a.e.\ $z\in W^{s, F}_\sigma(x)$, such that the following hold:
\begin{enumerate}
\item $\Hol^u_{x,y,\sigma}(x) = y$.
\item  $\Hol^u_{x,y,\sigma}(z) \in W^{s, F}_\sigma(y)\cap W^{u, F}_\sigma(z)$ for  Lebesgue a.e.\ $z\in W^{s, F}_\sigma(x)$.
\item The map $ \left(\Psi^{\sigma, F}_y\right)\inv \circ\Hol^u_{x,y,\sigma} \circ  \Psi^{\sigma, F}_x\colon V_{\sigma}^{s, F}\to V_{\sigma}^{s, F}$ is linear and,  moreover, diagonal with respect to the restriction of the basis $\Bsc$ in \eqref{basis} to $V_{\sigma}^{s, F}$.     
\end{enumerate}
\end{proposition}

{
Write $\Fsc(x)=p\inv(p(x))$ for the fiber of $M^\alpha$ through $x$.  
We define the map $H_x\colon V\to \Fsc(x)$ as follows: fix any nonempty proper subset $\sigma\subset \{1,\dots, n\}$ and set $\widehat \sigma = \{1,\dots, n\}\sm \sigma$.  
With respect to the splitting $V= V^{s, F}_\sigma\oplus V^{u, F}_\sigma=V^{s, F}_\sigma\oplus V^{s, F}_{\widehat \sigma}$, define 
$y = \Psi^{\widehat \sigma, F}_x(v_u)\in W^{u, F}_\sigma(x)$ and, assuming that $y\in W^{u, F}_\sigma(x)\cap R$, define
$$H_x(v_s,v_u)= \Hol^u_{x,y}\circ \Psi^{ \sigma, F}_x(v_s).$$

The following summarizes \cite[\S 4.1]{KRHarith}.  
\begin{proposition}\label[proposition]{prop:devel}  There exists a full measure set $X\subset M^{\alpha}$ such that the following holds for all $x\in X$:
\begin{enumerate}
\item $H_x\colon V\to M^F_x$ is defined on a full $\leb$-measure subset of $V$.
\item For $\leb$-a.e.\ $v\in V$, $H_x(v)$ is independent of the choice of (proper, nonempty) $\sigma\subset \{1,\dots, n\}$.
\item For $\leb$-a.e.\ $v\in V$, if $y = H_x(v)$ there is an affine map $L\colon V\to V$ with $L(0) = v$ and $ H_x\circ L = H_y$.  Moreover, with respect to the basis $\Bsc$ in \eqref{basis}, the linear part of $L$ is diagonal.
\item The restriction of $H_x$ to $V^{s, F}_\sigma$ and $V^{u, F}_\sigma$ is a $C^r$ diffeomorphism on to $W^{s, F}_\sigma(s)$ and $W^{s, F}_{\widehat \sigma}(s)$, respectively.  
\item\label{it9} The image of $H_{x}$ is contained in $ \Fsc(x) = p^{-1}(p(x)) $.  Moreover, if $\{\mu_x^\Fsc\}$ denotes a family of conditional measures of $\mu$ with respect to the partition of $M^\alpha$ into fibers, then $\mu^{\Fsc}_{x}(H_{x}(V))>0$

\end{enumerate}
Furthermore, for $b\in A$, there exists $X_{b}\in M^{\alpha}$ with $\mu(X_{b})=1$ such that the following holds:
\begin{enumerate}[resume]
\item\label{item6666} for all $x\in X_{b}$ with $\widetilde{\alpha}(b)(x)\in X_{b}$,  there is a linear map  $L_b=\pm \pi(b) \colon V\to V$ whose matrix, relative to the basis $\Bsc$ in \eqref{basis}, is of the form $$\pm \pi(b) = \pm \diag\left(  e^{\lambda^F_1(b)}, \dots,   e^{\lambda^F_n(b)}\right)$$ and 
$$H_{\wtd \alpha(b)(x)}\circ L_b=  \wtd \alpha(b)\circ H_x.$$ 

\end{enumerate}
\end{proposition}

\begin{definition} For $x\in X$, we define the set $I_{x}$  to be $I_{x}=\{H_{x}, H_{x}\circ (-\Id)\}$.  \end{definition}
With this notation, conclusion \eqref{item6666}  of \cref{prop:devel} implies  for every $b\in A$ and $\mu$-a.e. $x$ that
\[\wtd{\alpha}(b)\circ \widehat{H}_{x} \circ L_b\inv = \wtd{\alpha}(b)\circ \widehat{H}_{x} \circ \diag(e^{-\lambda^{F}_{1}(b)},\dots, e^{-\lambda^{F}_{n}(g)})\in I_{\wtd{\alpha}(b)(x)}\] for all $\widehat{H}_{x}\in I_{x}$.

We note that the cardinality of $I_x$ is at most 2.  (The cardinality could be $1$ though.   For example, when  $M= \T^n/\{\pm 1\}$ is an infratorus equipped with the standard $\Sl(n,\Z)$-action, $H_x$ and $H_x\circ (-\Id)$ coincide.)

The following is the main new technical  result of this section.  Roughly, \cref{prop:devel} asserts the (at most two possible choices at each point of $X$) development maps $H_x$ are equivariant up to a linear representation for the action of $A\subset G$.  We assert a similar equivariance for the action of  $G$.
\begin{proposition}\label[proposition]{prop:Gaffine}
For each $g\in G$, there exists $X_{g}\subset M^{\alpha}$ such that $\mu(X_{g})=1$ and such that for all $x\in X_{g}$ with $\widetilde{\alpha}(g)(x)\in X_{g}$ and  for all $\widehat{H}_{x}\in I_{x}$, 
\begin{equation}\label{eq:equiv1}\widetilde{\alpha}(g)\circ \widehat{H}_{x}\circ \pi(g^{-1}) \in I_{\widetilde{\alpha}(g)(x)}.\end{equation} 
\end{proposition}
\begin{proof}
It suffices to verify \eqref{eq:equiv1} for each root $\beta\in \Phi(A,G)$ and each  fixed  $g\in U^\beta$.  

Recall that the roots  $\beta\in \Phi(A,G)$ are of the form $\beta=\lambda_i^F-\lambda_j^F$ for $i\neq j$.  Fix $\beta=  \lambda_i^F-\lambda_j^F$ and $g\in U^\beta$.   Fix $a\in A$ with $\lambda_i^F(a) = \lambda_j^F(a) = -\frac 1 2$ and $\lambda_k^F(a)=\frac {1}{n-2}$ for all $k\notin \{i,j\}$.  Let $\sigma = \{i,j\}$ and $\widehat \sigma = \{1, \dots, n\} \sm \{i, j\}$.   Then $a\in \Ccal_{\sigma}$ and $g\in C_G(a)$.

By \cref{cor:summary}, there is $\epsilon\in \{0, 1\}$ such that  for every $v\in V^{s, F}_{\sigma}$ and $w\in V^{s, F}_{\widehat \sigma}$,  
$$\left(\Psi_{\widetilde{\alpha}(g)(x)}^{\sigma, F}\right)^{-1}\circ \widetilde{\alpha}(g)\circ\Psi_{x}^{\sigma, F} (v)= (-1)^\epsilon \pi(g) v$$
and 
$$\left(\Psi_{\widetilde{\alpha}(g)(x)}^{\widehat \sigma, F}\right)^{-1}\circ \widetilde{\alpha}(g)\circ\Psi_{x}^{\widehat \sigma, F}(w) = (-1)^\epsilon \pi(g) w.$$
Write $V= V^{s, F}_{\sigma}\oplus  V^{s, F}_{\widehat \sigma}$.

Let 
$\hat x= \widetilde{\alpha}(g)(x)$, $ y=  \Psi_{x}^{\widehat \sigma, F}(w)$,  $\hat y=  \widetilde{\alpha}(g)\circ\Psi_{x}^{\widehat \sigma, F}(w)$, $  z =  \Psi_{x}^{\sigma, F} (v)$, and $\hat z =  \widetilde{\alpha}(g)\circ\Psi_{x}^{\sigma, F} (v)$.  Because $\wtd \alpha(g)$ intertwines stable and unstable manifolds for $\wtd \alpha(a)$, it also intertwines the holonomy maps.  Thus, for $\widehat{H}_{x}\in I_{x}$, there is  $\epsilon'\in\{0,1\}$ such that 
\begin{align*}
\wtd \alpha(g)\circ \widehat{H}_x(v,w)&=
\widetilde{\alpha}(g) \circ \Hol^u_{  x,   y,\sigma} (  z)
\\&=
\Hol^u_{\hat x, \hat y,\sigma} (\hat z)\\&=
\widehat{H}_{\widetilde{\alpha}(g)(x)}((-1)^{\epsilon'} \pi(g) v, (-1)^{\epsilon'} \pi(g) w)
\end{align*}
and so 
$\wtd \alpha(g) \circ \widehat{H}_x\circ \pi(g)\inv \in I_{\alpha(g) (x)}$ and the result then follows.
\end{proof}
From \cref{prop:Gaffine} and Fubini's theorem, we can find a conull set $X'\subset G\times M^{\alpha}$ such that for all $(g,x)\in X'$, \begin{equation}\label{eq:equiv}\wtd{\alpha}(g)\circ \widehat{H}_{x}\circ \pi(g^{-1})\in I_{\wtd{\alpha}(g)(x)}\end{equation}for all $\widehat{H}_{x}\in I_{x}$. Using Fubini's theorem again, we can find $X_{M}\subset M^{\alpha}$ such that $\mu(X_{M})=1$ and $\{g\in G:(g,x)\notin X\}$ has a  Haar measure  $0$.  

For  $x\in X_{M}$ and $g\in G$, define $$\wtd{I}_{\wtd{\alpha}(g)(x)}=\wtd{\alpha}(g)\circ I_{x}\circ \pi(g^{-1}).$$  We claim that $\wtd{I}_{y}$ is well defined for every $y$ in the orbit $\wtd{\alpha}(G)(X_{M})\subset M^\alpha$. Indeed, if $g,h\in G$ and $x,y\in X_{M}$ satisfy $\wtd{\alpha}(g)(x)=\wtd{\alpha}(h)(y)$ then, since $y=\wtd{\alpha}(h^{-1}g)(x)\in X_{M}$,  \eqref{eq:equiv} holds and so 
 \[\wtd{\alpha}(g)\circ I_{x}\circ \pi(g^{-1})=\wtd{\alpha}(h)\circ I_{y}\circ \pi(h^{-1}).\] 
Thus, after redefining $I_{x}$ on a measure zero set in $M^\alpha$, we may assume  the following:
\begin{proposition}\label[proposition]{prop:GequivI}
There exists a full $\mu$-measure set $Y\subset M^{\alpha}$ such that the following hold:
\begin{enumerate}
\item $Y$ is $G$-invariant,  and hence projects onto $G/\Gamma$ under the projection $M^{\alpha}\to G/\Gamma$.
\item For every $y\in Y$, every $g\in G$,  and every $\widehat{H}_{x}\in I_{x}$,
\[\wtd{\alpha}(g)\circ \widehat{H}_{x}\circ \pi(g^{-1})\in I_{\wtd{\alpha}(g)(x)}.\]  
\end{enumerate}
\end{proposition}

We fix a choice of measurable section on $Y$,   $x\mapsto \wtd{H}_{x}\in I_{x}$.  Then 
\[\widetilde{\alpha}(g)\circ \wtd{H}_{x} = \wtd{H}_{x} \circ (\pm \pi(g))\] for all $g\in G$.
 Since $Y$, $\mu$, and $I_x$ are $G$-invariant (or equivariant) and since $G$ acts transitively on $G/\Gamma$, we may restrict to the fiber over the identity and  to  obtain the following.

\begin{corollary}\label[corollary]{coro:GammaaffineI} 
There exists an 
$\alpha(\Gamma)$-invariant $\nu$-measurable set $Y_0\subset M$ with $\nu(Y_0)=1$ and a measurable family of measurable maps $\{h_{y}\colon V \to M : y\in Y_0\}$ 
such that, for all $y\in Y_0$ the following hold:
\begin{enumerate}
\item\label{it16}  $h_{y}(0)=y$ and $\nu(h_{y}(V))>0$.
\item\label{it17} For Lebesgue almost every $v\in V$,  if $z=h_{y}(v)$ then there exits an affine map $L\in \Aff(V)$ such that $h_{z}\circ L=h_{y}$ for Lebesgue almost everywhere.
\item\label{it19} Conversely, there is a full (Lebesgue) measure set $R\subset V$ such that if $v, v' \in R$ and $h_y (v)=h_y (v')$ then there is an affine map $L\colon V\to V$ such that $h_y \circ L=h_{y}$ almost everywhere.
\item\label{it18} For all $\gamma\in \Gamma$  \[\alpha(\gamma)\circ h_{y} =h_{\alpha(\gamma)(y)}\circ\left(\pm{\pi}(\gamma)\right)\] Lebesgue almost everywhere.   

\end{enumerate}

\end{corollary} 
\begin{proof}[Proof of \Cref{coro:GammaaffineI}] 
Identify $M$ with the fiber over the identity coset $p\inv (\1 \Gamma)$ in $M^\alpha$.  
By $G$-invariance of $\mu$, we may define a family of fiberwise conditional measures $\mu_{g\Gamma}$ defined for every $g\Gamma\in G/\Gamma$.  Moreover, from the construction of $\mu$, under the identification of $M$ with $p\inv (\1 \Gamma)$, we have $\mu_{\1\Gamma}= \nu$.   

Let $Y_0 = Y\cap p\inv (\1 \Gamma)$.  Then $\nu(Y_0) = 1$.  Given $y\in Y_0$, under the identification of $M$ with $p\inv (\1 \Gamma)$, set $h_y \colon V\to M$ to be $\wtd H_y \colon V\to M^\alpha$.  Given $x\in p\inv (\1 \Gamma)$ and $\gamma\in \Gamma$ we have $\wtd \alpha(\gamma)(x)\in p\inv (\1 \Gamma)$ and, under the identification of $M$ with $p\inv (\1 \Gamma)$, $\wtd \alpha(\gamma)(x)= \alpha(\gamma)(x)$.  

The result then follows.
\end{proof}

\subsection{Proof of \texorpdfstring{\Cref{thm:SLnZintro}}{Theorem 1.4}}\label{sec:hgroup}
In this section, we continue the  proof of  \Cref{thm:SLnZintro}. 
Recall that by \cref{thm:abs}, $\nu$ is absolutely continuous.  
Also recall that in \Cref{coro:GammaaffineI}, we constructed a measurable family of development maps $h_{x}$, defined  for $\nu$-almost every $x\in M$.   

For $x\in Y_0$, write $U_{x}=h_{x}(V)$.  We have the following from  item \eqref{it17} of \Cref{coro:GammaaffineI}:
\begin{claim}\label[claim]{claim:finiteindex1}
For  $\nu$-a.e.\ $y\in U_{x}$, we have $U_x =U_y$. 
\end{claim} 
 From \cref{coro:GammaaffineI}, for a.e.\ $x$ and any $\gamma\in \Gamma$, we have either $\nu(\alpha(\gamma)(U_x)\triangle U_x)=0$ (where $\triangle$ denotes the symmetric difference) or $\nu(\alpha(\gamma)(U_x)\cap U_x)=0$.  Since $\nu$ is $\alpha(\Gamma)$-invariant, we conclude that the set $\{U_{x}:x\in Y_0\}$ is finite.   
Fix one choice of element $U_{0}=U_{x}$ for some $x\in Y_0$.  We have $\nu(U_0)>0$.  Set $\nu_0:= \frac{1}{\nu(U_0)}\restrict{\nu}{U_0}.$  
Then the subgroup $\Gamma_{0}<\Gamma$, $$\Gamma_0 := \{\gamma\in \Gamma: \nu(\alpha(\gamma)(U_0)\triangle U_0)=0\},$$
has finite index in $\Gamma$ and $\nu_0$ is an ergodic, $\alpha(\Gamma_0)$-invariant probability measure on $M$.  
Since $\nu$ was assumed absolutely continuous, $\nu_0$ is also absolutely continuous.  

\subsubsection{The homoclinic group}
Given $x\in U_0$, we let \[\Lambda_{x}=\left\{L\in \Aff(V): {h}_{x}\circ L=h_{x}\right\}. \]   Recall that  for $\nu$-almost every $y\in U_{x}$, there exists $L\in \Aff(V)$ such that $h_{x}\circ L=h_{y}$.  Thus, for every $\gamma\in \Gamma_0$ we can find an affine map $\wtd{L}_{\gamma}$ that satisfies following:
\begin{proposition}\label[proposition]{prop:conjh}\label[proposition]{prop:Gammaprop}
For $\nu$-a.e.\ $x\in U_0$ and for every $\gamma\in \Gamma_0$ there exists an affine map $\wtd{L}_{\gamma}\in \Aff(V)$ such that \begin{equation}\label{bananaface}\alpha(\gamma)\circ h_{x}=h_{x}\circ \wtd{L}_{\gamma}.\end{equation}Furthermore, the following hold: 
\begin{enumerate}
\item  $\Lambda_{x}\wtd{L}_{\gamma}=\wtd{L}_{\gamma}\Lambda_{x}$ for every $\gamma\in \Gamma_0$.  
\item There is an identification  of vector spaces $V\simeq \R^n$ relative to which the group of translations by $\Zbb^{n}$ in $\R^n$ is a finite (at most two) index subgroup of $\Lambda_{x}$.  

More precisely, under this identification either $\Lambda_{x}= \Zbb^{n}\rtimes\left\{\pm I\right\}$  or $\Lambda_{x}= \Zbb^{n}$. 
In particular, $V/\Lambda_{x}$ is either a $n$-torus or infra-torus.

\item $h_{x}$ descends to a function $h\colon  V/\Lambda_{x} \to  M$ defined on a Lebesgue-full measure subset of $V/\Lambda_{x}$.  Moreover, $\lambda=(h^{-1})_{*}\nu_{0}$ is the (normalized) Haar measure on  $V/\Lambda_{x}$ and 
 $h \colon (V/\Lambda_{x},\lambda)\to  (M,\nu_{0})$ is a  measurable isomorphism.  

\end{enumerate}\end{proposition}
\begin{proof} We have $\alpha(\gamma)\circ h_{x}=h_{\alpha(\gamma)(x)}\circ (\pm \pi(\gamma))$. Since $\alpha(\gamma)(x)\in U_{0}$, we can find an affine map $L_0\in \Aff(V)$ so that 
$h_{\alpha(\gamma)(x)}=h_{x}\circ L_0$.
Defining $\wtd{L}_{\gamma}=L_{0}\circ (\pm \pi(\gamma))$,   we  have that $\alpha(\gamma)\circ h_{x}=h_{x}\circ \wtd{L}_{\gamma}$.

For all $L\in \Lambda_{x}$, 
\[h_{x}\circ \wtd{L}_{\gamma}\circ L\circ \wtd{L}_{\gamma}^{-1}=\alpha(\gamma)\circ h_{x}\circ L\circ \wtd{L}_{\gamma}^{-1}=\alpha(\gamma)\circ h_{x}\circ \wtd{L}_{\gamma}^{-1}=h_{x}.\] This shows the first item. The second and third items can be proven exactly same as in \cite[Prop.\ 4.6 and Cor.\ 4.7]{KRHarith}.
\end{proof}

\subsubsection{The affine action}

 By the first item in \Cref{prop:Gammaprop},  $\wtd{L}_{\gamma}\colon V\to V$ induces a map $[\wtd{L}_{\gamma}]\colon V/\Lambda_{x}\to V/\Lambda_{x}$ where the quotient space $V/\Lambda_{x}$ is given by the affine  action of $\Lambda_{x}$ on $V$.  
Applying \eqref{bananaface} for $\gamma_1,\gamma_2\in \Gamma_0$, 
\begin{align*}
h_{x}&=
\alpha(\gamma_2)\inv \circ \alpha(\gamma_1)\inv \circ h_{x}\circ \wtd{L}_{\gamma_1\gamma_2}\\
&=
\alpha(\gamma_2)\inv \circ \left(  h_{x}\circ  \wtd{L}_{\gamma_1} \inv \right) \circ \wtd{L}_{\gamma_1\gamma_2}\\
&=
  h_{x}\circ \wtd{L}_{\gamma_2} \inv \circ  \wtd{L}_{\gamma_1} \inv  \circ \wtd{L}_{\gamma_1\gamma_2}
\end{align*}
and so conclude $$\wtd{L}_{\gamma_2} \inv \circ \wtd{L}_{\gamma_1} \inv  \circ \wtd{L}_{\gamma_1\gamma_2}\in \Lambda_x.$$
In particular,  $\gamma\mapsto [\wtd L_\gamma]$ is a well-defined  action of $\Gamma_0$ on $V/\Lambda_x$  by affine orbifold automorphisms.

\subsubsection{Assembling the proof of \Cref{thm:SLnZintro}}
 When $\Lambda_{x}= \Zbb^{n}$, we also write $L_\gamma = [\wtd{L}_{\gamma}]$ for the induced affine map of the torus $V/\Lambda_{x}= \T^n$.  
When $\Lambda_{x}= \Zbb^{n}\rtimes\left\{\pm I\right\}$, we let $L_\gamma$ be the affine map of $\T^n=\R^n/\Z^n$ induced by the choice of $\wtd L_\gamma\colon \R^n\to \R^n$.  Then $L_\gamma\colon \R^n/\Z^n\to \R^n/\Z^n$ is a lift of the affine orbifold transformation $[\wtd{L}_{\gamma}]\colon V/\Lambda_{x}\to V/\Lambda_{x}$.  Note in the case that $\Lambda_{x}= \Zbb^{n}\rtimes\left\{\pm I\right\}$ that $L_{\gamma_1}\circ L_{\gamma_2} =\pm L_{\gamma_1\gamma_2}$ but we need not have 
$L_{\gamma_1}\circ L_{\gamma_2}=L_{\gamma_1\gamma_2}$.

Let $\gamma_{1},\dots,\gamma_{m}$ be a choice of generators of $\Gamma_{0}$.  Let $\wtd{\Gamma}_{0}$ be the subgroup of $\Aff(\Tbb^{n})$ generated by the choice of $L_{\gamma_{i}}$ for $i=1,\dots, m$. Then there exists a finite group $F$ (in fact, $F=\{\pm \Id\}$) such that the following sequence is exact:
\[1\to F\to \wtd{\Gamma_0} \xrightarrow[r]{} \Gamma_0 \to 1.\]

The affine group of the torus $\Aff(\Tbb^{n})$ is a linear group.  Since $\wtd{\Gamma_0}$ is a subgroup of $\Aff(\Tbb^{n})$, there exists a finite index subgroup $\wtd{\Gamma}_{1}$ in $\wtd{\Gamma_{0}}$ such that $\wtd{\Gamma}_{1}$ is torsion-free. Since $F=\ker(r)$ is torsion,  the restriction $\restrict{r}{\wtd{\Gamma}_{1}}$ is injective and the image $r(\wtd{\Gamma}_{1})$ is a finite index subgroup of $\Gamma_{0}$.

We have that $\Aff(\Tbb^{n})\simeq\Gl(n,\Z)\ltimes \Tbb^n$.  Given $\gamma\in \Aff(\Tbb^{n})$, let $\rho(\gamma)= D\gamma\in \Gl(n,\Z)$ denote the linear part of $\gamma$.

As a consequence of Margulis' superrigidity  theorem   the first cohomology group of $\wtd{\Gamma}_{1}$ with coefficient in the $\wtd{\Gamma}_{1}$-module $V_\rho$ vanishes, 
$H^{1}(\wtd{\Gamma}_{1}, V_{\rho})=0$ (see \cite[Thm.\ 3(iii)]{Mar91}).   
By  \cite[Thm.\  3]{Hurder}, there is a finite index subgroup $\wtd \Gamma' $ of $\wtd \Gamma_{1}<\Aff(\Tbb^{n})$ such that $\wtd{\Gamma}'$ acts on $V/\Zbb^{n}$ by   automorphisms; that is  $\wtd{\Gamma}'$ is a subgroup of $\Aut(\Tbb^{n})\simeq \GL_{n}(\Zbb)$.  
In particular, there is a linear representation $\wtd \rho \colon \wtd{\Gamma}'\to \Gl(n,\Z)$ such that $\gamma\in \Aff(\Tbb^{n})$ coincides with the automorphism $\wtd \rho(\gamma)\in \Aut(\T^n)$ for every $\gamma\in \wtd{\Gamma}'$.

Let $\Gamma_{1}=r(\wtd{\Gamma}_{1})$ and let  $\Gamma'=r(\wtd{\Gamma}')$.  
Then $\Gamma'$ has  finite index in $\Gamma_{1}$  thus has finite index in  $\Gamma_{0}$ and $\Gamma$. This shows that $\Gamma$ contains a finite index subgroup $\Gamma'$ that is isomorphic to a finite index subgroup of $\SL(n, \Zbb)$. 

Finally, let  $h\colon (V/\Lambda_{x},\lambda) \to   (M,\nu_{0})$ be the measurable isomorphism  induced by $h_{x}$ as in \Cref{prop:conjh}.  
The representation  $\wtd  \rho\colon \wtd \Gamma'\to \Gl(n,\Z)$ descends (via the isomorphism $r$) to an affine action $\what \rho\colon \Gamma'\to \Aff(V/\Lambda_x)$, $$\what \rho(\gamma)(x\Lambda_x) = \wtd \rho(\gamma)(x)\Lambda_x.$$   By  \Cref{prop:conjh}, we have
\[ \alpha(\gamma)\circ h= h\circ \what \rho(\gamma)\] for every $\gamma\in \Gamma'$. This shows \Cref{thm:SLnZintro}.

\appendix
\section{Semicontinuity of fiber entropy}\label{app:yomdinnewhouse}

Although our main application is to the dynamics induced by translation by elements $g\in G$ on the fiber bundle $M^\alpha\to G/\Gamma$, we formulate a version of the classical results of Newhouse \cite{MR986792}, following the results of Yomdin \cite{MR0889979}, of upper-semicontinuity of fiber metric entropy in a more general setting of fibered dynamics.  Since such a formulation seems to not appear in the literature, we hope such a formulation will be of use in many other settings.

\subsection{Abstract setting for fiberwise dynamics}\label{s:setup} 
\subsubsection{Fibered space and uniformly bi-Lipschitz Borel trivialization}
We let $Z$ and $Y$ be complete, second countable metric spaces.  
We fix $p\colon Z\to Y$, a proper, continuous, surjective map.

Let $M$ be a compact Riemannian manifold with induced distance $d_M$.  Let $I \colon Y\times M\to Z$ be a Borel bijection 
such that $$p(I(y,x))= y$$ for all $(y,x)\in Y\times M$.
Write $I_y\colon M\to p\inv (y)$ for the map identifying $M$ with the fiber $p\inv(y)$; that is, $$I_y (x) = I(y,x).$$  
We will moreover assume there exists $L>1$ such that for every $y\in Y$, the restriction $I_y\colon M\to  p\inv (y)$ is a bi-Lipschitz homeomorphism  with   \begin{equation}\frac 1 L d _M(x,x')\le  d_Z\bigl(I_y(x), I_y(x')\bigr) \le L d_M(x,x').\label{eq:jasminepearls}\end{equation}

Later, given a probability measure $\mu$ on $Z$, we may need to modify the  trivialization $I$ so that its discontinuity set has $\mu$-measure zero.

\subsubsection{Fibered dynamics}
Fix $r>1$ for the remainder.  
Let $F\colon Z\to Z$ and $g\colon Y\to Y$ be homeomorphisms with the following properties:
\begin{enumerate}
\item The map $g$ is a topological factor of $F$ through $p$: for every $z\in Z$ we have $p(F(z)) = g(p(z))$.
\item For every $y\in Y$, the map $f_y:= I_{g(y)} \inv \circ F\circ I_y\colon M\to M$ is a $C^r$ diffeomorphism.
\end{enumerate}
In particular, $$I(g(y),f_y(x)) = F(I(y,x)).$$

Given $n\ge 1$, we write $$f_y^{(n)}:= f_{g^{n-1}(y)}\circ \dots \circ f_y\colon M\to M$$ and $f_y^{(0)} = \id$ for the iterated fiber dynamics relative to the trivialization $I$.  
 
\subsubsection{The $C^k$ size of a diffeomorphism}
In order to define a $C^k$ size of each diffeomorphism $f_y$ that will be convenient for future estimates, we follow, for example, \cite[\S 3.8]{MR880035} and assume  that $M$ is smoothly embedded in some  $\R^N$.    
All definitions below are independent of Lipschitz change of metric so we may equip $M$ with the restriction of the Euclidean metric.  

Let $NM$ denote the normal bundle to $M$ as a submanifold of $\R^N$.  Fix $\rho>0$ and let $U$ be the neighborhood  in $NM$ of radius $\rho$ centered at the zero section.  Given any $\rho$, we may first  perform a homothetic rescaling of $M$ to ensure the map $U\to \R^N,$ $(x,v)\mapsto x+v$ is injective.  In particular, to use estimates in \cite{MR889980}, we may assume $\rho=2$ or, to use estimates in \cite{MR880035}, we may assume $\rho\ge \sqrt m$ where $m=\dim M$.   We will identify $U\subset NM$ with its image $U\subset \R^N$ in what follows.

With the assumptions on $U$ as above, every $C^k$ map $f\colon M\to M$ extends to a  $C^k$ map $\wtd f\colon U\to M$ by precomposition with orthogonal projection from $U\subset NM$ onto the zero section $M$.

Consider an open set $V\subset \R^N$ and a $C^k$ function $\wtd f\colon V\to \R^{N'}$ for some $N'\in \N$.    Given $1\le s\le k$ and $x\in V$,  we let $D^s_x\wtd f$ denote the unique symmetric $s$-multilinear function
such that \begin{equation}\label{eq:taylor}v\mapsto \wtd f(x) + \sum_{s=1}^k D^s_x\wtd f(v^{\otimes s}) \end{equation}  is the Taylor polynomial of $\wtd f$ at $x$; that is, \eqref{eq:taylor} is tangent to $\wtd f$ up to  order $k$ at $x$.  
We often ignore the $C^0$ part of $\wtd f$ and write $$\|\wtd f\|_{C^k,\ast}= \sup _{x\in V} \max _{1\le s\le k} \left \{ \|D_x^s \wtd f\| \right \}$$
where we equip multilinear maps with their operator norms.  

Returning to the setup $M\subset U\subset \R^N$ as fixed above, given a $C^k$ map $f\colon M\to M$, let $\wtd f\colon U\to M$ denote the unique extension given by precomposition by orthogonal projection.  We then write $\|f\|_{C^k,\ast} := \|\wtd f\|_{C^k,\ast}.$

\subsection{Entropy theory}

Let  $\mu$ be  a Borel probability measure on $Z$.  
Let $\scrA$ be a measurable partition of $Z$ and let $\{\mu_x^\scrA\}$ denote a family of conditional probability measures relative to the partition $\scrA$.  Let  $F\inv \scrA$ denote the partition $F\inv \scrA:=\{F\inv (A): A\in \scrA\}$; we  say $\scrA$ is \emph{$F$-invariant} if $F\inv \scrA=\scrA.$
Let $\calP$ be a finite partition of $Z$; denote by $\calP(x)$ the atom of $\calP$ containing $x$.   The  \emph{entropy} of $\calP$ \emph{conditioned on $\scrA$} is 
\begin{align*}
H_\mu(\calP\mid \scrA) &:= \int - \log (\mu_x^\scrA(\calP(x)) ) \, d  \mu(x)\\
&\phantom{:}= \int -\sum_{P\in \calP} \mu_x^\scrA(P) \log (\mu_x^\scrA(P) ) \, d  \mu(x). 
\end{align*}

We now assume that  $\scrA$  and $\mu$ are $F$-invariant.  Given $n\ge 1$, we write $$\calP^n_F= \calP\vee F\inv( \calP)\vee\dots \vee F^{-(n-1)}(\calP).$$
When the dynamics $F$ is clear from context, we   simply write $\calP^n$ rather than $ \calP^n_F$.  
The \emph{entropy of $F$ relative to $\calP$ conditioned on $\scrA$} is 
\begin{equation}\label{defent}h_\mu(F,\calP\mid \scrA):= \inf_{n\to \infty}  \frac 1 n H_\mu(\calP^n\mid \scrA)= \lim_{n\to \infty}  \frac 1 n H_\mu(\calP^n\mid \scrA).\end{equation}
Finally, the  \emph{$\mu$-metric  entropy of $F$ conditioned on $\scrA$} is 
$$h_\mu(F\mid \scrA):= \sup h_\mu(F,\calP\mid \scrA)$$
where the supremum is taken over all finite partitions $\calP$ of $Z$.

\subsection{Fiber entropy}\label{fiberentsec}
We take $\fol$ to be the partition of $Z$ into preimages under $p\colon Z\to Y$ and observe that $\fol$ is $F$-invariant by $p$-equivariance.  Given an $F$-invariant Borel probability measure $\mu$, the quantity $h_\mu(F\mid \fol)$ is the \emph{fiber $\mu$-metric entropy of $F$}.

\subsection{Borel trivializations adapted to a measure}
Let $\mu$ be a Borel probability measure on $Z$.  We assume for every such  $\mu$ that there exist a Borel bijection $I_\mu \colon Y\times M\to Z$ 
such that  the following hold:
\begin{enumerate}\item $p(I_\mu(y,x))= y$ for all $(y,x)\in Y\times M$.
\item If $I_{\mu,y}\colon M\to p\inv (y)$ denotes the map $I_{\mu,y} (x) = I_\mu(y,x)$, there is $L_\mu>1$ such that for all $y\in Y$, the restriction $I_{\mu,y}\colon M\to  p\inv (y)$ is a bi-Lipschitz homeomorphism with $$\frac 1 {L_\mu} d _M(x,x')\le  d_Z\bigl(I_{\mu,y}(x), I_{\mu,y}(x')\bigr) \le L _\mu d_M(x,x').$$
\item The discontinuity set of $I_\mu$ has $\mu$-measure $0$.
\end{enumerate}

 Let  $\beta$ be a finite partition of $M$.  
Write $\diam \beta = \max \{\diam (B): B\in \beta\}$ and  $\partial \beta =\bigcup _{B\in \beta} \partial B$.
Given a finite partition $\calP$ of $Z$, we similarly write $\partial \calP= \bigcup _{P\in \calP} \partial P$.  Given Borel probability measure $\mu$ on $Z$ and  a finite partition $\beta$ of $M$, we let $\wtd \beta_\mu$ denote the  finite partition of $Z$ given by \begin{equation}\label{eq:pushpart}\wtd\beta_\mu:= \{I_\mu \bigl(Y\times B\bigr): B\in \beta\}.\end{equation}

\begin{claim}
Given any Borel probability measure $\mu$ on $Z$ and $\epsilon>0$, there exists a finite partition $\beta$ of $M$ such that 
\begin{enumerate}
	\item $\diam  \beta<\epsilon$ and 
	\item $\mu (\partial \wtd \beta_\mu)=0$.  
\end{enumerate}
Moreover, there exists  sequence of finite partitions $\{\fol_n\}$ of $Z$ such that 
\begin{enumerate}[resume]
\item $\fol_n\nearrow\fol$ and $\mu(\partial \fol_n)=0$ for every $n$.  
\end{enumerate}
\end{claim}

We enumerate a number of properties of the above definitions and constructions.

\begin{proposition}\label[proposition]{prop:entfacts}
Let $\mu$ be an $F$-invariant Borel probability measure on $Z$.  
\begin{enumerate}
\item\label{basicfact1} $h_\mu(F\mid \fol)= \sup h_\mu(F,\wtd \beta_\mu\mid \fol)$ where the supremum is taken over all finite partitions $\beta$ of $M$.
\item\label{basicfact2}  $h_\mu(F^k\mid \fol)= kh_\mu(F\mid \fol)$ for any $k\ge 1$.  
\item\label{basicfact3} If $\calP$ is a finite partition and  $\scrA$ is a  measurable partition,
then $$H_\mu(\calP\mid \scrA)\le \int \log \card(\restrict{\calP}{ \scrA(x)}) \, d \mu (x).$$
where $\restrict{\calP}{ \scrA(x)}$ denotes the restriction of $\calP$ to the atom $\scrA(x)$ of $\scrA$ containing $x$.  
 In particular if   $\card(\restrict{\calP}{ \scrA(x)})\le k$ for almost every $x$ then $H_\mu(\restrict{\calP}{ \scrA(x)})\le \log k$. 
\item\label{basicfact4} For any finite partition $\calP$ of $Z$,  $H_\mu(\calP\mid \Fol)  = \inf_{n} H_\mu(\calP\mid \Fol_n)$.  
\end{enumerate}
Let $\beta$ be a finite partition of $M$ such that $\mu (\partial \wtd \beta_\mu)=0$ and let 
$\{\fol_m\}$ be a sequence of finite partitions of $Z$ with $\fol_m\nearrow\fol$ and $\mu(\partial \fol_m)=0$ for every $m$.  

\begin{enumerate}[resume]
\item\label{basicfact5} For every $n$ and $m$, the function $\nu\mapsto H_\nu((\wtd \beta_\mu)^n\mid \fol_m)$ is continuous at $\nu=\mu$.  
\item\label{basicfact6} For every $n$, the functions $$\nu\mapsto H_\nu((\wtd \beta_\mu)^n\mid\fol)\quad \text{and}\quad \nu \mapsto h_\nu(F,\wtd \beta_\mu\mid\fol)$$ are upper semicontinuous at $\nu=\mu$.  
\end{enumerate}
\end{proposition}
 \eqref{basicfact4}  above appears as  \cite[5.11]{MR0217258}.  
For  \eqref{basicfact5}, we have $H_\nu((\wtd \beta_\mu)^n\mid \fol_m) = H_\nu((\wtd \beta_\mu)^n\vee \fol_m) - H_\nu(\fol_m)$.
We have $\mu(\partial ((\wtd \beta_\mu)^n\vee \fol_m))=0$ by continuity of $F$, the assumptions on the boundaries of $\beta$ and $ \fol_m$, and that  $\partial ( \fol_m\vee\wtd \beta_\mu)\subset \partial \fol_m\cup \partial\wtd \beta_\mu $.    
\eqref{basicfact6} then follows from \eqref{basicfact4} and \eqref{defent}.

\subsection{Local fiber entropy}
Fix $y\in Y$ and $n\ge 1$.  Define a metric $d_{y,n;F}$ on the fiber $p\inv(y)$ of $Z$ over $y\in Y$ as follows: for $z,z'\in p\inv (y)$ let $$d_{y,n;F}(z,z') = \max \{d_Z(F^j(z), F^j(z')):  0\le j\le n-1\}.$$
Fix $\epsilon>0$, $y\in Y$, and $z\in p\inv (y)$.  We write 
$ B^\Fol(z,\epsilon) = \{z'\in p\inv (y): d_{y} (z,z')<\epsilon\} $ for the metric ball in $p\inv(y)$ centered at $z$ of radius $\epsilon$; given $n\in \N$, we write 
\begin{align*}
B^\Fol_n(z,\epsilon;F)&:= \{z'\in p\inv (y): d_{y,n;F} (z,z')<\epsilon\} \\
&\phantom{:}=\{z'\in p\inv (y): d_Z(F^j(z), F^{j}(z'))<\epsilon \text{ for all $0\le j\le n-1$}\}
\end{align*}
for the \emph{fiberwise $n$-step Bowen ball} at $z$ relative to the dynamics of $F$.

Given $\delta>0$ and  subset $A\subset p\inv (y)$, we say that a set $S\subset p\inv(y)$ 
\emph{$\delta$-spans} $A$ if $A\subset \bigcup _{z\in S} B^\Fol(z,\delta)$.  A set $S$  is said to   \emph{$(n,\delta;F)$-span} $A$ if $A\subset \bigcup_{z\in S} B^\Fol_n(z,\delta;F)$. We write $$N(\delta, n,A;F):= \min \{\card S : \text{$S$   $(n,\delta;F)$-spans $A$} \}.$$

A collection $\calA$ of subsets of  $p\inv (y)$ is  a \emph{\emph{$(\delta,n;F)$}-cover} of $A\subset p\inv (y)$ if 
${d_{y,n;F}}$-$\diam E <\delta$ for every $E\in \calA$ and  $A\subset \bigcup_{E\in \calA} E.$

We similarly let $$\cov(\delta, n,A;F):= \min \{\card \calA : \text{$\calA$ is a $(n,\delta;F)$-cover of $A$} \}.$$
Observe  for any $\delta$, $n$,  and $A\subset p\inv (y)$ that 
\begin{equation}\label{Puzzletable}\cov(2\delta, n,A;F)\le N(\delta, n,A;F)\le \cov(\delta, n,A;F).\end{equation}

Given $y\in Y,$ $z\in p\inv(y)$, $\epsilon>0$, and  $\delta>0$, define
\begin{align*}
r(z, n,\delta, \epsilon;F)&= N\left(\delta, n, B^\Fol_n(z,\epsilon;F);F\right)\\
\wtd r(z, n,\delta, \epsilon;F)&=\cov \left(\delta, n, B^\Fol_n(z,\epsilon;F);F\right).\end{align*}
Define the \emph{local fiber entropy} at \emph{scale $\epsilon$} for the fiberwise dynamics $F$ over the fiber of $y\in Y$ by  \begin{align*}
\overline h_{y,\loc}(F,\epsilon\mid \fol)&:= \lim_{\delta\to 0}\limsup_{n\to \infty} \frac 1 n \log\left( \sup_{z\in p\inv (y)}r(z, n,\delta, \epsilon;F)\right)\\
&\phantom{:}= \lim_{\delta\to 0}\limsup_{n\to \infty} \frac 1 n \log\left( \sup_{z\in p\inv (y)}\wtd r(z, n,\delta, \epsilon;F)\right)\end{align*}
where the equality follows from \eqref{Puzzletable}

Given a collection $\scrM$ of $F$-invariant, Borel probability measures  on $Z$,  define the \emph{local fiber entropy} of $F$ with respect to $\scrM$ to be  \begin{equation}\label{eq:locent}h_{\text{$\scrM$-$\loc$}}(F\mid \fol) = \lim _{\epsilon\to 0} \sup _{\mu\in \scrM} \int \overline h_{y,\loc}(F,\epsilon\mid \fol) \, d (p_* \mu)(y).\end{equation}

\subsection{Local fiber entropy under iteration}
Under reasonable hypothesis on the collection $\scrM$, we expect the local entropy to be additive: 
$$h_{\text{$\scrM$-$\loc$}}(F^k\mid \fol) = kh_{\text{$\scrM$-$\loc$}}(F\mid \fol).$$
For instance, this holds whenever the family $\{f_y:y\in Y\}$ of homeomorphisms is assumed equicontinuous.  

We will need to consider situations where such equicontinuity fails (as happens for the fiber dynamics on $M^\alpha$ when $\Gamma$ is nonuniform ).  As our primary concern will be the defect of upper-semicontinuity, we only require an inequality in the sequel.  We thus provide a detailed proof of the following.
\begin{claim}\label[claim]{claim:submult}\ 
\begin{enumerate}
\item\label{anus} Let $\mu$ be an $F$-invariant Borel probability measure on $Z$ for which  $y \mapsto \log \|f_y\|_{C^1} $ is $L^1( p_*\mu).$
  Then for every $\epsilon>0$ and  $p_*\mu$-a.e. $y\in Y$,
$$  k \overline h_{y,\loc}(\epsilon, F\mid \fol)\le \overline h_{y,\loc}(\epsilon, F^k\mid \fol)$$
\item\label{bnus} Let $\scrM$ be a collection of $F$-invariant, Borel probability measures  on $Z$ such that  for every $\mu\in \scrM$, the function 
$y \mapsto \log^+ \|f_y\|_{C^1} $ is $L^1(p_*\mu)$.  
Then 
$$ kh_{\text{$\scrM$-$\loc$}}(F\mid \fol)\le h_{\text{$\scrM$-$\loc$}}(F^k\mid \fol) .$$

\end{enumerate}
\end{claim}
We remark that equality in  conclusion \eqref{bnus} should hold if the function $y\mapsto \log\|f_y \|_{C^1}$ is uniformly integrable with respect to the collection $\scrM$ (as defined in \cref{unifint} below.) As this won't be needed, we only establish the upper bound.

\begin{proof}[Proof of \cref{claim:submult}]
Conclusion \eqref{bnus} is a direct   consequence of \eqref{anus}.  We thus establish \eqref{anus}.

Fix $k$.  Note by compactness of $M$, we have $\|f_y\|_{C^1}\ge 1$ for every $y\in Y$.  
Let $\phi\colon Y\to [0,\infty)$ be defined as 
$$\phi(y) = \prod_{j=0}^{k-1}  \|f_{g^j(y)}\|_{C^1}.$$  We have $\phi(y)\ge 1$ for every $y$; moreover, by hypothesis, we have  $y\mapsto \log(\phi(y))$ is $L^1(p_*\mu)$.
Fix $0<\delta $ and $0<\delta'<\delta$ sufficiently small.  
Let $$G(\delta') =\left\{y\in Y:  \delta' \le  \tfrac{\delta} {\phi(y)L^2}\right\}$$  
where $L\ge 1$ is the bi-Lipschitz constant in \eqref{eq:jasminepearls}.

Fix $y\in Y$, $z\in p\inv (y)$, and $\epsilon>0$.  Observe for $n\ge 1$ and any $(n-1)k< \ell \le nk$ that 
$$B^\Fol_{\ell}(z,\epsilon;F)\subset B^\Fol_n(z,\epsilon;F^k).$$ 
Write  $A_n = B^\Fol_n(z,\epsilon;F^k).$

Fix $n\ge 1$ and $0<\delta'<\delta$.  Let $\calA = \{E_1,\dots, E_p\}$ be a $(n,\delta';F^k)$-cover of $A_n$.  

We induct on $1\le j\le n$ and claim there exists collections of subsets   $$\calA =\calA_0 \prec  \calA_1\prec   \dots \prec \calA_{n}$$ (where $\prec$ is the natural pre-order on covers) such that 
\begin{enumerate}
\item for each $1\le j\le n$, 
$\calA_{j}$ is $\bigl(jk,\delta;F\bigr)$-cover of  $A_n$, and 
\item for some constant $ C_{M,\epsilon,\delta}$ depending only on $M$, $\epsilon$, and $\delta$,  $$\card \calA_{j} \le \card \calA  \cdot \left( \prod_{\substack{0\le i < j\\g^{ik}(y)\notin G(\delta')}} L^{2m} C_{M,\epsilon,\delta} \cdot ( \phi(g^{ik}(y)))^m\right)$$
where $m= \dim M$.  
\end{enumerate}

Suppose for some $0\le j\le n-1$ that we have constructed a collection $\calA_{j}$; moreover, if $j\ge 1$ suppose that $\calA_{j}$ has the properties enumerated above.  
By the  inductive hypothesis,  we have 
\begin{enumerate}
\item $F^{\ell}(\calA_{j})$ is a $\delta$-cover of $F^{\ell}(A_n)$ for every  $0\le \ell \le jk-1$
\end{enumerate}
and since $\calA_0\prec \calA_j$ and $\calA_0$ is a $(n,\delta';F^k)$-cover of $A_n$, 
\begin{enumerate}[resume]
\item $F^{jk}(\calA_{j})$ is a $\delta'$-cover for  $F^{jk}(A_n)$.  
\end{enumerate}

If $ \delta' \le \frac{\delta} {\phi\left(g^{jk}(y)\right)L^2}$ then, using that $F^{jk}(\calA_{j})$  is $\delta'$-cover for  $F^{jk}(A_n)$, we have that $F^{jk+\ell}(\calA_{j})$  is also a $\delta$-cover for $F^{jk+\ell}(A_n)$ for every  $0\le \ell \le k-1$; then $\calA_{j}$ is a 
$\bigl((j+1)k,\delta;F\bigr)$-cover of  $A_n$ and we may take $\calA_{j+1} = \calA_{j}.$ 

If $ \delta' > \frac{\delta} {\phi\left(g^{jk}(y)\right)L^2}$,  
there exists $C_M$ depending only on $M$ and a simplicial partition $\calP_j$ of  $ B^\Fol(F^{jk}(z),\epsilon)$ such that 

\begin{enumerate}
\item $\diam(P)<\frac {\delta}{ \phi(g^{jk}(y))L^2}$ for every $P\in \calP_j$ and 
\item $\card \calP_j \le C_M \frac {\epsilon^m}{\delta^m} \left (L^2\phi\left(g^{jk}(y)\right)\right)^m$ where $m=\dim M$.
\end{enumerate}
Let $\calA_{j+1}$ be the collection of sets of the form $$\{B\cap P: B\in \calA_{j}, P\in  F^{-jk} (\calP_j)\}.$$
Then for each $E\in \calA_{j+1}$, we have 
\begin{enumerate}
\item $\diam(F^\ell(E))\le \delta$ for all $0\le \ell<jk$
\item $\diam(F^{jk}(E))\le \frac{\delta}{ \phi(g^{jk}(y))L^2}$ whence,
\item $\diam(F^{jk+\ell}(E))\le \delta$ for all $0\le \ell \le  k-1$. 
\end{enumerate}
 Moreover, we have
\begin{enumerate}[resume]
\item$\card \calA_{j+1}\le \card \calA_{j}\cdot \card \calP_j$.  
\end{enumerate}
The existence of such $\calA_1\prec \dots\prec \calA_n$ with the desired properties  thus follows from induction on $j$.

For every $(n-1)k<  \ell\le nk$ we conclude that 
\begin{align*}
 \wtd r(z, \ell, &\delta, \epsilon;F)\\ 
 &= \cov\bigl(\delta, \ell,B^\Fol_{\ell}(z,\epsilon;F);F\bigr)\\
& \le  \cov\left(\delta,\ell,B^\Fol_n(z,\epsilon;F^k);F\right)\\
& \le \wtd r(z, n,\delta', \epsilon;F^k)\cdot  \left( \prod_{\substack{ 0\le i <n\\ g^{ik}(y)\notin G(\delta')}} L^{2m}C_{M,\epsilon,\delta} \cdot (\phi(g^{ik}(y)))^m\right)
\end{align*}

By the pointwise ergodic theorem for $L^1$ functions, for $(p_*\mu)$-a.e.\ $y\in Y$ there is a (ergodic) $g^k$-invariant Borel probability measure $\nu^y$ (the $g^k$-ergodic component of $p_*\mu$ containing $y$) on $Y$ such that $\phi\in L^1(\nu^y)$ and for every rational $\delta'>0$,
\begin{align*}
\frac 1 n 	\sum_{\substack{0\le j\le n-1 \\g^{jk}(y)\notin G(\delta')}} & \left(\log( L^{2m}C_{M,\epsilon,\delta}) + m \log\left( \phi(g^{jk}(y))\right)\right) \
\\ & \to 
 \int_{Y\sm G(\delta')} \log( L^{2m}C_{M,\epsilon,\delta}) + m \log\left( \phi(\cdot )\right) d  \nu^y (\cdot). 
\end{align*}

Given $(n-1)k< \ell \le  nk$, let $n_-(\ell) = n-1$ and $n_+(\ell) = n$.  Then, 
\begin{align*}
k \limsup_{\ell\to \infty}  & \frac 1 \ell \log\left( \sup_{z\in p\inv (y)}\wtd r(z, \ell,\delta, \epsilon;F)\right)\\ 
&\le 	 k\limsup_{\ell\to \infty}  \frac 1 {k n_-(\ell)} \log \left(\sup_{z\in p\inv (y)}\wtd r(z, \ell,\delta, \epsilon;F)\right)\\ 
&\le 	 \limsup_{\ell\to \infty}  \frac 1 {n_-(\ell)} \log\left( \sup_{z\in p\inv (y)}\wtd r(z, n_+(\ell),\delta', \epsilon;F^k)\right)\\ 
&\quad \quad  + \int_{Y\sm G(\delta')} \log( L^{2m}C_{M,\epsilon,\delta}) + m \log\left( \phi(\cdot )\right) d  \nu^y (\cdot)\\
	&= 	\limsup_{n\to \infty} \frac 1 {n-1} \log\left( \sup_{z\in p\inv (y)}\wtd r(z, n,\delta', \epsilon;F^k)\right) \\&\quad \quad  
	\int_{Y\sm G(\delta')} \log( L^{2m}C_{M,\epsilon,\delta}) + m \log\left( \phi(\cdot )\right) d  \nu^y (\cdot). \end{align*}
Taking first $\delta'\to 0$  and then $\delta\to 0$ in the rationals, we obtain $$  k \overline h_{y,\loc}(\epsilon, F\mid \fol)\le \overline h_{y,\loc}(\epsilon, F^k\mid \fol)$$ as desired.  
\end{proof}

\subsection{Sufficient conditions for upper semicontinuity of fiber entropy}\label{unifint}

The main results of  this appendix are the following results, \cref{lem:locentzeromeansUSC} and  \cref{prop:Yomdin}, which provide sufficient criteria for upper semicontinuity of fiber entropy.

For our first result, assuming the local fiber entropy  $h_{\text{$\scrM$-$\loc$}}(F\mid \fol)$ vanishes, we obtain upper semicontinuity of the fiber entropy.

\begin{lemma}\label[lemma]{lem:locentzeromeansUSC}
Let $\scrM$ be a collection of $F$-invariant Borel probability measures on $Z$ such that $h_{\text{$\scrM$-$\loc$}}(F\mid \fol)=0.$ 
Then the function $\nu\mapsto h_\nu(F\mid \fol)$ is upper semicontinuous when restricted to $\scrM$. 
\end{lemma}

Given  a collection $\scrM$ of Borel probability measures  on $Z$, a Borel function $\phi\colon Z\to \R$ is \emph{uniformly integrable with respect to $\scrM$} if  
$$\lim _{K\to +\infty} \sup_{\mu\in \scrM} \int_{|\phi(z)|\ge K} |\phi(z)| \, d\mu(z) = 0.$$ 

Given $y\in Y$, let $$\Lambda(y) = \limsup_{n\to \infty} \frac 1 n \log (\|f_y^{(n)}\| _{C^1,\ast})= \limsup_{n\to \infty} \frac 1 n \log \left(\sup_{x\in M}\|D_xf_y^{(n)}\|\right).$$
Also write $R_{y,k}:= \|f_y\|_{C^k,\ast}.$  Observe that $R_{y,k}\ge R_{y,1}\ge 1$ for every $y\in Y$.

Our second main result is the following upper bound for the local fiber entropy  $h_{\text{$\scrM$-$\loc$}}(F\mid \fol)$; assuming $r=\infty$, this provides a sufficient condition for vanishing of local fiber entropy.     
\begin{theorem}\label{prop:Yomdin}
Assume  $r\ge k$ and let $\scrM$ be a collection of $F$-invariant, Borel probability measures on $Z$ such that the function $y\mapsto \log R_{y,k}$ is uniformly integrable with respect to $p_*\scrM=\{p_*\mu: \mu\in \scrM\}$. 

Then
$$h_{\text{$\scrM$-$\loc$}}(F\mid \fol)   \le \frac {\dim M}{k} 
\sup _{\mu\in \scrM} \int  \Lambda(y) \, d p_*\mu(y) .$$ 
\end{theorem}

\subsection{Proof of \texorpdfstring{\cref{lem:locentzeromeansUSC}}{Lemma A.4}}
We have the following version of \cite[Thm.\ 1, (1.1)]{MR986792}.  
 \begin{proposition}\label[proposition]{SheldonIsClever} 
Fix $\epsilon>0$ and let $\calP$ be a finite partition
of $Z$ such that for every $z\in Z$,
 $$\diam (\calP(z)\cap \fol(z))<\epsilon.$$
  Then  for any $F$-invariant Borel probability measure $\mu$, 
$$h_\mu(F\mid \fol) \le  h_\mu(F,\calP\mid \fol) +  \int \overline h_{y,\loc}(F,\epsilon\mid \fol) \, d (p_* \mu)(y).$$

\end{proposition}

 \cref{lem:locentzeromeansUSC} follows immediately from \cref{SheldonIsClever}.  
\begin{proof}[Proof of \cref{lem:locentzeromeansUSC}]
We follow \cite[Proof of Lem.\ 2, (2.1)]{MR986792}.  Fix $\mu\in \scrM$.  Fix $\eta>0$ and $\epsilon>0$ such that $$\sup _{\nu\in \scrM} \int \overline h_{y,\loc}(F,\epsilon\mid \fol) \, d (p_* \nu)(y)<\eta.$$

Fix a finite partition $\beta$ of $M$ with 
\begin{enumerate}
\item $\diam \beta \le \dfrac {\epsilon}{L_\mu}$; 
\item $\mu(\partial \wtd \beta _\mu) =0$.

\end{enumerate}

From \cref{SheldonIsClever}  and the upper semicontinuity of  $\nu\mapsto h_\nu(F,\wtd \beta_\mu\mid \fol)$ at $\nu=\mu$ in \cref{prop:entfacts}, 
we have $$\limsup_{\nu\to \mu} h_\nu (F\mid \fol) \le \limsup_{\nu\to \mu}  h_\nu(F,\wtd \beta_\mu\mid \fol) +\eta \le 
 h_\mu(F,\wtd \beta_\mu\mid \fol) +\eta 
\le h_\mu(F\mid \fol) +\eta. $$ The conclusion  follows immediately from arbitrariness of  $\eta$.
\end{proof}

It remains to establish \cref{SheldonIsClever}.  We follow the proof of \cite[Thm.\ 1, (1.1)]{MR986792}, especially as corrected in \cite{MR1043272}.

\begin{proof}[Proof of \cref{SheldonIsClever}]   
Let $\calP$ be a finite partition of $Z$ with 
\begin{enumerate}
\item  $\diam (\calP(z)\cap \fol(z))<\epsilon$ for every $z\in Z$.  
\end{enumerate}
Fix $N\in \N$.  
Let  $\alpha=\{K_1,\dots, K_\ell, K_{\ell+1}\} $ be a finite partition of $M$ such that 
\begin{enumerate}[resume]
\item   $K_i$ is compact for each $1\le i\le \ell$ and  $K_{\ell+1} = M\sm \bigcup_{i=1}^\ell K_i$;
\item $h_\mu(F^N\mid \fol) \le  h_\mu(F^N,\wtd \alpha_\mu\mid \fol) +13$ (following the notation in \eqref{eq:pushpart}).
  \end{enumerate} 
 For $1\le i\le \ell+1$  let $Q_i=I(Y\times K_i)$  and $\calQ$ be the finite partition $\calQ=\{Q_i: 1\le i\le \ell+1\}$.  
  Fix $$0<\delta<\frac 1 {2 L} \min \{d_M(x,y) : x\in K_i, y\in K_j, 1\le i,j\le \ell, i\neq j\}$$ sufficiently small.  
Then for every $y\in Y$,  \begin{equation}\label{eq:meatloaf}
\delta <\frac 1 2\min \left\{d_Z(z,z') : z\in Q_i\cap \fol(y), z'\in Q_j\cap \fol(y), 1\le i,j\le \ell, i\neq j\right\}.\end{equation}

  Given $k\in \N$, let $n=kN$.  We have 
\begin{align*}
 H_\mu\bigl(\calQ^k_{F^N} \mid \fol \bigr)
&\le   H_\mu\bigl(\calQ^k_{F^N} \vee \calP^n_{F}  \mid \fol \bigr)\\
&=
H_\mu\bigl(\calP^n_{F} \mid \fol \bigr) +  
H_\mu\bigl(\calQ^k_{F^N} \mid \calP^n_{F} \vee \fol \bigr).
 \end{align*}

Fix $y\in Y$ and $z\in p\inv (y)$.  
We have $\bigr(\calP^n_{F} \vee \fol \bigr)(z)\subset B^\Fol_n(z,\epsilon;F)$. 
 Fix $\eta>0$.  Having taken first    $\delta>0$  sufficiently small and then $n=Nk$  sufficiently large,  we may find a $(\delta,n;F)$-cover $\calA_n$ of 
$B^\Fol_n(z,\epsilon;F)$
with 
$$\card \calA_n\le e^{n(\overline h_{y,\loc}(F,\epsilon\mid \fol) +\eta)}.$$

Given $E\in \calA_n$ and $0\le j\le k-1$, by \eqref{eq:meatloaf} we have that $F^{jN}(E)$ meets at most one of the sets in $\left\{Q_1\cap \fol(g^{jN}(y)), \dots, Q_\ell\cap \fol(g^{jN}(y))\right\}$ as well as possibly meeting the set $Q_{\ell+1}\cap \fol(g^{jN}(y))$.

In particular,  each set $E\in \calA_n$ meets at most $2^k$ sets in \[ \restrict{\calQ^k_{F^N}}{ \bigr((\calP^n_F \vee \fol )(z)\bigr)}.\]

Thus 
\begin{align*}
H_\mu\bigl(\calQ^k_{F^N} \mid \calP^n_{F} \vee \fol \bigr)
&\le \int \log  \card \left(\restrict{\calQ^k_{F^N}}{ \bigr((\calP^n_{F} \vee \fol)(z) \bigr)}\right) \, d \mu(z)\\
&\le \int \log \left (2^k e^{n(\overline h_{y,\loc}(F,\epsilon\mid \fol) +\eta)}\right) \, d (p_*\mu)(y)\\
&\le k\log 2 + n \eta + n \int \overline h_{y,\loc}(F,\epsilon\mid \fol)  \, d (p_*\mu)(y).
\end{align*}
Hence ,
\begin{align*}
h_\mu(F\mid \fol)
&= \frac 1 N h_\mu(F^N\mid \fol)\\
& \le \frac 1 N   h_\mu(F^N,\calQ\mid \fol) +\frac {13} N\\
& =\frac {13} N +\frac 1 N \lim  _{k\to\infty } \frac 1 {k} H_\mu\bigl(\calQ^k_{F^N} \mid \fol \bigr) \\
& \le \frac {13} N +\lim  _{k\to\infty } 
\frac 1 {Nk} H_\mu\bigl(\calP^{Nk}_{F} \mid \fol \bigr) 
+\lim  _{k\to\infty }  \frac 1 {Nk} H_\mu\bigl(\calQ^k_{F^N} \mid \calP^{Nk}_{F} \vee \fol \bigr)\\
&= \frac {13} N +\lim  _{n\to\infty } 
\frac 1 {n} H_\mu\bigl(\calP^{n}_{F} \mid \fol \bigr) 
+\lim  _{k\to\infty }  \frac 1 {Nk} H_\mu\bigl(\calQ^k_{F^N} \mid \calP^{Nk}_{F} \vee \fol \bigr)\\
&\le \frac {13} N 
+  h_\mu(F,\calP\mid \fol) + \frac {\log 2}{N} +\eta +  \int \overline h_{y,\loc}(F,\epsilon\mid \fol) \, d (p_*\mu)(y) .
\end{align*}

Taking $N$ sufficiently large and $\eta>0$ sufficiently small, we obtain the desired inequality.  
\end{proof}

\subsection{Yomdin estimates} 
It remains to prove \cref{prop:Yomdin}.  This follows from the following key proposition of Yomdin.  See  \cite[Thm.\ 2.1]{MR0889979} (as well as \cite{MR889980}) and mild reformulation in \cite[Prop.\ 3.3]{MR1469107}.

We write $\bfQ^\ell=[0,1]^\ell$  for the $\ell$-dimensional unit cube in $\R^m$ (oriented along the first $\ell$ basis vectors relative to the standard basis on $\R^m$).   
Given $x\in \R^N$ we write $B(x,r)$ for the Euclidean ball centered at $x$ of radius $r>0$.  
\begin{proposition}\label[proposition]{Prop:YomdinGromov}
Fix $x\in \R^N$ and 
let $\sigma\colon \bfQ^\ell  \to \R^N $ and $f\colon B(x,2)\to \Rbb^N$ be $C^k$ functions with $\im (\sigma)\subset B(x,2)$, $\|\sigma\|_{C^k,\ast}\le 1$,   and $\|f\|_{C^k,\ast}\le R$.

There exists a constant $u=u(k,N,\ell)$ depending only on $k$, $N$, and $\ell$ 
and  at most $\kappa:=u\max\{R, 1\}^{\frac {\ell}{k}}$  maps $\psi_{i}\colon \bfQ^\ell\to \bfQ^\ell$ with the following properties: 
\begin{enumerate}
\item \label{contract} Each $\psi_{i}\colon \bfQ^\ell\to \bfQ^\ell$ is a $C^k$ diffeomorphism onto its image.  Moreover, $\|\psi_i\|_{C^1}\le 1$.  
\item $\sigma\inv\circ f \inv \bigl(B(f(x),1)\bigr )\subset \bigcup \im (\psi_i)$.  
\item for every $i,$ 
$ \im (f\circ \sigma \circ \psi_i)\subset    B(f(x),2)$.
\item  $\| f\circ \sigma\circ \psi_i\|_{C^k,\ast}\le 1$ for every $i$.
\end{enumerate}

\end{proposition}
The assertion that each $\psi_i$ can be taken to be a (non-strict) contraction in \eqref{contract} is not explicitly stated in \cite[Thm.\ 2.1]{MR0889979} but the estimate $\|\psi_i\|_{C^k,\ast}\le 1$ can be deduced from the proof.  See also the reformulation in \cite[Prop.\ 3.3]{MR1469107} where it is asserted explicitly that $\psi_i$ are (non-strict) contractions. 

\subsection{Proof of \texorpdfstring{\cref{prop:Yomdin}}{Theorem A.5}}  We roughly follow the idea of proof of  \cite[Thm.\ 1.8]{MR889980} and  \cite[Thm.\ 2.2]{MR1469107}.  

Let $m= \dim M$.  Recall we view $M\subset U\subset \R^N$.   
Fiberwise local entropy \eqref{eq:locent} is invariant under bi-Lipschitz change of coordinates and we thus equip $M$ with the Riemannian metric obtained by restricting the Euclidean metric to $M$.  
Recall we write $B(x,\rho)$ for the Euclidean ball in $\R^N$ centered at $x$ of radius $\rho$.  Given $x\in M$, let $B_M(x, \rho)$ denote the  ball  in $M$ with respect to the induced Riemannian metric centered at $x$ of radius $\rho$.  For  $\rho>0$, note $B_M(x, \rho)\subset B(x, \rho)$.  

Fix the constant $u=u(k,N,m)$ as in \cref{Prop:YomdinGromov}.  
Recall that given $y\in Y$ we write $$R_{y,k}:= \|f_y\|_{C^k,\ast}\ge 1.$$

  Given $\lambda>0$, let $d^\lambda \colon \R^N\to \R^N$ denote dilation by $\lambda>0$; that is, $d^\lambda(v) = \lambda v$.  
  
  Fix $0<\epsilon<1$.  Let $M^\epsilon = d^{1/\epsilon} (M)= \{\epsilon \inv x: x\in M\}$ and $U^\epsilon = d^{1/\epsilon} (U)$ be the rescaled sets.  Given $y\in Y$, let $f_{y,\epsilon}= d^{1/\epsilon}\circ f_y\circ d^{\epsilon}\colon M^\epsilon \to M^\epsilon$; that is, $f_{y,\epsilon} \colon v\mapsto \epsilon \inv f_y (\epsilon v).$
 For $1\le s \le k$ we have $$\|D^s f_{y,\epsilon}\| \le \epsilon ^{s-1} \|D^sf_{y} \|.$$
In particular, $\|f_{y,\epsilon}\|_{C^k,\ast}\le R_{y,k}$; moreover, if $R_{y, k} \le \epsilon \inv$ then $$\|f_{y,\epsilon}\|_{C^k,\ast}\le \|f_{y,\epsilon}\|_{C^1,\ast}= R_{y,1}.$$
Also write $f_{y,\epsilon}^{(0)}=\id$ and for $n\ge 1$ 
$$f_{y,\epsilon}^{(n)}= f_{g^{n-1}(y),\epsilon} \circ \dots \circ f_{y,\epsilon}.$$

For $x\in M$, fix an (ordered) orthonormal basis for $T_xM$.  Let $\bfQ_x= [-\frac 1 2, \frac 1 2]^m$ denote the cube in $T_xM$ of side-length $1$ centered at $0$ relative to this basis.  Let $I_x\colon \bfQ^m\to \bfQ_x$ denote the affine isometry defined relative to this basis.  

Let $\exp_x$ denote the exponential map of $M$ at $x$.  
Fixing $0<\lambda<1$ sufficiently small,  let $\sigma_x\colon \bfQ^m\to M$ be the map $$\sigma_x\colon v\mapsto \exp_x (\lambda I_x (v)).$$
Observe the image of $\sigma_x$ contains $B_M(x, \lambda)$ and is contained in $B_M(x, \lambda\sqrt m)$.  
By compactness there  exists $0<\lambda<\frac 1 {\sqrt m}$ so that $\|\sigma_x\|_{C^k,\ast}\le 1$ for all $x\in M$.

Given $0<\epsilon<1$, we let $\sigma_{x,\epsilon}\colon \bfQ^m\to M^\epsilon$ be $$\sigma_{x,\epsilon}\colon  v\mapsto \epsilon \inv \exp_x (\epsilon \lambda I_x (v)).$$
For $1\le s\le k$ we have $\|D^s \sigma_{x,\epsilon}\|\le\epsilon^{s-1}\|D^s \sigma_{x}\|$
whence $\|\sigma_{x,\epsilon}\|_{C^k,\ast}\le 1$ for all $x\in M$ and $0<\epsilon<1$.
We also have $$ B_{M^\epsilon}(x,\lambda)\subset \im \sigma_{x,\epsilon}\subset B_{M^\epsilon}(d^{1/\epsilon }(  x), \lambda\sqrt m)\subset B(d^{1/\epsilon }(  x), 1).$$

Write 
$$\kappa_{y,\epsilon}:= \begin{cases} u(k,N,m)(R_{y,1})^{\frac {m}{k}} & R_{y,k} \le \epsilon\inv\\
u(k,N,m)(R_{y,k})^{\frac {m}{k}} & R_{y,k} \ge \epsilon\inv.
\end{cases}$$

Fix $y_0\in Y$ and write $y_j = g^j(y)$ for $j\ge 0$.  
Fix $x_0 \in M$ and write  $x_j = f_{y_0}^{(j)}(x_0)$. 
Given $n\ge 1$, set 
\begin{equation}\label{eq:SN}S_n := \{x'\in  {M^\epsilon}:  f^{(j)}_{y_0 , \epsilon} (x') \subset B(d^{1/\epsilon}(x_j),1) \text { for all $0\le j\le n-1$}\}.\end{equation}

Recursive application of \cref{Prop:YomdinGromov} with the maps $f_{y_0,\epsilon}$, $f_{y_1,\epsilon}$, $f_{y_2,\epsilon}$, $\dots$ and the map $\sigma= \sigma_{x_0,\epsilon}$ yields the following.  

\begin{lemma}\label[lemma]{lem:abstyomdin}

Fix $0<\epsilon<1 $.  
For every $n\in \N$, there exist at most  $\prod_{j=0}^{n-1} \kappa_{y_{j},\epsilon}$ maps $\psi_i\colon \bfQ^m\to \bfQ^m$ with the following properties:
\begin{enumerate}
\item Each $\psi_i\colon \bfQ^m\to \bfQ^m$ is a $C^k$ diffeomorphism onto its image.   Moreover, $\|\psi_i\|_{C^1}\le 1$.  
\item $ \sigma_{x_0,\epsilon}\inv (S_n)\subset\bigcup \im ( \psi_i)$.
\item $\im\bigl(f_{y_0,\epsilon}^{(n)}\circ \sigma_{x_0,\epsilon}\circ \psi_i\bigr) \subset B(d^{1/\epsilon}(x_n, 2)$. 
\item $\|f_{y_0,\epsilon}^{(n)}\circ \sigma_{x_0,\epsilon}\circ \psi_i\|_{C^k,\ast} \le 1$ and $\|f_{y_0,\epsilon}^{(j)}\circ \sigma_{x_0,\epsilon}\circ \psi_i\|_{C^1} \le 1$ for each $1\le j\le n$.  
\end{enumerate}
\end{lemma}

\begin{proof}
When $n=1$, this is \cref{Prop:YomdinGromov}.  
Suppose the result is true for $n=\ell$.
Let $\psi_i$ be the maps guaranteed by the inductive hypothesis.  For a fixed $i$, let $$\sigma_i:= f_{y_0,\epsilon}^{(\ell)}\circ \sigma_{x_0,\epsilon}\circ \psi_i
=f_{y_{\ell-1},\epsilon}\circ \dots\circ f_{y_{0},\epsilon}\circ \sigma_{x_0,\epsilon}\circ \psi_i.$$
Applying \cref{Prop:YomdinGromov}, for each $i$  there exists at most $\kappa_{y_\ell,\epsilon}$ 
maps $\psi_{i,j}\colon \bfQ^m\to \bfQ^m$ such that 
\begin{enumerate}
\item Each $\psi_{i,j}\colon \bfQ^m\to \bfQ^m$ is a $C^k$ diffeomorphism onto its image and  $\|\psi_{i,j}\|_{C^1}\le 1$;
\item $\sigma_i\inv\circ f_{y_\ell,\epsilon} \inv \bigl(B(d^{1/\epsilon} (x_{\ell}),1)\bigr )\subset \bigcup_{j} \im (\psi_{i,j})$;
\item for every $j,$ 
$ \im (f_{y_\ell,\epsilon} \circ \sigma_i \circ \psi_{i,j})\subset   B(d^{1/\epsilon}(x_\ell), 2)$;
\item $\| f_{y_\ell,\epsilon} \circ \sigma_i \circ \psi_{i,j} \|_{C^k,\ast}\le 1$ for every $j$.
\end{enumerate}

We have $$S_{\ell+1}\subset \left(f^{(\ell)}_{y_0 , \epsilon}\right) ^{-1}B(d^{1/\epsilon}(x_{\ell}),1)\cap S_\ell.$$  By the inductive hypothesis we have
$$\sigma_{x_0,\epsilon}\inv (S_{\ell+1})\subset \bigcup_{i} \im (\psi_{i}).$$
Moreover, for each $i$ our application of \cref{Prop:YomdinGromov} gives
$$ \psi_i\inv \circ\sigma_{x_0,\epsilon}\inv (S_{\ell+1})\subset \bigcup_{j} \im (\psi_{i,j}).
$$
Thus $$\sigma_{x_0,\epsilon}\inv (S_{\ell+1})\subset \bigcup_{i,j} \im (\psi_i \circ\psi_{i,j}).$$
Additionally, for $1\le j\le n-1$, 
\begin{align*}
\|f_{y,\epsilon}^{(j)}\circ &\sigma_{x_0,\epsilon}\circ\psi_i \circ\psi_{i,j}\|_{C^1} \\
&\le \|f_{y,\epsilon}^{(j)}\circ \sigma_{x_0,\epsilon}\circ\psi_i\|_{C^1}  \|\psi_{i,j}\|_{C^1} \\
&\le 1.
\end{align*}

The collection of maps $\{\psi_i \circ\psi_{i,j}\}_{i,j}$ thus satisfy the requirements of the lemma for $n=\ell+1$  
and by the inductive hypothesis, the collection has at most $\prod_{j=0}^{\ell} \kappa_{y_{j},\epsilon}$ maps.  
\end{proof}

Let $L\ge 1$ be the bi-Lipschitz constant in \eqref{eq:jasminepearls}.  
\cref{lem:abstyomdin}   immediately implies the following.
\begin{corollary}\label[corollary]{cor:emtpyshell}
Fix a Borel probability measure $\mu$ on $Z$.  There exists $C_m$ depending only on $m$ such that for every $y_0\in Y$, $x_0\in M$, $\epsilon>0$, and  $0<\delta<1$, setting $z_0 = I(y_0,x_0)$ we have 
$$\cov \left(\delta, n, B^\Fol_n\left(z_0,\tfrac{\lambda}{L}\epsilon;F\right);F\right)\le  C_m\delta^{-m} L^m \prod_{j=0}^{n-1} \kappa_{y_{j},\epsilon}.$$
\end{corollary}
\begin{proof}
With $S_n$ as in \eqref{eq:SN}, we have that $$B^\Fol_n\left(z_0,\tfrac{\lambda}{L}\epsilon;F\right) \subset I_{ y_0}(d^{\epsilon}( S_n\cap \Im (\sigma_{x_0,\epsilon})
)).$$
There exists a simplicial $(\delta L\inv)$-cover $\calA$ of $\bfQ^m$ with $\card \calA\le C_m (\delta L\inv)^{-m}$ where $C_m$ depends only on $m$.

Retain all notation from \cref{lem:abstyomdin}.
Since $\|f_{y_0,\epsilon}^{(j)}\circ \sigma_{x_0,\epsilon}\circ \psi_i\|_{C^1} \le 1$, for each $0\le j\le n-1$, it follows that $$
\bigcup _{A\in \calA}   \sigma_{x_0,\epsilon}\circ \psi_i(A)$$
is a $(\delta L\inv,n;F)$-cover of $ S_n \cap \Im (\sigma_{x_0,\epsilon})$.
Hence 
 $$
 \bigcup _{A\in \calA}  I_{ y_0}  \circ  d^\epsilon \circ \sigma_{x_0,\epsilon}\circ \psi_i(A)$$
is a $(  \delta,n;F)$-cover of $B^\Fol_n\left(z_0,\frac{\lambda}{L}\epsilon;F\right)$.  
In particular, \[\cov \left(\delta, n, B^\Fol_n\left (z_0,\tfrac{\lambda}{L}\epsilon;F\right);F\right)\le  \card \calA \prod_{j=0}^{n-1} \kappa_{y_{j},\epsilon}.\qedhere\]
\end{proof}

Assembling the above, we conclude  \cref{prop:Yomdin}.  
\begin{proof}[Proof of \cref{prop:Yomdin}]
Fix $\ell>1$ and let $$B(\ell )=\{ y \in Y: R_{y,k} \ge \ell\}.$$
Let $\mu$ be an $F$-invariant, Borel probability measure on $Z$.  
Consider any $\wtd \epsilon\le \frac {\lambda} {L \ell}$ and set $\epsilon =\frac 1 \ell$.   
By the ergodic theorem and \cref{cor:emtpyshell},
\begin{align*}\int \overline h^\fol_{y,\loc}&(\wtd \epsilon) \, d (p_*\mu)(y) 
\\&\le  \int  \lim_{\delta\to 0} \limsup_{n\to \infty} \frac 1 n \log \left( C_m\delta^{-m} (L)^m \prod_{j=0}^{n-1} \kappa_{g^j(y),\epsilon}\right) d (p_*\mu)(y) 
\\
&\le  \int  \limsup_{n\to \infty} \frac 1 n \log \left(  \prod_{j=0}^{n-1} \kappa_{g^j(y),\epsilon}\right) d (p_*\mu)(y) 
\\
&=  \int  \lim_{n\to \infty} \frac 1 n \sum _{j=0}^{n-1} \log\left( \kappa_{g^j(y),\epsilon}\right) d (p_*\mu)(y) 
\\ &\le\log u(k,N,m) +\frac m k \left [ \int _{Y\sm B(\ell) } \log R_{y,1} \, d (p_*\mu)(y) +  \int _{B(\ell) } \log R_{y,k} \, d (p_*\mu)(y) \right].\end{align*}

From uniform integrability of $y\mapsto \log R_{y,k}$ over $\mu\in \scrM$, taking $\ell\to +\infty$ we obtain 
\begin{equation}\label{eq:onestep}h^\fol_{\text{$\scrM$-$\loc$}}(F) \le  \log u(k,N,m) + \sup _{\mu\in \scrM}\frac{m}{k}\int _{ Y}\log R_{y,1} \, d (p_*\mu)(y)\end{equation}
Replacing $F$ in  \eqref{eq:onestep} with the iterated dynamics $F^j$ (recall 
$u (k, N,m)$ in \cref{Prop:YomdinGromov} depends only on ambient dimension and degree of regularity but not on $\|f_y\|_{C^k,\ast}$) and applying
 \cref{claim:submult}, for every $j\ge 1$ we have 
\begin{align*}
h^\fol_{\text{$\scrM$-$\loc$}}(F) &\le \frac 1 j h^\fol_{\text{$\scrM$-$\loc$}}(F^j)\\
&\le \frac 1 j\left[ \log u(k,N,m) + \frac m k \sup _{\mu\in \scrM} \int _{Y } \log(\|f_y^{(j)}\|_{C^1,\ast,}) \, d (p_*\mu)(y)\right].
\end{align*}
Taking $j\to \infty$, we obtain  the desired upper bound
\[h^\fol_{\text{$\scrM$-$\loc$}}(F) \le \frac m k \sup _{\mu\in \scrM} \int _{ Y }  \Lambda(y) \, d (p_*\mu)(y). \qedhere\]
\end{proof}

\BibSpec{inbook}{%
  +{}  {\PrintAuthors}                {author}
  +{,} { \textit}                     {title}
  +{.} { }                            {part}
  +{:} { \textit}                     {subtitle}
  +{,} { \PrintContributions}         {contribution}
  +{,} { \PrintConference}            {conference}
  +{}  {\PrintBook}                   {book}
  +{. In} { \textit}                            {booktitle}
  +{,} { }                            {series}
  +{,} { }                            {publisher}
  +{,} { \PrintDateB}                 {date}
  +{,} { pp.~}                        {pages}
  +{,} { }                            {status}
  +{,} { available at \eprint}        {eprint}
  +{}  { \parenthesize}               {language}
  +{}  { \PrintTranslation}           {translation}
  +{;} { \PrintReprint}               {reprint}
  +{.} { }                            {note}
  +{.} {}                             {transition}
  +{}  {\SentenceSpace \PrintReviews} {review}
}



\end{document}